\documentclass{msc}

\usepackage{amsmath}
\usepackage{amsfonts} % Packages for mathematics
\usepackage{wasysym}
\usepackage{amsthm}
\usepackage{graphicx}
\usepackage{array}

\usepackage{enumerate,xspace}

\usepackage[all]{xy}
\usepackage{proof}
\usepackage{tikz}

\usepackage{stmaryrd} %need for special interpretation bracket
\usepackage{mathtools}
\usepackage{latexsym}
\usepackage{dsfont}
\usepackage{multicol}
\usepackage{cmll}
\usepackage{multirow}
\usepackage{longtable}

\usepackage{tikz}
\usetikzlibrary{tikzmark}
\usetikzlibrary{shapes.geometric, arrows}
\usepackage{tikz-cd}
\usetikzlibrary{cd}
\usepackage{framed,color}
\usepackage{smartdiagram}
\usesmartdiagramlibrary{additions}
\usetikzlibrary{positioning,arrows}
\usetikzlibrary{shapes}
\usetikzlibrary{shadows.blur}
\usetikzlibrary{shapes.symbols}
\usetikzlibrary{arrows.meta,calc}
\usetikzlibrary{shapes,%snakes
}
\pgfdeclarelayer{edgelayer}
\pgfdeclarelayer{nodelayer}
\pgfsetlayers{edgelayer,nodelayer,main}

\definecolor{hexcolor0xff0000}{rgb}{1.000,0.000,0.000}
\definecolor{hexcolor0x000000}{rgb}{0.000,0.000,0.000}
\definecolor{hexcolor0x00ff00}{rgb}{0.000,1.000,0.000}
\definecolor{hexcolor0xffff00}{rgb}{1.000,1.000,0.000}
\definecolor{hexcolor0x000000}{rgb}{0.000,0.000,0.000}
\definecolor{hexcolor0x000000}{rgb}{0.000,0.000,0.000}
\definecolor{white}{rgb}{1.000,1.000,1.000}

\tikzstyle{none}=[inner sep=0pt]
\tikzstyle{port}=[inner sep=0pt]
\tikzstyle{component}=[circle,fill=white,draw=black, inner sep=2.5pt]
\tikzstyle{integral}=[inner sep=0pt]
\tikzstyle{differential}=[inner sep=0pt]
\tikzstyle{codifferential}=[inner sep=0pt]
\tikzstyle{function}=[regular polygon,regular polygon sides=4,fill=white,draw=black]
\tikzstyle{function2}=[regular polygon,regular polygon sides=4,fill=white,draw=black, inner sep=2pt]
\tikzstyle{function3}=[regular polygon,regular polygon sides=4,fill=white,draw=black, inner sep=-2pt]
\tikzstyle{duplicate}=[circle,fill=white,draw=black, inner sep=1pt]
\tikzstyle{wire}=[-,draw=black,line width=1.000]
\tikzstyle{object}=[inner sep=2pt]
\tikzset{every picture/.append style={scale=.65}, transform shape}

\usepackage{lscape}
\usepackage{array}

\newtheorem{example}[therm]{Example}

\begin{document}

\lefttitle{Monoidal Reverse Differential Categories}
\righttitle{Mathematical Structures in Computer}

\papertitle{Article}

\jnlPage{1}{00}
\jnlDoiYr{2019}
\doival{10.1017/xxxxx}

\title{Monoidal Reverse Differential Categories}

\begin{authgrp}
\author{Geoff Cruttwell}
\affiliation{Mount Allison University, Department of Mathematics and Computer Science\\
        \email{gcruttwell@mta.ca}} 
        \author{Jonathan Gallagher}
\affiliation{HRL Laboratories Center for Secure and Resilient Systems\\
        \email{jonathan@infinitylab.io}} 
                \author{Jean-Simon Pacaud Lemay}
\affiliation{Mount Allison University, Department of Mathematics and Computer Science\\
        \email{jsplemay@gmail.com}} 
                               \author{ Dorette Pronk}
\affiliation{Dalhousie University, Department of Mathematics and Statistics\\
        \email{Dorette.Pronk@dal.ca}} 
\end{authgrp}

\history{(Received xx xxx xxx; revised xx xxx xxx; accepted xx xxx xxx)}
%\received{20 March 1995; revised 30 September 1998}

\begin{abstract}
Cartesian reverse differential categories (CRDCs) are a recently defined structure which categorically model the reverse differentiation operations used in supervised learning.  Here we define a related structure called a \emph{monoidal reverse differential category}, prove important results about its relationship to CRDCs, and provide examples of both structures, including examples coming from models of quantum computation.  
\end{abstract}


\maketitle

\section{Introduction}

To handle notions of differentiation that have become more prominent in computer science, two categorical structures have been useful: monoidal differential categories \cite[]{blute2006differential} and Cartesian differential categories \cite[]{blute2009cartesian}.  Each axiomatizes a different aspect of differentiation: monoidal differential categories axiomatize the linear maps and then derive the smooth maps from them; conversely, Cartesian differential categories axiomatize the smooth maps and derive the linear maps from them. While these structures have been very useful, they both only represent the ``forward'' aspect of differentiation. For uses of the derivative in supervised learning, the ``reverse'' derivative is more relevant.  

To understand the difference between forward and reverse differentiation, let us provide a simple example. Consider the smooth map $f: \mathbb{R}^2 \to \mathbb{R}$ defined by $f(x_1,x_2) = x_1^2x_2 + \sin(x_2)$. The Jacobian matrix of $f$, at $(x_1, x_2)$, is the $1 \times 2$ matrix whose components are the partial derivatives of $f$: 
\[ \mathbf{J}_f(x_1,x_2) := \begin{bmatrix} 2x_1x_2 & x_1^2 + \cos(x_2) \end{bmatrix} \]
The directional (forward) derivative of $f$ is the map $\mathsf{D}[f]: \mathbb{R}^2 \times \mathbb{R}^2 \to \mathbb{R}$ given by multiplying the Jacobian matrix of $f$ at the first input vector by the second input vector (seen as a $2 \times 1$ matrix): 
	\[ \mathsf{D}[f]\left ((x_1, x_2), (v_1, v_2) \right) = \mathbf{J}_f(x_1,x_2) \begin{bmatrix} v_1 \\
	v_2\end{bmatrix} =  2x_1x_2v_1 + (x_1^2 + \cos(x_2))v_2 \]
Note that this ``pushes vectors forwards'': at a point of $\mathbb{R}^2$, the directional derivative $\mathsf{D}[f]$ takes a vector in $\mathbb{R}^2$ to a vector in $\mathbb{R}$, that is, vectors are moved in the same direction as the map $f$ itself. 

Conversely, the reverse derivative moves vectors in the opposite direction.  The reverse derivative uses the \emph{transpose} of the Jacobian of $f$ at $(x_1, x_2)$, which is the $2 \times 1$ matrix: 
\[ \mathbf{J}^T_f(x_1,x_2) := \begin{bmatrix} 2x_1x_2 \\ x_1^2 + \cos(x_2) \end{bmatrix} \]
Then the reverse derivative of $f$ is $\mathsf{R}[f]: \mathbb{R}^2 \times \mathbb{R} \to \mathbb{R}^2$ defined by multiplying the transpose of the Jacobian $f$ at the first input vector by the second input vector (this time seen as a $1 \times 1$ matrix):
	\[ \mathsf{R}[f]\left ((x_1, x_2), t \right) = \mathbf{J}^T_f(x_1,x_2) t =  \begin{bmatrix} 2x_1x_2t \\ \left( x_1^2 + \cos(x_2) \right) t \end{bmatrix}  \]
Thus this operation indeed moves vectors in the opposite direction as $f$; that is,  it takes vectors from the codomain of $f$ and returns vectors in the domain. The reverse derivative is better suited for supervised learning situations, in which one knows a change in the codomain (e.g., the error of some function) and wants to know how much adjustment to make in the domain.    

Thus, a natural question is how to modify monoidal and Cartesian differential categories to handle reverse differentiation. For the Cartesian side of the picture, this was already accomplished in \cite[]{cockett_et_al:LIPIcs:2020:11661}.  While a Cartesian differential category (CDC) involves a category which comes equipped with an operator $\mathsf{D}$ which for any map $f: A \to B$ outputs a map $\mathsf{D}[f]: A \times A \to B$, a Cartesian \emph{reverse} differential category (CRDC) comes equipped with an operator $\mathsf{R}$ which for any map $f: A \to B$ outputs a map $\mathsf{R}[f]: A \times B \to A$. It was shown in \cite[]{cockett_et_al:LIPIcs:2020:11661} that a CRDC can be seen as a CDC with additional structure. Specifically, a CRDC is equivalent to giving a CDC in which the subcategory of linear maps in each simple slice has a transpose operator, which categorically speaking is a special type of of dagger structure. The explicit connection with supervised learning was then made in \cite[]{gradientBasedLearning}, which showed how to describe several supervised-learning techniques in the abstract setting of a CRDC.  

However, the first CRDC paper \cite[]{cockett_et_al:LIPIcs:2020:11661} left open the question of what a \emph{monoidal} reverse differential category (MRDC) should be. The goal of this paper is to fill in this gap by defining \emph{monoidal reverse differential categories} and establishing their fundamental relationships to the existing categorical differential structures described above. 

\begin{center}
\begin{tabular}{ |c|c|c| } 
% \hline
  & Cartesian & Monoidal \\ 
 % \hline
 Forward & CDC \cite[]{blute2009cartesian} & MDC \cite[]{blute2006differential} \\ 
 Reverse & CRDC \cite[]{cockett_et_al:LIPIcs:2020:11661} & MRDC (this paper) \\ 
% \hline
\end{tabular}
\end{center}

What should this structure look like?  As mentioned above, CDCs axiomatize smooth maps, while MDCs axiomatize linear maps.  However, as noted above, for a CRDC, its subcategory of linear maps has dagger structure. So at a minimum, an MRDC should have dagger structure.  However, we argue in this paper that an MRDC should be an even stronger: it should be \emph{self-dual compact closed}.  

Why do we ask for this additional structure?  There are two important requirements we ask of an MRDC. 
\begin{enumerate}
    \item Just as every Cartesian reverse differential category (CRDC) gives a Cartesian differential category (CDC) so should a monoidal reverse differential category (MRDC) give a monoidal differential category (MDC); moreover, we should be able to characterize precisely what structure is required of an MDC to make it an MRDC (as we can in the Cartesian case \cite[Theorem 41]{cockett_et_al:LIPIcs:2020:11661}).
    \item Just as the coKleisli category of an MDC is a CDC \cite[Proposition 3.2.1]{blute2009cartesian}, so should the coKleisli category of an MRDC be a CRDC.
\end{enumerate}
We shall see in Section \ref{sec:mrdc_is_sdcc} that these requirements force an MRDC to be self-dual compact closed.  

To prove these results, it will be helpful to investigate the basic structure of an MDC more closely.  In Section \ref{sec:diffcats}, we review monoidal differential categories, and add a new aspect to their story: a ``context fibration'' which helps to relate the structure of MDCs to CDCs (and then similarly between MRDCs and CRDCs). 

Thus, the main contributions of this paper are as follows:
\begin{itemize}
    \item Give the basic definition of a monoidal reverse differential category (MRDC), along with examples, including some unexpected ones in quantum computation.  
    \item Prove theorems that describe the relationships of MRDCs to CDCs, CRDCs, and MDCs.  
    \item Provide additional material about the relationship of MDCs to CDCs via a ``context fibration''.  
\end{itemize}
This work leaves open many future avenues for exploration; we describe some of these in Section \ref{sec:future_work}.

\paragraph*{Related Work:} In this paper, we study the linear logic categorical semantics for reverse differentiation. This was also studied by Smeding and V{\'a}k{\'a}r when they provide the categorical semantics for CHAD, their programming language for automatic differentiation (which includes both forward and reverse differentiation) \cite[]{vakar2021chad}. The work in this paper is also related to work done in categorical quantum mechanics. Indeed, the categorical semantics of (differential) linear logic that we consider in this paper also comes with the added assumption of dagger-compact closed structure. Compact closed categories have long been considered as models for linear logic \cite[]{hyland2003glueing, shirahata1996sequent}, and they form a setting that is often studied by those in the categorical quantum community, sometimes called multiplicative categorical quantum logic \cite[Chapter 4]{duncan2006types}. We are specifically interested in compact closed models of linear logic with exponentials. This is a setting that was studied by Selinger and Valiron when they developed a programming language for quantum computation with classical control \cite[]{journal:selinger-valiron-fully-abstract-quantum}, as well as by Vicary who studied categorical quantum harmonic oscillators \cite[]{vicary2008categorical}. Cockett, Comfort, and Srinivasan also provide a generalization of linear logic with exponentials for categorical quantum mechanics, by generalizing compact closed categories to linear distributive categories with daggers \cite[]{srinivasan2021dagger}. 

\paragraph*{Outline:} A reader interested in just the definition of MRDC can skip ahead to Section \ref{sec:mrdc}.  However, an important part of the paper is the justification of why we define MRDCs the way we do: in particular, why the self-dual compact-closed requirement is important.  For this, it was helpful for us to expand on a number of aspects of MDCs and CRDCs.  In particular, in Section \ref{sec:context_fibration}, we define a canonical ``context'' fibration associated to any MDC.  Then, in Section \ref{sec:fibration_equivalence}, we show that when we can build an associated CDC from an MDC, the canonical fibration associated to the MDC is equivalent to the canonical fibration of linear maps associated to a CDC.  This result is key in seeing why MRDCs must be self-dual compact closed.  Thus, prior to defining an MRDC, we review the background of MDCs, CDCs, and CRDCs, but also add some important new theory of these structures, which will in turn be helpful in understanding how our definition of an MRDC comes about.  Section \ref{sec:mrdc} contains the main definition of the paper, that of an MRDC.  We also describe examples, and prove the required properties.  We conclude the section by describing additional ways to build CRDCs.  Finally, in Section \ref{sec:future_work}, we describe some ways this work can be expanded on in the future.  

\paragraph*{Conventions:} In this paper, we will use the same terminology, notation, graphical calculus, and conventions as found in \cite[]{Blute2019}. In particular, our string diagrams are to be read from top to bottom, and we write composition diagrammatically, that is, the composition of maps $f: A \to B$ and $g: B \to C$ is denoted as $f;g: A \to C$.  

\section{Forward Differential Categories}\label{sec:diffcats}

In this section, we review the theory of monoidal and Cartesian differential categories, and add an important new element to the story: a canonical fibration associated to any coalgebra modality (in particular, to any differential category); see Section \ref{sec:context_fibration}.   When the differential category has products, so that we can build its associated Cartesian differential category, we show that this canonical fibration is isomorphic to the canonical linear fibration of the Cartesian differential category; see Theorem \ref{thm:fibration_equivalence}.  This isomorphism will be very useful when we go from monoidal reverse differential categories to Cartesian reverse differential categories.  

\subsection{Coalgebra Modalities}

The central structure on which a monoidal differential category rests is a coalgebra modality.  

\begin{definition}\label{coalgdef} A \textbf{coalgebra modality} \cite[Definition 2.1]{blute2006differential} on a symmetric monoidal category $\mathbb{X}$ is a quintuple $(\oc, \delta, \varepsilon, \Delta, e)$ consisting of an endofunctor $\oc: \mathbb{X} \to \mathbb{X}$ and four natural transformations: 
\begin{align*}
\delta_A: \oc A \to \oc \oc A && \varepsilon_A: \oc A \to A && \Delta_A: \oc A \to \oc A \otimes \oc A && e_A: \oc A \to k
\end{align*}
such that $(\oc, \delta, \varepsilon)$ is a comonad and for each object $A$,  $(\oc A, \Delta, e)$ is a cocommutative comonoid and $\delta_A$ is a comonoid morphism, that is, the diagrams found in \cite[Definition 1]{Blute2019} commute. 
 \end{definition}
 
Note that requiring that $\Delta$ and $e$ be natural transformations is equivalent to asking that for each map $f: A \to B$, $\oc(f): \oc A \to \oc B$ is also a comonoid morphism. In the graphical calculus, we will use functor boxes when dealing with string diagrams involving the endofunctor, that is, a mere map $f: A \to B$ will be encased in a circle while $\oc(f): \oc A \to \oc B$ will be encased in a box: 
 \begin{align*}
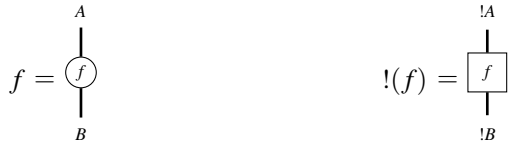
 \begin{array}[c]{c}
f
   \end{array}=
 \begin{array}[c]{c}
\begin{tikzpicture}
	\begin{pgfonlayer}{nodelayer}
		\node [style=circle] (0) at (0, 2.25) {$A$};
		\node [style=circle] (1) at (0, -0.25) {$B$};
		\node [style={component}] (2) at (0, 1) {$f$};
	\end{pgfonlayer}
	\begin{pgfonlayer}{edgelayer}
		\draw [style=wire] (0) to (2);
		\draw [style=wire] (2) to (1);
	\end{pgfonlayer}
\end{tikzpicture}
\end{array}
 &&  \begin{array}[c]{c}
\oc(f)
   \end{array}=
   \begin{array}[c]{c}
   \begin{tikzpicture}
	\begin{pgfonlayer}{nodelayer}
		\node [style=circle] (0) at (0, 2.25) {$\oc A$};
		\node [style=circle] (1) at (0, -0.25) {$\oc B$};
		\node [style={function}] (2) at (0, 1) {$f$};
	\end{pgfonlayer}
	\begin{pgfonlayer}{edgelayer}
		\draw [style=wire] (0) to (2);
		\draw [style=wire] (2) to (1);
	\end{pgfonlayer}
\end{tikzpicture}
   \end{array}
\end{align*}
The remaining coalgebra modality structure maps are drawn as follows: 
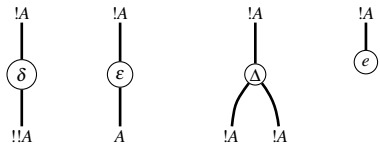
\begin{align*}
   \begin{array}[c]{c}
\begin{tikzpicture}
	\begin{pgfonlayer}{nodelayer}
		\node [style=object] (1) at (15, 3.5) {$\oc A$};
		\node [style=object] (2) at (15, 1) {$\oc\oc A$};
		\node [style=component] (3) at (15, 2.25) {$\delta$};
		\node [style=object] (5) at (19.75, 3.5) {$\oc A$};
		\node [style=object] (7) at (19.25, 1) {$\oc A$};
		\node [style=duplicate] (8) at (19.75, 2.25) {$\Delta$};
		\node [style=object] (10) at (20.25, 1) {$\oc A$};
		\node [style=object] (11) at (17, 3.5) {$\oc A$};
		\node [style=object] (12) at (17, 1) {$A$};
		\node [style=component] (13) at (17, 2.25) {$\varepsilon$};
		\node [style=object] (14) at (22, 3.5) {$\oc A$};
		\node [style=component] (16) at (22, 2.5) {$e$};
	\end{pgfonlayer}
	\begin{pgfonlayer}{edgelayer}
		\draw [style=wire] (1) to (3);
		\draw [style=wire] (3) to (2);
		\draw [style=wire, bend right=15, looseness=1.25] (8) to (7);
		\draw [style=wire] (5) to (8);
		\draw [style=wire, in=97, out=-53] (8) to (10);
		\draw [style=wire] (11) to (13);
		\draw [style=wire] (13) to (12);
		\draw [style=wire] (14) to (16);
	\end{pgfonlayer}
\end{tikzpicture}
   \end{array}
\end{align*}

We will occasionally make use of the following canonical natural transformation associated to any coalgebra modality.  

\begin{definition}\label{dcircdef} For a coalgebra modality $(\oc, \delta, \varepsilon, \Delta, e)$ on a symmetric monoidal category $\mathbb{X}$, its \textbf{coderiving transformation} \cite[Definition 2.2]{cockett_lemay_2018} is the natural transformation $\mathsf{d}^\circ_A: \oc A \to \oc A \otimes A$ defined as follows: 
\begin{align*}
\mathsf{d}^\circ_A := \xymatrixcolsep{5pc}\xymatrix{ \oc A \ar[r]^-{\Delta_A} & \oc A \otimes \oc A \ar[r]^-{1_{\oc A} \otimes \varepsilon_A} & \oc A \otimes A  
  } && \mathsf{d}^\circ := \begin{array}[c]{c} 
\begin{tikzpicture}
	\begin{pgfonlayer}{nodelayer}
		\node [style=component] (0) at (2, 1.5) {$\varepsilon$};
		\node [style=duplicate] (1) at (1.25, 2.25) {$\Delta$};
		\node [style=object] (2) at (0.5, 0.75) {$\oc A$};
		\node [style=object] (3) at (1.25, 3.25) {$\oc A$};
		\node [style=object] (4) at (2, 0.75) {$A$};
		\node [style=object] (5) at (-1.5, 1) {$A$};
		\node [style=object] (6) at (-2, 3) {$\oc A$};
		\node [style=integral] (7) at (-2, 2) {{\bf =\!=\!=\!=}};
		\node [style=object] (8) at (-2.5, 1) {$\oc A$};
		\node [style=object] (9) at (-0.25, 2) {$=$};
	\end{pgfonlayer}
	\begin{pgfonlayer}{edgelayer}
		\draw [style=wire] (3) to (1);
		\draw [style=wire, in=90, out=0, looseness=1.25] (1) to (0);
		\draw [style=wire] (0) to (4);
		\draw [style=wire, in=90, out=180] (1) to (2);
		\draw [style=wire, bend left] (7) to (5);
		\draw [style=wire] (6) to (7);
		\draw [style=wire, bend right] (7) to (8);
	\end{pgfonlayer}
\end{tikzpicture}
   \end{array}
\end{align*}
\end{definition}
See \cite[Proposition 2.1]{cockett_lemay_2018} for a list of identities the coderiving transformation satisfies. 

We now turn our attention to when our symmetric monoidal category also has finite products. 

\begin{definition} \label{Seelydef} For a coalgebra modality $(\oc, \delta, \varepsilon, \Delta, e)$ on a symmetric monoidal category $\mathbb{X}$ with finite products, the \textbf{Seely maps} consist of the natural transformations:
\begin{align*}
    \chi_{A,B}: \oc(A \times B) \to \oc A \otimes \oc B && \chi_\top: \oc \top \to k
\end{align*}
defined respectively as follows: 
\begin{align*}
  \chi_{A,B} := \xymatrixcolsep{3.5pc}\xymatrix{\oc(A \times B) \ar[r]^-{\Delta_{A,B}} &  \oc(A \times B) \otimes \oc(A \times B) \ar[r]^-{\oc(\pi_0) \otimes \oc(\pi_1)} & \oc A \otimes \oc B
  } 
\end{align*}
\begin{align*}
  \chi_\top := \xymatrixcolsep{3pc}\xymatrix{  \oc\top \ar[r]^-{e_\top} & k
  } 
\end{align*}
\begin{align*}
   \begin{array}[c]{c}
\begin{tikzpicture}
	\begin{pgfonlayer}{nodelayer}
		\node [style=function2] (0) at (1.25, 1.25) {$\pi_1$};
		\node [style=duplicate] (1) at (0.5, 2.25) {$\Delta$};
		\node [style=object] (2) at (-0.25, 0.25) {$\oc A$};
		\node [style=object] (3) at (0.5, 3.25) {$\oc (A \times B)$};
		\node [style=object] (4) at (1.25, 0.25) {$\oc B$};
		\node [style=object] (5) at (-1.5, 1) {$\oc B$};
		\node [style=object] (6) at (-2, 3) {$\oc(A \times B)$};
		\node [style=duplicate] (7) at (-2, 2) {$\chi$};
		\node [style=object] (8) at (-2.5, 1) {$\oc A$};
		\node [style=object] (9) at (-1, 2) {$=$};
		\node [style=function2] (10) at (-0.25, 1.25) {$\pi_0$};
		\node [style=component] (12) at (5.25, 1.5) {$e$};
		\node [style=object] (14) at (5.25, 2.5) {$\oc \top$};
		\node [style=object] (17) at (3.75, 2.5) {$\oc \top$};
		\node [style=duplicate] (18) at (3.75, 1.5) {$\chi_\top$};
		\node [style=object] (20) at (4.5, 2) {$=$};
	\end{pgfonlayer}
	\begin{pgfonlayer}{edgelayer}
		\draw [style=wire] (3) to (1);
		\draw [style=wire, in=90, out=-15, looseness=1.25] (1) to (0);
		\draw [style=wire] (0) to (4);
		\draw [style=wire, bend left] (7) to (5);
		\draw [style=wire] (6) to (7);
		\draw [style=wire, bend right] (7) to (8);
		\draw [style=wire, in=90, out=-165, looseness=1.25] (1) to (10);
		\draw [style=wire] (10) to (2);
		\draw [style=wire] (14) to (12);
		\draw [style=wire] (17) to (18);
	\end{pgfonlayer}
\end{tikzpicture}
   \end{array}
\end{align*}
 A coalgebra modality $(\oc, \delta, \varepsilon, \Delta, e)$ on a symmetric monoidal category $\mathbb{X}$ with finite products has \textbf{Seely isomorphisms} \cite[Definition 10]{Blute2019} if the Seely maps are isomorphisms, so that $\oc(A \times B) \cong \oc A \otimes \oc B$ and $\oc \top \cong k$. 
\end{definition}

Coalgebra modalities with Seely isomorphisms can equivalently be described as \textbf{monoidal coalgebra modalities} \cite[Definition 2]{Blute2019}, which are coalgebra modalities equipped with extra structure: a natural transformation ${\mathsf{m}_{A,B}: \oc A \otimes \oc B \to \oc(A \otimes B)}$ and a map $\mathsf{m}_k: k \to \oc k$ such that the underlying comonad $\oc$ is a symmetric monoidal comonad, and that $\Delta$ and $\mathsf{e}$ are both monoidal transformations and $\oc$-coalgebra morphisms (which imply that $\mathsf{m}_{A,B}$ and $\mathsf{m}_k$ are comonoid morphisms). See \cite[Section 7]{Blute2019} for how to build $\mathsf{m}$ and $\mathsf{m}_k$ from the Seely isomorphisms, and vice-versa. Note however that monoidal coalgebra modalities can be defined without the need of finite products; however, as they do not play a central role in this paper, we have elected to only briefly mention them. Many examples of (monoidal) coalgebra modalities can be found throughout the literature, see for example \cite[Section 2.4]{hyland2003glueing} for a very nice list of various kinds of examples of (monoidal) coalgebra modalities. 

We conclude this section by briefly discussing coalgebra modalities in the presence of additive structure. Indeed, the underlying categorical structure of a differential category is not only a symmetric monoidal category but an \emph{additive} symmetric monoidal category. 

\begin{definition}\label{addcatdef} An \textbf{additive symmetric monoidal category} \cite[Definition 3]{Blute2019} is a symmetric monoidal category $\mathbb{X}$ such that each hom-set $\mathbb{X}(A,B)$ is a commutative monoid with zero map $0 \in \mathbb{X}(A,B)$ and addition ${+: \mathbb{X}(A,B) \times \mathbb{X}(A,B) \to \mathbb{X}(A,B)}$, $(f,g) \mapsto f +g$, and, such that composition and the tensor product preserves the additive structure, that is, the following equalities hold: 
\begin{align*}
k;(f + g);h =k;f;h +k;g;h && k;0;h=0
\end{align*}
\begin{align*}
 k \otimes (f+g)\otimes h= k\otimes f\otimes h + k\otimes g \otimes h && k \otimes 0 \otimes h =0
\end{align*}
\end{definition}

By \cite[Theorem 1]{Blute2019}, for additive symmetric monoidal categories, monoidal coalgebra modalities can equivalently be describe as \textbf{additive bialgebra modalities} \cite[Definition 3]{Blute2019}. This implies that, in the additive case, we also have two extra natural transformations $\nabla_A: \oc A \otimes \oc A \to \oc A$ and $u_A: k \to \oc A$ such that $\oc A$ is a bimonoid. In particular, this implies that $\oc A$ is a commutative monoid. In the graphical calculus:
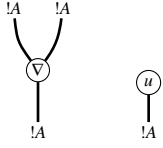
\begin{align*}
   \begin{array}[c]{c}
\begin{tikzpicture}
	\begin{pgfonlayer}{nodelayer}
		\node [style=object] (3) at (19.75, 1) {$\oc A$};
		\node [style=object] (4) at (19.25, 3.5) {$\oc A$};
		\node [style=duplicate] (5) at (19.75, 2.25) {$\nabla$};
		\node [style=object] (6) at (20.25, 3.5) {$\oc A$};
		\node [style=object] (10) at (22, 1) {$\oc A$};
		\node [style=component] (11) at (22, 2) {$u$};
	\end{pgfonlayer}
	\begin{pgfonlayer}{edgelayer}
		\draw [style=wire, bend left=15, looseness=1.25] (5) to (4);
		\draw [style=wire] (3) to (5);
		\draw [style=wire, in=-97, out=53] (5) to (6);
		\draw [style=wire] (10) to (11);
	\end{pgfonlayer}
\end{tikzpicture}
   \end{array}
\end{align*}
If an additive symmetric monoidal category has finite products, then the product $\times$ is in fact a biproduct and the terminal object $\top$ is a zero object. Thus, for an additive symmetric monoidal category with finite (bi)products, a coalgebra modality with Seely isomorphisms is an additive bialgebra modality and vice-versa \cite[Theorem 6]{Blute2019}. In particular, the inverse maps $\chi^{-1}_{A,B}: \oc A \otimes \oc B \to \oc (A \times B)$ and $\chi^{-1}_\top: k \to \oc \top$ are constructed as follows using the monoid structure of $\oc A$: 
\begin{align*}
  \chi^{-1}_{A,B} := \xymatrixcolsep{3.25pc}\xymatrix{\oc A \otimes \oc B  \ar[rr]^-{\oc(\iota_0) \otimes \oc(\iota_1)} && \oc(A \times B) \otimes \oc(A \times B) \ar[r]^-{\nabla_{A \times B}} & \oc(A \times B)
  } 
\end{align*}
\begin{align*}
  \chi^{-1}_\top := \xymatrixcolsep{3pc}
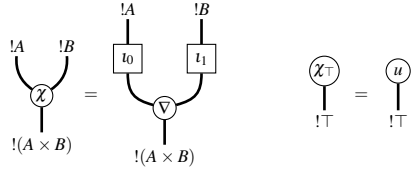
\xymatrix{  k \ar[r]^-{u_\top} & \oc\top 
  } 
\end{align*}
\begin{align*}
   \begin{array}[c]{c}
\begin{tikzpicture}
	\begin{pgfonlayer}{nodelayer}
		\node [style=function2] (0) at (4.25, 2.25) {$\iota_1$};
		\node [style=duplicate] (1) at (3.5, 1.25) {$\nabla$};
		\node [style=object] (2) at (2.75, 3.25) {$\oc A$};
		\node [style=object] (3) at (3.5, 0.25) {$\oc (A \times B)$};
		\node [style=object] (4) at (4.25, 3.25) {$\oc B$};
		\node [style=object] (5) at (1.5, 2.5) {$\oc B$};
		\node [style=object] (6) at (1, 0.5) {$\oc(A \times B)$};
		\node [style=duplicate] (7) at (1, 1.5) {$\chi$};
		\node [style=object] (8) at (0.5, 2.5) {$\oc A$};
		\node [style=object] (9) at (2, 1.5) {$=$};
		\node [style=function2] (10) at (2.75, 2.25) {$\iota_0$};
		\node [style=component] (11) at (8.25, 2) {$u$};
		\node [style=object] (12) at (8.25, 1) {$\oc \top$};
		\node [style=object] (13) at (6.75, 1) {$\oc \top$};
		\node [style=duplicate] (14) at (6.75, 2) {$\chi_\top$};
		\node [style=object] (15) at (7.5, 1.5) {$=$};
	\end{pgfonlayer}
	\begin{pgfonlayer}{edgelayer}
		\draw [style=wire] (3) to (1);
		\draw [style=wire, in=-90, out=15, looseness=1.25] (1) to (0);
		\draw [style=wire] (0) to (4);
		\draw [style=wire, bend right] (7) to (5);
		\draw [style=wire] (6) to (7);
		\draw [style=wire, bend left] (7) to (8);
		\draw [style=wire, in=-90, out=165, looseness=1.25] (1) to (10);
		\draw [style=wire] (10) to (2);
		\draw [style=wire] (12) to (11);
		\draw [style=wire] (13) to (14);
	\end{pgfonlayer}
\end{tikzpicture}
   \end{array}
\end{align*}
 where $\iota_0: A \to A \times B$ and $\iota_1: B \to A \times B$ are the injection maps of the biproduct.

\subsection{The context fibration associated to a coalgebra modality}\label{sec:context_fibration}

In this section we describe a canonical fibration associated to any coalgebra modality, whose individual fibres were studied in  \cite[]{ehrhard2021categorical,HYLAND1999127}. We assume the reader is familiar with the theory of fibrations (as, for example, presented in \cite[Section 2.1]{jacobs1999categorical}). The fibration in question will be over the coKleisli category of the comonad $\oc$. As we will be working with coKleisli categories, we will use the notation in \cite[]{blute2015cartesian}, where interpretation brackets $\llbracket - \rrbracket$ are used to translate between maps in the coKleisli category and maps in the base category.  That is, for a comonad $(\oc, \delta, \varepsilon)$ on a category $\mathbb{X}$ if $\mathbb{X}_\oc$ is its coKleisli category, then a map $f: A \to B$ in $\mathbb{X}_{\oc}$ corresponds to a map $\llbracket f \rrbracket: \oc A \to B$ in $\mathbb{X}$.  Using this notation, recall that composition and identities in $\mathbb{X}_\oc$ are defined as: 
\begin{align*}
\llbracket f;g \rrbracket = \delta_{A}; \oc(\llbracket f \rrbracket); \llbracket g \rrbracket && \llbracket 1_A \rrbracket = \varepsilon_A
\end{align*}
There are canonical adjoint functors $\mathsf{U}_\oc: \mathbb{X}_\oc \to \mathbb{X}$ and $\mathsf{F}_\oc: \mathbb{X} \to \mathbb{X}_\oc$ defined as follows: 
\begin{align*}
\mathsf{U}_\oc(A) = \oc A && \mathsf{U}_\oc(\llbracket f \rrbracket) = \delta_A;  \oc(\llbracket f \rrbracket) && \mathsf{F}_\oc(A) = A && \llbracket \mathsf{F}_\oc(f) \rrbracket = \varepsilon_A ; f
\end{align*}
We now describe the canonical fibration over the coKleisli category of a coalgebra modality. 

\begin{definition} Let $(\oc, \delta, \varepsilon, \Delta, e)$ be a coalgebra modality on a symmetric monoidal category $\mathbb{X}$. Define the category $\mathcal{L}_\oc[\mathbb{X}]$ as follows: 
\begin{enumerate}[{\em (i)}]
\item The objects of $\mathcal{L}_\oc[\mathbb{X}]$ are pairs of objects $(X,A)$ of $\mathbb{X}$; that is:
\[Ob\left( \mathcal{L}_\oc[\mathbb{X}] \right) = Ob\left( \mathbb{X} \right) \times Ob\left( \mathbb{X} \right);\]
\item The maps of $\mathcal{L}_\oc[\mathbb{X}]$ are pairs $(\llbracket f \rrbracket,g): (X,A) \to (Y,B)$ consisting of a coKleisli map ${\llbracket f \rrbracket: \oc X \to Y}$ and a map $g: \oc X \otimes A \to B$, that is, 
\[\mathcal{L}_\oc[\mathbb{X}]\left( (X,A), (Y,B) \right) = \mathbb{X}_\oc(X, Y) \times \mathbb{X}(\oc X \otimes A, B);\] 
\item The identity map of $(X,A)$ is defined as $(\llbracket 1_A \rrbracket, e_X \otimes 1_A) = (\varepsilon_A, e_X \otimes 1_A): (X,A) \to (X,A)$;
\begin{align*}
\left(\begin{array}[c]{c}
\begin{tikzpicture}
	\begin{pgfonlayer}{nodelayer}
		\node [style=object] (17) at (9.5, 1) {$\oc X$};
		\node [style=object] (20) at (9.5, -1) {$X$};
		\node [style=component] (21) at (9.5, 0) {$\varepsilon$};
	\end{pgfonlayer}
	\begin{pgfonlayer}{edgelayer}
		\draw [style=wire] (17) to (21);
		\draw [style=wire] (21) to (20);
	\end{pgfonlayer}
\end{tikzpicture}
   \end{array},
   \begin{array}[c]{c}
\begin{tikzpicture}
	\begin{pgfonlayer}{nodelayer}
		\node [style=object] (17) at (9.75, 1.75) {$\oc X$};
		\node [style=object] (19) at (10.5, 1.75) {$A$};
		\node [style=object] (21) at (10.5, -0.25) {$A$};
		\node [style=component] (22) at (9.75, 0.5) {$e$};
	\end{pgfonlayer}
	\begin{pgfonlayer}{edgelayer}
		\draw [style=wire] (19) to (21);
		\draw [style=wire] (17) to (22);
	\end{pgfonlayer}
\end{tikzpicture}
   \end{array} \right) 
\end{align*}
\item The composition of maps $(\llbracket f \rrbracket,g): (X,A) \to (Y,B)$ and $(\llbracket h \rrbracket, k): (Y,B) \to (Z,C)$ is defined as follows: 
\begin{align*} &(\llbracket f \rrbracket,g); (\llbracket h \rrbracket, k) = \\
&\left( \llbracket f;g \rrbracket, \xymatrixcolsep{4pc}\xymatrix{ \oc X \otimes A \ar[r]^-{\Delta_X \otimes 1_A} & \oc X \otimes \oc X \otimes A \ar[r]^-{\mathsf{U}_\oc\left( \llbracket f \rrbracket \right) \otimes g} & \oc Y \otimes B \ar[r]^-{k} & C} \right)
\end{align*}
\begin{align*}
\left(\begin{array}[c]{c}
\begin{tikzpicture}
	\begin{pgfonlayer}{nodelayer}
		\node [style=object] (17) at (9.5, 1) {$\oc X$};
		\node [style=object] (20) at (9.5, -1) {$Y$};
		\node [style=component] (21) at (9.5, 0) {$f$};
	\end{pgfonlayer}
	\begin{pgfonlayer}{edgelayer}
		\draw [style=wire] (17) to (21);
		\draw [style=wire] (21) to (20);
	\end{pgfonlayer}
\end{tikzpicture}
   \end{array},
   \begin{array}[c]{c}
\begin{tikzpicture}
	\begin{pgfonlayer}{nodelayer}
		\node [style=object] (17) at (9.75, 1.75) {$\oc X$};
		\node [style=object] (19) at (10.75, 1.75) {$A$};
		\node [style=component] (20) at (10.25, 0.75) {$g$};
		\node [style=object] (21) at (10.25, -0.25) {$B$};
	\end{pgfonlayer}
	\begin{pgfonlayer}{edgelayer}
		\draw [style=wire] (20) to (21);
		\draw [style=wire, in=165, out=-90] (17) to (20);
		\draw [style=wire, in=-90, out=15] (20) to (19);
	\end{pgfonlayer}
\end{tikzpicture}
   \end{array} \right) ; \left(\begin{array}[c]{c}
\begin{tikzpicture}
	\begin{pgfonlayer}{nodelayer}
		\node [style=object] (17) at (9.5, 1) {$\oc Y$};
		\node [style=object] (20) at (9.5, -1) {$Z$};
		\node [style=component] (21) at (9.5, 0) {$h$};
	\end{pgfonlayer}
	\begin{pgfonlayer}{edgelayer}
		\draw [style=wire] (17) to (21);
		\draw [style=wire] (21) to (20);
	\end{pgfonlayer}
\end{tikzpicture}
   \end{array},
   \begin{array}[c]{c}
\begin{tikzpicture}
	\begin{pgfonlayer}{nodelayer}
		\node [style=object] (17) at (9.75, 1.75) {$\oc Y$};
		\node [style=object] (19) at (10.75, 1.75) {$B$};
		\node [style=component] (20) at (10.25, 0.75) {$k$};
		\node [style=object] (21) at (10.25, -0.25) {$C$};
	\end{pgfonlayer}
	\begin{pgfonlayer}{edgelayer}
		\draw [style=wire] (20) to (21);
		\draw [style=wire, in=165, out=-90] (17) to (20);
		\draw [style=wire, in=-90, out=15] (20) to (19);
	\end{pgfonlayer}
\end{tikzpicture}
   \end{array} \right) = 
\left(\begin{array}[c]{c}
\begin{tikzpicture}
	\begin{pgfonlayer}{nodelayer}
		\node [style=object] (17) at (9.5, 1.75) {$\oc X$};
		\node [style=function2] (18) at (9.5, 0) {$f$};
		\node [style=component] (19) at (9.5, -1) {$h$};
		\node [style=object] (20) at (9.5, -1.75) {$Z$};
		\node [style=component] (21) at (9.5, 1) {$\delta$};
	\end{pgfonlayer}
	\begin{pgfonlayer}{edgelayer}
		\draw [style=wire] (18) to (19);
		\draw [style=wire] (19) to (20);
		\draw [style=wire] (17) to (21);
		\draw [style=wire] (21) to (18);
	\end{pgfonlayer}
\end{tikzpicture}
   \end{array},
   \begin{array}[c]{c}
\begin{tikzpicture}
	\begin{pgfonlayer}{nodelayer}
		\node [style=component] (15) at (9, 0) {$\delta$};
		\node [style=duplicate] (16) at (9.5, 1) {$\Delta$};
		\node [style=object] (17) at (9.5, 1.75) {$\oc X$};
		\node [style=component] (18) at (10.25, 0) {$g$};
		\node [style=object] (19) at (10.75, 1.75) {$A$};
		\node [style=component] (20) at (9.75, -1.75) {$k$};
		\node [style=object] (21) at (9.75, -2.5) {$C$};
		\node [style=function2] (23) at (9, -1) {$f$};
	\end{pgfonlayer}
	\begin{pgfonlayer}{edgelayer}
		\draw [style=wire, bend right] (16) to (15);
		\draw [style=wire] (17) to (16);
		\draw [style=wire, in=-90, out=15, looseness=0.75] (18) to (19);
		\draw [style=wire, in=150, out=-30, looseness=1.25] (16) to (18);
		\draw [style=wire] (20) to (21);
		\draw [style=wire, in=30, out=-90] (18) to (20);
		\draw [style=wire] (15) to (23);
		\draw [style=wire, bend right, looseness=1.25] (23) to (20);
	\end{pgfonlayer}
\end{tikzpicture}
   \end{array} \right)
\end{align*}
\end{enumerate}
Let $\mathsf{p}_\oc: \mathcal{L}_\oc[\mathbb{X}] \to \mathbb{X}_\oc$ be the forgetful functor, which is defined on objects as $\mathsf{p}_\oc(X,A) = X$ and on maps as ${\mathsf{p}_\oc(\llbracket f \rrbracket,g) = \llbracket f \rrbracket}$. 
\end{definition}

% Need to add references, eg., Robin, indicating this idea is not entirely new

The following is then straightforward:

\begin{proposition}\label{prop:context_fibration}  Let $(\oc, \delta, \varepsilon, \Delta, e)$ be a coalgebra modality on a symmetric monoidal category $\mathbb{X}$. Then $\mathsf{p}_\oc: \mathcal{L}_\oc[\mathbb{X}] \to \mathbb{X}_\oc$ is a fibration where the Cartesian maps are those of the form: 
\begin{align*}
(\llbracket f \rrbracket, e_X \otimes 1_A): (X,A) \to (Y,A) && 
\left(\begin{array}[c]{c}
\begin{tikzpicture}
	\begin{pgfonlayer}{nodelayer}
		\node [style=object] (17) at (9.5, 1) {$\oc X$};
		\node [style=object] (20) at (9.5, -1) {$Y$};
		\node [style=component] (21) at (9.5, 0) {$f$};
	\end{pgfonlayer}
	\begin{pgfonlayer}{edgelayer}
		\draw [style=wire] (17) to (21);
		\draw [style=wire] (21) to (20);
	\end{pgfonlayer}
\end{tikzpicture}
   \end{array},
   \begin{array}[c]{c}
\begin{tikzpicture}
	\begin{pgfonlayer}{nodelayer}
		\node [style=object] (17) at (9.75, 1.75) {$\oc X$};
		\node [style=object] (19) at (10.5, 1.75) {$A$};
		\node [style=object] (21) at (10.5, -0.25) {$A$};
		\node [style=component] (22) at (9.75, 0.5) {$e$};
	\end{pgfonlayer}
	\begin{pgfonlayer}{edgelayer}
		\draw [style=wire] (19) to (21);
		\draw [style=wire] (17) to (22);
	\end{pgfonlayer}
\end{tikzpicture}
   \end{array} \right) 
\end{align*}
\end{proposition}

We now describe the fibres of this fibration. The fibres are examples of Hyland and Schalk's \textbf{comonoid indexing} \cite[Section 4]{HYLAND1999127} over the cofree $\oc$-coalgebras, which are also used by Ehrhard and Jafarrahmani for studying fixed point formulas \cite[]{ehrhard2021categorical}. In particular, since $\oc X$ is a comonoid, $\oc X \otimes -$ is a comonad and, furthermore, its coKleisli category is precisely the fibre over $X$. 

 \begin{lemma}\label{lem:linfibre} Let $(\oc, \delta, \varepsilon, \Delta, e)$ be a coalgebra modality on a symmetric monoidal category $\mathbb{X}$. For any object $X \in \mathbb{X}$, the fibre over $X$ of the fibration $\mathsf{p}_\oc: \mathcal{L}_\oc[\mathbb{X}] \to \mathbb{X}_\oc$ is written as $\mathcal{L}_\oc[X]$ and given by
\begin{enumerate}[{\em (i)}]
\item The objects of $\mathcal{L}_\oc[X]$ are the same as the objects of $\mathbb{X}$, that is, $Ob\left( \mathcal{L}_\oc[X] \right) = Ob\left( \mathbb{X} \right)$;
\item The maps of $\mathcal{L}_\oc[X]$ are maps $f: \oc X \otimes A \to B$, that is, $\mathcal{L}_\oc[X]\left( A,B \right) = \mathbb{X}(\oc X \otimes A, B)$; 
\item The identity map of $A$ is defined as $e_X \otimes 1_A: \oc X \otimes A \to A$;
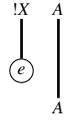
\[ \begin{array}[c]{c}
\begin{tikzpicture}
	\begin{pgfonlayer}{nodelayer}
		\node [style=object] (17) at (9.75, 1.75) {$\oc X$};
		\node [style=object] (19) at (10.5, 1.75) {$A$};
		\node [style=object] (21) at (10.5, -0.25) {$A$};
		\node [style=component] (22) at (9.75, 0.5) {$e$};
	\end{pgfonlayer}
	\begin{pgfonlayer}{edgelayer}
		\draw [style=wire] (19) to (21);
		\draw [style=wire] (17) to (22);
	\end{pgfonlayer}
\end{tikzpicture}
   \end{array} \]
\item The composition of maps $f: \oc X \otimes A \to B$ and $g: \oc X \otimes B \to C$ is defined as follows: 
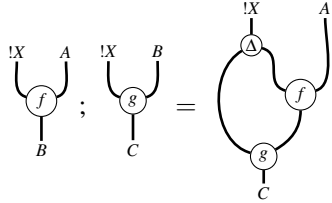
\begin{align*}
 \xymatrixcolsep{5pc}\xymatrix{ \oc X \otimes A \ar[r]^-{\Delta_X \otimes 1_A} & \oc X \otimes \oc X \otimes A \ar[r]^-{1_{\oc X} \otimes f} & \oc X \otimes B \ar[r]^-{g} & C} 
 \end{align*}
 \begin{align*}
  \begin{array}[c]{c}
\begin{tikzpicture}
	\begin{pgfonlayer}{nodelayer}
		\node [style=object] (17) at (9.75, 1.75) {$\oc X$};
		\node [style=object] (19) at (10.75, 1.75) {$A$};
		\node [style=component] (20) at (10.25, 0.75) {$f$};
		\node [style=object] (21) at (10.25, -0.25) {$B$};
	\end{pgfonlayer}
	\begin{pgfonlayer}{edgelayer}
		\draw [style=wire] (20) to (21);
		\draw [style=wire, in=165, out=-90] (17) to (20);
		\draw [style=wire, in=-90, out=15] (20) to (19);
	\end{pgfonlayer}
\end{tikzpicture}
   \end{array};    \begin{array}[c]{c}
\begin{tikzpicture}
	\begin{pgfonlayer}{nodelayer}
		\node [style=object] (17) at (9.75, 1.75) {$\oc X$};
		\node [style=object] (19) at (10.75, 1.75) {$B$};
		\node [style=component] (20) at (10.25, 0.75) {$g$};
		\node [style=object] (21) at (10.25, -0.25) {$C$};
	\end{pgfonlayer}
	\begin{pgfonlayer}{edgelayer}
		\draw [style=wire] (20) to (21);
		\draw [style=wire, in=165, out=-90] (17) to (20);
		\draw [style=wire, in=-90, out=15] (20) to (19);
	\end{pgfonlayer}
\end{tikzpicture}
   \end{array}=
   \begin{array}[c]{c}
\begin{tikzpicture}
	\begin{pgfonlayer}{nodelayer}
		\node [style=duplicate] (1) at (9.5, 1) {$\Delta$};
		\node [style=object] (2) at (9.5, 1.75) {$\oc X$};
		\node [style=component] (3) at (10.5, 0) {$f$};
		\node [style=object] (4) at (11, 1.75) {$A$};
		\node [style=component] (5) at (9.75, -1.25) {$g$};
		\node [style=object] (6) at (9.75, -2) {$C$};
	\end{pgfonlayer}
	\begin{pgfonlayer}{edgelayer}
		\draw [style=wire] (2) to (1);
		\draw [style=wire, in=-90, out=15, looseness=0.75] (3) to (4);
		\draw [style=wire] (5) to (6);
		\draw [style=wire, in=30, out=-90] (3) to (5);
		\draw [style=wire, in=165, out=-150] (1) to (5);
		\draw [style=wire, in=165, out=-15, looseness=1.25] (1) to (3);
	\end{pgfonlayer}
\end{tikzpicture}
   \end{array}
\end{align*}
\end{enumerate}
For every coKleisli map $\llbracket h \rrbracket: \oc X \to Y$, define the \textbf{substitution functor} $\llbracket h \rrbracket^\ast_\oc: \mathcal{L}_\oc[Y] \to \mathcal{L}_\oc[X]$ on objects as $\llbracket h \rrbracket^\ast_\oc(A) = A$ and on maps $f: \oc Y \otimes A \to B$ as follows: 
\begin{align*}
\llbracket h \rrbracket^\ast_\oc(f) := \xymatrixcolsep{5pc}\xymatrix{ \oc X \otimes A \ar[r]^-{\llbracket h \rrbracket \otimes 1_A} & \oc Y \otimes A \ar[r]^-{f} & B} &&  \llbracket h \rrbracket^\ast_\oc \left(   \begin{array}[c]{c}
\begin{tikzpicture}
	\begin{pgfonlayer}{nodelayer}
		\node [style=object] (17) at (9.75, 1.75) {$\oc Y$};
		\node [style=object] (19) at (10.75, 1.75) {$A$};
		\node [style=component] (20) at (10.25, 0.75) {$f$};
		\node [style=object] (21) at (10.25, -0.25) {$B$};
	\end{pgfonlayer}
	\begin{pgfonlayer}{edgelayer}
		\draw [style=wire] (20) to (21);
		\draw [style=wire, in=165, out=-90] (17) to (20);
		\draw [style=wire, in=-90, out=15] (20) to (19);
	\end{pgfonlayer}
\end{tikzpicture}
   \end{array} \right) =    \begin{array}[c]{c}
\begin{tikzpicture}
	\begin{pgfonlayer}{nodelayer}
		\node [style=component] (0) at (9, 0.25) {$\delta$};
		\node [style=object] (1) at (10.5, 1.25) {$A$};
		\node [style=component] (2) at (9.75, -1.75) {$f$};
		\node [style=object] (3) at (9.75, -2.75) {$B$};
		\node [style=function2] (4) at (9, -0.75) {$h$};
		\node [style=object] (5) at (9, 1.25) {$\oc X$};
	\end{pgfonlayer}
	\begin{pgfonlayer}{edgelayer}
		\draw [style=wire] (2) to (3);
		\draw [style=wire, bend right, looseness=1.25] (4) to (2);
		\draw [style=wire] (0) to (4);
		\draw [style=wire, in=30, out=-90, looseness=0.75] (1) to (2);
		\draw [style=wire] (5) to (0);
	\end{pgfonlayer}
\end{tikzpicture}
   \end{array} 
\end{align*}
\end{lemma}

Every fibre is also a symmetric monoidal category. 

\begin{lemma}\cite[Proposition 4.1]{HYLAND1999127} Let $(\oc, \delta, \varepsilon, \Delta, e)$ be a coalgebra modality on a symmetric monoidal category $\mathbb{X}$. For every object $X \in \mathbb{X}$, $\mathcal{L}_\oc[X]$ is a symmetric monoidal category where the tensor product $\otimes$ is defined on objects as the tensor product in $\mathbb{X}$, and on maps $f: \oc X \otimes A \to B$ and $g: \oc X \otimes C \to D$ as follows: 
  \[   \begin{array}[c]{c}  f \otimes g    \end{array} :=    \begin{array}[c]{c} \xymatrixcolsep{5pc}\xymatrixrowsep{1pc}\xymatrix{ \oc X \otimes A  \otimes C \ar[r]^-{\Delta_X \otimes 1_A \otimes 1_C} &  \oc X \otimes \oc X \otimes A \otimes C \ar[r]^-{1_{\oc X} \otimes \sigma_{\oc X, A} \otimes 1_C} & \\
   \oc X \otimes A \otimes \oc X \otimes C  \ar[r]^-{f \otimes g} & B \otimes D}    \end{array} \]
  \begin{align*}
   \begin{array}[c]{c}
\begin{tikzpicture}
	\begin{pgfonlayer}{nodelayer}
		\node [style=object] (17) at (9.75, 1.75) {$\oc X$};
		\node [style=object] (19) at (10.75, 1.75) {$A$};
		\node [style=component] (20) at (10.25, 0.75) {$f$};
		\node [style=object] (21) at (10.25, -0.25) {$B$};
		\node [style=object] (22) at (11.25, 0.75) {$\bigotimes$};
		\node [style=object] (23) at (11.75, 1.75) {$\oc X$};
		\node [style=object] (24) at (12.75, 1.75) {$C$};
		\node [style=component] (25) at (12.25, 0.75) {$g$};
		\node [style=object] (26) at (12.25, -0.25) {$D$};
		\node [style=object] (27) at (13.25, 0.75) {$=$};
		\node [style=component] (30) at (14.5, 0.25) {$f$};
		\node [style=object] (31) at (14.5, -0.75) {$B$};
		\node [style=object] (32) at (14.5, 2.25) {$\oc X$};
		\node [style=object] (33) at (15.5, 2.25) {$A$};
		\node [style=object] (34) at (16.5, 2.25) {$C$};
		\node [style=duplicate] (35) at (14.5, 1.5) {$\Delta$};
		\node [style=component] (36) at (15.75, 0.25) {$g$};
		\node [style=object] (37) at (15.75, -0.75) {$D$};
	\end{pgfonlayer}
	\begin{pgfonlayer}{edgelayer}
		\draw [style=wire] (20) to (21);
		\draw [style=wire, in=165, out=-90] (17) to (20);
		\draw [style=wire, in=-90, out=15] (20) to (19);
		\draw [style=wire] (25) to (26);
		\draw [style=wire, in=165, out=-90] (23) to (25);
		\draw [style=wire, in=-90, out=15] (25) to (24);
		\draw [style=wire] (30) to (31);
		\draw [style=wire] (32) to (35);
		\draw [style=wire] (36) to (37);
		\draw [style=wire, in=135, out=-15, looseness=0.75] (35) to (36);
		\draw [style=wire, in=-180, out=-150, looseness=1.50] (35) to (30);
		\draw [style=wire, in=0, out=-90] (33) to (30);
		\draw [style=wire, in=15, out=-90] (34) to (36);
	\end{pgfonlayer}
\end{tikzpicture}
   \end{array}
\end{align*}
and where the monoidal unit is the same as $\mathbb{X}$. Furthermore, for every coKleisli map $\llbracket h \rrbracket: \oc X \to Y$, the substitution functor $\llbracket h \rrbracket^\ast_\oc: \mathcal{L}_\oc[Y] \to \mathcal{L}_\oc[X]$ is a strict symmetric monoidal functor. 
\end{lemma}

That is, as described in \cite[Remark 3.5]{monoidal_fibrations} this fibration is a pseudomonoid in the 2-category of fibrations over $\mathbb{X}_\oc$.  This is different from a monoidal fibration, which is defined to be a pseudomonoid in the 2-category of fibrations with non-fixed base \cite[Definition 3.1]{monoidal_fibrations}.  

However, if the base category has finite products $\mathbb{X}$, then so does $\mathbb{X}_\oc$: on objects the product is defined as in $\mathbb{X}$, and the remaining data is defined as follows: 
\begin{align*}
\llbracket \pi_0 \rrbracket =  \varepsilon_{A \times B}; \pi_0 && \llbracket \pi_1 \rrbracket = \varepsilon_{A \times B}; \pi_1 && \llbracket \langle f, g \rangle \rrbracket = \left \langle \llbracket f \rrbracket, \llbracket g \rrbracket \right \rangle && \llbracket f \times g \rrbracket = \left \langle \oc(\pi_0); \llbracket f \rrbracket , \oc(\pi_1); \llbracket g \rrbracket \right \rangle
\end{align*}
Moreover, such a fibration in which the base category is Cartesian is a monoidal fibration: see \cite[Theorem  4.1]{monoidal_fibrations} and \cite[Theorem 12.8]{shulman_monoidal_fibrations}.  In particular this means that the total category of the fibration is monoidal, and the following corollary describes its structure.

\begin{corollary}  Let $(\oc, \delta, \varepsilon, \Delta, e)$ be a coalgebra modality on a symmetric monoidal category $\mathbb{X}$ with finite products. Then $\mathcal{L}_\oc[\mathbb{X}]$ is a symmetric monoidal category where the tensor product $\otimes_\oc$ is defined on objects as $(X,A) \otimes (Y,B) = (X \times Y, A \otimes B)$, and on maps $(\llbracket f \rrbracket, g): (X,A) \to (Y,B)$ and $(\llbracket h \rrbracket, k): (Z,C) \to (W,D)$, $ (\llbracket f \rrbracket, g) \otimes (\llbracket h \rrbracket, k)$ is defined as follows:
\[ \left(\begin{array}[c]{c} \llbracket f \times h \rrbracket  \end{array}, \begin{array}[c]{c} \xymatrixcolsep{4pc}\xymatrix{ \oc (X \times Y) \!\otimes\! A  \!\otimes\! C \ar[r]^-{\chi_{X,Y} \otimes 1_A \otimes 1_C} & \oc X \!\otimes\! \oc Y \!\otimes\! A \!\otimes\! C \ar[r]^-{1_{\oc X} \otimes \sigma_{\oc Y, A} \otimes 1_C} & \\
  \oc X \otimes A \otimes \oc X \otimes C  \ar[r]^-{g \otimes k} & B \otimes D}  \end{array} \right)  \]
\begin{align*}
\left(\begin{array}[c]{c}
\begin{tikzpicture}
	\begin{pgfonlayer}{nodelayer}
		\node [style=object] (17) at (9.5, 1) {$\oc X$};
		\node [style=object] (20) at (9.5, -1) {$Y$};
		\node [style=component] (21) at (9.5, 0) {$f$};
	\end{pgfonlayer}
	\begin{pgfonlayer}{edgelayer}
		\draw [style=wire] (17) to (21);
		\draw [style=wire] (21) to (20);
	\end{pgfonlayer}
\end{tikzpicture}
   \end{array},
   \begin{array}[c]{c}
\begin{tikzpicture}
	\begin{pgfonlayer}{nodelayer}
		\node [style=object] (17) at (9.75, 1.75) {$\oc X$};
		\node [style=object] (19) at (10.75, 1.75) {$A$};
		\node [style=component] (20) at (10.25, 0.75) {$g$};
		\node [style=object] (21) at (10.25, -0.25) {$B$};
	\end{pgfonlayer}
	\begin{pgfonlayer}{edgelayer}
		\draw [style=wire] (20) to (21);
		\draw [style=wire, in=165, out=-90] (17) to (20);
		\draw [style=wire, in=-90, out=15] (20) to (19);
	\end{pgfonlayer}
\end{tikzpicture}
   \end{array} \right) \otimes \left(\begin{array}[c]{c}
\begin{tikzpicture}
	\begin{pgfonlayer}{nodelayer}
		\node [style=object] (17) at (9.5, 1) {$\oc Z$};
		\node [style=object] (20) at (9.5, -1) {$W$};
		\node [style=component] (21) at (9.5, 0) {$h$};
	\end{pgfonlayer}
	\begin{pgfonlayer}{edgelayer}
		\draw [style=wire] (17) to (21);
		\draw [style=wire] (21) to (20);
	\end{pgfonlayer}
\end{tikzpicture}
   \end{array},
   \begin{array}[c]{c}
\begin{tikzpicture}
	\begin{pgfonlayer}{nodelayer}
		\node [style=object] (17) at (9.75, 1.75) {$\oc Z$};
		\node [style=object] (19) at (10.75, 1.75) {$C$};
		\node [style=component] (20) at (10.25, 0.75) {$k$};
		\node [style=object] (21) at (10.25, -0.25) {$D$};
	\end{pgfonlayer}
	\begin{pgfonlayer}{edgelayer}
		\draw [style=wire] (20) to (21);
		\draw [style=wire, in=165, out=-90] (17) to (20);
		\draw [style=wire, in=-90, out=15] (20) to (19);
	\end{pgfonlayer}
\end{tikzpicture}
   \end{array} \right) = 
\left(\begin{array}[c]{c}
\begin{tikzpicture}
	\begin{pgfonlayer}{nodelayer}
		\node [style=object] (17) at (9.5, 1.5) {$\oc (X \times Y)$};
		\node [style=object] (20) at (9.5, -1.5) {$Y$};
		\node [style=component] (21) at (9.5, 0) {$\llbracket f \times h \rrbracket$};
	\end{pgfonlayer}
	\begin{pgfonlayer}{edgelayer}
		\draw [style=wire] (17) to (21);
		\draw [style=wire] (21) to (20);
	\end{pgfonlayer}
\end{tikzpicture}   \end{array},
   \begin{array}[c]{c}
\begin{tikzpicture}
	\begin{pgfonlayer}{nodelayer}
		\node [style=component] (30) at (14.5, 0.25) {$f$};
		\node [style=object] (31) at (14.5, -0.75) {$B$};
		\node [style=object] (32) at (14.5, 2.25) {$\oc (X \times Y)$};
		\node [style=object] (33) at (15.5, 2.25) {$A$};
		\node [style=object] (34) at (16.5, 2.25) {$C$};
		\node [style=duplicate] (35) at (14.5, 1.5) {$\chi$};
		\node [style=component] (36) at (15.75, 0.25) {$g$};
		\node [style=object] (37) at (15.75, -0.75) {$D$};
	\end{pgfonlayer}
	\begin{pgfonlayer}{edgelayer}
		\draw [style=wire] (30) to (31);
		\draw [style=wire] (32) to (35);
		\draw [style=wire] (36) to (37);
		\draw [style=wire, in=135, out=-15, looseness=0.75] (35) to (36);
		\draw [style=wire, in=-180, out=-150, looseness=1.50] (35) to (30);
		\draw [style=wire, in=0, out=-90] (33) to (30);
		\draw [style=wire, in=15, out=-90] (34) to (36);
	\end{pgfonlayer}
\end{tikzpicture}
   \end{array} \right)
\end{align*}
and where the monoidal unit is $(\top, k)$. Furthermore, $\mathsf{p}_\oc: \mathcal{L}_\oc[\mathbb{X}] \to \mathbb{X}_\oc$ is a monoidal fibration in the sense of \cite[Definition 3.1]{monoidal_fibrations}.   
\end{corollary}

It also interesting to note that for monoidal coalgebra modalities, each fibre also comes equipped with monoidal coalgebra modality structure \cite[]{ehrhard2021categorical,HYLAND1999127}. Furthermore, if one also assumes finite products, we can extend this to a monoidal coalgebra modality on the whole fibration.  However, these results are not necessary for the rest of the story of this paper. 

\subsection{Monoidal Differential Categories}

We now recall one of the central structures of this paper: monoidal differential categories (these were originally simply called differential categories, but here we add ``monoidal'' to help differentiate the various structures we are considering). For a more detailed introduction to monoidal differential categories, we refer the reader to \cite[]{Blute2019,blute2006differential}. 

\begin{definition}\label{def:diffcat} A \textbf{ monoidal differential category} \cite[Definition 2.4]{blute2006differential} is an additive symmetric monoidal category $\mathbb{X}$ with a coalgebra modality $(\oc, \delta, \varepsilon, \Delta, e)$ which comes equipped with a \textbf{deriving transformation} \cite[Definition 7]{Blute2019}; that is, a natural transformation $\mathsf{d}_A: \oc A \otimes A \to \oc A$, which is drawn in the graphical calculus as: 
\[\mathsf{d}:= \begin{array}[c]{c} 
\begin{tikzpicture}
	\begin{pgfonlayer}{nodelayer}
		\node [style=object] (0) at (1.75, 2.75) {$A$};
		\node [style=object] (1) at (1.25, 1.25) {$\oc A$};
		\node [style=integral] (2) at (1.25, 2) {{\bf =\!=\!=\!=}};
		\node [style=object] (3) at (0.75, 2.75) {$\oc A$};
	\end{pgfonlayer}
	\begin{pgfonlayer}{edgelayer}
		\draw [style=wire, bend right] (2) to (0);
		\draw [style=wire] (1) to (2);
		\draw [style=wire, bend left] (2) to (3);
	\end{pgfonlayer}
\end{tikzpicture}
   \end{array}\]
and such that the following axioms hold: 
\begin{description}
\item[{\bf [d.1]}] Constant Rule:
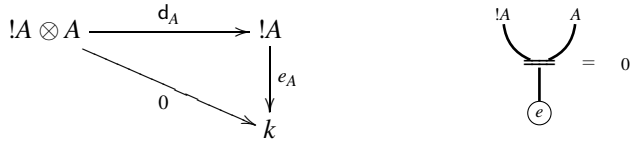
\begin{align*}
\begin{array}[c]{c}
 \xymatrixcolsep{5pc}\xymatrix{\oc A \otimes A \ar[dr]_-{0} \ar[r]^-{\mathsf{d}_A} & \oc A \ar[d]^-{e_A} \\
  & k} 
   \end{array}&&
   \begin{array}[c]{c}
\begin{tikzpicture}
	\begin{pgfonlayer}{nodelayer}
		\node [style=component] (0) at (1.25, 0) {$e$};
		\node [style=object] (1) at (2, 2) {$A$};
		\node [style=differential] (2) at (1.25, 1) {{\bf =\!=\!=\!=}};
		\node [style=object] (3) at (0.5, 2) {$\oc A$};
		\node [style=port] (4) at (2.25, 1) {$=$};
		\node [style=port] (5) at (3, 1) {$0$};
	\end{pgfonlayer}
	\begin{pgfonlayer}{edgelayer}
		\draw [style=wire, bend right] (2) to (1);
		\draw [style=wire, bend left] (2) to (3);
		\draw [style=wire] (2) to (0);
	\end{pgfonlayer}
\end{tikzpicture}
   \end{array}
\end{align*}
\item[{\bf [d.2]}] Leibniz Rule (or Product Rule):
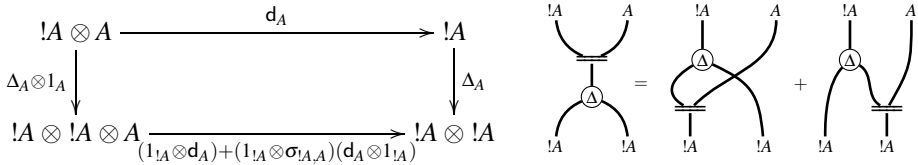
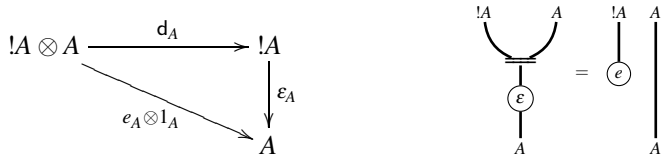
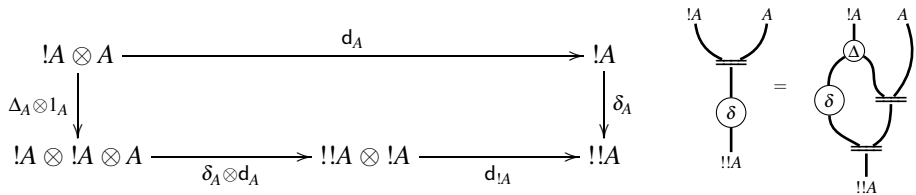
\begin{align*}
\begin{array}[c]{c}
\xymatrixcolsep{3.75pc}\xymatrix{\oc A \otimes A \ar[d]_-{\Delta_A \otimes 1_A} \ar[rr]^-{\mathsf{d}_A} && \oc A \ar[d]^-{\Delta_A} \\
    \oc A \otimes \oc A \otimes A \ar[rr]_-{(1_{\oc A} \otimes \mathsf{d}_A) + (1_{\oc A} \otimes \sigma_{\oc A, A})(\mathsf{d}_A \otimes 1_{\oc A})} && \oc A \otimes \oc A
  }
   \end{array} && 
   \begin{array}[c]{c}
\begin{tikzpicture}
	\begin{pgfonlayer}{nodelayer}
		\node [style=differential] (6) at (7, 1) {{\bf =\!=\!=\!=}};
		\node [style=object] (7) at (7.75, 2) {$A$};
		\node [style=object] (8) at (6.25, 2) {$\oc A$};
		\node [style=object] (9) at (6.25, -0.75) {$\oc A$};
		\node [style=duplicate] (10) at (7, 0.25) {$\Delta$};
		\node [style=object] (11) at (7.75, -0.75) {$\oc A$};
		\node [style=object] (12) at (9.25, 2) {$\oc A$};
		\node [style=differential] (13) at (9, 0) {{\bf =\!=\!=\!=}};
		\node [style=object] (14) at (10.75, 2) {$A$};
		\node [style=duplicate] (15) at (9.25, 1) {$\Delta$};
		\node [style=object] (16) at (9, -0.75) {$\oc A$};
		\node [style=object] (17) at (10.5, -0.75) {$\oc A$};
		\node [style=object] (18) at (13.5, 2) {$A$};
		\node [style=differential] (19) at (13, 0) {{\bf =\!=\!=\!=}};
		\node [style=object] (20) at (13, -0.75) {$\oc A$};
		\node [style=object] (21) at (11.75, -0.75) {$\oc A$};
		\node [style=object] (22) at (12.25, 2) {$\oc A$};
		\node [style=duplicate] (23) at (12.25, 1) {$\Delta$};
		\node [style=port] (24) at (8, 0.5) {$=$};
		\node [style=port] (25) at (11.25, 0.5) {$+$};
	\end{pgfonlayer}
	\begin{pgfonlayer}{edgelayer}
		\draw [style=wire, bend right] (6) to (7);
		\draw [style=wire, bend left] (6) to (8);
		\draw [style=wire, bend right] (10) to (9);
		\draw [style=wire, bend left] (10) to (11);
		\draw [style=wire] (6) to (10);
		\draw [style=wire, in=-90, out=45] (13) to (14);
		\draw [style=wire, in=150, out=-150, looseness=1.50] (15) to (13);
		\draw [style=wire] (12) to (15);
		\draw [style=wire] (13) to (16);
		\draw [style=wire, bend left, looseness=1.25] (15) to (17);
		\draw [style=wire, in=-90, out=60, looseness=1.25] (19) to (18);
		\draw [style=wire, in=91, out=-135, looseness=0.75] (23) to (21);
		\draw [style=wire, in=150, out=-30] (23) to (19);
		\draw [style=wire] (22) to (23);
		\draw [style=wire] (19) to (20);
	\end{pgfonlayer}
\end{tikzpicture}
   \end{array}
\end{align*}
\item[{\bf [d.3]}] Linear Rule:
\begin{align*}
\begin{array}[c]{c}
\xymatrixcolsep{5pc}\xymatrix{\oc A \otimes A \ar[r]^-{\mathsf{d}_A} \ar[dr]_-{e_A \otimes 1_A} & \oc A \ar[d]^-{\varepsilon_A} \\
  & A   
  }
   \end{array} &&
   \begin{array}[c]{c}
\begin{tikzpicture}
	\begin{pgfonlayer}{nodelayer}
		\node [style=object] (0) at (3.25, 1.75) {$\oc A$};
		\node [style=object] (1) at (4, 1.75) {$A$};
		\node [style=object] (2) at (4, -1) {$A$};
		\node [style=component] (3) at (3.25, 0.5) {$e$};
		\node [style=object] (4) at (2, 1.75) {$A$};
		\node [style=object] (5) at (0.5, 1.75) {$\oc A$};
		\node [style=component] (6) at (1.25, 0) {$\varepsilon$};
		\node [style=differential] (7) at (1.25, 0.75) {{\bf =\!=\!=\!=}};
		\node [style=object] (8) at (1.25, -1) {$A$};
		\node [style=port] (9) at (2.5, 0.5) {$=$};
	\end{pgfonlayer}
	\begin{pgfonlayer}{edgelayer}
		\draw [style=wire] (0) to (3);
		\draw [style=wire] (1) to (2);
		\draw [style=wire, bend right] (7) to (4);
		\draw [style=wire, bend left] (7) to (5);
		\draw [style=wire] (7) to (6);
		\draw [style=wire] (6) to (8);
	\end{pgfonlayer}
\end{tikzpicture}
   \end{array}
\end{align*}
\item[{\bf [d.4]}] Chain Rule: 
\begin{align*}
\begin{array}[c]{c}
\xymatrixcolsep{5pc}\xymatrix{\oc A \otimes A \ar[d]_-{\Delta_A \otimes 1_A} \ar[rr]^-{\mathsf{d}_A} && \oc A \ar[d]^-{\delta_A} \\
    \oc A \otimes \oc A \otimes A \ar[r]_-{\delta_A \otimes \mathsf{d}_A} & \oc \oc A \otimes \oc A \ar[r]_-{\mathsf{d}_{\oc A}} & \oc \oc A
  }
   \end{array}&&
   \begin{array}[c]{c}
\begin{tikzpicture}
	\begin{pgfonlayer}{nodelayer}
		\node [style=object] (10) at (7.75, 1.75) {$A$};
		\node [style=differential] (11) at (7, 0.75) {{\bf =\!=\!=\!=}};
		\node [style=object] (12) at (7, -1.25) {$\oc \oc A$};
		\node [style=object] (13) at (6.25, 1.75) {$\oc A$};
		\node [style=component] (14) at (7, -0.25) {$\delta$};
		\node [style=component] (15) at (9, 0) {$\delta$};
		\node [style=duplicate] (16) at (9.5, 1) {$\Delta$};
		\node [style=object] (17) at (9.5, 1.75) {$\oc A$};
		\node [style=differential] (18) at (10.25, 0) {{\bf =\!=\!=\!=}};
		\node [style=object] (19) at (10.5, 1.75) {$A$};
		\node [style=differential] (20) at (9.75, -1) {{\bf =\!=\!=\!=}};
		\node [style=object] (21) at (9.75, -1.75) {$\oc \oc A$};
		\node [style=port] (22) at (8, 0.25) {$=$};
	\end{pgfonlayer}
	\begin{pgfonlayer}{edgelayer}
		\draw [style=wire, bend right] (11) to (10);
		\draw [style=wire, bend left] (11) to (13);
		\draw [style=wire] (11) to (14);
		\draw [style=wire] (14) to (12);
		\draw [style=wire, bend right] (16) to (15);
		\draw [style=wire] (17) to (16);
		\draw [style=wire, bend right] (18) to (19);
		\draw [style=wire, in=150, out=-30, looseness=1.25] (16) to (18);
		\draw [style=wire] (20) to (21);
		\draw [style=wire, in=30, out=-90] (18) to (20);
		\draw [style=wire, in=150, out=-90] (15) to (20);
	\end{pgfonlayer}
\end{tikzpicture}
   \end{array}
\end{align*}
     \item[{\bf [d.5]}] Interchange Rule:
     \begin{align*}
\begin{array}[c]{c}
\xymatrixcolsep{5pc}\xymatrix{\oc A \otimes A \otimes A  \ar[d]_-{\mathsf{d} _A \otimes 1_A}\ar[r]^-{1_{\oc A} \otimes \sigma_{A,A}} & \oc A \otimes A \otimes A \ar[r]^-{\mathsf{d}_A \otimes 1_A} & \oc A \otimes A \ar[d]^-{\mathsf{d}_A} \\
         \oc A \otimes A \ar[rr]_-{\mathsf{d}_A} && \oc A 
  }
   \end{array}&&
   \begin{array}[c]{c}
\begin{tikzpicture}
	\begin{pgfonlayer}{nodelayer}
		\node [style=object] (23) at (14, 1.75) {$A$};
		\node [style=object] (24) at (13.5, 1.75) {$A$};
		\node [style=object] (25) at (13.5, -0.75) {$\oc A$};
		\node [style=object] (26) at (12.5, 1.75) {$\oc A$};
		\node [style=codifferential] (27) at (13.5, 0) {{\bf =\!=\!=\!=}};
		\node [style=codifferential] (28) at (13, 1) {{\bf =\!=\!=\!=}};
		\node [style=object] (29) at (15.25, 1.75) {$\oc A$};
		\node [style=codifferential] (30) at (15.75, 1) {{\bf =\!=\!=\!=}};
		\node [style=codifferential] (31) at (16.25, 0) {{\bf =\!=\!=\!=}};
		\node [style=object] (32) at (16.25, -0.75) {$\oc A$};
		\node [style=object] (33) at (16.25, 1.75) {$A$};
		\node [style=object] (34) at (16.75, 1.75) {$A$};
		\node [style=port] (35) at (14.75, 0.75) {$=$};
	\end{pgfonlayer}
	\begin{pgfonlayer}{edgelayer}
		\draw [style=wire, bend left=15, looseness=1.25] (27) to (28);
		\draw [style=wire, bend right] (27) to (23);
		\draw [style=wire] (25) to (27);
		\draw [style=wire, bend left] (28) to (26);
		\draw [style=wire, bend right] (28) to (24);
		\draw [style=wire, bend left=15, looseness=1.25] (31) to (30);
		\draw [style=wire] (32) to (31);
		\draw [style=wire, bend left] (30) to (29);
		\draw [style=wire, in=45, out=-90] (34) to (30);
		\draw [style=wire, in=45, out=-90, looseness=1.50] (33) to (31);
	\end{pgfonlayer}
\end{tikzpicture}
   \end{array}
\end{align*}
\end{description}
\end{definition}

For lists of many examples of differential categories, we invite the reader to see \cite[]{Blute2019,blute2006differential}. 

\begin{definition} A \textbf{monoidal differential storage category} \cite[Section 4]{blute2006differential} is a monoidal differential category with finite products whose coalgebra modality has Seely isomorphisms. 
\end{definition}

For a monoidal differential storage category, the differential structure can also equivalently be described in terms of a \textbf{codereliction} \cite[Definition 4.11]{blute2006differential}, which is a natural transformation $\eta_A: A \to \oc A$ satisfying the axioms found in \cite[Definition 9]{Blute2019}. By \cite[Theorem 4]{Blute2019}, for coalgebra modalities with Seely isomorphisms (or more generally monoidal coalgebra modalities), there is a bijective correspondence between coderelictions and deriving transformations. Starting with a deriving transformation $\mathsf{d}$, we construct a codereliction as follows: 
\begin{align*}
\eta_A := \xymatrixcolsep{5pc}\xymatrix{A \ar[r]^-{u_A \otimes 1_A} & \oc A \otimes A \ar[r]^-{\mathsf{d}_A} & \oc A } && \begin{array}[c]{c}
\begin{tikzpicture}
	\begin{pgfonlayer}{nodelayer}
		\node [style=circle] (0) at (0, 2.25) {$A$};
		\node [style=circle] (1) at (0, -0.25) {$\oc A$};
		\node [style={component}] (2) at (0, 1) {$\eta$};
	\end{pgfonlayer}
	\begin{pgfonlayer}{edgelayer}
		\draw [style=wire] (0) to (2);
		\draw [style=wire] (2) to (1);
	\end{pgfonlayer}
\end{tikzpicture}
   \end{array}=
   \begin{array}[c]{c}
\begin{tikzpicture}
	\begin{pgfonlayer}{nodelayer}
		\node [style=component] (0) at (0.5, 2) {$u$};
		\node [style=differential] (1) at (1, 1) {{\bf =\!=\!=}};
		\node [style=object] (2) at (1.75, 2.5) {$A$};
		\node [style=object] (3) at (1, 0.25) {$\oc A$};
	\end{pgfonlayer}
	\begin{pgfonlayer}{edgelayer}
		\draw [style=wire, bend left] (1) to (0);
		\draw [style=wire, bend right] (1) to (2);
		\draw [style=wire] (1) to (3);
	\end{pgfonlayer}
\end{tikzpicture}
   \end{array} 
\end{align*}
Conversely, starting with a codereliction $\eta$, we construct a deriving transformation as follows: 
\begin{align*}
 \mathsf{d}_A :=  \xymatrixcolsep{5pc}\xymatrix{\oc A \otimes A \ar[r]^-{1_{\oc A} \otimes \eta_A} & \oc A \otimes \oc A \ar[r]^-{\nabla_A} & \oc A } && \begin{array}[c]{c}
\begin{tikzpicture}
	\begin{pgfonlayer}{nodelayer}
		\node [style=object] (0) at (1.75, 2.75) {$A$};
		\node [style=object] (1) at (1.25, 1.25) {$\oc A$};
		\node [style=integral] (2) at (1.25, 2) {{\bf =\!=\!=\!=}};
		\node [style=object] (3) at (0.75, 2.75) {$\oc A$};
	\end{pgfonlayer}
	\begin{pgfonlayer}{edgelayer}
		\draw [style=wire, bend right] (2) to (0);
		\draw [style=wire] (1) to (2);
		\draw [style=wire, bend left] (2) to (3);
	\end{pgfonlayer}
\end{tikzpicture}
   \end{array}=
   \begin{array}[c]{c}
\begin{tikzpicture}
	\begin{pgfonlayer}{nodelayer}
		\node [style=component] (0) at (2, 2.5) {$\eta$};
		\node [style=duplicate] (1) at (1.25, 1.5) {$\nabla$};
		\node [style=object] (2) at (0.5, 3.25) {$\oc A$};
		\node [style=object] (3) at (1.25, 0.75) {$\oc A$};
		\node [style=object] (4) at (2, 3.25) {$A$};
	\end{pgfonlayer}
	\begin{pgfonlayer}{edgelayer}
		\draw [style=wire] (3) to (1);
		\draw [style=wire, in=-90, out=0, looseness=1.25] (1) to (0);
		\draw [style=wire] (0) to (4);
		\draw [style=wire, in=-90, out=180] (1) to (2);
	\end{pgfonlayer}
\end{tikzpicture}
   \end{array}
\end{align*}
These constructions are inverses of each other.

\subsection{Cartesian Differential Categories}

Another central structure of this paper is a \emph{Cartesian} differential category. For a full detailed introduction to Cartesian differential categories, see \cite[]{blute2009cartesian,garner2021cartesian}. 

The underlying structure of a Cartesian differential category is that of a Cartesian left additive category. A category is said to be \emph{left} additive if it is \emph{skew}-enriched over the category of commutative monoids \cite[Section 2.1]{garner2021cartesian}, or in other words, if each hom-set is a commutative monoid such that pre-composition preserves the additive structure. This allows one to have zero maps and sums of maps while allowing for maps which do not preserve the additive structure. Maps which do preserve the additive structure are called \emph{additive} maps. 

\begin{definition}\label{CLACdef} A \textbf{Cartesian left additive category} \cite[Definition 1.2.1]{blute2009cartesian} is a category $\mathbb{X}$ with finite products such that each hom-set $\mathbb{X}(A,B)$ is a commutative monoid with zero map $0 \in \mathbb{X}(A,B)$ and addition ${+: \mathbb{X}(A,B) \times \mathbb{X}(A,B) \to \mathbb{X}(A,B)}$, $(f,g) \mapsto f +g$, and, such that:
\begin{enumerate}[{\em (i)}]
\item Pre-composition preserves the additive structure; that is, the following equalities hold: $f;0 = 0$ and $f;(g+h) = f;g + f;h$
\item Post-composition by the projection maps preserve the additive structure; that is, the following equalities hold: $0;\pi_i= 0$ and $(f+g);\pi_i = f;\pi_i + g;\pi_i$. 
\end{enumerate}
In a Cartesian left additive category, a map $f: A \to B$ is \textbf{additive} \cite[Definition 1.1.1]{blute2009cartesian} if post-composition by $f$ preserves the additive structure; that is, the following equalities hold: $(g+h);f = g;f + h;f$ and $0;f = 0$ (note that the projection maps are additive). 
\end{definition}

Here are now some important maps for Cartesian differential categories that can be defined in any Cartesian left additive category. In a Cartesian left additive category $\mathbb{X}$: 
\begin{enumerate}[{\em (i)}]
\item \label{injdef} For each pair of objects $A$ and $B$, define the \textbf{injection maps} $\iota_0: A \to A \times B$ and $\iota_1: B \to A \times B$ respectively as $\iota_0 := \langle 1_A, 0 \rangle$ and $\iota_1 := \langle 0, 1_B \rangle$
\item \label{nabladef} For each object $A$, define the \textbf{sum map} $+_A: A \times A \to A$ as $+_A := \pi_0 + \pi_1$. 
\item \label{elldef} For each object $A$, define the \textbf{lifting map} $\ell_A: A \times A \to (A \times A) \times (A \times A)$ as follows $\ell_A := \iota_0 \times \iota_1$. 
\item \label{cdef} For each object $A$ define the \textbf{interchange map} $c_A: (A \times A) \times (A \times A) \to (A \times A) \times (A \times A)$ as follows $c_A : = \left \langle \pi_0 \times \pi_0, \pi_1 \times \pi_1 \right \rangle$. 
\end{enumerate}

Observe that while $c$ is natural in the obvious sense, the same cannot be said for the rest. Indeed, the injection maps $\iota_j$, the sum map $\nabla$, and the lifting map $\ell$ are not natural transformations. In particular, since the injection maps are not natural, it follows that these injection maps do not make the product a coproduct, and therefore not a biproduct. However, the well-known biproduct identities still hold in a Cartesian left additive category. 

\begin{definition}\label{cartdiffdef} A \textbf{Cartesian differential category} (CDC) \cite[Definition 2.1.1]{blute2009cartesian} is a Cartesian left additive category $\mathbb{X}$ equipped with a \textbf{differential combinator} $\mathsf{D}$ which is a family of operators:
\begin{align*}
    \mathsf{D}: \mathbb{X}(A,B) \to \mathbb{X}(A \times A,B) && (f: A\to B) \mapsto (\mathsf{D}[f]: A \times A \to B)
\end{align*}
where $\mathsf{D}[f]$ is called the \textbf{derivative} of $f$, such that the following seven axioms hold:  
\begin{enumerate}[{\bf [CD.1]}] 
\item Additivity of the differentiation: 
\begin{align*} 
\mathsf{D}[f+g] = \mathsf{D}[f] + \mathsf{D}[g] && \mathsf{D}[0]=0 
\end{align*}
\item Additivity of the derivative in its second variable: 
\begin{align*} (1_A \times +_A); \mathsf{D}[f] = (1_A \times \pi_0); \mathsf{D}[f] + (1_A \times \pi_1);\mathsf{D}[f] && \iota_0; \mathsf{D}[f]=0
\end{align*}
\item Coherence with identities and projections: 
\begin{align*} \mathsf{D}[1_A]=\pi_1 && \mathsf{D}[\pi_0] = \pi_1;\pi_0 && \mathsf{D}[\pi_1] = \pi_1;\pi_1 \end{align*}
\item Coherence with pairings: 
\begin{align*}
\mathsf{D}[\langle f, g \rangle] = \langle \mathsf{D}[f] , \mathsf{D}[g] \rangle
\end{align*}
\item Chain rule: 
\begin{align*} \mathsf{D}[fg] = \langle \pi_0 f, \mathsf{D}[f] \rangle; \mathsf{D}[g]
\end{align*}
\item Linearity of the derivative in its second variable: 
\begin{align*} \ell_A; \mathsf{D}\!\left[\mathsf{D}[f] \right] = \mathsf{D}[f]
\end{align*}
\item Symmetry of mixed partial derivatives: 
\begin{align*} c_A; \mathsf{D}\!\left[\mathsf{D}[f] \right]= \mathsf{D}\left[\mathsf{D}[f] \right]
\end{align*} 
\end{enumerate}
\end{definition}

More discussions on the intuition for the differential combinator axioms can be found in \cite[Remark 2.1.3]{blute2009cartesian}. There are many interesting (and sometimes very exotic) examples of Cartesian differential categories in the literature: see \cite[]{cockett2020linearizing,garner2021cartesian}.

\subsection{Linear Fibration of a Cartesian Differential Category}

Just as any monoidal differential category has a canonical fibration associated to it, so too does a Cartesian differential category.  To understand this fibration, we begin by describing what it means for a map in a Cartesian differential category to be linear.  

\begin{definition}\label{def:linear-def}  In a Cartesian differential category $\mathbb{X}$ with differential combinator $\mathsf{D}$, a map $f: A \to B$ is \textbf{linear} \cite[Definition 2.2.1]{blute2009cartesian} if the following diagram commutes: 
  \[  \xymatrixcolsep{5pc}\xymatrix{A \times A \ar[dr]_-{\pi_1} \ar[rr]^-{\mathsf{D}[f]} && B  \\
  & A \ar[ur]_-{f}
  } \]
or equivalently \cite[Lemma 12]{cockett_et_al:LIPIcs:2020:11661}, if the following diagram commutes: 
  \[  \xymatrixcolsep{5pc}\xymatrix{A \ar[dr]_-{\iota_1} \ar[rr]^-{f} && B  \\
  & A \times A \ar[ur]_-{\mathsf{D}[f]}
  } \]
Define the subcategory of linear maps $\mathsf{LIN}[\mathbb{X}]$ to be the category whose objects are the same as $\mathbb{X}$ and whose maps are linear in $\mathbb{X}$.
\end{definition}

A modification of this notion allows one to describe maps which are only ``linear in one variable''.  

\begin{definition}\label{defn:linear_fibration} In a Cartesian differential category $\mathbb{X}$ with differential combinator $\mathsf{D}$, a map ${f\!: X \!\times\! A \to B}$ is \textbf{linear in its second argument $A$} (or \textbf{linear in context $X$}) \cite[Definition 9]{cockett_et_al:LIPIcs:2020:11661} if the following diagram commutes: 
  \[  \xymatrixcolsep{5pc}\xymatrix{X \times (A \times A) \ar[r]^-{\left \langle 1_X \times \pi_0; 0 \times \pi_1 \right \rangle} \ar[d]_-{1_X \times \pi_1} & (X \times A) \times (X \times A) \ar[d]^-{\mathsf{D}[f]} \\
    X \times A \ar[r]_-{f} & B 
  } \]
or equivalently \cite[Lemma 12]{cockett_et_al:LIPIcs:2020:11661}, if the following diagram commutes: 
  \[  \xymatrixcolsep{5pc}\xymatrix{X \times A \ar[dr]_-{\iota_0 \times \iota_1} \ar[rr]^-{f} && B  \\
  & (X \times A) \times (X \times A) \ar[ur]_-{\mathsf{D}[f]}
  } \]
\end{definition}

With such maps we can define the canonical fibration associated to a Cartesian differential category.  

\begin{definition} For a Cartesian differential category $\mathbb{X}$ with differential combinator $\mathsf{D}$, define the category $\mathcal{L}[\mathbb{X}]$ as follows: 
\begin{enumerate}[{\em (i)}]
\item The objects of $\mathcal{L}[\mathbb{X}]$ are pairs of elements $(X,A)$ of $\mathbb{X}$, that is, $Ob\left( \mathcal{L}[\mathbb{X}] \right) = Ob\left( \mathbb{X} \right) \times Ob\left( \mathbb{X} \right)$;
\item The maps of $\mathcal{L}[\mathbb{X}]$ are pairs of maps $(f,g): (X,A) \to (Y,B)$ consisting of an arbitrary map ${f: X \to Y}$ and a map $g: X \times A \to B$ which is linear in context $X$; 
\item The identity map of $(X,A)$ is the pair $(1_X, \pi_1): (X,A) \to (X,A)$;
\item The composition of maps $(f,g): (X,A) \to (Y,B)$ and $(h, k): (Y,B) \to (Z,C)$ is defined as follows: 
  \[ (f,g); (h, k) = \left( (f;h) , \xymatrixcolsep{3pc}\xymatrix{ X \times A \ar[r]^-{\langle \pi_0, 1_{X \times A} \rangle } & X \times (X \times A) \ar[r]^-{f \times g} &Y \times B \ar[r]^-{k} & C} \right) \]
  \end{enumerate}
Let $\mathsf{p}: \mathcal{L}[\mathbb{X}] \to \mathbb{X}$ be the forgetful functor defined on objects as $\mathsf{p}(X,A) = X$ and on maps as $\mathsf{p}(f,g) = f$. 
\end{definition}

Note that this is a subcategory of the simple fibration over $\mathbb{X}$ \cite[Definition 1.3.1]{jacobs1999categorical}.  It is then straightforward to show that:

\begin{proposition}\label{prop:linear_fibration}  Let $\mathbb{X}$ be a Cartesian differential category with differential combinator $\mathsf{D}$. Then the forgetful functor ${\mathsf{p}: \mathcal{L}[\mathbb{X}] \to \mathbb{X}}$ is a fibration where the Cartesian maps are those of the form $(f, \pi_1): (X,A) \to (Y,A)$. 
\end{proposition}

It will be useful to have an explicit description of the fibres of this fibration:

\begin{lemma}\label{lemma:L_fibres} Let $\mathbb{X}$ be a Cartesian differential category with differential combinator $\mathsf{D}$. For any object $X \in \mathbb{X}$, the fibre over $X$ of the fibration $\mathsf{p}: \mathcal{L}[\mathbb{X}] \to \mathbb{X}$ is written as $\mathcal{L}[X]$ and given by
\begin{enumerate}[{\em (i)}]
\item The objects of $\mathcal{L}[X]$ are the same as the objects of $\mathbb{X}$, that is, $Ob\left( \mathcal{L}[X] \right) = Ob\left( \mathbb{X} \right)$;
\item The maps of $\mathcal{L}[X]$ are maps $f: X \times A \to B$ which are linear in context $X$; 
\item The identity map of $A$ is defined as $\pi_1: X \times A \to A$;
\item The composition of maps $f: X \times A \to B$ and $g: X \times B \to C$ is defined as follows: 
\begin{align*}
 \xymatrixcolsep{5pc}\xymatrix{ X \times A \ar[r]^-{\langle \pi_0, 1_{X \times A} \rangle} &X \times (X \times A) \ar[r]^-{1_X \times f} &X \times B \ar[r]^-{g} & C} \end{align*}
\end{enumerate}
For every map $h: X \to Y$, define the \textbf{substitution functor} $h^\ast: \mathcal{L}[Y] \to \mathcal{L}[X]$ on objects as $h^\ast(A) = A$ and on maps $f: Y \times A \to B$ as follows: 
\begin{align*}
h^\ast(f) := \xymatrixcolsep{5pc}\xymatrix{ X \times A \ar[r]^-{h \times 1_A} & Y \times A \ar[r]^-{f} & B} \end{align*}
Furthermore, note that for the terminal object $\top$, there is an isomorphism $\mathcal{L}[\top] \cong \mathsf{LIN}[\mathbb{X}]$. 
\end{lemma}

\subsection{The coKleisli construction}\label{cokleislisection}

In this section we review a very important source of Cartesian differential categories: the coKleisli categories of monoidal differential categories. 

Before constructing the differential combinator, we must first describe the additive structure of the coKleisli category. So let $(\oc, \delta, \varepsilon)$ be a comonad on a category $\mathbb{X}$ with finite biproducts. Then $\mathbb{X}_\oc$ is a Cartesian left additive category \cite[Proposition 1.3.3]{blute2009cartesian} where the additive structure is defined as follows: 
\begin{align*}
\llbracket f+g \rrbracket = \llbracket f \rrbracket + \llbracket g \rrbracket && \llbracket 0 \rrbracket = 0 \end{align*}
Furthermore, the injection maps, sum maps, lifting maps, and interchange maps in the coKleisli category are easily computed out to be: 
\begin{align*}
\llbracket \iota_0 \rrbracket = \varepsilon_A; \iota_0 && \llbracket \iota_1 \rrbracket = \varepsilon_B; \iota_1 && \llbracket +_A \rrbracket = \varepsilon_{A \times A}; +_A &&  \llbracket \ell_A \rrbracket = \varepsilon_{A \times A}; \ell_A &&  \llbracket c_A \rrbracket = \varepsilon_{(A \times A) \times (A \times A)}; c_A     \end{align*}
If one starts with a differential category, then using the deriving transformation, we are able to construct a differential combinator for the coKleisli category. 

\begin{proposition}\label{coKleisliCDC} \cite[Proposition 3.2.1]{blute2009cartesian} Let $\mathbb{X}$ be a monoidal differential category with coalgebra modality $(\oc, \delta, \varepsilon, \Delta, e)$ and deriving transformation $\mathsf{d}: \oc A \otimes A \to \oc A$, and finite (bi)products (which we denote here using the product notation). Then the coKleisli category $\mathbb{X}_\oc$ is a Cartesian differential category with Cartesian left additive structure defined above and differential combinator $\mathsf{D}$ defined as follows on a coKleisli map $\llbracket f \rrbracket: \oc A \to B$: 
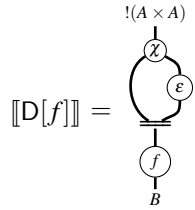
 \begin{align*}
 \llbracket \mathsf{D}[f] \rrbracket := \xymatrixcolsep{3pc}\xymatrix{\oc(A \times A) \ar[r]^-{\chi_{A \times A}} & \oc A \otimes \oc A \ar[r]^-{1_{\oc A} \otimes \varepsilon_A} & \oc A \otimes A \ar[r]^-{\mathsf{d}_A} & \oc A \ar[r]^-{\llbracket f \rrbracket} & B 
 } \end{align*}
  \begin{align*}
\begin{array}[c]{c}
\llbracket \mathsf{D}[f] \rrbracket 
   \end{array}=
   \begin{array}[c]{c}
\begin{tikzpicture}
	\begin{pgfonlayer}{nodelayer}
		\node [style=component] (0) at (7.5, 1.75) {$\varepsilon$};
		\node [style=differential] (1) at (7, 1) {{\bf =\!=\!=\!=}};
		\node [style=object] (2) at (7, -0.5) {$B$};
		\node [style=component] (4) at (7, 0.25) {$f$};
		\node [style=duplicate] (6) at (7, 2.5) {$\chi$};
		\node [style=object] (7) at (7, 3.25) {$\oc (A \times A)$};
	\end{pgfonlayer}
	\begin{pgfonlayer}{edgelayer}
		\draw [style=wire, in=-90, out=30, looseness=1.50] (1) to (0);
		\draw [style=wire] (1) to (4);
		\draw [style=wire] (4) to (2);
		\draw [style=wire] (7) to (6);
		\draw [style=wire, in=90, out=-30, looseness=1.25] (6) to (0);
		\draw [style=wire, in=150, out=-150] (6) to (1);
	\end{pgfonlayer}
\end{tikzpicture}
   \end{array}
\end{align*}
where $\chi_{A \times A}: \oc (A \times A) \to \oc A \otimes \oc A$ is defined as in Definition \ref{Seelydef}.  
\end{proposition} 

It is important to note that the above proposition does not require the coalgebra modality to be monoidal or, equivalently, to have Seely isomorphisms.  In addition, we could have also expressed the differential combinator of the coKleisli category in terms of the coderiving transformation $\mathsf{d}^\circ$ (Definition \ref{dcircdef}) as follows on a coKleisli map $\llbracket f \rrbracket: \oc A \to B$: 
 \begin{align*} \llbracket \mathsf{D}[f] \rrbracket := \xymatrixcolsep{2.25pc}\xymatrix{\oc(A \times A) \ar[r]^-{\mathsf{d}^\circ_{A \times A}} & \oc (A \times A) \otimes (A \times A) \ar[r]^-{\oc(\pi_0) \otimes \pi_1} & \oc A \otimes A \ar[r]^-{\mathsf{d}_A} & \oc A \ar[r]^-{\llbracket f \rrbracket} & B 
 }
  \end{align*}
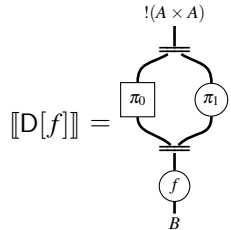
  \begin{align*}
\begin{array}[c]{c}
\llbracket \mathsf{D}[f] \rrbracket 
   \end{array}=
   \begin{array}[c]{c}
\begin{tikzpicture}
	\begin{pgfonlayer}{nodelayer}
		\node [style=differential] (1) at (7, 0.75) {{\bf =\!=\!=\!=}};
		\node [style=object] (2) at (7, -0.75) {$B$};
		\node [style=component] (4) at (7, 0) {$f$};
		\node [style=object] (7) at (7, 3.5) {$\oc (A \times A)$};
		\node [style=differential] (8) at (7, 2.75) {{\bf =\!=\!=\!=}};
		\node [style=component] (9) at (7.75, 1.75) {$\pi_1$};
		\node [style=function2] (10) at (6.25, 1.75) {$\pi_0$};
	\end{pgfonlayer}
	\begin{pgfonlayer}{edgelayer}
		\draw [style=wire] (1) to (4);
		\draw [style=wire] (4) to (2);
		\draw [style=wire] (7) to (8);
		\draw [style=wire, in=90, out=-150, looseness=1.25] (8) to (10);
		\draw [style=wire, in=135, out=-90] (10) to (1);
		\draw [style=wire, in=-90, out=45, looseness=1.25] (1) to (9);
		\draw [style=wire, in=90, out=-30, looseness=1.25] (8) to (9);
	\end{pgfonlayer}
\end{tikzpicture}
   \end{array}
\end{align*}
We now turn our attention to giving an explicit description of the linear maps in the coKleisli category. 

\begin{lemma}\label{lem:cokleisli-linear}  Let $\mathbb{X}$ be a monoidal differential category with coalgebra modality $(\oc, \delta, \varepsilon, \Delta, e)$ and deriving transformation $\mathsf{d}: \oc A \otimes A \to \oc A$, and finite (bi)products. Then: 
\begin{enumerate}[{\em (i)}]
\item\label{lem:cokleisli-linear.i} A coKleisli map $\llbracket f \rrbracket: \oc A \to B$ is linear in $\mathbb{X}_\oc$ if and only if the following diagram commutes: 
 \begin{align*}
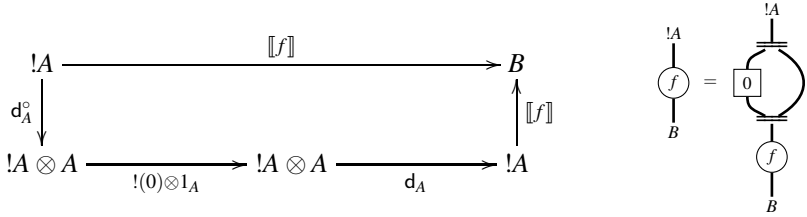
   \begin{array}[c]{c}
  \xymatrixcolsep{5pc}\xymatrix{ \oc A \ar[rr]^-{\llbracket f \rrbracket} \ar[d]_-{\mathsf{d}^\circ_A} && B \\
   \oc A \otimes A \ar[r]_-{\oc(0) \otimes 1_A} & \oc A \otimes A \ar[r]_-{\mathsf{d}_A} & \oc A \ar[u]_-{\llbracket f \rrbracket} }
    %  \xymatrixcolsep{5pc}\xymatrix{ \oc A  \ar[drrr]_-{\llbracket f \rrbracket}  \ar[r]^-{\Delta_A} & \oc A \otimes \oc A \ar[r]^-{\oc(0) \otimes \varepsilon_A} & \oc A \otimes A \ar[r]^-{\mathsf{d}}  & \oc A \ar[d]^-{\llbracket f \rrbracket} \\
%&& & B  }  
 \end{array} && 
   \begin{array}[c]{c}
\begin{tikzpicture}
	\begin{pgfonlayer}{nodelayer}
		\node [style=differential] (0) at (6.75, 1) {{\bf =\!=\!=\!=}};
		\node [style=object] (1) at (6.75, -0.75) {$B$};
		\node [style=component] (2) at (6.75, 0.25) {$f$};
		\node [style=object] (3) at (6.75, 3.25) {$ \oc A$};
		\node [style=function2] (6) at (6.25, 1.75) {$0$};
		\node [style=object] (7) at (4.75, 2.75) {$\oc A$};
		\node [style=object] (8) at (4.75, 0.75) {$B$};
		\node [style=component] (9) at (4.75, 1.75) {$f$};
		\node [style=object] (10) at (5.5, 1.75) {$=$};
		\node [style=differential] (11) at (6.75, 2.5) {{\bf =\!=\!=\!=}};
	\end{pgfonlayer}
	\begin{pgfonlayer}{edgelayer}
		\draw [style=wire] (0) to (2);
		\draw [style=wire] (2) to (1);
		\draw [style=wire, in=150, out=-90, looseness=1.25] (6) to (0);
		\draw [style=wire] (7) to (9);
		\draw [style=wire] (9) to (8);
		\draw [style=wire] (3) to (11);
		\draw [style=wire, in=90, out=-150, looseness=1.25] (11) to (6);
		\draw [style=wire, in=30, out=-30, looseness=1.50] (11) to (0);
	\end{pgfonlayer}
\end{tikzpicture}
   \end{array}
\end{align*}
\item \label{lem:cokleisli-linear.ii} For every map $g: A \to B$ in $\mathbb{X}$, $\llbracket \mathsf{F}_\oc(g) \rrbracket = \varepsilon_A ; g: \oc A \to B$ is linear in $\mathbb{X}_\oc$. Therefore there is a functor $\mathsf{F}_{\mathsf{L}}: \mathbb{X} \to \mathsf{LIN}[\mathbb{X}_\oc]$ defined on objects as $\mathsf{F}_{\mathsf{L}}(A) = A$ and on maps $g: A \to B$ as $\llbracket \mathsf{F}_{\mathsf{L}}(g) \rrbracket = \llbracket \mathsf{F}_\oc(g) \rrbracket = \varepsilon_A ; g$. 
\end{enumerate}
\end{lemma}
\begin{proof} For $(i)$, first observe that for any coKleisli map $\llbracket k \rrbracket: \oc (A \times A) \to B$, precomposing by $\llbracket \iota_1 \rrbracket = \varepsilon_A; \iota_1$ is equal to $\llbracket \iota_1; k \rrbracket = \oc(\iota_1); \llbracket k \rrbracket$. Therefore, for any coKleisli map $\llbracket f \rrbracket: \oc A \to B$, we compute the following: 
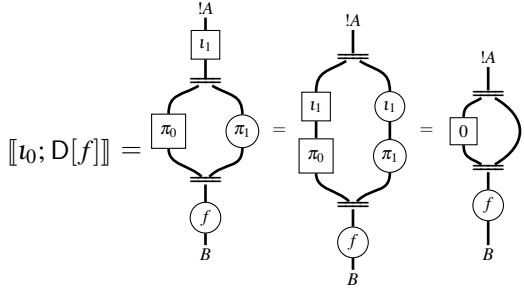
\begin{align*}
\begin{array}[c]{c}
\llbracket \iota_0; \mathsf{D}[f] \rrbracket 
   \end{array}=
   \begin{array}[c]{c}
\begin{tikzpicture}
	\begin{pgfonlayer}{nodelayer}
		\node [style=differential] (0) at (5.5, 2.5) {{\bf =\!=\!=\!=}};
		\node [style=object] (1) at (5.5, 0.75) {$B$};
		\node [style=component] (2) at (5.5, 1.75) {$f$};
		\node [style=object] (3) at (5.5, 4.75) {$ \oc A$};
		\node [style=function2] (6) at (5, 3.25) {$0$};
		\node [style=object] (10) at (1.25, 3.25) {$=$};
		\node [style=differential] (11) at (5.5, 4) {{\bf =\!=\!=\!=}};
		\node [style=differential] (12) at (-0.25, 2.25) {{\bf =\!=\!=\!=}};
		\node [style=object] (13) at (-0.25, 0.75) {$B$};
		\node [style=component] (14) at (-0.25, 1.5) {$f$};
		\node [style=object] (15) at (-0.25, 5.75) {$\oc A$};
		\node [style=differential] (16) at (-0.25, 4.25) {{\bf =\!=\!=\!=}};
		\node [style=component] (17) at (0.5, 3.25) {$\pi_1$};
		\node [style=function2] (18) at (-1, 3.25) {$\pi_0$};
		\node [style=function2] (19) at (-0.25, 5) {$\iota_1$};
		\node [style=object] (23) at (2.75, 5.5) {$\oc A$};
		\node [style=differential] (28) at (2.75, 4.75) {{\bf =\!=\!=\!=}};
		\node [style=component] (29) at (3.5, 3.75) {$\iota_1$};
		\node [style=function2] (30) at (2, 3.75) {$\iota_1$};
		\node [style=differential] (31) at (2.75, 1.75) {{\bf =\!=\!=\!=}};
		\node [style=object] (32) at (2.75, 0.25) {$B$};
		\node [style=component] (33) at (2.75, 1) {$f$};
		\node [style=component] (34) at (3.5, 2.75) {$\pi_1$};
		\node [style=function2] (35) at (2, 2.75) {$\pi_0$};
		\node [style=object] (36) at (4.25, 3.25) {$=$};
	\end{pgfonlayer}
	\begin{pgfonlayer}{edgelayer}
		\draw [style=wire] (0) to (2);
		\draw [style=wire] (2) to (1);
		\draw [style=wire, in=150, out=-90, looseness=1.25] (6) to (0);
		\draw [style=wire] (3) to (11);
		\draw [style=wire, in=90, out=-150, looseness=1.25] (11) to (6);
		\draw [style=wire, in=30, out=-30, looseness=1.50] (11) to (0);
		\draw [style=wire] (12) to (14);
		\draw [style=wire] (14) to (13);
		\draw [style=wire, in=90, out=-150, looseness=1.25] (16) to (18);
		\draw [style=wire, in=135, out=-90] (18) to (12);
		\draw [style=wire, in=-90, out=45, looseness=1.25] (12) to (17);
		\draw [style=wire, in=90, out=-30, looseness=1.25] (16) to (17);
		\draw [style=wire] (15) to (19);
		\draw [style=wire] (19) to (16);
		\draw [style=wire, in=90, out=-150, looseness=1.25] (28) to (30);
		\draw [style=wire, in=90, out=-30, looseness=1.25] (28) to (29);
		\draw [style=wire] (23) to (28);
		\draw [style=wire] (31) to (33);
		\draw [style=wire] (33) to (32);
		\draw [style=wire, in=135, out=-90] (35) to (31);
		\draw [style=wire, in=-90, out=45, looseness=1.25] (31) to (34);
		\draw [style=wire] (30) to (35);
		\draw [style=wire] (29) to (34);
	\end{pgfonlayer}
\end{tikzpicture}
   \end{array}
\end{align*}
%\begin{align*} 
%\llbracket \iota_1; \mathsf{D}[f] \rrbracket &=~ \delta_A; \oc\left( \llbracket \iota_1 \rrbracket \right);  \llbracket  \mathsf{D}[f] \rrbracket  \\
%&=~ \delta_A; \oc(\varepsilon_A; \iota_1);  \llbracket  \mathsf{D}[f] \rrbracket  \\
%&=~ \delta_A; \oc(\varepsilon_A); \oc(\iota_1);  \llbracket  \mathsf{D}[f] \rrbracket  \\
%&=~  \oc(\iota_1);  \llbracket  \mathsf{D}[f] \rrbracket \tag{Comonad Identity} \\
%&=~  \oc(\iota_1); \mathsf{d}^\circ_{A \times A}; \left( \oc(\pi_0) \otimes \pi_1 \right); \mathsf{d}_A; \llbracket f \rrbracket \\ 
%&=~ \mathsf{d}^\circ_A;  \left(  \oc(\iota_1) \otimes \iota_1 \right);  \left( \oc(\pi_0) \otimes \pi_1 \right); \mathsf{d}_A; \llbracket f \rrbracket \tag{Nat. of $\mathsf{d}^\circ$} \\
%&=~  \mathsf{d}^\circ_A; \left( \oc(\iota_1; \pi_0) \otimes (\iota_1; \pi_1) \right);  \mathsf{d}_A; \llbracket f \rrbracket \\
%&=~   \mathsf{d}^\circ_A; \left( \oc(0) \otimes 1_A \right);  \mathsf{d}_A; \llbracket f \rrbracket \tag{Biproduct Identity}
%\end{align*}
So $\llbracket \iota_1; \mathsf{D}[f] \rrbracket =  \mathsf{d}^\circ_A; \left( \oc(0) \otimes 1_A \right);  \mathsf{d}_A; \llbracket f \rrbracket$. Therefore, by definition, $\llbracket f \rrbracket: \oc A \to B$ is linear if and only if $ \mathsf{d}^\circ_A; \left( \oc(0) \otimes 1_A \right);  \mathsf{d}_A; \llbracket f \rrbracket  = \llbracket \iota_1; \mathsf{D}[f] \rrbracket  = \llbracket f \rrbracket$. For $(ii)$, we use the linear rule \textbf{[d.3]} to compute: 
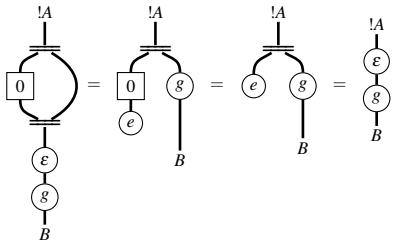
\begin{align*}
\begin{tikzpicture}
	\begin{pgfonlayer}{nodelayer}
		\node [style=differential] (0) at (-4.25, 3) {{\bf =\!=\!=\!=}};
		\node [style=object] (1) at (-4.25, 0.75) {$B$};
		\node [style=object] (3) at (-4.25, 5.25) {$ \oc A$};
		\node [style=function2] (6) at (-4.75, 3.75) {$0$};
		\node [style=differential] (11) at (-4.25, 4.5) {{\bf =\!=\!=\!=}};
		\node [style=object] (36) at (-3.25, 3.75) {$=$};
		\node [style=component] (47) at (-4.25, 2.25) {$\varepsilon$};
		\node [style=component] (48) at (-4.25, 1.5) {$g$};
		\node [style=object] (49) at (-2, 5.25) {$ \oc A$};
		\node [style=function2] (50) at (-2.5, 3.75) {$0$};
		\node [style=differential] (51) at (-2, 4.5) {{\bf =\!=\!=\!=}};
		\node [style=component] (52) at (-2.5, 3) {$e$};
		\node [style=object] (53) at (-1.5, 2.25) {$B$};
		\node [style=component] (54) at (-1.5, 3.75) {$g$};
		\node [style=object] (55) at (0.5, 5.25) {$ \oc A$};
		\node [style=component] (56) at (0, 3.75) {$e$};
		\node [style=differential] (57) at (0.5, 4.5) {{\bf =\!=\!=\!=}};
		\node [style=object] (59) at (1, 2.5) {$B$};
		\node [style=component] (60) at (1, 3.75) {$g$};
		\node [style=object] (61) at (-0.75, 3.75) {$=$};
		\node [style=object] (62) at (2.5, 2.75) {$B$};
		\node [style=component] (63) at (2.5, 4.25) {$\varepsilon$};
		\node [style=component] (64) at (2.5, 3.5) {$g$};
		\node [style=object] (65) at (1.75, 3.75) {$=$};
		\node [style=object] (66) at (2.5, 5) {$ \oc A$};
	\end{pgfonlayer}
	\begin{pgfonlayer}{edgelayer}
		\draw [style=wire, in=150, out=-90, looseness=1.25] (6) to (0);
		\draw [style=wire] (3) to (11);
		\draw [style=wire, in=90, out=-150, looseness=1.25] (11) to (6);
		\draw [style=wire, in=30, out=-30, looseness=1.50] (11) to (0);
		\draw [style=wire] (0) to (47);
		\draw [style=wire] (47) to (48);
		\draw [style=wire] (48) to (1);
		\draw [style=wire] (49) to (51);
		\draw [style=wire, in=90, out=-150, looseness=1.25] (51) to (50);
		\draw [style=wire] (50) to (52);
		\draw [style=wire] (54) to (53);
		\draw [style=wire, in=90, out=-30, looseness=1.25] (51) to (54);
		\draw [style=wire] (55) to (57);
		\draw [style=wire, in=90, out=-150, looseness=1.25] (57) to (56);
		\draw [style=wire] (60) to (59);
		\draw [style=wire, in=90, out=-30, looseness=1.25] (57) to (60);
		\draw [style=wire] (63) to (64);
		\draw [style=wire] (64) to (62);
		\draw [style=wire] (66) to (63);
	\end{pgfonlayer}
\end{tikzpicture}
\end{align*}
Therefore, $\llbracket \mathsf{F}_\oc(g) \rrbracket = \varepsilon_A ; g$ is linear. As a consequence, $\mathsf{F}_{\mathsf{L}}: \mathbb{X} \to \mathsf{LIN}[\mathbb{X}_\oc]$ is well-defined and is functor since $\mathsf{F}_\oc: \mathbb{X} \to \mathbb{X}_\oc$ is a functor. 
\end{proof} 

It is important to note that for an arbitrary differential category $\mathbb{X}$ with finite products, not every linear map in the coKleisli category is of the form $\varepsilon_B; g$. Therefore, $\mathsf{LIN}[\mathbb{X}_\oc]$ is not necessarily isomorphic to the base category $\mathbb{X}$. However, for differential storage categories, the desired isomorphism holds, which is a fundamental concept in differential linear logic. 

\begin{corollary}\label{cor:seely-lin} Let $\mathbb{X}$ be a monoidal differential storage category with coalgebra modality $(\oc, \delta, \varepsilon, \Delta, e)$ with Seely isomorphisms, deriving transformation $\mathsf{d}: \oc A \otimes A \to \oc A$ (or equivalently codereliction $\eta_A: A \to \oc A$), and finite (bi)products. Then: 
\begin{enumerate}[{\em (i)}]
\item A coKleisli map $\llbracket f \rrbracket: \oc A \to B$ is linear in $\mathbb{X}_\oc$ if and only if the following diagram commutes: 
 \begin{align*}
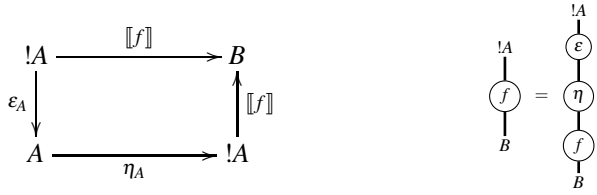
   \begin{array}[c]{c}
  \xymatrixcolsep{5pc}\xymatrix{ \oc A \ar[r]^-{\llbracket f \rrbracket} \ar[d]_-{\varepsilon_A} & B \\
  A \ar[r]_-{\eta_A} & \oc A \ar[u]_-{\llbracket f \rrbracket} }   \end{array} && 
   \begin{array}[c]{c}
\begin{tikzpicture}
	\begin{pgfonlayer}{nodelayer}
		\node [style=component] (0) at (6.25, 1.75) {$\eta$};
		\node [style=object] (1) at (6.25, 0) {$B$};
		\node [style=component] (2) at (6.25, 0.75) {$f$};
		\node [style=object] (3) at (6.25, 3.5) {$ \oc A$};
		\node [style=component] (4) at (6.25, 2.75) {$\varepsilon$};
		\node [style=object] (5) at (4.75, 2.75) {$\oc A$};
		\node [style=object] (6) at (4.75, 0.75) {$B$};
		\node [style=component] (7) at (4.75, 1.75) {$f$};
		\node [style=object] (8) at (5.5, 1.75) {$=$};
	\end{pgfonlayer}
	\begin{pgfonlayer}{edgelayer}
		\draw [style=wire] (0) to (2);
		\draw [style=wire] (2) to (1);
		\draw [style=wire] (3) to (4);
		\draw [style=wire] (5) to (7);
		\draw [style=wire] (7) to (6);
		\draw [style=wire] (4) to (0);
	\end{pgfonlayer}
\end{tikzpicture}
   \end{array}
\end{align*}
\item $\mathsf{F}_{\mathsf{L}}: \mathbb{X} \to \mathsf{LIN}[\mathbb{X}_\oc]$ is an isomorphism with inverse $\mathsf{F}^{-1}_{\mathsf{L}}: \mathsf{LIN}[\mathbb{X}_\oc] \to \mathbb{X}$ defined on objects as $\mathsf{F}^{-1}_{\mathsf{L}}(A) = A$ and on maps $\llbracket f \rrbracket: \oc A \to B$ as $\mathsf{F}^{-1}_{L}\left( \llbracket f \rrbracket \right) = \eta_A; \llbracket f \rrbracket$. 
\end{enumerate}
In other words, a coKleisli map $\llbracket f \rrbracket: \oc A \to B$ is linear in $\mathbb{X}_\oc$ if and only if $\llbracket f \rrbracket = \varepsilon_A; g$ for some (necessarily unique) map $g: A \to B$ in $\mathbb{X}$. 
\end{corollary}
\begin{proof} For $(i)$, recall that in an additive bialgebra modality, $\oc(0) = e_A; u_A$. Therefore, for any coKleisli map $\llbracket f \rrbracket: \oc A \to B$ we have that: 
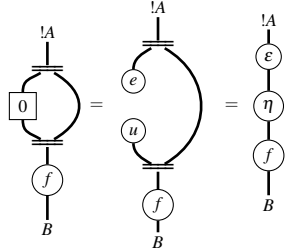
\begin{align*}
\begin{tikzpicture}
	\begin{pgfonlayer}{nodelayer}
		\node [style=component] (0) at (-1.75, 4.5) {$\eta$};
		\node [style=object] (1) at (-1.75, 2.5) {$B$};
		\node [style=component] (2) at (-1.75, 3.5) {$f$};
		\node [style=object] (3) at (-1.75, 6.25) {$ \oc A$};
		\node [style=component] (4) at (-1.75, 5.5) {$\varepsilon$};
		\node [style=object] (8) at (-5.25, 4.5) {$=$};
		\node [style=differential] (9) at (-6.25, 3.75) {{\bf =\!=\!=\!=}};
		\node [style=object] (10) at (-6.25, 2) {$B$};
		\node [style=component] (11) at (-6.25, 3) {$f$};
		\node [style=object] (12) at (-6.25, 6) {$ \oc A$};
		\node [style=function2] (13) at (-6.75, 4.5) {$0$};
		\node [style=differential] (18) at (-6.25, 5.25) {{\bf =\!=\!=\!=}};
		\node [style=differential] (19) at (-4, 3.25) {{\bf =\!=\!=\!=}};
		\node [style=object] (20) at (-4, 1.75) {$B$};
		\node [style=component] (21) at (-4, 2.5) {$f$};
		\node [style=object] (22) at (-4, 6.5) {$ \oc A$};
		\node [style=differential] (24) at (-4, 5.75) {{\bf =\!=\!=\!=}};
		\node [style=component] (25) at (-4.5, 5) {$e$};
		\node [style=component] (26) at (-4.5, 4) {$u$};
		\node [style=object] (27) at (-2.5, 4.5) {$=$};
	\end{pgfonlayer}
	\begin{pgfonlayer}{edgelayer}
		\draw [style=wire] (0) to (2);
		\draw [style=wire] (2) to (1);
		\draw [style=wire] (3) to (4);
		\draw [style=wire] (4) to (0);
		\draw [style=wire] (9) to (11);
		\draw [style=wire] (11) to (10);
		\draw [style=wire, in=150, out=-90, looseness=1.25] (13) to (9);
		\draw [style=wire] (12) to (18);
		\draw [style=wire, in=90, out=-150, looseness=1.25] (18) to (13);
		\draw [style=wire, in=30, out=-30, looseness=1.50] (18) to (9);
		\draw [style=wire] (19) to (21);
		\draw [style=wire] (21) to (20);
		\draw [style=wire] (22) to (24);
		\draw [style=wire, in=30, out=-30, looseness=1.25] (24) to (19);
		\draw [style=wire, in=150, out=-90] (26) to (19);
		\draw [style=wire, in=90, out=-150] (24) to (25);
	\end{pgfonlayer}
\end{tikzpicture}
\end{align*}
So $\mathsf{d}^\circ_A; \left( \oc(0) \otimes 1_A \right);  \mathsf{d}_A; \llbracket f \rrbracket = \eta_A; \varepsilon_A; \llbracket f \rrbracket$. Then by Lemma \ref{lem:cokleisli-linear}.(\ref{lem:cokleisli-linear.i}),  $\llbracket f \rrbracket: \oc A \to B$ is linear if and only if $\llbracket f \rrbracket= \mathsf{d}^\circ_A; \left( \oc(0) \otimes 1_A \right);  \mathsf{d}_A; \llbracket f \rrbracket = \eta_A; \varepsilon_A; \llbracket f \rrbracket$. 

For (ii), usually we would first have to check that  $\mathsf{F}^{-1}_{\mathsf{L}}$ is a functor; that is, $\mathsf{F}^{-1}_{\mathsf{L}}$ preserves composition and identities. However, it turns out that there is a way around this by applying \cite[Chapter IV, Theorem 2]{mac2013categories} to isomorphisms. Briefly, if $\mathsf{F}: \mathbb{X} \to \mathbb{Y}$ is a functor and $\mathsf{G}: \mathbb{Y} \to \mathbb{X}$ is a well-defined mapping on objects and maps such that $\mathsf{F} \circ \mathsf{G} = 1_\mathbb{Y}$ and $\mathsf{G} \circ \mathsf{F} = 1_\mathbb{X}$, then $\mathsf{G}$ is a functor and so $\mathsf{F}$ is an isomorphism with inverse $\mathsf{G}$. By Lemma \ref{lem:cokleisli-linear}.(\ref{lem:cokleisli-linear.ii}), $\mathsf{F}_{\mathsf{L}}: \mathbb{X} \to \mathsf{LIN}[\mathbb{X}_\oc]$ is a functor, so it remains to show that $\mathsf{F}^{-1}_{\mathsf{L}} \circ \mathsf{F}_\mathsf{L} = 1_\mathbb{X}$ and $\mathsf{F}_{\mathsf{L}} \circ \mathsf{F}^{-1}_\mathsf{L} = 1_{\mathsf{LIN}[\mathbb{X}_\oc]}$. Starting with the former, on objects this is immediate, $\mathsf{F}^{-1}_{\mathsf{L}}\mathsf{F}_\mathsf{L} (A) = A$. While on maps, recall that the linear rule for the codereliction says that $\eta_A; \varepsilon_A = 1_A$, therefore:
\[\mathsf{F}^{-1}_{\mathsf{L}}\mathsf{F}_\mathsf{L} (f) = \eta_A; \varepsilon_A; f = f\] 
So $\mathsf{F}^{-1}_{\mathsf{L}} \circ \mathsf{F}_\mathsf{L} = 1_\mathbb{X}$. For the other direction, on objects this is again immediate $\mathsf{F}_{\mathsf{L}}\mathsf{F}^{-1}_\mathsf{L} (A) = A$. For a linear coKleisli map $\llbracket f \rrbracket: \oc A \to B$, by Lemma \ref{lem:cokleisli-linear}.(\ref{lem:cokleisli-linear.i}), we have that:
\[\mathsf{F}_{\mathsf{L}}\mathsf{F}^{-1}_\mathsf{L}  \left( \llbracket f \rrbracket \right) = \varepsilon_A ; \eta_A ; \llbracket f \rrbracket = \llbracket f \rrbracket\]
So $\mathsf{F}_{\mathsf{L}} \circ \mathsf{F}^{-1}_\mathsf{L} = 1_{\mathsf{LIN}[\mathbb{X}_\oc]}$. Therefore, $\mathsf{F}^{-1}_{\mathsf{L}}: \mathsf{LIN}[\mathbb{X}_\oc] \to \mathbb{X}$ is a functor and is an inverse of $\mathsf{F}_{\mathsf{L}}: \mathbb{X} \to \mathsf{LIN}[\mathbb{X}_\oc]$. So we conclude that $\mathbb{X} \cong \mathsf{LIN}[\mathbb{X}_\oc]$.\end{proof} 

\subsection{Equivalence of Linear Fibrations}
\label{sec:fibration_equivalence}
Consider a monoidal differential category with finite products.  On the one hand, we have the fibration ${\mathsf{p}_\oc: \mathcal{L}_\oc[\mathbb{X}] \to \mathbb{X}_\oc}$
of Proposition \ref{prop:context_fibration} associated to any coalgebra modality.  On the other hand, by Proposition \ref{coKleisliCDC}, $\mathbb{X}_{\oc}$ is a Cartesian differential category, and so we also have its associated linear fibration ${\mathsf{p}: \mathcal{L}[\mathbb{X}_\oc] \to \mathbb{X}_\oc}$ of Proposition \ref{prop:linear_fibration}. The objective of this section is to show that they are in fact isomorphic (as fibrations over $\mathbb{X}_{\oc}$). We begin by providing an explicit description of maps which are linear in context in the coKleisli category. 

\begin{lemma}\label{lem:cokleisli-linearcontext}  Let $\mathbb{X}$ be a differential category with coalgebra modality $(\oc, \delta, \varepsilon, \Delta, e)$ and deriving transformation $\mathsf{d}: \oc A \otimes A \to \oc A$, and finite (bi)products. Then: 
\begin{enumerate}[{\em (i)}]
\item \label{lem:cokleisli-linearcontext.i} For every object $X$ and $A$, the following diagram commutes: 
\begin{align*}
  \begin{array}[c]{c}
  \xymatrixcolsep{5pc}\xymatrix{ \oc X \otimes A  \ar[d]_-{\oc(\iota_0) \otimes \iota_1} \ar@{=}[rr]^-{} & & \oc X \otimes A \\
  \oc(X \times A) \!\otimes\! (X \times A) \ar[r]_-{\mathsf{d}_{X \times A}} & \oc (X \times A) \ar[r]_-{\mathsf{d}^\circ_{X \times A}} & \oc(X \times A) \!\otimes\!  (X \times A) \ar[u]_-{\oc(\pi_0) \otimes \pi_1} }   
    \end{array} \end{align*}
    \begin{align*}   \begin{array}[c]{c}
\begin{tikzpicture}
	\begin{pgfonlayer}{nodelayer}
		\node [style=differential] (20) at (0.75, 4.5) {{\bf =\!=\!=\!=}};
		\node [style=component] (21) at (1.5, 3.5) {$\pi_1$};
		\node [style=function2] (22) at (0, 3.5) {$\pi_0$};
		\node [style=differential] (23) at (0.75, 5.25) {{\bf =\!=\!=\!=}};
		\node [style=component] (26) at (1.5, 6.25) {$\iota_1$};
		\node [style=function2] (27) at (0, 6.25) {$\iota_0$};
		\node [style=object] (43) at (2.25, 4.75) {$=$};
		\node [style=object] (67) at (0, 7) {$\oc X$};
		\node [style=object] (68) at (1.5, 7) {$A$};
		\node [style=object] (69) at (0, 2.75) {$\oc X$};
		\node [style=object] (70) at (1.5, 2.75) {$A$};
		\node [style=object] (77) at (3.25, 7) {$\oc X$};
		\node [style=object] (78) at (4.75, 7) {$A$};
		\node [style=object] (79) at (3.25, 2.75) {$\oc X$};
		\node [style=object] (80) at (4.75, 2.75) {$A$};
	\end{pgfonlayer}
	\begin{pgfonlayer}{edgelayer}
		\draw [style=wire, in=90, out=-150, looseness=1.25] (20) to (22);
		\draw [style=wire, in=90, out=-30, looseness=1.25] (20) to (21);
		\draw [style=wire, in=135, out=-90] (27) to (23);
		\draw [style=wire, in=-90, out=45, looseness=1.25] (23) to (26);
		\draw [style=wire] (23) to (20);
		\draw [style=wire] (22) to (69);
		\draw [style=wire] (21) to (70);
		\draw [style=wire] (67) to (27);
		\draw [style=wire] (68) to (26);
		\draw [style=wire] (77) to (79);
		\draw [style=wire] (78) to (80);
	\end{pgfonlayer}
\end{tikzpicture}
   \end{array}
\end{align*}
\item\label{lem:cokleisli-linearcontext.ii} A coKleisli map $\llbracket f \rrbracket: \oc(X \times A) \to B$ is linear in context $X$ if and only if the following diagram commutes: 
 \begin{align*}
   \begin{array}[c]{c}
  \xymatrixcolsep{3.5pc}\xymatrix{ \oc (X \times A)  \ar[d]_-{\mathsf{d}^\circ_{X \times A}} \ar[rrr]^-{\llbracket f \rrbracket} & && B \\
  \oc(X \times A) \!\otimes\! (X \times A) \ar[r]_-{\oc(\pi_0) \otimes \pi_1} & \oc X \otimes A \ar[r]_-{\oc(\iota_0) \otimes \iota_1} & \oc(X \times A) \!\otimes\!  (X \times A) \ar[r]_-{\mathsf{d}_{X \times A}} & \oc(X \times A) \ar[u]_-{\llbracket f \rrbracket} }   
%  \xymatrixcolsep{3pc}\xymatrix{ \oc (X \times A)  \ar[ddrrr]_-{\llbracket h \rrbracket}  \ar[r]^-{\chi_{X \times A}} & \oc X \otimes \oc A \ar[r]^-{1_{\oc X} \otimes \varepsilon_A} & \oc X \otimes A \ar[r]^-{\oc(\iota_0) \otimes \iota_1} & \oc(X \times A) \otimes (X \times A) \ar[d]^-{\mathsf{d}} \\
%  &&& \oc (X \times A) \ar[d]^-{\llbracket h \rrbracket} \\
%&& &  B }
    \end{array}
    \end{align*}
    \begin{align*} 
    \begin{array}[c]{c}
\begin{tikzpicture}
	\begin{pgfonlayer}{nodelayer}
		\node [style=object] (15) at (-1.75, 5.5) {$\oc(X \times A)$};
		\node [style=differential] (16) at (-1.75, 4.75) {{\bf =\!=\!=\!=}};
		\node [style=component] (17) at (-1, 3.75) {$\pi_1$};
		\node [style=function2] (18) at (-2.5, 3.75) {$\pi_0$};
		\node [style=differential] (19) at (-1.75, 1.75) {{\bf =\!=\!=\!=}};
		\node [style=object] (20) at (-1.75, 0) {$B$};
		\node [style=component] (21) at (-1.75, 1) {$f$};
		\node [style=component] (22) at (-1, 2.75) {$\iota_1$};
		\node [style=function2] (23) at (-2.5, 2.75) {$\iota_0$};
		\node [style=object] (32) at (-4, 4.25) {$\oc (X \times A)$};
		\node [style=object] (33) at (-4, 2.25) {$B$};
		\node [style=component] (34) at (-4, 3.25) {$f$};
		\node [style=object] (35) at (-3.25, 3.25) {$=$};
	\end{pgfonlayer}
	\begin{pgfonlayer}{edgelayer}
		\draw [style=wire, in=90, out=-150, looseness=1.25] (16) to (18);
		\draw [style=wire, in=90, out=-30, looseness=1.25] (16) to (17);
		\draw [style=wire] (15) to (16);
		\draw [style=wire] (19) to (21);
		\draw [style=wire] (21) to (20);
		\draw [style=wire, in=150, out=-90] (23) to (19);
		\draw [style=wire, in=-90, out=30] (19) to (22);
		\draw [style=wire] (18) to (23);
		\draw [style=wire] (17) to (22);
		\draw [style=wire] (32) to (34);
		\draw [style=wire] (34) to (33);
	\end{pgfonlayer}
\end{tikzpicture}
   \end{array} 
\end{align*}
\item\label{lem:cokleisli-linearcontext.iii} For every map $g: \oc X \otimes A \to B$ in $\mathbb{X}$, the composite: 
\begin{align*}
   \begin{array}[c]{c}
\xymatrixcolsep{3pc}\xymatrix{\oc(X \times A) \ar[r]^-{\mathsf{d}^\circ_{X \times A}} & \oc (X \times A) \times (X \times A) \ar[r]^-{\oc(\pi_0) \times \pi_1} & \oc X \otimes A \ar[r]^-{g} & B } 
   \end{array} &&    \begin{array}[c]{c}
\begin{tikzpicture}
	\begin{pgfonlayer}{nodelayer}
		\node [style=component] (0) at (7, 0.75) {$g$};
		\node [style=object] (2) at (7, 0) {$B$};
		\node [style=object] (3) at (7, 3.5) {$\oc (X \times A)$};
		\node [style=differential] (4) at (7, 2.75) {{\bf =\!=\!=\!=}};
		\node [style=component] (5) at (7.75, 1.75) {$\pi_1$};
		\node [style=function2] (6) at (6.25, 1.75) {$\pi_0$};
	\end{pgfonlayer}
	\begin{pgfonlayer}{edgelayer}
		\draw [style=wire] (0) to (2);
		\draw [style=wire] (3) to (4);
		\draw [style=wire, in=90, out=-150, looseness=1.25] (4) to (6);
		\draw [style=wire, in=135, out=-90] (6) to (0);
		\draw [style=wire, in=-90, out=45, looseness=1.25] (0) to (5);
		\draw [style=wire, in=90, out=-30, looseness=1.25] (4) to (5);
	\end{pgfonlayer}
\end{tikzpicture}
   \end{array}
\end{align*}
is linear in context $X$ in $\mathbb{X}_\oc$. 
\end{enumerate}
\end{lemma}
\begin{proof} For $(i)$, recall the following useful compatibility relation between the deriving transformation and coderiving transformation \cite[Proposition 4.1]{cockett_lemay_2018}: 
\begin{align*}
\begin{tikzpicture}
	\begin{pgfonlayer}{nodelayer}
		\node [style=differential] (20) at (0.75, 4.5) {{\bf =\!=\!=\!=}};
		\node [style=object] (21) at (1.5, 3.5) {$A$};
		\node [style=object] (22) at (0, 3.5) {$\oc A$};
		\node [style=differential] (23) at (0.75, 5.25) {{\bf =\!=\!=\!=}};
		\node [style=object] (26) at (1.5, 6.25) {$A$};
		\node [style=object] (27) at (0, 6.25) {$\oc A$};
		\node [style=object] (43) at (2.25, 4.75) {$=$};
		\node [style=differential] (44) at (3.25, 4.25) {{\bf =\!=\!=\!=}};
		\node [style=object] (45) at (4.75, 3.5) {$A$};
		\node [style=object] (46) at (3.25, 3.5) {$\oc A$};
		\node [style=differential] (47) at (3.25, 5.5) {{\bf =\!=\!=\!=}};
		\node [style=object] (48) at (4.75, 6.25) {$A$};
		\node [style=object] (49) at (3.25, 6.25) {$\oc A$};
		\node [style=object] (50) at (5.5, 4.75) {$+$};
		\node [style=object] (52) at (8, 3.5) {$A$};
		\node [style=object] (53) at (6.5, 3.5) {$\oc A$};
		\node [style=object] (55) at (8, 6.25) {$A$};
		\node [style=object] (56) at (6.5, 6.25) {$\oc A$};
	\end{pgfonlayer}
	\begin{pgfonlayer}{edgelayer}
		\draw [style=wire, in=90, out=-150, looseness=1.25] (20) to (22);
		\draw [style=wire, in=90, out=-30, looseness=1.25] (20) to (21);
		\draw [style=wire, in=135, out=-90] (27) to (23);
		\draw [style=wire, in=-90, out=45, looseness=1.25] (23) to (26);
		\draw [style=wire] (23) to (20);
		\draw [style=wire] (44) to (46);
		\draw [style=wire] (49) to (47);
		\draw [style=wire] (56) to (53);
		\draw [style=wire] (55) to (52);
		\draw [style=wire, in=45, out=-90] (48) to (44);
		\draw [style=wire, in=90, out=-60, looseness=1.50] (47) to (45);
		\draw [style=wire, in=135, out=-135, looseness=1.25] (47) to (44);
	\end{pgfonlayer}
\end{tikzpicture}
\end{align*}
Then by using the above identity and the biproduct coherences, we compute: 
\begin{align*}
\begin{tikzpicture}
	\begin{pgfonlayer}{nodelayer}
		\node [style=differential] (0) at (1.25, 9.5) {{\bf =\!=\!=\!=}};
		\node [style=component] (1) at (2, 8.5) {$\pi_1$};
		\node [style=function2] (2) at (0.5, 8.5) {$\pi_0$};
		\node [style=differential] (3) at (1.25, 10.25) {{\bf =\!=\!=\!=}};
		\node [style=component] (4) at (2, 11.25) {$\iota_1$};
		\node [style=function2] (5) at (0.5, 11.25) {$\iota_0$};
		\node [style=object] (6) at (2.5, 10) {$=$};
		\node [style=object] (7) at (0.5, 12) {$\oc X$};
		\node [style=object] (8) at (2, 12) {$A$};
		\node [style=object] (9) at (0.5, 7.75) {$\oc X$};
		\node [style=object] (10) at (2, 7.75) {$A$};
		\node [style=differential] (11) at (3.5, 9.25) {{\bf =\!=\!=\!=}};
		\node [style=component] (12) at (4.75, 8.5) {$\pi_1$};
		\node [style=function2] (13) at (3.5, 8.5) {$\pi_0$};
		\node [style=differential] (14) at (3.5, 10.5) {{\bf =\!=\!=\!=}};
		\node [style=component] (15) at (4.75, 11.25) {$\iota_1$};
		\node [style=function2] (16) at (3.5, 11.25) {$\iota_0$};
		\node [style=object] (17) at (3.5, 12) {$\oc X$};
		\node [style=object] (18) at (4.75, 12) {$A$};
		\node [style=object] (19) at (3.5, 7.75) {$\oc X$};
		\node [style=object] (20) at (4.75, 7.75) {$A$};
		\node [style=object] (21) at (5.25, 10) {$+$};
		\node [style=component] (22) at (7, 9) {$\pi_1$};
		\node [style=function2] (23) at (6, 9) {$\pi_0$};
		\node [style=component] (24) at (7, 10.75) {$\iota_1$};
		\node [style=function2] (25) at (6, 10.75) {$\iota_0$};
		\node [style=object] (26) at (6, 12) {$\oc X$};
		\node [style=object] (27) at (7, 12) {$A$};
		\node [style=object] (28) at (6, 7.75) {$\oc X$};
		\node [style=object] (29) at (7, 7.75) {$A$};
		\node [style=object] (30) at (7.75, 10) {$=$};
		\node [style=differential] (31) at (9.25, 8.5) {{\bf =\!=\!=\!=}};
		\node [style=component] (32) at (10.5, 8.5) {$\pi_1$};
		\node [style=function2] (33) at (8.75, 9.5) {$\pi_0$};
		\node [style=differential] (34) at (9.25, 10.5) {{\bf =\!=\!=\!=}};
		\node [style=component] (35) at (10.5, 11.25) {$\iota_1$};
		\node [style=function2] (36) at (9.25, 11.25) {$\iota_0$};
		\node [style=object] (37) at (9.25, 12) {$\oc X$};
		\node [style=object] (38) at (10.5, 12) {$A$};
		\node [style=object] (39) at (9.25, 7.75) {$\oc X$};
		\node [style=object] (40) at (10.5, 7.75) {$A$};
		\node [style=object] (41) at (11, 10) {$+$};
		\node [style=object] (42) at (11.75, 12) {$\oc X$};
		\node [style=object] (43) at (12.75, 12) {$A$};
		\node [style=object] (44) at (11.75, 7.75) {$\oc X$};
		\node [style=object] (45) at (12.75, 7.75) {$A$};
		\node [style=component] (46) at (10.5, 10.25) {$\pi_0$};
		\node [style=differential] (47) at (1.75, 3.25) {{\bf =\!=\!=\!=}};
		\node [style=component] (48) at (3, 3.25) {$\pi_1$};
		\node [style=function2] (49) at (1.25, 4.25) {$\pi_0$};
		\node [style=differential] (50) at (1.75, 5.25) {{\bf =\!=\!=\!=}};
		\node [style=function2] (51) at (1.75, 6) {$\iota_0$};
		\node [style=object] (52) at (1.75, 6.75) {$\oc X$};
		\node [style=object] (53) at (3, 6.75) {$A$};
		\node [style=object] (54) at (1.75, 2.5) {$\oc X$};
		\node [style=object] (55) at (3, 2.5) {$A$};
		\node [style=object] (56) at (3.5, 4.75) {$+$};
		\node [style=object] (57) at (4.25, 6.75) {$\oc X$};
		\node [style=object] (58) at (5.25, 6.75) {$A$};
		\node [style=object] (59) at (4.25, 2.5) {$\oc X$};
		\node [style=object] (60) at (5.25, 2.5) {$A$};
		\node [style=component] (61) at (3, 5) {$0$};
		\node [style=object] (62) at (0.25, 4.75) {$=$};
		\node [style=object] (63) at (5.75, 4.75) {$=$};
		\node [style=object] (64) at (7.75, 6.75) {$\oc X$};
		\node [style=object] (65) at (8.75, 6.75) {$A$};
		\node [style=object] (66) at (7.75, 2.5) {$\oc X$};
		\node [style=object] (67) at (8.75, 2.5) {$A$};
		\node [style=object] (68) at (9.25, 4.75) {$=$};
		\node [style=object] (69) at (10, 6.75) {$\oc X$};
		\node [style=object] (70) at (11, 6.75) {$A$};
		\node [style=object] (71) at (10, 2.5) {$\oc X$};
		\node [style=object] (72) at (11, 2.5) {$A$};
		\node [style=object] (73) at (6.5, 4.75) {$0$};
		\node [style=object] (74) at (7.25, 4.75) {$+$};
	\end{pgfonlayer}
	\begin{pgfonlayer}{edgelayer}
		\draw [style=wire, in=90, out=-150, looseness=1.25] (0) to (2);
		\draw [style=wire, in=90, out=-30, looseness=1.25] (0) to (1);
		\draw [style=wire, in=135, out=-90] (5) to (3);
		\draw [style=wire, in=-90, out=45, looseness=1.25] (3) to (4);
		\draw [style=wire] (3) to (0);
		\draw [style=wire] (2) to (9);
		\draw [style=wire] (1) to (10);
		\draw [style=wire] (7) to (5);
		\draw [style=wire] (8) to (4);
		\draw [style=wire] (11) to (13);
		\draw [style=wire] (16) to (14);
		\draw [style=wire, in=135, out=-135] (14) to (11);
		\draw [style=wire] (13) to (19);
		\draw [style=wire] (12) to (20);
		\draw [style=wire] (17) to (16);
		\draw [style=wire] (18) to (15);
		\draw [style=wire, in=-90, out=30] (11) to (15);
		\draw [style=wire, in=90, out=-30] (14) to (12);
		\draw [style=wire] (23) to (28);
		\draw [style=wire] (22) to (29);
		\draw [style=wire] (26) to (25);
		\draw [style=wire] (27) to (24);
		\draw [style=wire] (25) to (23);
		\draw [style=wire] (24) to (22);
		\draw [style=wire] (36) to (34);
		\draw [style=wire] (32) to (40);
		\draw [style=wire] (37) to (36);
		\draw [style=wire] (38) to (35);
		\draw [style=wire, in=90, out=-30] (34) to (32);
		\draw [style=wire] (42) to (44);
		\draw [style=wire] (43) to (45);
		\draw [style=wire, in=90, out=-135] (34) to (33);
		\draw [style=wire, in=135, out=-90] (33) to (31);
		\draw [style=wire] (31) to (39);
		\draw [style=wire] (35) to (46);
		\draw [style=wire, in=45, out=-90] (46) to (31);
		\draw [style=wire] (51) to (50);
		\draw [style=wire] (48) to (55);
		\draw [style=wire] (52) to (51);
		\draw [style=wire, in=90, out=-30] (50) to (48);
		\draw [style=wire] (57) to (59);
		\draw [style=wire] (58) to (60);
		\draw [style=wire, in=90, out=-135] (50) to (49);
		\draw [style=wire, in=135, out=-90] (49) to (47);
		\draw [style=wire] (47) to (54);
		\draw [style=wire, in=45, out=-90] (61) to (47);
		\draw [style=wire] (53) to (61);
		\draw [style=wire] (64) to (66);
		\draw [style=wire] (65) to (67);
		\draw [style=wire] (69) to (71);
		\draw [style=wire] (70) to (72);
	\end{pgfonlayer}
\end{tikzpicture}
\end{align*}
 For $(ii)$, first note that for ${\llbracket k \rrbracket: \oc\left( (X \times A) \times (X \times A) \right) \to B}$, precomposing by ${\llbracket \iota_0 \times \iota_1 \rrbracket = \varepsilon_{X \times A}; (\iota_0 \times \iota_1)}$ is equal to $\llbracket (\iota_0 \times \iota_1); k \rrbracket = \oc (\iota_0 \times \iota_1); \llbracket k \rrbracket$. Therefore, for any coKleisli map $\llbracket f \rrbracket: \oc(X \times A) \to B$, we compute: 
\begin{align*}
\begin{array}[c]{c}
\llbracket (\iota_0 \times \iota_1); \mathsf{D}[f] \rrbracket 
   \end{array}=
   \begin{array}[c]{c}
\begin{tikzpicture}
	\begin{pgfonlayer}{nodelayer}
		\node [style=object] (41) at (4.75, 6.75) {$=$};
		\node [style=differential] (43) at (3.25, 5.75) {{\bf =\!=\!=\!=}};
		\node [style=object] (44) at (3.25, 4.25) {$B$};
		\node [style=component] (45) at (3.25, 5) {$f$};
		\node [style=object] (46) at (3.25, 10) {$\oc(X \times A)$};
		\node [style=differential] (47) at (3.25, 7.75) {{\bf =\!=\!=\!=}};
		\node [style=component] (48) at (4, 6.75) {$\pi_1$};
		\node [style=function2] (49) at (2.5, 6.75) {$\pi_0$};
		\node [style=function3] (50) at (3.25, 8.75) {$\iota_0 \times \iota_1$};
		\node [style=object] (51) at (6.75, 9) {$\oc(X \times A)$};
		\node [style=differential] (52) at (6.75, 8.25) {{\bf =\!=\!=\!=}};
		\node [style=component] (53) at (7.5, 7) {$\iota_0 \times \iota_1$};
		\node [style=function3] (54) at (6, 7) {$\iota_0 \times \iota_1$};
		\node [style=differential] (55) at (6.75, 4.75) {{\bf =\!=\!=\!=}};
		\node [style=object] (56) at (6.75, 3.25) {$B$};
		\node [style=component] (57) at (6.75, 4) {$f$};
		\node [style=component] (58) at (7.5, 5.75) {$\pi_1$};
		\node [style=function2] (59) at (6, 5.75) {$\pi_0$};
		\node [style=object] (60) at (8.75, 6.75) {$=$};
		\node [style=object] (61) at (10.5, 9) {$\oc(X \times A)$};
		\node [style=differential] (62) at (10.5, 8.25) {{\bf =\!=\!=\!=}};
		\node [style=component] (63) at (11.25, 7.25) {$\pi_1$};
		\node [style=function2] (64) at (9.75, 7.25) {$\pi_0$};
		\node [style=differential] (65) at (10.5, 5.25) {{\bf =\!=\!=\!=}};
		\node [style=object] (66) at (10.5, 3.75) {$B$};
		\node [style=component] (67) at (10.5, 4.5) {$f$};
		\node [style=component] (68) at (11.25, 6.25) {$\iota_1$};
		\node [style=function2] (69) at (9.75, 6.25) {$\iota_0$};
	\end{pgfonlayer}
	\begin{pgfonlayer}{edgelayer}
		\draw [style=wire] (43) to (45);
		\draw [style=wire] (45) to (44);
		\draw [style=wire, in=90, out=-150, looseness=1.25] (47) to (49);
		\draw [style=wire, in=135, out=-90] (49) to (43);
		\draw [style=wire, in=-90, out=45, looseness=1.25] (43) to (48);
		\draw [style=wire, in=90, out=-30, looseness=1.25] (47) to (48);
		\draw [style=wire] (46) to (50);
		\draw [style=wire] (50) to (47);
		\draw [style=wire, in=90, out=-150, looseness=1.25] (52) to (54);
		\draw [style=wire, in=90, out=-30, looseness=1.25] (52) to (53);
		\draw [style=wire] (51) to (52);
		\draw [style=wire] (55) to (57);
		\draw [style=wire] (57) to (56);
		\draw [style=wire, in=135, out=-90] (59) to (55);
		\draw [style=wire, in=-90, out=45, looseness=1.25] (55) to (58);
		\draw [style=wire] (54) to (59);
		\draw [style=wire] (53) to (58);
		\draw [style=wire, in=90, out=-150, looseness=1.25] (62) to (64);
		\draw [style=wire, in=90, out=-30, looseness=1.25] (62) to (63);
		\draw [style=wire] (61) to (62);
		\draw [style=wire] (65) to (67);
		\draw [style=wire] (67) to (66);
		\draw [style=wire, in=135, out=-90] (69) to (65);
		\draw [style=wire, in=-90, out=45, looseness=1.25] (65) to (68);
		\draw [style=wire] (64) to (69);
		\draw [style=wire] (63) to (68);
	\end{pgfonlayer}
\end{tikzpicture}
   \end{array}
\end{align*}
So 
	\[ \llbracket (\iota_0 \times \iota_1); \mathsf{D}[f] \rrbracket  = \mathsf{d}^\circ_{X \times A}; (\oc( \pi_0) \otimes \pi_1); (\oc( \iota_0 \otimes \iota_1); \mathsf{d}_{X \times A}; \llbracket f \rrbracket. \]
Therefore, by definition, $\llbracket f \rrbracket: \oc(X \times A) \to B$ is linear in context $X$ if and only if 
	\[ \llbracket f \rrbracket  = \llbracket (\iota_0 \times \iota_1); \mathsf{D}[f] \rrbracket  = \mathsf{d}^\circ_{X \times A}; (\oc( \pi_0) \otimes \pi_1); (\oc( \iota_0 \otimes \iota_1); \mathsf{d}_{X \times A}; \llbracket f \rrbracket. \]
	
For $(iii)$, it is automatic by $(i)$ that we have that: 
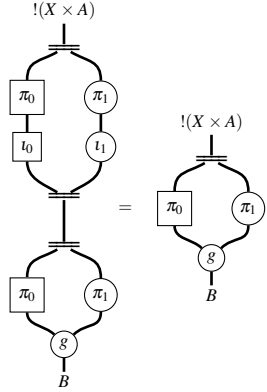
\[ \begin{tikzpicture}
	\begin{pgfonlayer}{nodelayer}
		\node [style=component] (0) at (7, 0.75) {$g$};
		\node [style=object] (1) at (7, 0) {$B$};
		\node [style=differential] (3) at (7, 2.75) {{\bf =\!=\!=\!=}};
		\node [style=component] (4) at (7.75, 1.75) {$\pi_1$};
		\node [style=function2] (5) at (6.25, 1.75) {$\pi_0$};
		\node [style=object] (6) at (7, 7.5) {$\oc(X \times A)$};
		\node [style=differential] (7) at (7, 6.75) {{\bf =\!=\!=\!=}};
		\node [style=component] (8) at (7.75, 5.75) {$\pi_1$};
		\node [style=function2] (9) at (6.25, 5.75) {$\pi_0$};
		\node [style=differential] (10) at (7, 3.75) {{\bf =\!=\!=\!=}};
		\node [style=component] (13) at (7.75, 4.75) {$\iota_1$};
		\node [style=function2] (14) at (6.25, 4.75) {$\iota_0$};
		\node [style=object] (15) at (8.25, 3.5) {$=$};
		\node [style=component] (16) at (10, 2.5) {$g$};
		\node [style=object] (17) at (10, 1.75) {$B$};
		\node [style=differential] (18) at (10, 4.5) {{\bf =\!=\!=\!=}};
		\node [style=component] (19) at (10.75, 3.5) {$\pi_1$};
		\node [style=function2] (20) at (9.25, 3.5) {$\pi_0$};
		\node [style=object] (21) at (10, 5.25) {$\oc(X \times A)$};
	\end{pgfonlayer}
	\begin{pgfonlayer}{edgelayer}
		\draw [style=wire] (0) to (1);
		\draw [style=wire, in=90, out=-150, looseness=1.25] (3) to (5);
		\draw [style=wire, in=135, out=-90] (5) to (0);
		\draw [style=wire, in=-90, out=45, looseness=1.25] (0) to (4);
		\draw [style=wire, in=90, out=-30, looseness=1.25] (3) to (4);
		\draw [style=wire, in=90, out=-150, looseness=1.25] (7) to (9);
		\draw [style=wire, in=90, out=-30, looseness=1.25] (7) to (8);
		\draw [style=wire] (6) to (7);
		\draw [style=wire, in=150, out=-90] (14) to (10);
		\draw [style=wire, in=-90, out=30] (10) to (13);
		\draw [style=wire] (9) to (14);
		\draw [style=wire] (8) to (13);
		\draw [style=wire] (10) to (3);
		\draw [style=wire] (16) to (17);
		\draw [style=wire, in=90, out=-150, looseness=1.25] (18) to (20);
		\draw [style=wire, in=135, out=-90] (20) to (16);
		\draw [style=wire, in=-90, out=45, looseness=1.25] (16) to (19);
		\draw [style=wire, in=90, out=-30, looseness=1.25] (18) to (19);
		\draw [style=wire] (21) to (18);
	\end{pgfonlayer}
\end{tikzpicture}\] 
Then by $(ii)$, it follows that $\mathsf{d}^\circ_{X \times A}; (\oc(\pi_0) \otimes \pi_1); g$ is linear in context $X$. %Lastly $(iv)$ is a consequence of $(i)$, $(ii)$, and $(iii)$ -- though we omit the proof here and will instead prove it in the proof of Theorem \ref{thm:fibration_equivalence}. 
\end{proof} 

Below we will show that, in fact, a coKleisli map $\llbracket f \rrbracket: \oc(X \times A) \to B$ is linear in context $X$ if and only if it is of the form $\llbracket f \rrbracket = \mathsf{d}^\circ_{X \times A}; (\oc(\pi_0) \otimes \pi_1); g$ for some (necessarily unique) map $g: \oc X \otimes A \to B$. If we have the Seely isomorphisms, we may also re-express linearity in context using the codereliction. 

\begin{corollary} Let $\mathbb{X}$ be a differential storage category with coalgebra modality $(\oc, \delta, \varepsilon, \Delta, e)$ with Seely isomorphisms, deriving transformation $\mathsf{d}: \oc A \otimes A \to \oc A$ (or equivalently codereliction $\eta_A: A \to \oc A$), and finite (bi)products. Then a coKleisli map $\llbracket f \rrbracket: \oc(X \times A) \to B$ is linear in context $X$ if and only if the following diagram commutes: 
 \begin{align*}
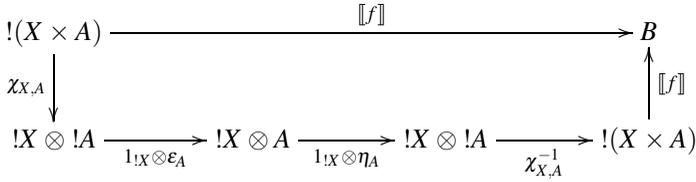
   \begin{array}[c]{c}
     \xymatrixcolsep{3pc}\xymatrix{ \oc (X \times A)  \ar[d]_-{\chi_{X,A}} \ar[rrr]^-{\llbracket f \rrbracket} & && B \\
  \oc X \otimes \oc A \ar[r]_-{1_{\oc X} \otimes \varepsilon_A} & \oc X \otimes A \ar[r]_-{1_{\oc X} \otimes \eta_A} & \oc X \otimes \oc A \ar[r]_-{\chi^{-1}_{X, A}} & \oc(X \times A) \ar[u]_-{\llbracket f \rrbracket} }   
%  \xymatrixcolsep{3pc}\xymatrix{ \oc (X \times A)  \ar[ddrrr]_-{\llbracket h \rrbracket}  \ar[r]^-{\chi_{X \times A}} & \oc X \otimes \oc A \ar[r]^-{1_{\oc X} \otimes \varepsilon_A} & \oc X \otimes A \ar[r]^-{1_{\oc X} \otimes \eta_A} & \oc X \otimes \oc A \ar[d]^-{\chi^{-1}_{X,A}} \\
%  &&& \oc (X \times A) \ar[d]^-{\llbracket h \rrbracket} \\
% && &  B  }
    \end{array}
\end{align*}
\begin{align*}
    \begin{array}[c]{c}
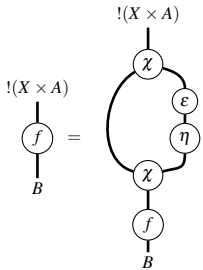
\begin{tikzpicture}
	\begin{pgfonlayer}{nodelayer}
		\node [style=component] (0) at (7, 1) {$\chi$};
		\node [style=object] (1) at (7, -0.75) {$B$};
		\node [style=component] (2) at (7, 0) {$f$};
		\node [style=object] (3) at (7, 4.25) {$\oc(X \times A)$};
		\node [style=component] (4) at (7, 3.25) {$\chi$};
		\node [style=component] (5) at (7.75, 1.75) {$\eta$};
		\node [style=object] (6) at (4.75, 2.75) {$\oc (X \times A)$};
		\node [style=object] (7) at (4.75, 0.75) {$B$};
		\node [style=component] (8) at (4.75, 1.75) {$f$};
		\node [style=object] (9) at (5.5, 1.75) {$=$};
		\node [style=component] (10) at (7.75, 2.5) {$\varepsilon$};
	\end{pgfonlayer}
	\begin{pgfonlayer}{edgelayer}
		\draw [style=wire] (0) to (2);
		\draw [style=wire] (2) to (1);
		\draw [style=wire] (3) to (4);
		\draw [style=wire, in=-90, out=15, looseness=1.50] (0) to (5);
		\draw [style=wire] (6) to (8);
		\draw [style=wire] (8) to (7);
		\draw [style=wire] (10) to (5);
		\draw [style=wire, in=90, out=-15] (4) to (10);
		\draw [style=wire, in=165, out=-150] (4) to (0);
	\end{pgfonlayer}
\end{tikzpicture}
   \end{array}
\end{align*}
\end{corollary}
\begin{proof} By expressing the deriving transformation in terms of the multiplication and codereliction, for any coKleisli map $\llbracket f \rrbracket: \oc(X \times A) \to B$, we compute: 
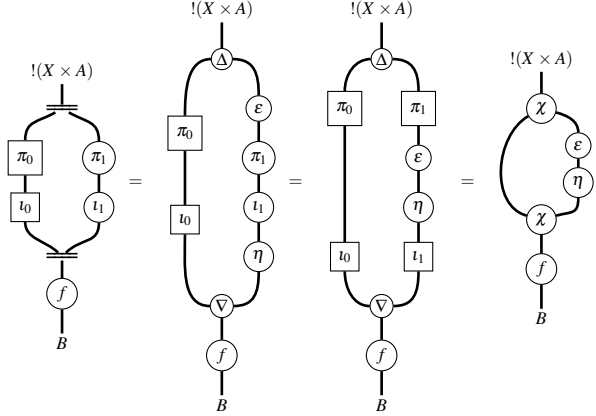
\begin{align*}
\begin{tikzpicture}
	\begin{pgfonlayer}{nodelayer}
		\node [style=object] (19) at (0.75, 9) {$\oc(X \times A)$};
		\node [style=differential] (20) at (0.75, 8.25) {{\bf =\!=\!=\!=}};
		\node [style=component] (21) at (1.5, 7.25) {$\pi_1$};
		\node [style=function2] (22) at (0, 7.25) {$\pi_0$};
		\node [style=differential] (23) at (0.75, 5.25) {{\bf =\!=\!=\!=}};
		\node [style=object] (24) at (0.75, 3.5) {$B$};
		\node [style=component] (25) at (0.75, 4.5) {$f$};
		\node [style=component] (26) at (1.5, 6.25) {$\iota_1$};
		\node [style=function2] (27) at (0, 6.25) {$\iota_0$};
		\node [style=object] (28) at (4, 10.25) {$\oc(X \times A)$};
		\node [style=object] (33) at (4, 2.25) {$B$};
		\node [style=component] (34) at (4, 3.25) {$f$};
		\node [style=component] (35) at (4.75, 6.25) {$\iota_1$};
		\node [style=function2] (36) at (3.25, 6) {$\iota_0$};
		\node [style=component] (37) at (4.75, 5.25) {$\eta$};
		\node [style=duplicate] (38) at (4, 4.25) {$\nabla$};
		\node [style=component] (39) at (4.75, 7.25) {$\pi_1$};
		\node [style=function2] (40) at (3.25, 7.75) {$\pi_0$};
		\node [style=component] (41) at (4.75, 8.25) {$\varepsilon$};
		\node [style=duplicate] (42) at (4, 9.25) {$\Delta$};
		\node [style=object] (43) at (2.25, 6.75) {$=$};
		\node [style=object] (44) at (7.25, 10.25) {$\oc(X \times A)$};
		\node [style=object] (45) at (7.25, 2.25) {$B$};
		\node [style=component] (46) at (7.25, 3.25) {$f$};
		\node [style=component] (47) at (8, 6.25) {$\eta$};
		\node [style=function2] (48) at (6.5, 5.25) {$\iota_0$};
		\node [style=function2] (49) at (8, 5.25) {$\iota_1$};
		\node [style=duplicate] (50) at (7.25, 4.25) {$\nabla$};
		\node [style=component] (51) at (8, 7.25) {$\varepsilon$};
		\node [style=function2] (52) at (6.5, 8.25) {$\pi_0$};
		\node [style=function2] (53) at (8, 8.25) {$\pi_1$};
		\node [style=duplicate] (54) at (7.25, 9.25) {$\Delta$};
		\node [style=object] (55) at (5.5, 6.75) {$=$};
		\node [style=component] (56) at (10.5, 6) {$\chi$};
		\node [style=object] (57) at (10.5, 4) {$B$};
		\node [style=component] (58) at (10.5, 5) {$f$};
		\node [style=object] (59) at (10.5, 9.25) {$\oc(X \times A)$};
		\node [style=component] (60) at (10.5, 8.25) {$\chi$};
		\node [style=component] (61) at (11.25, 6.75) {$\eta$};
		\node [style=object] (65) at (9, 6.75) {$=$};
		\node [style=component] (66) at (11.25, 7.5) {$\varepsilon$};
	\end{pgfonlayer}
	\begin{pgfonlayer}{edgelayer}
		\draw [style=wire, in=90, out=-150, looseness=1.25] (20) to (22);
		\draw [style=wire, in=90, out=-30, looseness=1.25] (20) to (21);
		\draw [style=wire] (19) to (20);
		\draw [style=wire] (23) to (25);
		\draw [style=wire] (25) to (24);
		\draw [style=wire, in=135, out=-90] (27) to (23);
		\draw [style=wire, in=-90, out=45, looseness=1.25] (23) to (26);
		\draw [style=wire] (22) to (27);
		\draw [style=wire] (21) to (26);
		\draw [style=wire] (34) to (33);
		\draw [style=wire, in=-90, out=0, looseness=1.25] (38) to (37);
		\draw [style=wire] (38) to (34);
		\draw [style=wire] (35) to (37);
		\draw [style=wire, in=-90, out=180] (38) to (36);
		\draw [style=wire, in=90, out=0, looseness=1.25] (42) to (41);
		\draw [style=wire] (39) to (41);
		\draw [style=wire, in=90, out=-180] (42) to (40);
		\draw [style=wire] (40) to (36);
		\draw [style=wire] (39) to (35);
		\draw [style=wire] (28) to (42);
		\draw [style=wire] (46) to (45);
		\draw [style=wire, in=-90, out=0, looseness=1.25] (50) to (49);
		\draw [style=wire] (50) to (46);
		\draw [style=wire] (47) to (49);
		\draw [style=wire, in=-90, out=180] (50) to (48);
		\draw [style=wire, in=90, out=0, looseness=1.25] (54) to (53);
		\draw [style=wire] (51) to (53);
		\draw [style=wire, in=90, out=-180] (54) to (52);
		\draw [style=wire] (52) to (48);
		\draw [style=wire] (51) to (47);
		\draw [style=wire] (44) to (54);
		\draw [style=wire] (56) to (58);
		\draw [style=wire] (58) to (57);
		\draw [style=wire] (59) to (60);
		\draw [style=wire, in=-90, out=15, looseness=1.50] (56) to (61);
		\draw [style=wire] (66) to (61);
		\draw [style=wire, in=90, out=-15] (60) to (66);
		\draw [style=wire, in=165, out=-150] (60) to (56);
	\end{pgfonlayer}
\end{tikzpicture}
\end{align*}
So we have that:
\[\mathsf{d}^\circ_{X \times A}; (\oc( \pi_0) \otimes \pi_1); (\oc( \iota_0 \otimes \iota_1); \mathsf{d}_{X \times A}; \llbracket f \rrbracket = \chi_{X,A}; (1_{\oc X} \otimes \varepsilon_A); (1_{\oc X} \otimes \eta_A); \chi^{-1}_{X,A}; \llbracket f \rrbracket \]
Then by Lemma \ref{lem:cokleisli-linearcontext}.(\ref{lem:cokleisli-linearcontext.ii}),  $\llbracket f \rrbracket: \oc (X \times A) \to B$ is linear if and only if the following equality holds: 
\[ \llbracket f \rrbracket =  \mathsf{d}^\circ_{X \times A}; (\oc( \pi_0) \otimes \pi_1); (\oc( \iota_0 \otimes \iota_1); \mathsf{d}_{X \times A}; \llbracket f \rrbracket = \chi_{X,A}; (1_{\oc X} \otimes \varepsilon_A); (1_{\oc X} \otimes \eta_A); \chi^{-1}_{X,A}; \llbracket f \rrbracket\]
So the desired equivalence holds. 
\end{proof} 

We now prove the main result of this section, that we have an isomorphism of fibrations. It is important to note that the following result does not require the Seely isomorphisms. 

\begin{therm}\label{thm:fibration_equivalence}  Let $\mathbb{X}$ be a monoidal differential category with coalgebra modality $(\oc, \delta, \varepsilon, \Delta, e)$ and deriving transformation $\mathsf{d}: \oc A \otimes A \to \oc A$, and finite (bi)products. Then the fibrations ${\mathsf{p}_\oc: \mathcal{L}_\oc[\mathbb{X}] \to \mathbb{X}_\oc}$ and ${\mathsf{p}: \mathcal{L}[\mathbb{X}_\oc] \to \mathbb{X}_\oc}$ are isomorphic via the functors $\mathsf{E}: \mathcal{L}_\oc[\mathbb{X}] \to \mathcal{L}[\mathbb{X}_\oc]$ and $\mathsf{E}^{-1}: \mathcal{L}[\mathbb{X}_\oc] \to  \mathcal{L}_\oc[\mathbb{X}]$ where: 
\begin{enumerate}[{\em (i)}]
\item $\mathsf{E}$ is defined on objects as $\mathsf{E}(X,A) = (X,A)$, and on maps $(\llbracket f \rrbracket, g): (X,A) \to (Y,B)$ as follows: 
\begin{align*}
\mathsf{E}(\llbracket f \rrbracket, g) 
%&= \left( \llbracket f \rrbracket,  \xymatrixcolsep{3pc}\xymatrix{\oc(X \times A) \ar[r]^-{\chi_{X,A}} & \oc X \times \oc A \ar[r]^-{1_{\oc X} \otimes \varepsilon_A} & \oc X \otimes A \ar[r]^-{g} & B } \right) \\
&= \left( \llbracket f \rrbracket,  \xymatrixcolsep{3pc}\xymatrix{\oc(X \times A) \ar[r]^-{\mathsf{d}^\circ_{X \times A}} & \oc (X \times A) \times (X \times A) \ar[r]^-{\oc(\pi_0) \times \pi_1} & \oc X \otimes A \ar[r]^-{g} & B } \right) \end{align*}
\begin{align*} \mathsf{E}\left(\begin{array}[c]{c}
\begin{tikzpicture}
	\begin{pgfonlayer}{nodelayer}
		\node [style=object] (17) at (9.5, 1) {$\oc X$};
		\node [style=object] (20) at (9.5, -1) {$Y$};
		\node [style=component] (21) at (9.5, 0) {$f$};
	\end{pgfonlayer}
	\begin{pgfonlayer}{edgelayer}
		\draw [style=wire] (17) to (21);
		\draw [style=wire] (21) to (20);
	\end{pgfonlayer}
\end{tikzpicture}
   \end{array},
   \begin{array}[c]{c}
\begin{tikzpicture}
	\begin{pgfonlayer}{nodelayer}
		\node [style=object] (17) at (9.75, 1.75) {$\oc X$};
		\node [style=object] (19) at (10.75, 1.75) {$A$};
		\node [style=component] (20) at (10.25, 0.75) {$g$};
		\node [style=object] (21) at (10.25, -0.25) {$B$};
	\end{pgfonlayer}
	\begin{pgfonlayer}{edgelayer}
		\draw [style=wire] (20) to (21);
		\draw [style=wire, in=165, out=-90] (17) to (20);
		\draw [style=wire, in=-90, out=15] (20) to (19);
	\end{pgfonlayer}
\end{tikzpicture}
   \end{array} \right) 
   %\left(\begin{array}[c]{c}
   %\begin{tikzpicture}
	%\begin{pgfonlayer}{nodelayer}
%		\node [style=object] (17) at (9.5, 1) {$\oc X$};
%		\node [style=object] (20) at (9.5, -1) {$Y$};
%		\node [style=component] (21) at (9.5, 0) {$f$};
%	\end{pgfonlayer}
%	\begin{pgfonlayer}{edgelayer}
%		\draw [style=wire] (17) to (21);
%		\draw [style=wire] (21) to (20);
%	\end{pgfonlayer}
%\end{tikzpicture}
 %  \end{array},
 %  \begin{array}[c]{c}
%\begin{tikzpicture}
%	\begin{pgfonlayer}{nodelayer}
%		\node [style=component] (0) at (7.5, 1.75) {$\varepsilon$};
%		\node [style=component] (1) at (7, 1) {$g$};
%		\node [style=object] (3) at (7, 0.25) {$B$};
%		\node [style=duplicate] (4) at (7, 2.5) {$\chi$};
%		\node [style=object] (5) at (7, 3.25) {$\oc (X \times A)$};
%	\end{pgfonlayer}
%	\begin{pgfonlayer}{edgelayer}
%		\draw [style=wire, in=-90, out=30] (1) to (0);
%		\draw [style=wire] (1) to (3);
%		\draw [style=wire] (5) to (4);
%		\draw [style=wire, in=90, out=-30, looseness=1.25] (4) to (0);
%		\draw [style=wire, in=150, out=-150, looseness=1.25] (4) to (1);
%	\end{pgfonlayer}
%\end{tikzpicture}
 %  \end{array} \right) 
 =  \left(\begin{array}[c]{c}
\begin{tikzpicture}
	\begin{pgfonlayer}{nodelayer}
		\node [style=object] (17) at (9.5, 1) {$\oc X$};
		\node [style=object] (20) at (9.5, -1) {$Y$};
		\node [style=component] (21) at (9.5, 0) {$f$};
	\end{pgfonlayer}
	\begin{pgfonlayer}{edgelayer}
		\draw [style=wire] (17) to (21);
		\draw [style=wire] (21) to (20);
	\end{pgfonlayer}
\end{tikzpicture}
   \end{array},
   \begin{array}[c]{c}
\begin{tikzpicture}
	\begin{pgfonlayer}{nodelayer}
		\node [style=component] (0) at (7, 0.75) {$g$};
		\node [style=object] (2) at (7, 0) {$B$};
		\node [style=object] (3) at (7, 3.5) {$\oc (X \times A)$};
		\node [style=differential] (4) at (7, 2.75) {{\bf =\!=\!=\!=}};
		\node [style=component] (5) at (7.75, 1.75) {$\pi_1$};
		\node [style=function2] (6) at (6.25, 1.75) {$\pi_0$};
	\end{pgfonlayer}
	\begin{pgfonlayer}{edgelayer}
		\draw [style=wire] (0) to (2);
		\draw [style=wire] (3) to (4);
		\draw [style=wire, in=90, out=-150, looseness=1.25] (4) to (6);
		\draw [style=wire, in=135, out=-90] (6) to (0);
		\draw [style=wire, in=-90, out=45, looseness=1.25] (0) to (5);
		\draw [style=wire, in=90, out=-30, looseness=1.25] (4) to (5);
	\end{pgfonlayer}
\end{tikzpicture}
   \end{array} \right)
\end{align*}
\item $\mathsf{E}^{-1}$ is defined on objects as $\mathsf{E}^{-1}(X,A) = (X,A)$, and on maps $(\llbracket f \rrbracket, \llbracket g \rrbracket): (X,A) \to (Y,B)$ as follows: 
\begin{align*}
\mathsf{E}^{-1}(\llbracket f \rrbracket, \llbracket g \rrbracket) = \left( \llbracket f \rrbracket,  \xymatrixcolsep{2.5pc}\xymatrix{\oc X \otimes A \ar[r]^-{\oc(\iota_0) \otimes \iota_1} & \oc(X \times A) \times (X \times A) \ar[r]^-{\mathsf{d}_{X \times A}} & \oc (X \times A) \ar[r]^-{\llbracket g \rrbracket} & B } \right)\end{align*} \begin{align*} \mathsf{E}^{-1}\left(\begin{array}[c]{c}
\begin{tikzpicture}
	\begin{pgfonlayer}{nodelayer}
		\node [style=object] (17) at (9.5, 1) {$\oc X$};
		\node [style=object] (20) at (9.5, -1) {$Y$};
		\node [style=component] (21) at (9.5, 0) {$f$};
	\end{pgfonlayer}
	\begin{pgfonlayer}{edgelayer}
		\draw [style=wire] (17) to (21);
		\draw [style=wire] (21) to (20);
	\end{pgfonlayer}
\end{tikzpicture}
   \end{array},
   \begin{array}[c]{c}
\begin{tikzpicture}
	\begin{pgfonlayer}{nodelayer}
		\node [style=object] (17) at (9.5, 1) {$\oc (X \times A)$};
		\node [style=object] (20) at (9.5, -1) {$B$};
		\node [style=component] (21) at (9.5, 0) {$g$};
	\end{pgfonlayer}
	\begin{pgfonlayer}{edgelayer}
		\draw [style=wire] (17) to (21);
		\draw [style=wire] (21) to (20);
	\end{pgfonlayer}
\end{tikzpicture}
   \end{array} \right) = \left(\begin{array}[c]{c}
\begin{tikzpicture}
	\begin{pgfonlayer}{nodelayer}
		\node [style=object] (17) at (9.5, 1) {$\oc X$};
		\node [style=object] (20) at (9.5, -1) {$Y$};
		\node [style=component] (21) at (9.5, 0) {$f$};
	\end{pgfonlayer}
	\begin{pgfonlayer}{edgelayer}
		\draw [style=wire] (17) to (21);
		\draw [style=wire] (21) to (20);
	\end{pgfonlayer}
\end{tikzpicture}
   \end{array},
   \begin{array}[c]{c}
\begin{tikzpicture}
	\begin{pgfonlayer}{nodelayer}
		\node [style=differential] (0) at (7, 1) {{\bf =\!=\!=\!=}};
		\node [style=object] (1) at (7, -0.5) {$B$};
		\node [style=component] (2) at (7, 0.25) {$g$};
		\node [style=component] (5) at (7.75, 1.75) {$\iota_1$};
		\node [style=function2] (6) at (6.25, 1.75) {$\iota_0$};
		\node [style=object] (7) at (6.25, 2.5) {$\oc X$};
		\node [style=object] (8) at (7.75, 2.5) {$A$};
	\end{pgfonlayer}
	\begin{pgfonlayer}{edgelayer}
		\draw [style=wire] (0) to (2);
		\draw [style=wire] (2) to (1);
		\draw [style=wire, in=135, out=-90] (6) to (0);
		\draw [style=wire, in=-90, out=45, looseness=1.25] (0) to (5);
		\draw [style=wire] (7) to (6);
		\draw [style=wire] (8) to (5);
	\end{pgfonlayer}
\end{tikzpicture}
   \end{array} \right)
\end{align*}
\end{enumerate}
\end{therm}
\begin{proof} For this proof, we will follow the same strategy as in the proof of Corollary \ref{cor:seely-lin}. We will first prove that $\mathsf{E}: \mathcal{L}_\oc[\mathbb{X}] \to \mathcal{L}[\mathbb{X}_\oc]$ is a well-defined morphism of fibrations (that is, a functor that preserves Cartesian maps). Then we will prove that $\mathsf{E} \circ \mathsf{E}^{-1} = 1_{ \mathcal{L}[\mathbb{X}_\oc]}$ and $\mathsf{E}^{-1} \circ \mathsf{E} = 1_{ \mathcal{L}_\oc[\mathbb{X}]}$. Therefore, it follows that ${\mathsf{E}^{-1}: \mathcal{L}[\mathbb{X}_\oc] \to  \mathcal{L}_\oc[\mathbb{X}]}$ is also a morphism of fibrations, and that $\mathsf{E}$ and $\mathsf{E}^{-1}$ are isomorphisms and inverses of each other. 

By Lemma \ref{lem:cokleisli-linearcontext}.(\ref{lem:cokleisli-linearcontext.iii}), $\mathsf{E}$ is well-defined. To show that $\mathsf{E}$ preserve composition, first observe that by using Lemma \ref{lem:cokleisli-linearcontext}.(\ref{lem:cokleisli-linearcontext.i}) we see that composition in $\mathcal{L}[\mathbb{X}_\oc]$ can be expressed as follows (which we leave as an exercise for the reader to work out for themselves): 
\[ \left(\begin{array}[c]{c}
\begin{tikzpicture}
	\begin{pgfonlayer}{nodelayer}
		\node [style=object] (17) at (9.5, 1) {$\oc X$};
		\node [style=object] (20) at (9.5, -1) {$Y$};
		\node [style=component] (21) at (9.5, 0) {$f$};
	\end{pgfonlayer}
	\begin{pgfonlayer}{edgelayer}
		\draw [style=wire] (17) to (21);
		\draw [style=wire] (21) to (20);
	\end{pgfonlayer}
\end{tikzpicture}
   \end{array},
   \begin{array}[c]{c}
\begin{tikzpicture}
	\begin{pgfonlayer}{nodelayer}
		\node [style=object] (17) at (9.5, 1) {$\oc (X \times A)$};
		\node [style=object] (20) at (9.5, -1) {$B$};
		\node [style=component] (21) at (9.5, 0) {$g$};
	\end{pgfonlayer}
	\begin{pgfonlayer}{edgelayer}
		\draw [style=wire] (17) to (21);
		\draw [style=wire] (21) to (20);
	\end{pgfonlayer}
\end{tikzpicture}
   \end{array} \right) ;  \left( \begin{array}[c]{c}
\begin{tikzpicture}
	\begin{pgfonlayer}{nodelayer}
		\node [style=object] (17) at (9.5, 1) {$\oc Y$};
		\node [style=object] (20) at (9.5, -1) {$Z$};
		\node [style=component] (21) at (9.5, 0) {$h$};
	\end{pgfonlayer}
	\begin{pgfonlayer}{edgelayer}
		\draw [style=wire] (17) to (21);
		\draw [style=wire] (21) to (20);
	\end{pgfonlayer}
\end{tikzpicture}
   \end{array},
   \begin{array}[c]{c}
\begin{tikzpicture}
	\begin{pgfonlayer}{nodelayer}
		\node [style=object] (17) at (9.5, 1) {$\oc (Y \times B)$};
		\node [style=object] (20) at (9.5, -1) {$C$};
		\node [style=component] (21) at (9.5, 0) {$k$};
	\end{pgfonlayer}
	\begin{pgfonlayer}{edgelayer}
		\draw [style=wire] (17) to (21);
		\draw [style=wire] (21) to (20);
	\end{pgfonlayer}
\end{tikzpicture}
   \end{array} \right) = \left(\begin{array}[c]{c}
\begin{tikzpicture}
	\begin{pgfonlayer}{nodelayer}
		\node [style=object] (17) at (9.5, 1.75) {$\oc X$};
		\node [style=function2] (18) at (9.5, 0) {$f$};
		\node [style=component] (19) at (9.5, -1) {$h$};
		\node [style=object] (20) at (9.5, -1.75) {$Z$};
		\node [style=component] (21) at (9.5, 1) {$\delta$};
	\end{pgfonlayer}
	\begin{pgfonlayer}{edgelayer}
		\draw [style=wire] (18) to (19);
		\draw [style=wire] (19) to (20);
		\draw [style=wire] (17) to (21);
		\draw [style=wire] (21) to (18);
	\end{pgfonlayer}
\end{tikzpicture}
   \end{array},
   \begin{array}[c]{c}\resizebox{!}{4.75cm}{%
\begin{tikzpicture}
	\begin{pgfonlayer}{nodelayer}
		\node [style=component] (0) at (4.75, 0) {$\delta$};
		\node [style=object] (2) at (6, 3.75) {$\oc (X \times A)$};
		\node [style=object] (6) at (5.5, -4.75) {$C$};
		\node [style=function2] (7) at (4.75, -1.25) {$f$};
		\node [style=differential] (23) at (6, 2.75) {{\bf =\!=\!=\!=}};
		\node [style=component] (24) at (6.75, 1.75) {$\pi_1$};
		\node [style=function2] (25) at (5.25, 1.75) {$\pi_0$};
		\node [style=duplicate] (26) at (5.25, 0.75) {$\Delta$};
		\node [style=differential] (27) at (6.25, -0.75) {{\bf =\!=\!=\!=}};
		\node [style=component] (28) at (6.75, 0) {$\iota_1$};
		\node [style=function2] (29) at (5.75, 0) {$\iota_0$};
		\node [style=differential] (30) at (5.5, -3.25) {{\bf =\!=\!=\!=}};
		\node [style=component] (31) at (6.25, -2.5) {$\iota_1$};
		\node [style=function2] (32) at (4.75, -2.5) {$\iota_0$};
		\node [style=component] (33) at (5.5, -4) {$k$};
		\node [style=component] (34) at (6.25, -1.5) {$g$};
	\end{pgfonlayer}
	\begin{pgfonlayer}{edgelayer}
		\draw [style=wire] (0) to (7);
		\draw [style=wire, in=90, out=-150, looseness=1.25] (23) to (25);
		\draw [style=wire, in=90, out=-30, looseness=1.25] (23) to (24);
		\draw [style=wire] (2) to (23);
		\draw [style=wire] (25) to (26);
		\draw [style=wire, in=90, out=-150] (26) to (0);
		\draw [style=wire, in=135, out=-90] (29) to (27);
		\draw [style=wire, in=-90, out=45, looseness=1.25] (27) to (28);
		\draw [style=wire, in=135, out=-90] (32) to (30);
		\draw [style=wire, in=-90, out=45, looseness=1.25] (30) to (31);
		\draw [style=wire] (7) to (32);
		\draw [style=wire] (30) to (33);
		\draw [style=wire] (33) to (6);
		\draw [style=wire] (27) to (34);
		\draw [style=wire] (34) to (31);
		\draw [style=wire, in=90, out=-30, looseness=1.25] (26) to (29);
		\draw [style=wire] (24) to (28);
	\end{pgfonlayer}
\end{tikzpicture}
} % 
   \end{array}  \right)\] 
Then by Lemma \ref{lem:cokleisli-linearcontext}.(\ref{lem:cokleisli-linearcontext.ii}), we compute that:   
   \begin{align*}
   &\mathsf{E}\left(\begin{array}[c]{c}
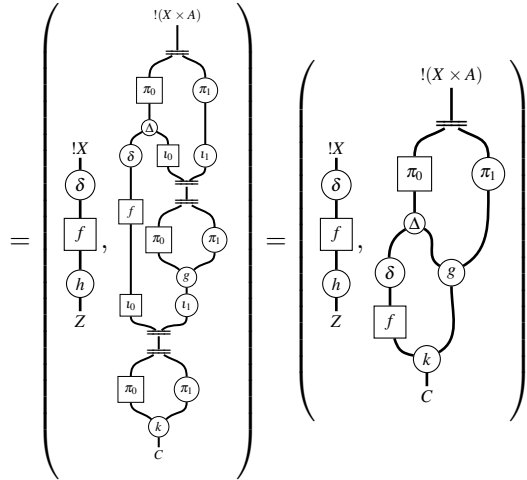
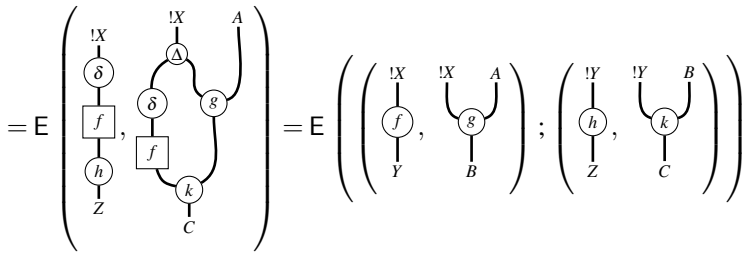
\begin{tikzpicture}
	\begin{pgfonlayer}{nodelayer}
		\node [style=object] (17) at (9.5, 1) {$\oc X$};
		\node [style=object] (20) at (9.5, -1) {$Y$};
		\node [style=component] (21) at (9.5, 0) {$f$};
	\end{pgfonlayer}
	\begin{pgfonlayer}{edgelayer}
		\draw [style=wire] (17) to (21);
		\draw [style=wire] (21) to (20);
	\end{pgfonlayer}
\end{tikzpicture}
   \end{array},
   \begin{array}[c]{c}
\begin{tikzpicture}
	\begin{pgfonlayer}{nodelayer}
		\node [style=object] (17) at (9.75, 1.75) {$\oc X$};
		\node [style=object] (19) at (10.75, 1.75) {$A$};
		\node [style=component] (20) at (10.25, 0.75) {$g$};
		\node [style=object] (21) at (10.25, -0.25) {$B$};
	\end{pgfonlayer}
	\begin{pgfonlayer}{edgelayer}
		\draw [style=wire] (20) to (21);
		\draw [style=wire, in=165, out=-90] (17) to (20);
		\draw [style=wire, in=-90, out=15] (20) to (19);
	\end{pgfonlayer}
\end{tikzpicture}
   \end{array} \right) ; \mathsf{E}\left(\begin{array}[c]{c}
\begin{tikzpicture}
	\begin{pgfonlayer}{nodelayer}
		\node [style=object] (17) at (9.5, 1) {$\oc Y$};
		\node [style=object] (20) at (9.5, -1) {$Z$};
		\node [style=component] (21) at (9.5, 0) {$h$};
	\end{pgfonlayer}
	\begin{pgfonlayer}{edgelayer}
		\draw [style=wire] (17) to (21);
		\draw [style=wire] (21) to (20);
	\end{pgfonlayer}
\end{tikzpicture}
   \end{array},
   \begin{array}[c]{c}
\begin{tikzpicture}
	\begin{pgfonlayer}{nodelayer}
		\node [style=object] (17) at (9.75, 1.75) {$\oc Y$};
		\node [style=object] (19) at (10.75, 1.75) {$B$};
		\node [style=component] (20) at (10.25, 0.75) {$k$};
		\node [style=object] (21) at (10.25, -0.25) {$C$};
	\end{pgfonlayer}
	\begin{pgfonlayer}{edgelayer}
		\draw [style=wire] (20) to (21);
		\draw [style=wire, in=165, out=-90] (17) to (20);
		\draw [style=wire, in=-90, out=15] (20) to (19);
	\end{pgfonlayer}
\end{tikzpicture}
   \end{array} \right) =  \left(\begin{array}[c]{c}
\begin{tikzpicture}
	\begin{pgfonlayer}{nodelayer}
		\node [style=object] (17) at (9.5, 1) {$\oc X$};
		\node [style=object] (20) at (9.5, -1) {$Y$};
		\node [style=component] (21) at (9.5, 0) {$f$};
	\end{pgfonlayer}
	\begin{pgfonlayer}{edgelayer}
		\draw [style=wire] (17) to (21);
		\draw [style=wire] (21) to (20);
	\end{pgfonlayer}
\end{tikzpicture}
   \end{array},
   \begin{array}[c]{c}
\begin{tikzpicture}
	\begin{pgfonlayer}{nodelayer}
		\node [style=component] (0) at (7, 0.75) {$g$};
		\node [style=object] (2) at (7, 0) {$B$};
		\node [style=object] (3) at (7, 3.5) {$\oc (X \times A)$};
		\node [style=differential] (4) at (7, 2.75) {{\bf =\!=\!=\!=}};
		\node [style=component] (5) at (7.75, 1.75) {$\pi_1$};
		\node [style=function2] (6) at (6.25, 1.75) {$\pi_0$};
	\end{pgfonlayer}
	\begin{pgfonlayer}{edgelayer}
		\draw [style=wire] (0) to (2);
		\draw [style=wire] (3) to (4);
		\draw [style=wire, in=90, out=-150, looseness=1.25] (4) to (6);
		\draw [style=wire, in=135, out=-90] (6) to (0);
		\draw [style=wire, in=-90, out=45, looseness=1.25] (0) to (5);
		\draw [style=wire, in=90, out=-30, looseness=1.25] (4) to (5);
	\end{pgfonlayer}
\end{tikzpicture}
   \end{array} \right) ; \left(\begin{array}[c]{c}
\begin{tikzpicture}
	\begin{pgfonlayer}{nodelayer}
		\node [style=object] (17) at (9.5, 1) {$\oc Y$};
		\node [style=object] (20) at (9.5, -1) {$Z$};
		\node [style=component] (21) at (9.5, 0) {$h$};
	\end{pgfonlayer}
	\begin{pgfonlayer}{edgelayer}
		\draw [style=wire] (17) to (21);
		\draw [style=wire] (21) to (20);
	\end{pgfonlayer}
\end{tikzpicture}
   \end{array},
   \begin{array}[c]{c}
\begin{tikzpicture}
	\begin{pgfonlayer}{nodelayer}
		\node [style=component] (0) at (7, 0.75) {$k$};
		\node [style=object] (2) at (7, 0) {$C$};
		\node [style=object] (3) at (7, 3.5) {$\oc (Y \times B)$};
		\node [style=differential] (4) at (7, 2.75) {{\bf =\!=\!=\!=}};
		\node [style=component] (5) at (7.75, 1.75) {$\pi_1$};
		\node [style=function2] (6) at (6.25, 1.75) {$\pi_0$};
	\end{pgfonlayer}
	\begin{pgfonlayer}{edgelayer}
		\draw [style=wire] (0) to (2);
		\draw [style=wire] (3) to (4);
		\draw [style=wire, in=90, out=-150, looseness=1.25] (4) to (6);
		\draw [style=wire, in=135, out=-90] (6) to (0);
		\draw [style=wire, in=-90, out=45, looseness=1.25] (0) to (5);
		\draw [style=wire, in=90, out=-30, looseness=1.25] (4) to (5);
	\end{pgfonlayer}
\end{tikzpicture}
   \end{array} \right) 
\end{align*}
\begin{align*}   
   &=  \left(\begin{array}[c]{c}
\begin{tikzpicture}
	\begin{pgfonlayer}{nodelayer}
		\node [style=object] (17) at (9.5, 1.75) {$\oc X$};
		\node [style=function2] (18) at (9.5, 0) {$f$};
		\node [style=component] (19) at (9.5, -1) {$h$};
		\node [style=object] (20) at (9.5, -1.75) {$Z$};
		\node [style=component] (21) at (9.5, 1) {$\delta$};
	\end{pgfonlayer}
	\begin{pgfonlayer}{edgelayer}
		\draw [style=wire] (18) to (19);
		\draw [style=wire] (19) to (20);
		\draw [style=wire] (17) to (21);
		\draw [style=wire] (21) to (18);
	\end{pgfonlayer}
\end{tikzpicture}
   \end{array},
   \begin{array}[c]{c}\resizebox{!}{6cm}{%
\begin{tikzpicture}
	\begin{pgfonlayer}{nodelayer}
		\node [style=component] (0) at (4.75, 1.5) {$\delta$};
		\node [style=object] (2) at (6, 5.25) {$\oc (X \times A)$};
		\node [style=object] (6) at (5.5, -6.5) {$C$};
		\node [style=function2] (7) at (4.75, 0) {$f$};
		\node [style=differential] (23) at (6, 4.25) {{\bf =\!=\!=\!=}};
		\node [style=component] (24) at (6.75, 3.25) {$\pi_1$};
		\node [style=function2] (25) at (5.25, 3.25) {$\pi_0$};
		\node [style=duplicate] (26) at (5.25, 2.25) {$\Delta$};
		\node [style=differential] (27) at (6.25, 0.75) {{\bf =\!=\!=\!=}};
		\node [style=component] (28) at (6.75, 1.5) {$\iota_1$};
		\node [style=function2] (29) at (5.75, 1.5) {$\iota_0$};
		\node [style=differential] (30) at (5.5, -3.25) {{\bf =\!=\!=\!=}};
		\node [style=component] (31) at (6.25, -2.5) {$\iota_1$};
		\node [style=function2] (32) at (4.75, -2.5) {$\iota_0$};
		\node [style=component] (35) at (5.5, -5.75) {$k$};
		\node [style=differential] (38) at (5.5, -3.75) {{\bf =\!=\!=\!=}};
		\node [style=component] (39) at (6.25, -4.75) {$\pi_1$};
		\node [style=function2] (40) at (4.75, -4.75) {$\pi_0$};
		\node [style=component] (41) at (6.25, -1.75) {$g$};
		\node [style=differential] (44) at (6.25, 0.25) {{\bf =\!=\!=\!=}};
		\node [style=component] (45) at (7, -0.75) {$\pi_1$};
		\node [style=function2] (46) at (5.5, -0.75) {$\pi_0$};
	\end{pgfonlayer}
	\begin{pgfonlayer}{edgelayer}
		\draw [style=wire] (0) to (7);
		\draw [style=wire, in=90, out=-150, looseness=1.25] (23) to (25);
		\draw [style=wire, in=90, out=-30, looseness=1.25] (23) to (24);
		\draw [style=wire] (2) to (23);
		\draw [style=wire] (25) to (26);
		\draw [style=wire, in=90, out=-150] (26) to (0);
		\draw [style=wire, in=135, out=-90] (29) to (27);
		\draw [style=wire, in=-90, out=45, looseness=1.25] (27) to (28);
		\draw [style=wire, in=135, out=-90] (32) to (30);
		\draw [style=wire, in=-90, out=45, looseness=1.25] (30) to (31);
		\draw [style=wire] (7) to (32);
		\draw [style=wire, in=90, out=-30, looseness=1.25] (26) to (29);
		\draw [style=wire] (24) to (28);
		\draw [style=wire, in=90, out=-150, looseness=1.25] (38) to (40);
		\draw [style=wire, in=135, out=-90] (40) to (35);
		\draw [style=wire, in=-90, out=45, looseness=1.25] (35) to (39);
		\draw [style=wire, in=90, out=-30, looseness=1.25] (38) to (39);
		\draw [style=wire] (30) to (38);
		\draw [style=wire] (35) to (6);
		\draw [style=wire, in=90, out=-150, looseness=1.25] (44) to (46);
		\draw [style=wire, in=135, out=-90] (46) to (41);
		\draw [style=wire, in=-90, out=45, looseness=1.25] (41) to (45);
		\draw [style=wire, in=90, out=-30, looseness=1.25] (44) to (45);
		\draw [style=wire] (27) to (44);
		\draw [style=wire] (41) to (31);
	\end{pgfonlayer}
\end{tikzpicture} }% 
   \end{array} \right)
   = \left(\begin{array}[c]{c}
\begin{tikzpicture}
	\begin{pgfonlayer}{nodelayer}
		\node [style=object] (17) at (9.5, 1.75) {$\oc X$};
		\node [style=function2] (18) at (9.5, 0) {$f$};
		\node [style=component] (19) at (9.5, -1) {$h$};
		\node [style=object] (20) at (9.5, -1.75) {$Z$};
		\node [style=component] (21) at (9.5, 1) {$\delta$};
	\end{pgfonlayer}
	\begin{pgfonlayer}{edgelayer}
		\draw [style=wire] (18) to (19);
		\draw [style=wire] (19) to (20);
		\draw [style=wire] (17) to (21);
		\draw [style=wire] (21) to (18);
	\end{pgfonlayer}
\end{tikzpicture}
   \end{array},
   \begin{array}[c]{c} 
\begin{tikzpicture}
	\begin{pgfonlayer}{nodelayer}
		\node [style=object] (2) at (6, 5.25) {$\oc (X \times A)$};
		\node [style=differential] (23) at (6, 4.25) {{\bf =\!=\!=\!=}};
		\node [style=component] (24) at (6.75, 3.25) {$\pi_1$};
		\node [style=function2] (25) at (5.25, 3.25) {$\pi_0$};
		\node [style=component] (42) at (4.75, 1.25) {$\delta$};
		\node [style=duplicate] (43) at (5.25, 2.25) {$\Delta$};
		\node [style=component] (45) at (6, 1.25) {$g$};
		\node [style=component] (47) at (5.5, -0.5) {$k$};
		\node [style=object] (48) at (5.5, -1.25) {$C$};
		\node [style=function2] (49) at (4.75, 0.25) {$f$};
	\end{pgfonlayer}
	\begin{pgfonlayer}{edgelayer}
		\draw [style=wire, in=90, out=-150, looseness=1.25] (23) to (25);
		\draw [style=wire, in=90, out=-30, looseness=1.25] (23) to (24);
		\draw [style=wire] (2) to (23);
		\draw [style=wire, bend right] (43) to (42);
		\draw [style=wire, in=150, out=-30, looseness=1.25] (43) to (45);
		\draw [style=wire] (47) to (48);
		\draw [style=wire, in=30, out=-90] (45) to (47);
		\draw [style=wire] (42) to (49);
		\draw [style=wire, bend right, looseness=1.25] (49) to (47);
		\draw [style=wire] (25) to (43);
		\draw [style=wire, in=30, out=-90] (24) to (45);
	\end{pgfonlayer}
\end{tikzpicture}
   \end{array} \right)   
\end{align*}
\begin{align*}   
&= \mathsf{E}\left(\begin{array}[c]{c}
\begin{tikzpicture}
	\begin{pgfonlayer}{nodelayer}
		\node [style=object] (17) at (9.5, 1.75) {$\oc X$};
		\node [style=function2] (18) at (9.5, 0) {$f$};
		\node [style=component] (19) at (9.5, -1) {$h$};
		\node [style=object] (20) at (9.5, -1.75) {$Z$};
		\node [style=component] (21) at (9.5, 1) {$\delta$};
	\end{pgfonlayer}
	\begin{pgfonlayer}{edgelayer}
		\draw [style=wire] (18) to (19);
		\draw [style=wire] (19) to (20);
		\draw [style=wire] (17) to (21);
		\draw [style=wire] (21) to (18);
	\end{pgfonlayer}
\end{tikzpicture}
   \end{array},
   \begin{array}[c]{c}
\begin{tikzpicture}
	\begin{pgfonlayer}{nodelayer}
		\node [style=component] (15) at (9, 0) {$\delta$};
		\node [style=duplicate] (16) at (9.5, 1) {$\Delta$};
		\node [style=object] (17) at (9.5, 1.75) {$\oc X$};
		\node [style=component] (18) at (10.25, 0) {$g$};
		\node [style=object] (19) at (10.75, 1.75) {$A$};
		\node [style=component] (20) at (9.75, -1.75) {$k$};
		\node [style=object] (21) at (9.75, -2.5) {$C$};
		\node [style=function2] (23) at (9, -1) {$f$};
	\end{pgfonlayer}
	\begin{pgfonlayer}{edgelayer}
		\draw [style=wire, bend right] (16) to (15);
		\draw [style=wire] (17) to (16);
		\draw [style=wire, in=-90, out=15, looseness=0.75] (18) to (19);
		\draw [style=wire, in=150, out=-30, looseness=1.25] (16) to (18);
		\draw [style=wire] (20) to (21);
		\draw [style=wire, in=30, out=-90] (18) to (20);
		\draw [style=wire] (15) to (23);
		\draw [style=wire, bend right, looseness=1.25] (23) to (20);
	\end{pgfonlayer}
\end{tikzpicture}
   \end{array} \right) = \mathsf{E} \left( 
   \left(\begin{array}[c]{c}
\begin{tikzpicture}
	\begin{pgfonlayer}{nodelayer}
		\node [style=object] (17) at (9.5, 1) {$\oc X$};
		\node [style=object] (20) at (9.5, -1) {$Y$};
		\node [style=component] (21) at (9.5, 0) {$f$};
	\end{pgfonlayer}
	\begin{pgfonlayer}{edgelayer}
		\draw [style=wire] (17) to (21);
		\draw [style=wire] (21) to (20);
	\end{pgfonlayer}
\end{tikzpicture}
   \end{array},
   \begin{array}[c]{c}
\begin{tikzpicture}
	\begin{pgfonlayer}{nodelayer}
		\node [style=object] (17) at (9.75, 1.75) {$\oc X$};
		\node [style=object] (19) at (10.75, 1.75) {$A$};
		\node [style=component] (20) at (10.25, 0.75) {$g$};
		\node [style=object] (21) at (10.25, -0.25) {$B$};
	\end{pgfonlayer}
	\begin{pgfonlayer}{edgelayer}
		\draw [style=wire] (20) to (21);
		\draw [style=wire, in=165, out=-90] (17) to (20);
		\draw [style=wire, in=-90, out=15] (20) to (19);
	\end{pgfonlayer}
\end{tikzpicture}
   \end{array} \right) ; \left(\begin{array}[c]{c}
\begin{tikzpicture}
	\begin{pgfonlayer}{nodelayer}
		\node [style=object] (17) at (9.5, 1) {$\oc Y$};
		\node [style=object] (20) at (9.5, -1) {$Z$};
		\node [style=component] (21) at (9.5, 0) {$h$};
	\end{pgfonlayer}
	\begin{pgfonlayer}{edgelayer}
		\draw [style=wire] (17) to (21);
		\draw [style=wire] (21) to (20);
	\end{pgfonlayer}
\end{tikzpicture}
   \end{array},
   \begin{array}[c]{c}
\begin{tikzpicture}
	\begin{pgfonlayer}{nodelayer}
		\node [style=object] (17) at (9.75, 1.75) {$\oc Y$};
		\node [style=object] (19) at (10.75, 1.75) {$B$};
		\node [style=component] (20) at (10.25, 0.75) {$k$};
		\node [style=object] (21) at (10.25, -0.25) {$C$};
	\end{pgfonlayer}
	\begin{pgfonlayer}{edgelayer}
		\draw [style=wire] (20) to (21);
		\draw [style=wire, in=165, out=-90] (17) to (20);
		\draw [style=wire, in=-90, out=15] (20) to (19);
	\end{pgfonlayer}
\end{tikzpicture}
   \end{array} \right) \right) 
\end{align*}
Next we show that $\mathsf{E}$ preserves identities. First note that the identity in $\mathcal{L}[\mathbb{X}_\oc]$ is easily computed out to be $(\llbracket 1_X \rrbracket, \llbracket \pi_1 \rrbracket) = (\varepsilon_X, \varepsilon_{X \times A}; \pi_1)$. Then we compute: 
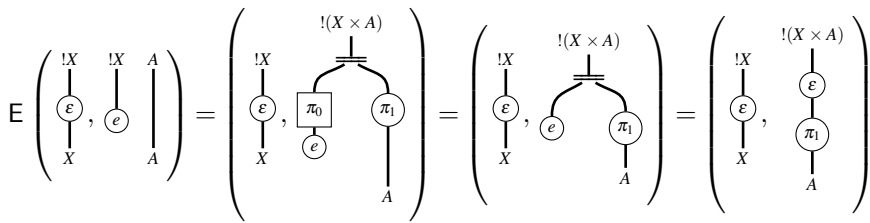
\[ \mathsf{E}\left(\begin{array}[c]{c}
\begin{tikzpicture}
	\begin{pgfonlayer}{nodelayer}
		\node [style=object] (17) at (9.5, 1) {$\oc X$};
		\node [style=object] (20) at (9.5, -1) {$X$};
		\node [style=component] (21) at (9.5, 0) {$\varepsilon$};
	\end{pgfonlayer}
	\begin{pgfonlayer}{edgelayer}
		\draw [style=wire] (17) to (21);
		\draw [style=wire] (21) to (20);
	\end{pgfonlayer}
\end{tikzpicture}
   \end{array},
   \begin{array}[c]{c}
\begin{tikzpicture}
	\begin{pgfonlayer}{nodelayer}
		\node [style=object] (17) at (9.75, 1.75) {$\oc X$};
		\node [style=object] (19) at (10.5, 1.75) {$A$};
		\node [style=object] (21) at (10.5, -0.25) {$A$};
		\node [style=component] (22) at (9.75, 0.5) {$e$};
	\end{pgfonlayer}
	\begin{pgfonlayer}{edgelayer}
		\draw [style=wire] (19) to (21);
		\draw [style=wire] (17) to (22);
	\end{pgfonlayer}
\end{tikzpicture}
   \end{array} \right) = \left(\begin{array}[c]{c}
\begin{tikzpicture}
	\begin{pgfonlayer}{nodelayer}
		\node [style=object] (17) at (9.5, 1) {$\oc X$};
		\node [style=object] (20) at (9.5, -1) {$X$};
		\node [style=component] (21) at (9.5, 0) {$\varepsilon$};
	\end{pgfonlayer}
	\begin{pgfonlayer}{edgelayer}
		\draw [style=wire] (17) to (21);
		\draw [style=wire] (21) to (20);
	\end{pgfonlayer}
\end{tikzpicture}
   \end{array},
   \begin{array}[c]{c}
\begin{tikzpicture}
	\begin{pgfonlayer}{nodelayer}
		\node [style=object] (1) at (7.75, 0) {$A$};
		\node [style=object] (2) at (7, 3.5) {$\oc(X \times A)$};
		\node [style=differential] (3) at (7, 2.75) {{\bf =\!=\!=\!=}};
		\node [style=component] (4) at (7.75, 1.75) {$\pi_1$};
		\node [style=function2] (5) at (6.25, 1.75) {$\pi_0$};
		\node [style=component] (6) at (6.25, 1) {$e$};
	\end{pgfonlayer}
	\begin{pgfonlayer}{edgelayer}
		\draw [style=wire] (2) to (3);
		\draw [style=wire, in=90, out=-150, looseness=1.25] (3) to (5);
		\draw [style=wire, in=90, out=-30, looseness=1.25] (3) to (4);
		\draw [style=wire] (5) to (6);
		\draw [style=wire] (4) to (1);
	\end{pgfonlayer}
\end{tikzpicture}
   \end{array} \right) =   \left(\begin{array}[c]{c}
\begin{tikzpicture}
	\begin{pgfonlayer}{nodelayer}
		\node [style=object] (17) at (9.5, 1) {$\oc X$};
		\node [style=object] (20) at (9.5, -1) {$X$};
		\node [style=component] (21) at (9.5, 0) {$\varepsilon$};
	\end{pgfonlayer}
	\begin{pgfonlayer}{edgelayer}
		\draw [style=wire] (17) to (21);
		\draw [style=wire] (21) to (20);
	\end{pgfonlayer}
\end{tikzpicture}
   \end{array},
   \begin{array}[c]{c}
\begin{tikzpicture}
	\begin{pgfonlayer}{nodelayer}
		\node [style=object] (1) at (7.75, 0.75) {$A$};
		\node [style=object] (2) at (7, 3.5) {$\oc(X \times A)$};
		\node [style=differential] (3) at (7, 2.75) {{\bf =\!=\!=\!=}};
		\node [style=component] (4) at (7.75, 1.75) {$\pi_1$};
		\node [style=component] (5) at (6.25, 1.75) {$e$};
	\end{pgfonlayer}
	\begin{pgfonlayer}{edgelayer}
		\draw [style=wire] (2) to (3);
		\draw [style=wire, in=90, out=-150, looseness=1.25] (3) to (5);
		\draw [style=wire, in=90, out=-30, looseness=1.25] (3) to (4);
		\draw [style=wire] (4) to (1);
	\end{pgfonlayer}
\end{tikzpicture}
   \end{array} \right) =   \left(\begin{array}[c]{c}
\begin{tikzpicture}
	\begin{pgfonlayer}{nodelayer}
		\node [style=object] (17) at (9.5, 1) {$\oc X$};
		\node [style=object] (20) at (9.5, -1) {$X$};
		\node [style=component] (21) at (9.5, 0) {$\varepsilon$};
	\end{pgfonlayer}
	\begin{pgfonlayer}{edgelayer}
		\draw [style=wire] (17) to (21);
		\draw [style=wire] (21) to (20);
	\end{pgfonlayer}
\end{tikzpicture}
   \end{array},
   \begin{array}[c]{c}
\begin{tikzpicture}
	\begin{pgfonlayer}{nodelayer}
		\node [style=object] (1) at (7.75, 1) {$A$};
		\node [style=object] (2) at (7.75, 4) {$\oc(X \times A)$};
		\node [style=component] (4) at (7.75, 2) {$\pi_1$};
		\node [style=component] (5) at (7.75, 3) {$\varepsilon$};
	\end{pgfonlayer}
	\begin{pgfonlayer}{edgelayer}
		\draw [style=wire] (4) to (1);
		\draw [style=wire] (2) to (5);
		\draw [style=wire] (5) to (4);
	\end{pgfonlayer}
\end{tikzpicture}
   \end{array} \right)  \] 
   So $\mathsf{E}$ is a functor. To show that $\mathsf{E}$ is a fibration morphism, we must show that $\mathsf{p} \circ \mathsf{E} = \mathsf{p}_\oc$ and that $\mathsf{E}$ preserves Cartesian maps. Starting with the former, this is straightforward on both maps and objects since $\mathsf{p}\mathsf{E}(X,A) = X = \mathsf{p}_\oc(X,A)$ and $\mathsf{p}\mathsf{E}(\llbracket f \rrbracket, g) = \llbracket f \rrbracket = \mathsf{p}_\oc(\llbracket f \rrbracket,g)$. Next we must show that $\mathsf{E}$ also preserves Cartesian maps. By Proposition \ref{prop:linear_fibration}, note that in $\mathcal{L}[\mathbb{X}_\oc]$ these are easily computed out to be the maps of the form $(\llbracket f \rrbracket, \llbracket \pi_1 \rrbracket) = (\llbracket f \rrbracket, \varepsilon_{X \times A}; \pi_1)$, while we recall that by Proposition \ref{prop:context_fibration}, Cartesian maps in $\mathcal{L}_\oc[\mathbb{X}]$ are of the form $(\llbracket f \rrbracket, e_X \otimes 1_A)$. By a similar calculation as the one above, we easily compute that $\mathsf{E} (\llbracket f \rrbracket, e_X \otimes 1_A) = (\llbracket f \rrbracket, \varepsilon_{X \times A}; \pi_1)$. So $\mathsf{E}$ preserves Cartesian maps, and we conclude that $\mathsf{E}$ is a fibration morphism. 
   
Next we show that  $\mathsf{E}$ and $\mathsf{E}^{-1}$ are inverses of each other. Starting with $\mathsf{E} \circ \mathsf{E}^{-1}$, clearly on objects $\mathsf{E}\mathsf{E}^{-1}(X,A) = (X,A)$, while on maps we use Lemma \ref{lem:cokleisli-linearcontext}.(\ref{lem:cokleisli-linearcontext.ii}): 
\[ \mathsf{E}\mathsf{E}^{-1}\left(\begin{array}[c]{c}
\begin{tikzpicture}
	\begin{pgfonlayer}{nodelayer}
		\node [style=object] (17) at (9.5, 1) {$\oc X$};
		\node [style=object] (20) at (9.5, -1) {$Y$};
		\node [style=component] (21) at (9.5, 0) {$f$};
	\end{pgfonlayer}
	\begin{pgfonlayer}{edgelayer}
		\draw [style=wire] (17) to (21);
		\draw [style=wire] (21) to (20);
	\end{pgfonlayer}
\end{tikzpicture}
   \end{array},
   \begin{array}[c]{c}
\begin{tikzpicture}
	\begin{pgfonlayer}{nodelayer}
		\node [style=object] (17) at (9.5, 1) {$\oc (X \times A)$};
		\node [style=object] (20) at (9.5, -1) {$B$};
		\node [style=component] (21) at (9.5, 0) {$g$};
	\end{pgfonlayer}
	\begin{pgfonlayer}{edgelayer}
		\draw [style=wire] (17) to (21);
		\draw [style=wire] (21) to (20);
	\end{pgfonlayer}
\end{tikzpicture}
   \end{array} \right) = \mathsf{E} \left(\begin{array}[c]{c}
\begin{tikzpicture}
	\begin{pgfonlayer}{nodelayer}
		\node [style=object] (17) at (9.5, 1) {$\oc X$};
		\node [style=object] (20) at (9.5, -1) {$Y$};
		\node [style=component] (21) at (9.5, 0) {$f$};
	\end{pgfonlayer}
	\begin{pgfonlayer}{edgelayer}
		\draw [style=wire] (17) to (21);
		\draw [style=wire] (21) to (20);
	\end{pgfonlayer}
\end{tikzpicture}
   \end{array},
   \begin{array}[c]{c}
\begin{tikzpicture}
	\begin{pgfonlayer}{nodelayer}
		\node [style=differential] (0) at (7, 1) {{\bf =\!=\!=\!=}};
		\node [style=object] (1) at (7, -0.5) {$B$};
		\node [style=component] (2) at (7, 0.25) {$g$};
		\node [style=component] (5) at (7.75, 1.75) {$\iota_1$};
		\node [style=function2] (6) at (6.25, 1.75) {$\iota_0$};
		\node [style=object] (7) at (6.25, 2.5) {$\oc X$};
		\node [style=object] (8) at (7.75, 2.5) {$A$};
	\end{pgfonlayer}
	\begin{pgfonlayer}{edgelayer}
		\draw [style=wire] (0) to (2);
		\draw [style=wire] (2) to (1);
		\draw [style=wire, in=135, out=-90] (6) to (0);
		\draw [style=wire, in=-90, out=45, looseness=1.25] (0) to (5);
		\draw [style=wire] (7) to (6);
		\draw [style=wire] (8) to (5);
	\end{pgfonlayer}
\end{tikzpicture}
   \end{array} \right) =  \left(\begin{array}[c]{c}
\begin{tikzpicture}
	\begin{pgfonlayer}{nodelayer}
		\node [style=object] (17) at (9.5, 1) {$\oc X$};
		\node [style=object] (20) at (9.5, -1) {$Y$};
		\node [style=component] (21) at (9.5, 0) {$f$};
	\end{pgfonlayer}
	\begin{pgfonlayer}{edgelayer}
		\draw [style=wire] (17) to (21);
		\draw [style=wire] (21) to (20);
	\end{pgfonlayer}
\end{tikzpicture}
   \end{array},
   \begin{array}[c]{c}
\begin{tikzpicture}
	\begin{pgfonlayer}{nodelayer}
		\node [style=object] (0) at (2.75, 5.5) {$\oc(X \times A)$};
		\node [style=differential] (1) at (2.75, 4.75) {{\bf =\!=\!=\!=}};
		\node [style=component] (2) at (3.5, 3.75) {$\pi_1$};
		\node [style=function2] (3) at (2, 3.75) {$\pi_0$};
		\node [style=differential] (4) at (2.75, 1.75) {{\bf =\!=\!=\!=}};
		\node [style=object] (5) at (2.75, 0.25) {$B$};
		\node [style=component] (6) at (2.75, 1) {$g$};
		\node [style=component] (7) at (3.5, 2.75) {$\iota_1$};
		\node [style=function2] (8) at (2, 2.75) {$\iota_0$};
	\end{pgfonlayer}
	\begin{pgfonlayer}{edgelayer}
		\draw [style=wire, in=90, out=-150, looseness=1.25] (1) to (3);
		\draw [style=wire, in=90, out=-30, looseness=1.25] (1) to (2);
		\draw [style=wire] (0) to (1);
		\draw [style=wire] (4) to (6);
		\draw [style=wire] (6) to (5);
		\draw [style=wire, in=150, out=-90] (8) to (4);
		\draw [style=wire, in=-90, out=30] (4) to (7);
		\draw [style=wire] (3) to (8);
		\draw [style=wire] (2) to (7);
	\end{pgfonlayer}
\end{tikzpicture}
   \end{array} \right) =\left(\begin{array}[c]{c}
\begin{tikzpicture}
	\begin{pgfonlayer}{nodelayer}
		\node [style=object] (17) at (9.5, 1) {$\oc X$};
		\node [style=object] (20) at (9.5, -1) {$Y$};
		\node [style=component] (21) at (9.5, 0) {$f$};
	\end{pgfonlayer}
	\begin{pgfonlayer}{edgelayer}
		\draw [style=wire] (17) to (21);
		\draw [style=wire] (21) to (20);
	\end{pgfonlayer}
\end{tikzpicture}
   \end{array},
   \begin{array}[c]{c}
\begin{tikzpicture}
	\begin{pgfonlayer}{nodelayer}
		\node [style=object] (17) at (9.5, 1) {$\oc (X \times A)$};
		\node [style=object] (20) at (9.5, -1) {$B$};
		\node [style=component] (21) at (9.5, 0) {$g$};
	\end{pgfonlayer}
	\begin{pgfonlayer}{edgelayer}
		\draw [style=wire] (17) to (21);
		\draw [style=wire] (21) to (20);
	\end{pgfonlayer}
\end{tikzpicture}
   \end{array} \right)  \]
So  $\mathsf{E} \circ \mathsf{E}^{-1} = 1_{ \mathcal{L}[\mathbb{X}_\oc]}$. Next for $\mathsf{E}^{-1} \circ \mathsf{E}$, again this is clear on objects since $\mathsf{E}^{-1}\mathsf{E}(X,A) = (X,A)$, while on maps we use Lemma \ref{lem:cokleisli-linearcontext}.(\ref{lem:cokleisli-linearcontext.i}): 
\[  \mathsf{E}^{-1}\mathsf{E}\left(\begin{array}[c]{c}
\begin{tikzpicture}
	\begin{pgfonlayer}{nodelayer}
		\node [style=object] (17) at (9.5, 1) {$\oc X$};
		\node [style=object] (20) at (9.5, -1) {$Y$};
		\node [style=component] (21) at (9.5, 0) {$f$};
	\end{pgfonlayer}
	\begin{pgfonlayer}{edgelayer}
		\draw [style=wire] (17) to (21);
		\draw [style=wire] (21) to (20);
	\end{pgfonlayer}
\end{tikzpicture}
   \end{array},
   \begin{array}[c]{c}
\begin{tikzpicture}
	\begin{pgfonlayer}{nodelayer}
		\node [style=object] (17) at (9.75, 1.75) {$\oc X$};
		\node [style=object] (19) at (10.75, 1.75) {$A$};
		\node [style=component] (20) at (10.25, 0.75) {$g$};
		\node [style=object] (21) at (10.25, -0.25) {$B$};
	\end{pgfonlayer}
	\begin{pgfonlayer}{edgelayer}
		\draw [style=wire] (20) to (21);
		\draw [style=wire, in=165, out=-90] (17) to (20);
		\draw [style=wire, in=-90, out=15] (20) to (19);
	\end{pgfonlayer}
\end{tikzpicture}
   \end{array} \right) 
 =  \mathsf{E}^{-1} \left(\begin{array}[c]{c}
\begin{tikzpicture}
	\begin{pgfonlayer}{nodelayer}
		\node [style=object] (17) at (9.5, 1) {$\oc X$};
		\node [style=object] (20) at (9.5, -1) {$Y$};
		\node [style=component] (21) at (9.5, 0) {$f$};
	\end{pgfonlayer}
	\begin{pgfonlayer}{edgelayer}
		\draw [style=wire] (17) to (21);
		\draw [style=wire] (21) to (20);
	\end{pgfonlayer}
\end{tikzpicture}
   \end{array},
   \begin{array}[c]{c}
\begin{tikzpicture}
	\begin{pgfonlayer}{nodelayer}
		\node [style=component] (0) at (7, 0.75) {$g$};
		\node [style=object] (2) at (7, 0) {$B$};
		\node [style=object] (3) at (7, 3.5) {$\oc (X \times A)$};
		\node [style=differential] (4) at (7, 2.75) {{\bf =\!=\!=\!=}};
		\node [style=component] (5) at (7.75, 1.75) {$\pi_1$};
		\node [style=function2] (6) at (6.25, 1.75) {$\pi_0$};
	\end{pgfonlayer}
	\begin{pgfonlayer}{edgelayer}
		\draw [style=wire] (0) to (2);
		\draw [style=wire] (3) to (4);
		\draw [style=wire, in=90, out=-150, looseness=1.25] (4) to (6);
		\draw [style=wire, in=135, out=-90] (6) to (0);
		\draw [style=wire, in=-90, out=45, looseness=1.25] (0) to (5);
		\draw [style=wire, in=90, out=-30, looseness=1.25] (4) to (5);
	\end{pgfonlayer}
\end{tikzpicture}
   \end{array} \right) =  \left(\begin{array}[c]{c}
\begin{tikzpicture}
	\begin{pgfonlayer}{nodelayer}
		\node [style=object] (17) at (9.5, 1) {$\oc X$};
		\node [style=object] (20) at (9.5, -1) {$Y$};
		\node [style=component] (21) at (9.5, 0) {$f$};
	\end{pgfonlayer}
	\begin{pgfonlayer}{edgelayer}
		\draw [style=wire] (17) to (21);
		\draw [style=wire] (21) to (20);
	\end{pgfonlayer}
\end{tikzpicture}
   \end{array},
   \begin{array}[c]{c}
\begin{tikzpicture}
	\begin{pgfonlayer}{nodelayer}
		\node [style=component] (9) at (7, 1) {$g$};
		\node [style=object] (10) at (7, 0.25) {$B$};
		\node [style=differential] (12) at (7, 3) {{\bf =\!=\!=\!=}};
		\node [style=component] (13) at (7.75, 2) {$\pi_1$};
		\node [style=function2] (14) at (6.25, 2) {$\pi_0$};
		\node [style=differential] (15) at (7, 3.75) {{\bf =\!=\!=\!=}};
		\node [style=component] (16) at (7.75, 4.75) {$\pi_1$};
		\node [style=function2] (17) at (6.25, 4.75) {$\pi_0$};
		\node [style=object] (18) at (6.25, 5.5) {$\oc X$};
		\node [style=object] (19) at (7.75, 5.5) {$A$};
	\end{pgfonlayer}
	\begin{pgfonlayer}{edgelayer}
		\draw [style=wire] (9) to (10);
		\draw [style=wire, in=90, out=-150, looseness=1.25] (12) to (14);
		\draw [style=wire, in=135, out=-90] (14) to (9);
		\draw [style=wire, in=-90, out=45, looseness=1.25] (9) to (13);
		\draw [style=wire, in=90, out=-30, looseness=1.25] (12) to (13);
		\draw [style=wire, in=-90, out=150, looseness=1.25] (15) to (17);
		\draw [style=wire, in=-90, out=30, looseness=1.25] (15) to (16);
		\draw [style=wire] (15) to (12);
		\draw [style=wire] (18) to (17);
		\draw [style=wire] (19) to (16);
	\end{pgfonlayer}
\end{tikzpicture}
   \end{array} \right) = \left(\begin{array}[c]{c}
\begin{tikzpicture}
	\begin{pgfonlayer}{nodelayer}
		\node [style=object] (17) at (9.5, 1) {$\oc X$};
		\node [style=object] (20) at (9.5, -1) {$Y$};
		\node [style=component] (21) at (9.5, 0) {$f$};
	\end{pgfonlayer}
	\begin{pgfonlayer}{edgelayer}
		\draw [style=wire] (17) to (21);
		\draw [style=wire] (21) to (20);
	\end{pgfonlayer}
\end{tikzpicture}
   \end{array},
   \begin{array}[c]{c}
\begin{tikzpicture}
	\begin{pgfonlayer}{nodelayer}
		\node [style=object] (17) at (9.75, 1.75) {$\oc X$};
		\node [style=object] (19) at (10.75, 1.75) {$A$};
		\node [style=component] (20) at (10.25, 0.75) {$g$};
		\node [style=object] (21) at (10.25, -0.25) {$B$};
	\end{pgfonlayer}
	\begin{pgfonlayer}{edgelayer}
		\draw [style=wire] (20) to (21);
		\draw [style=wire, in=165, out=-90] (17) to (20);
		\draw [style=wire, in=-90, out=15] (20) to (19);
	\end{pgfonlayer}
\end{tikzpicture}
   \end{array} \right) \] 
   So $\mathsf{E}^{-1} \circ \mathsf{E} = 1_{ \mathcal{L}_\oc[\mathbb{X}]}$. As a consequence, it follows that $\mathsf{E}^{-1}$ is also a fibration morphism. Therefore, $\mathsf{E}$ and $\mathsf{E}^{-1}$ are fibration isomorphisms and inverses of each other, and so we conclude that the fibrations ${\mathsf{p}_\oc: \mathcal{L}_\oc[\mathbb{X}] \to \mathbb{X}_\oc}$ and ${\mathsf{p}: \mathcal{L}[\mathbb{X}_\oc] \to \mathbb{X}_\oc}$ are isomorphic.
\end{proof}

As an immediate consequence, we have that fibres over the same object are isomorphic. 

\begin{corollary}\label{cor:fibres-equiv} Let $\mathbb{X}$ be a monoidal differential category with coalgebra modality $(\oc, \delta, \varepsilon, \Delta, e)$ and deriving transformation $\mathsf{d}: \oc A \otimes A \to \oc A$, and finite (bi)products. Then for each object $X$, $\mathcal{L}_\oc[X]$ is isomorphic to $\mathcal{L}[\mathbb{X}]$ via the functors $\mathsf{E}_X: \mathcal{L}_\oc[X] \to \mathcal{L}[X]$ and $\mathsf{E}^{-1}_X: \mathcal{L}[X] \to  \mathcal{L}_\oc[X]$ where: 
\begin{enumerate}[{\em (i)}]
\item $\mathsf{E}_X$ is defined on objects as $\mathsf{E}_X(A) = A$, and on maps $g: \oc X \otimes A \to B$ as follows: 
\begin{align*}
\mathsf{E}_X(g) = \left(\xymatrixcolsep{3pc}\xymatrix{\oc(X \times A) \ar[r]^-{\mathsf{d}^\circ_{X \times A}} & \oc (X \times A) \times (X \times A) \ar[r]^-{\oc(\pi_0) \times \pi_1} & \oc X \otimes A \ar[r]^-{g} & B } \right) \end{align*}
\begin{align*} \mathsf{E}\left(   \begin{array}[c]{c}
\begin{tikzpicture}
	\begin{pgfonlayer}{nodelayer}
		\node [style=object] (17) at (9.75, 1.75) {$\oc X$};
		\node [style=object] (19) at (10.75, 1.75) {$A$};
		\node [style=component] (20) at (10.25, 0.75) {$g$};
		\node [style=object] (21) at (10.25, -0.25) {$B$};
	\end{pgfonlayer}
	\begin{pgfonlayer}{edgelayer}
		\draw [style=wire] (20) to (21);
		\draw [style=wire, in=165, out=-90] (17) to (20);
		\draw [style=wire, in=-90, out=15] (20) to (19);
	\end{pgfonlayer}
\end{tikzpicture}
   \end{array} \right)  =    \begin{array}[c]{c}
\begin{tikzpicture}
	\begin{pgfonlayer}{nodelayer}
		\node [style=component] (0) at (7, 0.75) {$g$};
		\node [style=object] (2) at (7, 0) {$B$};
		\node [style=object] (3) at (7, 3.5) {$\oc (X \times A)$};
		\node [style=differential] (4) at (7, 2.75) {{\bf =\!=\!=\!=}};
		\node [style=component] (5) at (7.75, 1.75) {$\pi_1$};
		\node [style=function2] (6) at (6.25, 1.75) {$\pi_0$};
	\end{pgfonlayer}
	\begin{pgfonlayer}{edgelayer}
		\draw [style=wire] (0) to (2);
		\draw [style=wire] (3) to (4);
		\draw [style=wire, in=90, out=-150, looseness=1.25] (4) to (6);
		\draw [style=wire, in=135, out=-90] (6) to (0);
		\draw [style=wire, in=-90, out=45, looseness=1.25] (0) to (5);
		\draw [style=wire, in=90, out=-30, looseness=1.25] (4) to (5);
	\end{pgfonlayer}
\end{tikzpicture}
   \end{array} 
\end{align*}
\item $\mathsf{E}_X^{-1}$ is defined on objects as $\mathsf{E}^{-1}_X(A) = A$, and on maps $\llbracket g \rrbracket: \oc(X \times A) \to B$ as follows: 
\begin{align*}
\mathsf{E}_X^{-1}(\llbracket g \rrbracket) = \left( \xymatrixcolsep{3pc}\xymatrix{\oc X \otimes A \ar[r]^-{\oc(\iota_0) \otimes \iota_1} & \oc(X \times A) \times (X \times A) \ar[r]^-{\mathsf{d}_{X \times A}} & \oc (X \times A) \ar[r]^-{\llbracket g \rrbracket} & B } \right) \end{align*} \begin{align*}\mathsf{E}_X^{-1}\left(
   \begin{array}[c]{c}
\begin{tikzpicture}
	\begin{pgfonlayer}{nodelayer}
		\node [style=object] (17) at (9.5, 1) {$\oc (X \times A)$};
		\node [style=object] (20) at (9.5, -1) {$B$};
		\node [style=component] (21) at (9.5, 0) {$g$};
	\end{pgfonlayer}
	\begin{pgfonlayer}{edgelayer}
		\draw [style=wire] (17) to (21);
		\draw [style=wire] (21) to (20);
	\end{pgfonlayer}
\end{tikzpicture}
   \end{array} \right) =    \begin{array}[c]{c}
\begin{tikzpicture}
	\begin{pgfonlayer}{nodelayer}
		\node [style=differential] (0) at (7, 1) {{\bf =\!=\!=\!=}};
		\node [style=object] (1) at (7, -0.5) {$B$};
		\node [style=component] (2) at (7, 0.25) {$g$};
		\node [style=component] (5) at (7.75, 1.75) {$\iota_1$};
		\node [style=function2] (6) at (6.25, 1.75) {$\iota_0$};
		\node [style=object] (7) at (6.25, 2.5) {$\oc X$};
		\node [style=object] (8) at (7.75, 2.5) {$A$};
	\end{pgfonlayer}
	\begin{pgfonlayer}{edgelayer}
		\draw [style=wire] (0) to (2);
		\draw [style=wire] (2) to (1);
		\draw [style=wire, in=135, out=-90] (6) to (0);
		\draw [style=wire, in=-90, out=45, looseness=1.25] (0) to (5);
		\draw [style=wire] (7) to (6);
		\draw [style=wire] (8) to (5);
	\end{pgfonlayer}
\end{tikzpicture}
   \end{array}
\end{align*}
\end{enumerate}
As such, a coKleisli map $\llbracket f \rrbracket: \oc(X \times A) \to B$ is linear in context $X$ if and only if there exists a (necessarily unique) map $g: \oc X \otimes A \to B$ such that $\llbracket f \rrbracket = \mathsf{d}^\circ_{X \times A}; (\oc(\pi_0) \otimes \pi_1); g$. 
\end{corollary}

\section{Cartesian reverse differential categories}\label{sec:CRDC}

In this section we recall the key definitions and results on Cartesian reverse differential categories from \cite[]{cockett_et_al:LIPIcs:2020:11661}.  

\subsection{Definition}

\begin{definition}\label{cartrevdiffdef}\cite[Definition 13]{cockett_et_al:LIPIcs:2020:11661} A \textbf{Cartesian reverse differential category} (CRDC) is a Cartesian left additive category $\mathbb{X}$ equipped with a \textbf{reverse differential combinator} $\mathsf{R}$, which is a family of operators ${\mathsf{R}: \mathbb{X}(A,B) \to \mathbb{X}(A \times B,A)}$, $f \mapsto \mathsf{R}[f]$, where $\mathsf{R}[f]$ is called the \textbf{reverse derivative} of $f$, such that the following seven axioms hold:  
\begin{enumerate}[{\bf [RD.1]}]
\item Additivity of reverse differentiation: 
\begin{align*}
\mathsf{R}[f+g] = \mathsf{R}[f] + \mathsf{R}[g] && \mathsf{R}[0]=0
\end{align*}
\item Additivity of the reverse derivative in its second variable: 
\begin{align*} (1_A \times +_B); \mathsf{R}[f] = (1_A \times \pi_0);\mathsf{R}[f] + (1_A \times \pi_1)\mathsf{R}[f] && \iota_0; \mathsf{R}[f]=0
\end{align*}
\item Coherence with identities and projections: 
\begin{align*} \mathsf{R}[1_A]=\pi_1 && \mathsf{R}[\pi_0] = \pi_1;\iota_0 && \mathsf{R}[\pi_1] = \pi_1;\iota_1 \end{align*}
\item Coherence with pairings: 
\begin{align*} \mathsf{R}[\langle f, g \rangle] = (1_A \times \pi_0); \mathsf{R}[f] + (1_A \times \pi_1);\mathsf{R}[g]
\end{align*}
\item Reverse chain rule: 
\begin{align*} \mathsf{R}[fg] = \left \langle \pi_0, \langle \pi_0; f, \pi_1 \rangle; \mathsf{R}[g] \right \rangle; \mathsf{R}[f] \end{align*}
\item Linearity of the reverse derivative in its second variable: 
\begin{align*} (\iota_0 \times \iota_1); (\iota_0 \times 1_{A \times B});  \mathsf{R}\!\left[\mathsf{R}\!\left[\mathsf{R}[f] \right] \right]; \pi_1 = \mathsf{R}[f]
\end{align*}
\item Symmetry of mixed partial derivatives: 
\begin{align*} & c_A; (\iota_0 \times 1_{A \times A}); \mathsf{R}\!\left[\mathsf{R}\!\left[(\iota_0 \times 1_A); \mathsf{R}\!\left[\mathsf{R}[f] \right]; \pi_1 \right] \right]; \pi_1 \\
&= (\iota_0 \times 1_{A \times A});\mathsf{R}\!\left[\mathsf{R}\!\left[(\iota_0 \times 1_A); \mathsf{R}\!\left[\mathsf{R}[f] \right]; \pi_1 \right] \right]; \pi_1 \end{align*}
\end{enumerate}
\end{definition}

For more discussion on the definition and examples, see \cite[]{cockett_et_al:LIPIcs:2020:11661}.  One of the central results of that paper is that any CRDC also has the structure of a CDC:

\begin{therm}\label{thm:crdc_to_cdc}\cite[Theorem 16]{cockett_et_al:LIPIcs:2020:11661} If $\mathbb{X}$ is a Cartesian reverse differential category with reverse differential combinator $\mathsf{R}$, then $\mathbb{X}$ is also a Cartesian differential category, where for a map $f: A \to B$, its derivative $D[f]: A \times A \to B$ is defined as follows:
  \[  D[f] :=  \xymatrixcolsep{5pc}\xymatrix{ A \times A \ar[r]^-{\iota_0 \times 1_A} & A \times B \times A  \ar[r]^-{R[R[f]]} & A \times B \ar[r]^-{\pi_1} & B
  } \]
\end{therm}

In general, however, there is no reason why a CDC should have the structure of a CRDC.  In the next two sections we look at what additional structure is needed on a CDC to get a CRDC.

\subsection{Dagger fibrations}

In this section, we review the structure necessary to go from a Cartesian differential category to a reverse Cartesian differential category: a \emph{dagger fibration} structure on its fibration of linear maps.

The idea of a dagger fibration is to capture (from the fibrational point of view) the idea of each fibre being a dagger category.  Recalling that a dagger category involves a functor from a category to its opposite, from the fibrational point of view, this must then involve a map from a fibration to its \emph{dual fibration}, a fibration which takes the original fibration and takes the opposite category in each fibre.  This fibration can be defined directly as follows:

\begin{definition}\cite[Defn. 1.10.11]{jacobs1999categorical}
If $p: \mathbb{E} \to \mathbb{B}$ is a fibration, its \textbf{dual fibration} is the fibration $p^*: \mathbb{E}^* \to \mathbb{B}$ given by: 
\begin{enumerate}[{\em (i)}]
    \item The objects of $\mathbb{E}^\ast$ are the same as the objects of $\mathbb{E}$; that is, $Ob(\mathbb{E}^\ast) = Ob(\mathbb{E})$; 
    \item A map from $E$ to $E'$ in $\mathbb{E}^*$ consists of an equivalence class of spans
    \[ \xymatrix{ & S \ar[dl]_{v} \ar[dr]^{c} & \\
    E & & E'} \]
    where $v$ is vertical and $c$ Cartesian. Such a span is equivalent to $(S',v',c')$ if there is a vertical isomorphism $\alpha: S \to S'$ which makes the relevant triangles commute.   
\end{enumerate}
\end{definition}

The following are our two primary examples of the dual fibration construction.  

\begin{example} \normalfont
The dual fibration of the fibration of Proposition \ref{prop:context_fibration} has objects pairs $(X,A)$, with a map from $(X,A)$ to $(Y,B)$ consisting of a coKleisli map $\llbracket f \rrbracket: \oc X \to Y$ and a map
    \[ g: \oc X \otimes B \to A \]
\end{example}

\begin{example} \normalfont
The dual fibration of the fibration of Proposition \ref{prop:linear_fibration} has objects pairs $(X,A)$, with a map from $(X,A)$ to $(Y,B)$ consisting of a map $f: X \to Y$ and a map
    \[ g: X \times B \to A \]
\end{example}

\begin{lemma}\cite[Lemma 1.10.12]{jacobs1999categorical}
If $p: \mathbb{E} \to \mathbb{B}$ is a fibration, then:
\begin{enumerate}[{\em (i)}]
    \item For each object $B$ of $\mathbb{B}$ there is an isomorphism of categories
    \[ [p^{-1}(B)]^{\mathsf{op}} \cong (p^*)^{-1}(B) \]
which is natural in $B$; 
    \item There is an isomorphism of fibrations $(\mathbb{E}^*)^* \cong \mathbb{E}$ over $\mathbb{B}$.  
\end{enumerate}
\end{lemma}

The following does not appear in any published accounts on the dual fibration, but is straightforward:

\begin{lemma}
If $(p: \mathbb{E} \to \mathbb{B})$ and $(p': \mathbb{E}' \to \mathbb{B}')$ are fibrations and 
    \[ \xymatrix{\mathbb{E} \ar[r]^{F} \ar[d]_p  & \mathbb{E}' \ar[d]^{p'} \\ \mathbb{B} \ar[r]_G & \mathbb{B}'} \]
is a morphism of fibrations, then there is a morphism of fibrations
    \[ \xymatrix{\mathbb{E}^* \ar[r]^{F^*} \ar[d]_{p^*}  & (\mathbb{E}')^* \ar[d]^{(p')^*} \\ \mathbb{B} \ar[r]_G & \mathbb{B}'} \]
which sends a span $(S, v, C): X \to X'$ to $(FS,F(v),F(c)): FX \to FX'$.  
\end{lemma}

We can now succinctly define what it means for a fibration to have dagger category structure in each fibre.

\begin{definition}\cite[Definition 33]{cockett_et_al:LIPIcs:2020:11661} A \textbf{dagger fibration} consists of a fibration $p: \mathbb{E} \to \mathbb{B}$ together with a morphism of fibrations
    \[ (-)^{\dagger}: \mathbb{E} \to \mathbb{E}^* \]
which is stationary on objects and ``is its own inverse''; that is, the composite
    \[ \mathbb{E} \to^{(-)^{\dagger}} \mathbb{E}^* \to^{((-)^{\dagger})^*} (\mathbb{E}^*)^* \cong \mathbb{E} \]
is the identity functor.  
\end{definition}

\begin{example}\label{ex:daggerlin} \normalfont
Let us consider what it would mean to have dagger structure on the linear fibration $\mathcal{L}[X]$ (Definition \ref{defn:linear_fibration}) of a Cartesian differential category $\mathbb{X}$.  In particular, this would mean that for each map $f: X \times A \to B$ which is linear in context $X$, we would need to give a map $f^{\dagger[X]}: X \times B \to A$ which is also linear in context $X$. The axioms for a dagger fibration are then equivalent to asking that each fibre $\mathcal{L}[X]$ (Lemma \ref{lemma:L_fibres}) be a dagger category with dagger $\dagger[X]: \mathcal{L}[X]^{\mathsf{op}} \to \mathcal{L}[X]$ such that each substitution functor preserves the dagger. Explicitly, the operation $\dagger[-]$ satisfies the following: 
\begin{enumerate}[{\em (i)}]
  \item Contravariant functoriality: $(\langle \pi_0, f \rangle; g)^{\dagger[X]} = \langle \pi_0, g^{\dagger[X]} \rangle; f^{\dagger[X]}$ and $\pi_1^{\dagger[X]} = \pi_1$
    \item Involutive: ${f^{\dagger[X]}}^{\dagger[X]} =f$
    \item Change of Base: for every map $h: Y \to X$ in $\mathbb{X}$, its associated substitution functor $h^\ast: \mathcal{L}[X] \to \mathcal{L}[Y]$ preserves the dagger, that is, $(h^\ast(f))^{\dagger[Y]} = h^\ast\left( f^{\dagger[X]} \right)$. 
\end{enumerate}
%and $\mathcal{L}[\mathbb{X}][C]$ has dagger biproducts, that is, the following equalities hold: 
%\begin{align*} 
%(\pi_1\pi_j)^{\dagger[C]} = \pi_1\iota_j && %(\pi_1\iota_j)^{\dagger[C]} = \pi_1\pi_j 
%\end{align*}
\end{example}

If $\mathbb{X}$ is a Cartesian \emph{reverse} differential category, then by Theorem \ref{thm:crdc_to_cdc}, it is also a Cartesian differential category, and in this case its associated linear fibration has dagger structure: 

\begin{therm}\label{thm:crdc_dagger_fibration}\cite[Theorem 37]{cockett_et_al:LIPIcs:2020:11661}
If $\mathbb{X}$ is a CRDC, then its associated fibration of linear maps $\mathcal{L}[\mathbb{X}]$ is a dagger fibration, where for a map $f: X \times A \to B$ which is linear in context $X$, 
    \[ f^{\dagger[X]} :=  X \times B \to^{\iota_0 \times 1} X \times A \times B \to^{R[f]} X \times A \to^{\pi_1} A \]
\end{therm}

It will also be useful (see Lemma \ref{lemma:compactclosed_to_dagger}) to characterize when the fibration associated to any coalgebra modality has dagger fibration structure.  

\begin{example} \normalfont \label{ex:!daggerfibration} To give a dagger fibration structure on the fibration of Proposition \ref{prop:context_fibration} corresponds to associating every map $f: \oc X \otimes A \to B$ to a map $f^{\dagger[X]}: \oc X \otimes B \to A$, which we draw in the graphical calculus simply as: 
\[ f^\dagger[X] := \left( \begin{array}[c]{c}
\begin{tikzpicture}
	\begin{pgfonlayer}{nodelayer}
		\node [style=object] (17) at (9.75, 1.75) {$\oc X$};
		\node [style=object] (19) at (10.75, 1.75) {$A$};
		\node [style=component] (20) at (10.25, 0.75) {$f$};
		\node [style=object] (21) at (10.25, -0.25) {$B$};
	\end{pgfonlayer}
	\begin{pgfonlayer}{edgelayer}
		\draw [style=wire] (20) to (21);
		\draw [style=wire, in=165, out=-90] (17) to (20);
		\draw [style=wire, in=-90, out=15] (20) to (19);
	\end{pgfonlayer}
\end{tikzpicture}
   \end{array} \right)^{\dagger[X]} = \begin{array}[c]{c}
\begin{tikzpicture}
	\begin{pgfonlayer}{nodelayer}
		\node [style=object] (17) at (9.75, 1.75) {$\oc X$};
		\node [style=object] (19) at (10.75, 1.75) {$B$};
		\node [style=component] (20) at (10.25, 0.75) {$f^\dagger$};
		\node [style=object] (21) at (10.25, -0.25) {$A$};
	\end{pgfonlayer}
	\begin{pgfonlayer}{edgelayer}
		\draw [style=wire] (20) to (21);
		\draw [style=wire, in=165, out=-90] (17) to (20);
		\draw [style=wire, in=-90, out=15] (20) to (19);
	\end{pgfonlayer}
\end{tikzpicture}
   \end{array} \] 
 Once again, the axioms for a dagger fibration in this case are equivalent to asking that each of the fibres $\mathcal{L}_\oc[X]$ be a dagger category with dagger $\dagger[X]: \mathcal{L}_\oc[X]^{\mathsf{op}} \to \mathcal{L}_\oc[X]$ such that each substitution functor preserves the dagger.  Explicitly, the operation $\dagger[-]$ satisfies the following: 
\begin{enumerate}[{\em (i)}]
\item Contravariant functorality: 
\begin{align*}
\left((\Delta_X \otimes 1_A);(1_{\oc X} \otimes f);g \right)^{\dagger[X]} = (\Delta_X \otimes 1_C);(1_{\oc X} \otimes g^{\dagger[X]}); f^{\dagger[X]} && {(e_X \otimes 1_A)^{\dagger[X]} = e_X \otimes 1_A}
\end{align*}
\begin{align*}
\left(    \begin{array}[c]{c}
\begin{tikzpicture}
	\begin{pgfonlayer}{nodelayer}
		\node [style=duplicate] (1) at (9.5, 1) {$\Delta$};
		\node [style=object] (2) at (9.5, 1.75) {$\oc X$};
		\node [style=component] (3) at (10.5, 0) {$f$};
		\node [style=object] (4) at (11, 1.75) {$A$};
		\node [style=component] (5) at (9.75, -1.25) {$g$};
		\node [style=object] (6) at (9.75, -2) {$C$};
	\end{pgfonlayer}
	\begin{pgfonlayer}{edgelayer}
		\draw [style=wire] (2) to (1);
		\draw [style=wire, in=-90, out=15, looseness=0.75] (3) to (4);
		\draw [style=wire] (5) to (6);
		\draw [style=wire, in=30, out=-90] (3) to (5);
		\draw [style=wire, in=165, out=-150] (1) to (5);
		\draw [style=wire, in=165, out=-15, looseness=1.25] (1) to (3);
	\end{pgfonlayer}
\end{tikzpicture}
   \end{array} \right)^{\dagger[X]} =     \begin{array}[c]{c}
\begin{tikzpicture}
	\begin{pgfonlayer}{nodelayer}
		\node [style=duplicate] (1) at (9.5, 1) {$\Delta$};
		\node [style=object] (2) at (9.5, 1.75) {$\oc X$};
		\node [style=component] (3) at (10.5, 0) {$g^\dagger$};
		\node [style=object] (4) at (11, 1.75) {$C$};
		\node [style=component] (5) at (9.75, -1.25) {$f^\dagger$};
		\node [style=object] (6) at (9.75, -2.25) {$A$};
	\end{pgfonlayer}
	\begin{pgfonlayer}{edgelayer}
		\draw [style=wire] (2) to (1);
		\draw [style=wire, in=-90, out=15, looseness=0.75] (3) to (4);
		\draw [style=wire] (5) to (6);
		\draw [style=wire, in=30, out=-90] (3) to (5);
		\draw [style=wire, in=165, out=-150] (1) to (5);
		\draw [style=wire, in=165, out=-15, looseness=1.25] (1) to (3);
	\end{pgfonlayer}
\end{tikzpicture}
   \end{array} && \left( \begin{array}[c]{c}
\begin{tikzpicture}
	\begin{pgfonlayer}{nodelayer}
		\node [style=object] (17) at (9.75, 1.75) {$\oc X$};
		\node [style=object] (19) at (10.5, 1.75) {$A$};
		\node [style=object] (21) at (10.5, -0.25) {$A$};
		\node [style=component] (22) at (9.75, 0.5) {$e$};
	\end{pgfonlayer}
	\begin{pgfonlayer}{edgelayer}
		\draw [style=wire] (19) to (21);
		\draw [style=wire] (17) to (22);
	\end{pgfonlayer}
\end{tikzpicture}
   \end{array}  \right)^{\dagger[X]} = \begin{array}[c]{c}
\begin{tikzpicture}
	\begin{pgfonlayer}{nodelayer}
		\node [style=object] (17) at (9.75, 1.75) {$\oc X$};
		\node [style=object] (19) at (10.5, 1.75) {$A$};
		\node [style=object] (21) at (10.5, -0.25) {$A$};
		\node [style=component] (22) at (9.75, 0.5) {$e$};
	\end{pgfonlayer}
	\begin{pgfonlayer}{edgelayer}
		\draw [style=wire] (19) to (21);
		\draw [style=wire] (17) to (22);
	\end{pgfonlayer}
\end{tikzpicture}
   \end{array} 
\end{align*}
\item Involution: ${f^{\dagger[X]}}^{\dagger[X]} =f$
\[ \left( \begin{array}[c]{c}
\begin{tikzpicture}
	\begin{pgfonlayer}{nodelayer}
		\node [style=object] (17) at (9.75, 1.75) {$\oc X$};
		\node [style=object] (19) at (10.75, 1.75) {$B$};
		\node [style=component] (20) at (10.25, 0.75) {$f^\dagger$};
		\node [style=object] (21) at (10.25, -0.25) {$A$};
	\end{pgfonlayer}
	\begin{pgfonlayer}{edgelayer}
		\draw [style=wire] (20) to (21);
		\draw [style=wire, in=165, out=-90] (17) to (20);
		\draw [style=wire, in=-90, out=15] (20) to (19);
	\end{pgfonlayer}
\end{tikzpicture}
   \end{array} \right)^{\dagger[X]} = \begin{array}[c]{c}
\begin{tikzpicture}
	\begin{pgfonlayer}{nodelayer}
		\node [style=object] (17) at (9.75, 1.75) {$\oc X$};
		\node [style=object] (19) at (10.75, 1.75) {$A$};
		\node [style=component] (20) at (10.25, 0.75) {$f$};
		\node [style=object] (21) at (10.25, -0.25) {$B$};
	\end{pgfonlayer}
	\begin{pgfonlayer}{edgelayer}
		\draw [style=wire] (20) to (21);
		\draw [style=wire, in=165, out=-90] (17) to (20);
		\draw [style=wire, in=-90, out=15] (20) to (19);
	\end{pgfonlayer}
\end{tikzpicture}
   \end{array} \] 
\item Change of Base: For every coKleisli map $\llbracket h \rrbracket: \oc Y \to X$, the substitution functor ${\llbracket h \rrbracket^\ast: \mathcal{L}_\oc[Y] \to \mathcal{L}_\oc[X]}$ preserves the dagger, that is, $ \left( (\delta \otimes 1) (\oc(h) \otimes 1) f \right)^{\dagger[Y]} = (\delta \otimes 1) (\oc(h) \otimes 1)f^{\dagger[X]}$
\begin{align*}
\left( \begin{array}[c]{c}
\begin{tikzpicture}
	\begin{pgfonlayer}{nodelayer}
		\node [style=component] (0) at (9, 0.25) {$\delta$};
		\node [style=object] (1) at (10.5, 1.25) {$A$};
		\node [style=component] (2) at (9.75, -1.75) {$f$};
		\node [style=object] (3) at (9.75, -2.75) {$B$};
		\node [style=function2] (4) at (9, -0.75) {$h$};
		\node [style=object] (5) at (9, 1.25) {$\oc X$};
	\end{pgfonlayer}
	\begin{pgfonlayer}{edgelayer}
		\draw [style=wire] (2) to (3);
		\draw [style=wire, bend right, looseness=1.25] (4) to (2);
		\draw [style=wire] (0) to (4);
		\draw [style=wire, in=30, out=-90, looseness=0.75] (1) to (2);
		\draw [style=wire] (5) to (0);
	\end{pgfonlayer}
\end{tikzpicture}
   \end{array} \right)^{\dagger[X]} =  \begin{array}[c]{c}
\begin{tikzpicture}
	\begin{pgfonlayer}{nodelayer}
		\node [style=component] (0) at (9, 0.25) {$\delta$};
		\node [style=object] (1) at (10.5, 1.25) {$B$};
		\node [style=component] (2) at (9.75, -1.75) {$f^\dagger$};
		\node [style=object] (3) at (9.75, -2.75) {$A$};
		\node [style=function2] (4) at (9, -0.75) {$h$};
		\node [style=object] (5) at (9, 1.25) {$\oc X$};
	\end{pgfonlayer}
	\begin{pgfonlayer}{edgelayer}
		\draw [style=wire] (2) to (3);
		\draw [style=wire, bend right, looseness=1.25] (4) to (2);
		\draw [style=wire] (0) to (4);
		\draw [style=wire, in=30, out=-90, looseness=0.75] (1) to (2);
		\draw [style=wire] (5) to (0);
	\end{pgfonlayer}
\end{tikzpicture}
   \end{array} 
\end{align*} 
%\item $\dagger[C]$ preserves the tensor product: 
%\[ \left( (\Delta \otimes 1)(1 \otimes \sigma \otimes 1)(f \otimes g) \right)^{\dagger[C]}= (\Delta \otimes 1)(1 \otimes \sigma \otimes 1)(f^{\dagger[C]} \otimes g^{\dagger[C]}) \]
\end{enumerate}
\end{example}

\subsection{Characterization theorem for CRDCs}

With one extra ingredient, the structure described in Theorem \ref{thm:crdc_dagger_fibration} is enough to characterize CRDCs.

\begin{definition} A CDC $\mathbb{X}$ has a \textbf{contextual linear dagger} if $\mathcal{L}[\mathbb{X}]$ has a dagger fibration structure for which each fibre $\mathcal{L}[\mathbb{X}]$ has dagger biproducts.  
\end{definition}

\begin{therm}\label{thm:characterization_of_crdc}\cite[Theorem 42] {cockett_et_al:LIPIcs:2020:11661}
A Cartesian reverse differential category is precisely the same as a Cartesian differential $\mathbb{X}$ with a contextual linear dagger. 
\end{therm}

In particular, given such a CDC, its reverse combinator is given by taking the dagger of its (forward) derivative $\mathsf{D}$. So for a map $f: A \to B$, its derivative is $\mathsf{D}[f]: A \times A \to B$ (which is linear in its second variable by \textbf{[CD.6]}), so we define the reverse derivative as follows:
    \[ R[f] := A \times B \to^{D[f]^{\dagger[A]}} A \]

\section{Monoidal Reverse Differential Categories}\label{sec:mrdc}

This section introduces the main subject of the article: monoidal reverse differential categories.  As noted in the introduction, we would like these structures to satisfy several requirements:
\begin{enumerate}
    \item Just as every Cartesian reverse differential category is a Cartesian differential category, so should every monoidal reverse differential category be a monoidal differential category.
    \item Just as every monoidal differential storage category has Cartesian differential category structure on its coKleisli category, so should every monoidal reverse differential storage category have Cartesian reverse differential structure on its coKleisli category.
    \item Examples of this structure should be interesting and varied.  
\end{enumerate}

In the next section, we will see that requirements 1 and 2 will force monoidal reverse differential categories to be self-dual compact closed.  

\subsection{Monoidal reverse differential categories should be self-dual compact closed}\label{sec:mrdc_is_sdcc}

First, let us recall the relevant definitions.

\begin{definition}\label{SDCC} In a symmetric monoidal category, a \textbf{self-dual object} \cite[Definition 3.1]{heunen2019categories} is a triple $(A, \cup_A, \cap_A)$ consisting of an object $A$ and two maps $\cup_A: A \otimes A \to k$ and $\cap_A: k \to A \otimes A$, drawn in the graphical calculus as follows: 
\begin{align*}
  \begin{array}[c]{c} \cup_A    \end{array} : =    \begin{array}[c]{c}
\begin{tikzpicture}
	\begin{pgfonlayer}{nodelayer}
		\node [style=object] (1) at (1.25, 1.25) {$A$};
		\node [style=object] (7) at (2.5, 1.25) {$A$};
	\end{pgfonlayer}
	\begin{pgfonlayer}{edgelayer}
		\draw [style=wire, bend right=90, looseness=2.00] (1) to (7);
	\end{pgfonlayer}
\end{tikzpicture}
   \end{array} &&   \begin{array}[c]{c} \cap_A    \end{array} : =    \begin{array}[c]{c}
\begin{tikzpicture}
	\begin{pgfonlayer}{nodelayer}
		\node [style=object] (1) at (1.25, 1.25) {$A$};
		\node [style=object] (7) at (2.5, 1.25) {$A$};
	\end{pgfonlayer}
	\begin{pgfonlayer}{edgelayer}
		\draw [style=wire, bend left=90, looseness=2.00] (1) to (7);
	\end{pgfonlayer}
\end{tikzpicture}
   \end{array}
\end{align*}
 such that the following diagram commutes (often called the \textbf{snake equations)}:
 \begin{align*}
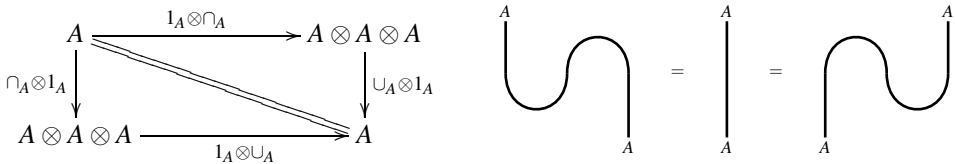
   \begin{array}[c]{c}
 \xymatrixcolsep{5pc}\xymatrix{ A \ar@{=}[dr]_-{} \ar[d]_-{\cap_A \otimes 1_A} \ar[r]^-{1_A \otimes \cap_A} & A \otimes A \otimes A \ar[d]^-{\cup_A \otimes 1_A} \\
 A \otimes A \otimes A  \ar[r]_-{1_A \otimes \cup_A} & A   
  } 
   \end{array} &&    \begin{array}[c]{c}
\begin{tikzpicture}
	\begin{pgfonlayer}{nodelayer}
		\node [style=none] (0) at (1.75, 2.75) {};
		\node [style=none] (1) at (0.5, 2.75) {};
		\node [style=object] (3) at (0.5, 4) {$A$};
		\node [style=none] (4) at (3, 2.75) {};
		\node [style=object] (5) at (3, 1.25) {$A$};
		\node [style=none] (7) at (1.75, 2.75) {};
		\node [style=object] (8) at (4, 2.75) {$=$};
		\node [style=object] (9) at (5, 4) {$A$};
		\node [style=object] (10) at (5, 1.25) {$A$};
		\node [style=object] (11) at (6, 2.75) {$=$};
		\node [style=none] (12) at (8.25, 2.75) {};
		\node [style=none] (13) at (9.5, 2.75) {};
		\node [style=object] (14) at (9.5, 4) {$A$};
		\node [style=none] (15) at (7, 2.75) {};
		\node [style=object] (16) at (7, 1.25) {$A$};
		\node [style=none] (17) at (8.25, 2.75) {};
	\end{pgfonlayer}
	\begin{pgfonlayer}{edgelayer}
		\draw [style=wire, bend left=90, looseness=2.00] (0.center) to (4.center);
		\draw [style=wire] (4.center) to (5);
		\draw [style=wire, bend right=90, looseness=2.00] (1.center) to (7.center);
		\draw [style=wire] (3) to (1.center);
		\draw [style=wire] (9) to (10);
		\draw [style=wire, bend right=90, looseness=2.00] (12.center) to (15.center);
		\draw [style=wire] (15.center) to (16);
		\draw [style=wire, bend left=90, looseness=2.00] (13.center) to (17.center);
		\draw [style=wire] (14) to (13.center);
	\end{pgfonlayer}
\end{tikzpicture}
   \end{array}
\end{align*}
A self-dual object $(A, \cup_A, \cap_A)$ is said to satisfy the \textbf{twist equations} if the cup and cap are symmetry invariant, that is, the following diagram commutes: 
 \begin{align*}
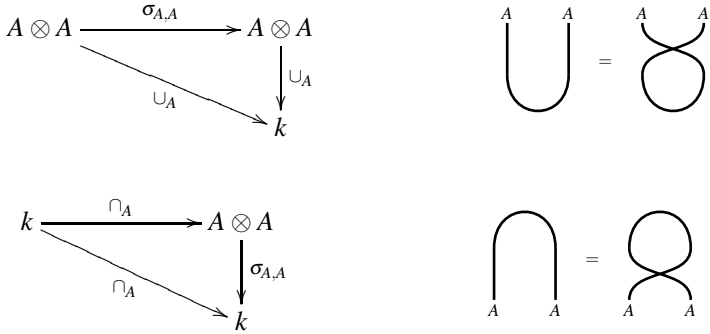
   \begin{array}[c]{c}
 \xymatrixcolsep{5pc}\xymatrix{ A \otimes A \ar[dr]_-{\cup_A}  \ar[r]^-{\sigma_{A,A}} & A \otimes A \ar[d]^-{\cup_A} \\
& k   
  } 
   \end{array} &&    \begin{array}[c]{c}
\begin{tikzpicture}
	\begin{pgfonlayer}{nodelayer}
		\node [style=none] (1) at (1.5, 1) {};
		\node [style=none] (7) at (2.75, 1) {};
		\node [style=object] (8) at (1.5, 2.25) {$A$};
		\node [style=object] (9) at (2.75, 2.25) {$A$};
		\node [style=none] (10) at (-1.25, 1) {};
		\node [style=none] (11) at (0, 1) {};
		\node [style=object] (12) at (-1.25, 2.25) {$A$};
		\node [style=object] (13) at (0, 2.25) {$A$};
		\node [style=object] (14) at (0.75, 1.25) {$=$};
	\end{pgfonlayer}
	\begin{pgfonlayer}{edgelayer}
		\draw [style=wire, bend right=90, looseness=2.00] (1.center) to (7.center);
		\draw [style=wire, in=-90, out=90] (1.center) to (9);
		\draw [style=wire, in=90, out=-90] (8) to (7.center);
		\draw [style=wire, bend right=90, looseness=2.00] (10.center) to (11.center);
		\draw [style=wire] (12) to (10.center);
		\draw [style=wire] (13) to (11.center);
	\end{pgfonlayer}
\end{tikzpicture}
   \end{array}
\end{align*} 
 \begin{align*}
   \begin{array}[c]{c}
 \xymatrixcolsep{5pc}\xymatrix{ k \ar[dr]_-{\cap_A}  \ar[r]^-{\cap_A} & A \otimes A \ar[d]^-{\sigma_{A,A}} \\
& k   
  } 
   \end{array} &&    \begin{array}[c]{c}
\begin{tikzpicture}
	\begin{pgfonlayer}{nodelayer}
		\node [style=none] (1) at (1.5, 2.25) {};
		\node [style=none] (7) at (2.75, 2.25) {};
		\node [style=object] (8) at (1.5, 1) {$A$};
		\node [style=object] (9) at (2.75, 1) {$A$};
		\node [style=none] (10) at (-1.25, 2.25) {};
		\node [style=none] (11) at (0, 2.25) {};
		\node [style=object] (12) at (-1.25, 1) {$A$};
		\node [style=object] (13) at (0, 1) {$A$};
		\node [style=object] (14) at (0.75, 2) {$=$};
	\end{pgfonlayer}
	\begin{pgfonlayer}{edgelayer}
		\draw [style=wire, bend left=90, looseness=2.00] (1.center) to (7.center);
		\draw [style=wire, in=90, out=-90] (1.center) to (9);
		\draw [style=wire, in=-90, out=90] (8) to (7.center);
		\draw [style=wire, bend left=90, looseness=2.00] (10.center) to (11.center);
		\draw [style=wire] (12) to (10.center);
		\draw [style=wire] (13) to (11.center);
	\end{pgfonlayer}
\end{tikzpicture}
   \end{array}
\end{align*} 
  A \textbf{self-dual compact closed category} \cite[Section 5]{selinger2010autonomous} is a symmetric monoidal category $\mathbb{X}$ equipped with a family of maps $\cup_A: A \otimes A \to k$ and $\cap_A: k \to A \otimes A$ such that for each object $A$, $(A, \cup_A, \cap_A)$ is a self-dual object which satisfies the twist equations.
\end{definition}

Without loss of generality, in a self-dual compact closed category, we use the convention that for each pair of objects $A$ and $B$: \begin{align*}
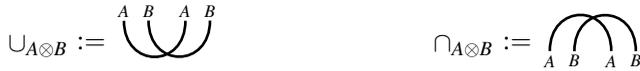
  \begin{array}[c]{c} \cup_{A \otimes B}    \end{array} : =    \begin{array}[c]{c}
\begin{tikzpicture}
	\begin{pgfonlayer}{nodelayer}
		\node [style=object] (135) at (7, -4.5) {$A$};
		\node [style=object] (136) at (8.25, -4.5) {$A$};
		\node [style=object] (137) at (7.5, -4.5) {$B$};
		\node [style=object] (138) at (8.75, -4.5) {$B$};
	\end{pgfonlayer}
	\begin{pgfonlayer}{edgelayer}
		\draw [style=wire, bend right=90, looseness=2.00] (135) to (136);
		\draw [style=wire, bend right=90, looseness=2.00] (137) to (138);
	\end{pgfonlayer}
\end{tikzpicture}
   \end{array} &&   \begin{array}[c]{c} \cap_{A \otimes B}    \end{array} : =    \begin{array}[c]{c}
\begin{tikzpicture}
	\begin{pgfonlayer}{nodelayer}
		\node [style=object] (135) at (7, -4.5) {$A$};
		\node [style=object] (136) at (8.25, -4.5) {$A$};
		\node [style=object] (137) at (7.5, -4.5) {$B$};
		\node [style=object] (138) at (8.75, -4.5) {$B$};
	\end{pgfonlayer}
	\begin{pgfonlayer}{edgelayer}
		\draw [style=wire, bend left=90, looseness=2.00] (135) to (136);
		\draw [style=wire, bend left=90, looseness=2.00] (137) to (138);
	\end{pgfonlayer}
\end{tikzpicture}
   \end{array}
\end{align*}
and that for the unit $k$, $\cup_k = \cap_k = 1_k$ (which are just empty drawings in the graphical calculus). We also point out that the twist equations are not strictly necessary for the story of this paper, or for defining a compact closed category where $A = A^\ast$ as in \cite[]{selinger2010autonomous}. However, in this paper we have elected to include them in the definition as it is more practical, greatly simplifying our string diagram computations, and all of the examples of monoidal reverse differential categories that we have discovered so far satisfy this twist equation. Furthermore, in most of the literature when considering self-dual compact closed categories, the twist equations are often taken as axioms such as for the ZX calculus \cite[]{Coecke08interactingquantum}, quantum computing (which considers the free self-dual compact closed PROP) \cite[]{7174913}, and hypergraph categories \cite[]{fong2019hypergraph}. 

We now discuss the induced dagger functor of a self-dual compact closed category: 

\begin{lemma} \cite[Remark 4.5]{selinger2010autonomous} \label{sliding} Let $\mathbb{X}$ be a self-dual compact closed category with cups $\cup$ and caps $\cap$. Then $\mathbb{X}$ is a dagger category whose dagger functor $(\_)^\ast: \mathbb{X}^{\mathsf{op}} \to \mathbb{X}$ is defined on objects as $A^\ast = A$ and for a map $f: A \to B$, $f^\ast: B \to A$ \cite[Definition 3.9]{heunen2019categories} is defined as follows: 
\begin{align*}
   \begin{array}[c]{c}
f^{\ast} := \xymatrixcolsep{5pc}\xymatrix{B \ar[r]^-{\cap_A \otimes 1_A} & \oc A \otimes A \otimes B  \ar[r]^-{1_A \otimes f \otimes 1_A} & A \otimes B \otimes B \ar[r]^-{1_A \otimes \cap_A} & A   
  }
   \end{array} 
   \end{align*}  
   \begin{align*}
   \begin{array}[c]{c}
\begin{tikzpicture}
	\begin{pgfonlayer}{nodelayer}
		\node [style=object] (9) at (5, 3.25) {$B$};
		\node [style=object] (10) at (5, 0.75) {$A$};
		\node [style=object] (11) at (6, 2) {$=$};
		\node [style=none] (12) at (8.25, 2.75) {};
		\node [style=none] (13) at (9.5, 1.5) {};
		\node [style=object] (14) at (9.5, 4) {$B$};
		\node [style=none] (15) at (7, 2.75) {};
		\node [style=object] (16) at (7, 0.5) {$A$};
		\node [style=none] (17) at (8.25, 1.5) {};
		\node [style=component] (18) at (5, 2) {$f^\ast$};
		\node [style=component] (19) at (8.25, 2) {$f$};
	\end{pgfonlayer}
	\begin{pgfonlayer}{edgelayer}
		\draw [style=wire, bend right=90, looseness=2.00] (12.center) to (15.center);
		\draw [style=wire] (15.center) to (16);
		\draw [style=wire, bend left=90, looseness=2.00] (13.center) to (17.center);
		\draw [style=wire] (14) to (13.center);
		\draw [style=wire] (9) to (18);
		\draw [style=wire] (18) to (10);
		\draw [style=wire] (12.center) to (19);
		\draw [style=wire] (19) to (17.center);
	\end{pgfonlayer}
\end{tikzpicture}
   \end{array}
\end{align*}  
Explicitly, $1_A^\ast = 1_A$, $(f;g)^\ast = g^\ast;f^\ast$ and $f^{\ast\ast} = f$. Furthermore, for any map $f: A \to B$, the following diagrams commute \cite[Lemma 3.12 \& 3.26]{heunen2019categories}: 
\begin{align*}
   \begin{array}[c]{c}
  \xymatrixcolsep{5pc}\xymatrix{ k \ar[r]^-{\cap_A} \ar[d]_-{\cap_B} & A \otimes A \ar[d]^-{1_A \otimes f}  \\
  B \otimes B \ar[r]_-{f^\ast \otimes 1_B} & A \otimes B 
  } 
   \end{array} && 
   \begin{array}[c]{c}
\begin{tikzpicture}
	\begin{pgfonlayer}{nodelayer}
		\node [style=object] (11) at (-3, 4.5) {$=$};
		\node [style=none] (12) at (-4, 5.25) {};
		\node [style=none] (15) at (-5.25, 5.25) {};
		\node [style=object] (16) at (-5.25, 3.5) {$A$};
		\node [style=component] (19) at (-4, 4.5) {$f$};
		\node [style=object] (20) at (-4, 3.5) {$B$};
		\node [style=none] (21) at (-2, 5.25) {};
		\node [style=none] (22) at (-0.75, 5.25) {};
		\node [style=object] (23) at (-0.75, 3.5) {$B$};
		\node [style=component] (24) at (-2, 4.5) {$f^\ast$};
		\node [style=object] (25) at (-2, 3.5) {$A$};
	\end{pgfonlayer}
	\begin{pgfonlayer}{edgelayer}
		\draw [style=wire, bend right=90, looseness=2.00] (12.center) to (15.center);
		\draw [style=wire] (15.center) to (16);
		\draw [style=wire] (12.center) to (19);
		\draw [style=wire] (19) to (20);
		\draw [style=wire, bend left=90, looseness=2.00] (21.center) to (22.center);
		\draw [style=wire] (22.center) to (23);
		\draw [style=wire] (21.center) to (24);
		\draw [style=wire] (24) to (25);
	\end{pgfonlayer}
\end{tikzpicture}
   \end{array} \\
   \begin{array}[c]{c}
  \xymatrixcolsep{5pc}\xymatrix{ B \otimes A \ar[d]_-{f^\ast \otimes 1_A} \ar[r]^-{1_B \otimes f} & B \otimes B \ar[d]^-{\cup_B}  \\
  A \otimes A \ar[r]_-{\cup_A} & k 
  } 
   \end{array}   &&  \begin{array}[c]{c}
\begin{tikzpicture}
	\begin{pgfonlayer}{nodelayer}
		\node [style=object] (0) at (2.75, 4.25) {$=$};
		\node [style=none] (1) at (1.75, 3.5) {};
		\node [style=none] (2) at (0.5, 3.5) {};
		\node [style=object] (3) at (0.5, 5.25) {$A$};
		\node [style=component] (4) at (1.75, 4.25) {$f$};
		\node [style=object] (5) at (1.75, 5.25) {$B$};
		\node [style=none] (6) at (3.75, 3.5) {};
		\node [style=none] (7) at (5, 3.5) {};
		\node [style=object] (8) at (5, 5.25) {$B$};
		\node [style=component] (9) at (3.75, 4.25) {$f^\ast$};
		\node [style=object] (10) at (3.75, 5.25) {$A$};
	\end{pgfonlayer}
	\begin{pgfonlayer}{edgelayer}
		\draw [style=wire, bend left=90, looseness=2.00] (1.center) to (2.center);
		\draw [style=wire] (2.center) to (3);
		\draw [style=wire] (1.center) to (4);
		\draw [style=wire] (4) to (5);
		\draw [style=wire, bend right=90, looseness=2.00] (6.center) to (7.center);
		\draw [style=wire] (7.center) to (8);
		\draw [style=wire] (6.center) to (9);
		\draw [style=wire] (9) to (10);
	\end{pgfonlayer}
\end{tikzpicture}
   \end{array} 
\end{align*}
The above identities are sometimes referred to as the \textbf{sliding equations}. 
\end{lemma}

We will now justify that monoidal reverse differential categories should be self-dual compact closed. We begin in the Seely case.  By requirement 1. at the beginning of this section, if we start with an MRDC which satisfies the Seely requirements, then it should be a differential storage category $\mathbb{X}$, with coalgebra modality $(\oc, \delta, \varepsilon, \Delta, e)$ (with the Seely isomorphisms), and deriving transformation $\mathsf{d}: \oc A \otimes A \to \oc A$ (or equivalently codereliction $\eta_A: A \to \oc A$). Of course this means that the coKleisli category $\mathbb{X}_\oc$ is a Cartesian differential category. But then by requirement (2), $\mathbb{X}_\oc$ should be a Cartesian reverse differential category, which means its fibration of linear maps $\mathcal{L}[\mathbb{X}_\oc]$ is a dagger fibration. As explained in Example \ref{ex:daggerlin}, each of the fibres $\mathcal{L}[X]$ is then a $\dagger$-category. By Corollary \ref{cor:seely-lin}, we also have the isomorphisms $\mathcal{L}[0] \cong \mathsf{LIN}[\mathbb{X}_\oc] \cong \mathbb{X}$, which implies that the base category $\mathbb{X}$ is itself a $\dagger$-category. To distinguish between the two dagger structures, we will use $\dagger$ for the daggers in $\mathbb{X}_\oc$, and $\ast$ for the daggers in $\mathbb{X}$. If we assume that $\dagger$ is monoidal, then using the comonad counit $\varepsilon$ and the codereliction $\eta$, we can now build cups and caps to make every object of $\mathbb{X}$ a self-dual object. Unfortunately, since the twist equation is non-canonical, it does not appear to come for free from this approach, and so $\mathbb{X}$ is not a self-dual compact closed category as defined in this paper. Of course this should somewhat be expected since we asked for the twist equation for practical reasons. Nevertheless, this still justifies the link between reverse differentiation and self-duality. 

Using Theorem \ref{thm:fibration_equivalence}, pushing the dagger through the fibration equivalence $\mathcal{L}[\mathbb{X}_\oc] \cong \mathcal{L}_\oc[\mathbb{X}]$, we also have that $\mathcal{L}_\oc[\mathbb{X}]$ is a dagger fibration as in Example \ref{ex:!daggerfibration}. Consider then the map $\varepsilon_A: \oc A \to A$ interpreted as a map in the fibre $\mathcal{L}_\oc[\mathbb{X}](k,A)$. Taking its dagger we obtain a map $\varepsilon_A^{\dagger[A]}: \oc A \otimes A \to k$. Precomposing this map with the codereliction we obtain our cup $\cup_A: A \otimes A \to k$: 
\[ \cup_A :=  \xymatrixcolsep{5pc}\xymatrix{  A \otimes A \ar[r]^-{\eta_A \otimes 1_A} & \oc A \otimes A  \ar[r]^-{\varepsilon_A^{\dagger[A]}} & k } \]
To build the cap $\cap_A: k \to A \otimes A$, we use the dagger $\ast$ on the base category $\mathbb{X}$:
\[ \cap_A :=  \xymatrixcolsep{5pc}\xymatrix{  k \ar[r]^-{{\varepsilon_A^{\dagger[A]}}^\ast} & \oc A \otimes A  \ar[r]^-{\eta^\ast_A \otimes 1_A} & A \otimes A } \]

\begin{lemma} Let $\mathbb{X}$ be a differential storage category. Suppose $\mathbb{X}_\oc$ is a Cartesian reverse differential category, and therefore has a contextual linear dagger $\dagger$. If $\dagger$ is (strict) monoidal then every object of $\mathbb{X}$ is a self-dual object where the cups and caps are defined as above. 
\end{lemma}
\begin{proof} We must show that the cups and caps satisfy the snake equations. To do so we will need some simple identities. First observe that for a map $f: \oc X \otimes A \to B$ and a map $g: B \to D$, since the dagger is contravariant, it straightforward to show that: 
\[(f;g)^{\dagger[X]} = (1_X \otimes g^\ast); f^{\dagger[X]}\]
Next, using the assumption that $\dagger$ is monoidal, we have that for any map $f: \oc X \otimes A \to B$ and any object $C$: 
\[ (f \otimes 1_C)^{\dagger[X]} = f^{\dagger[X]} \otimes 1_C\]
The last required identity comes from the fact that in any Cartesian differential category with a contextual linear dagger, the dagger preserves linearity in context. Translating this in terms of the fibration $\mathcal{L}_\oc[\mathbb{X}]$, by Corollary \ref{cor:seely-lin} we have that for any map $f: \oc X \otimes A \to B$: 
\[ \left( (\varepsilon_X; \eta_X) \otimes 1_A \right) ;f^{\dagger[X]} = \left( \left( (\varepsilon_X; \eta_X) \otimes 1_A \right); f \right)^{\dagger[X]}  \]
Now we compute one of the snake equations: 
\begin{align*}
(1_A \otimes \cap_A);(\cup_A \otimes 1_A) &=~ (1_A \otimes {\varepsilon_A^{\dagger[A]}}^\ast);  (1_A \otimes \eta^\ast_A \otimes 1_A) ;(\eta_A \otimes 1_A \otimes 1_A); (\varepsilon^{\dagger[A]}_A \otimes 1_A) \\
&=~ \eta_A; (1_{\oc A} \otimes {\varepsilon_A^{\dagger[A]}}^\ast); (1_A \otimes \eta^\ast_A \otimes 1_A); (\varepsilon^{\dagger[A]}_A \otimes 1_A) \\
&=~ \eta_A; (1_{\oc A} \otimes {\varepsilon_A^{\dagger[A]}}^\ast); (1_A \otimes (\varepsilon_A \otimes 1_A)^\ast) ; (\eta_A \otimes 1_A)^{\dagger[A]} \\
&=~ \eta_A; (1_{\oc A} \otimes {\varepsilon_A^{\dagger[A]}}^\ast); \left( (\varepsilon_A \otimes 1_A);(\eta_A \otimes 1_A)  \right)^{\dagger[A]} \\
&=~ \eta_A; \left(  \left( (\varepsilon_A; \eta_A) \otimes 1_A \right) ; \varepsilon_A^{\dagger[A]} \right)^{\dagger[A]} \\ 
&=~ \eta_A; \left( \left( \varepsilon_A; \eta_A; \varepsilon_A \right)^{\dagger[A]} \right)^{\dagger[A]} \\
&=~ \eta_A ; {\varepsilon_A^{\dagger[A]}}^{\dagger[A]} \\
&=~ \eta_A; \varepsilon_A \\
&=~ 1_A 
\end{align*}
The proof for the other snake equation is similar. So we conclude that $(A, \cup_A, \cap_A)$ is a self-dual object, and that $\mathbb{X}$ is a self-dual compact closed category. 
\end{proof}

Even if the coalgebra modality does not have the Seely isomorphisms, it is still possible to show that in each of the fibres $\mathcal{L}_\oc[X]$ every object is self-dual. However, the computations in the proof are more complicated and not necessarily more enlightening. So we will omit the proof and simply provide the construction. In the fibre $\mathcal{L}_\oc[X]$, the cup is the map $\cup^X_A: \oc X \otimes A \otimes A \to k$ defined as follows: 
\[ \cup^X_A :=  \xymatrixcolsep{5pc}\xymatrix{ \oc X \otimes A \otimes A \ar[r]^-{\oc(0) \otimes 1_A \otimes 1_A} & \oc A \otimes A \otimes A \ar[r]^-{\mathsf{d}_A \otimes 1_A} & \oc A \otimes A  \ar[r]^-{\varepsilon_A^{\dagger[A]}} & k } \]
while the cap $\cap^X_A: \oc X \to A \otimes A$ is defined as the dagger in the fibre of the cap: $\cap^X_A := {\cup^X_A}^{\dagger[X]}$.

\subsection{Definition and examples}

Given the discussion of the previous section, monoidal reverse differential categories should at least be self-dual compact closed. We now give the full definition.  

\begin{definition} \label{MRDC} A \textbf{monoidal reverse differential category} (MRDC) is an additive symmetric monoidal category $\mathbb{X}$, such that $\mathbb{X}$ is a self-dual compact closed category, equipped with a coalgebra modality $(\oc, \delta, \varepsilon, \Delta, e)$ and a \textbf{reverse deriving transformation} which is a family of maps $\mathsf{r}_A: \oc A \otimes \oc A \to A$ drawn in the graphical calculus as: 
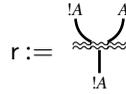
\[\mathsf{r}:= \begin{array}[c]{c} 
\begin{tikzpicture}
	\begin{pgfonlayer}{nodelayer}
		\node [style=object] (0) at (1.75, 2.75) {$A$};
		\node [style=object] (1) at (1.25, 1.25) {$\oc A$};
		\node [style=integral] (2) at (1.25, 2) {{\bf \aquarius\!\aquarius\!\aquarius}};
		\node [style=object] (3) at (0.75, 2.75) {$\oc A$};
	\end{pgfonlayer}
	\begin{pgfonlayer}{edgelayer}
		\draw [style=wire, bend right] (2) to (0);
		\draw [style=wire] (1) to (2);
		\draw [style=wire, bend left] (2) to (3);
	\end{pgfonlayer}
\end{tikzpicture}
   \end{array}\] 
such that the following axioms hold: 
\begin{description}
 \item[{\bf [r.N]}]  Reverse Naturality Rule: 
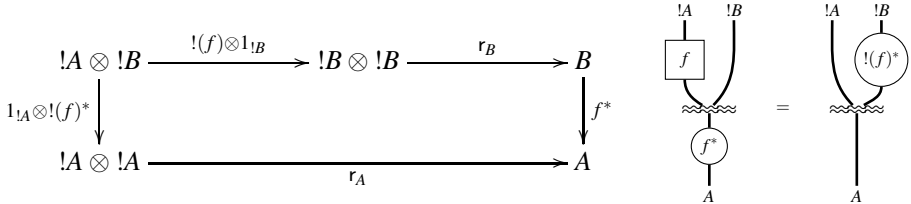
\begin{align*}
   \begin{array}[c]{c}
 \xymatrixcolsep{5pc}\xymatrix{\oc A \otimes \oc B \ar[d]_-{1_{\oc A} \otimes \oc(f)^\ast} \ar[r]^-{\oc(f) \otimes 1_{\oc B}} & \oc B \otimes \oc B \ar[r]^-{\mathsf{r}_B} & B  \ar[d]^-{f^\ast} \\
 \oc A \otimes \oc A \ar[rr]_-{\mathsf{r}_A}  && A
  }
   \end{array} &&    \begin{array}[c]{c}
\begin{tikzpicture}
	\begin{pgfonlayer}{nodelayer}
		\node [style=integral] (28) at (0, 8) {{\bf \aquarius\!\aquarius\!\aquarius}};
		\node [style=function] (29) at (-0.5, 9) {$f$};
		\node [style=component] (30) at (0, 7.25) {$f^\ast$};
		\node [style=object] (31) at (0, 6.25) {$A$};
		\node [style=object] (32) at (-0.5, 10) {$\oc A$};
		\node [style=object] (33) at (0.5, 10) {$\oc B$};
		\node [style=object] (34) at (1.5, 8) {$=$};
		\node [style=integral] (35) at (3, 8) {{\bf \aquarius\!\aquarius\!\aquarius}};
		\node [style=component] (36) at (3.5, 9) {$\oc(f)^\ast$};
		\node [style=object] (38) at (3, 6.25) {$A$};
		\node [style=object] (39) at (3.5, 10) {$\oc B$};
		\node [style=object] (40) at (2.5, 10) {$\oc A$};
	\end{pgfonlayer}
	\begin{pgfonlayer}{edgelayer}
		\draw [style=wire] (32) to (29);
		\draw [style=wire, in=150, out=-90] (29) to (28);
		\draw [style=wire, in=-90, out=45] (28) to (33);
		\draw [style=wire] (28) to (30);
		\draw [style=wire] (30) to (31);
		\draw [style=wire] (39) to (36);
		\draw [style=wire, in=30, out=-90] (36) to (35);
		\draw [style=wire, in=-90, out=135] (35) to (40);
		\draw [style=wire] (35) to (38);
	\end{pgfonlayer}
\end{tikzpicture}
   \end{array}
\end{align*}
\item[{\bf [r.1]}] Reverse Constant Rule:
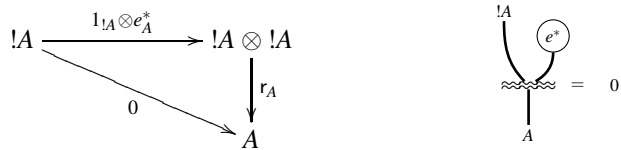
\begin{align*}
   \begin{array}[c]{c}
 \xymatrixcolsep{5pc}\xymatrix{\oc A  \ar[dr]_-{0} \ar[r]^-{1_{\oc A} \otimes e^\ast_A} & \oc A \otimes \oc A \ar[d]^-{\mathsf{r}_A} \\
  & A }
   \end{array}&&   \begin{array}[c]{c}
\begin{tikzpicture}
	\begin{pgfonlayer}{nodelayer}
		\node [style=object] (41) at (5.25, 0) {$A$};
		\node [style=component] (42) at (5.75, 2) {$e^\ast$};
		\node [style=differential] (43) at (5.25, 1) {{\bf \aquarius\!\aquarius\!\aquarius}};
		\node [style=object] (44) at (4.75, 2.5) {$\oc A$};
		\node [style=port] (45) at (6.25, 1) {$=$};
		\node [style=port] (46) at (7, 1) {$0$};
	\end{pgfonlayer}
	\begin{pgfonlayer}{edgelayer}
		\draw [style=wire, in=-90, out=30] (43) to (42);
		\draw [style=wire, in=-90, out=135] (43) to (44);
		\draw [style=wire] (43) to (41);
	\end{pgfonlayer}
\end{tikzpicture}
   \end{array}
\end{align*}
\item[{\bf [r.2]}] Reverse Leibniz Rule (or Reverse Product Rule):
  \begin{align*}
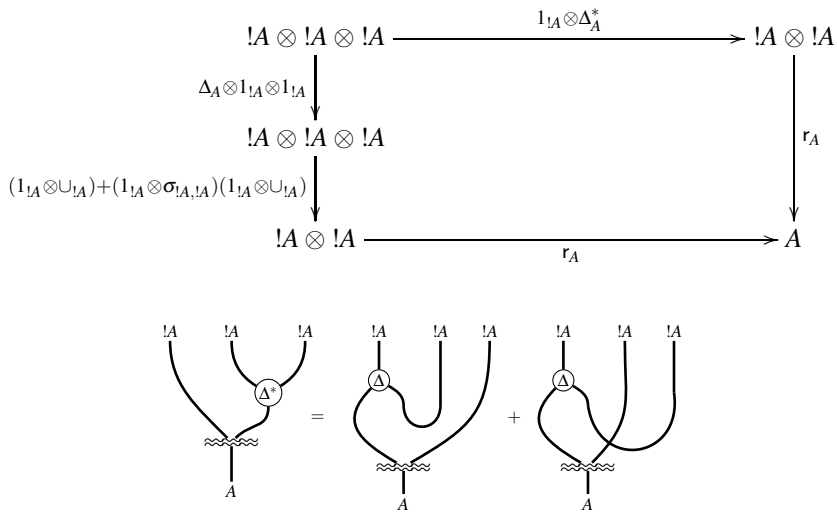
     \begin{array}[c]{c}
\xymatrixcolsep{11pc}\xymatrix{\oc A \otimes \oc A \otimes \oc A \ar[d]_-{\Delta_A \otimes 1_{\oc A} \otimes 1_{\oc A}} \ar[r]^-{1_{\oc A} \otimes \Delta^\ast_A} &  \oc A \otimes \oc A \ar[dd]^-{\mathsf{r}_A} \\
 \oc A \otimes \oc A \otimes \oc A \ar[d]_-{(1_{\oc A} \otimes \cup_{\oc A}) + (1_{\oc A} \otimes \sigma_{\oc A, \oc A})(1_{\oc A} \otimes \cup_{\oc A})} \\ 
    \oc A \otimes \oc A \ar[r]_-{\mathsf{r}_A} & A }
   \end{array}
\end{align*}
\begin{align*}
  \begin{array}[c]{c}
\begin{tikzpicture}
	\begin{pgfonlayer}{nodelayer}
		\node [style=object] (41) at (-3.5, 2.75) {$A$};
		\node [style=differential] (43) at (-3.5, 3.75) {{\bf \aquarius\!\aquarius\!\aquarius}};
		\node [style=object] (44) at (-4.75, 6) {$\oc A$};
		\node [style=object] (67) at (-3.5, 6) {$\oc A$};
		\node [style=duplicate] (68) at (-2.75, 4.75) {$\Delta^\ast$};
		\node [style=object] (69) at (-2, 6) {$\oc A$};
		\node [style=object] (70) at (-1.75, 4.25) {$=$};
		\node [style=object] (71) at (-0.5, 6) {$\oc A$};
		\node [style=differential] (72) at (0, 3.25) {{\bf \aquarius\!\aquarius\!\aquarius}};
		\node [style=object] (73) at (1.75, 6) {$\oc A$};
		\node [style=duplicate] (74) at (-0.5, 5) {$\Delta$};
		\node [style=object] (75) at (0, 2.5) {$A$};
		\node [style=object] (76) at (0.75, 6) {$\oc A$};
		\node [style=port] (83) at (2.25, 4.25) {$+$};
		\node [style=none] (84) at (0.75, 4.5) {};
		\node [style=none] (85) at (0, 4.5) {};
		\node [style=object] (86) at (3.25, 6) {$\oc A$};
		\node [style=differential] (87) at (3.75, 3.25) {{\bf \aquarius\!\aquarius\!\aquarius}};
		\node [style=object] (88) at (4.5, 6) {$\oc A$};
		\node [style=duplicate] (89) at (3.25, 5) {$\Delta$};
		\node [style=object] (90) at (3.75, 2.5) {$A$};
		\node [style=object] (91) at (5.5, 6) {$\oc A$};
		\node [style=none] (92) at (5.5, 4.75) {};
		\node [style=none] (93) at (3.75, 4.5) {};
	\end{pgfonlayer}
	\begin{pgfonlayer}{edgelayer}
		\draw [style=wire, in=-90, out=135] (43) to (44);
		\draw [style=wire] (43) to (41);
		\draw [style=wire, bend left] (68) to (67);
		\draw [style=wire, bend right] (68) to (69);
		\draw [style=wire, in=45, out=-90, looseness=1.25] (68) to (43);
		\draw [style=wire, in=-90, out=30, looseness=1.25] (72) to (73);
		\draw [style=wire, in=150, out=-135, looseness=1.50] (74) to (72);
		\draw [style=wire] (71) to (74);
		\draw [style=wire] (72) to (75);
		\draw [style=wire, bend left=90, looseness=2.00] (84.center) to (85.center);
		\draw [style=wire] (76) to (84.center);
		\draw [style=wire, in=90, out=-30] (74) to (85.center);
		\draw [style=wire, in=-90, out=45, looseness=1.25] (87) to (88);
		\draw [style=wire, in=150, out=-135, looseness=1.50] (89) to (87);
		\draw [style=wire] (86) to (89);
		\draw [style=wire] (87) to (90);
		\draw [style=wire, bend left=90, looseness=2.00] (92.center) to (93.center);
		\draw [style=wire] (91) to (92.center);
		\draw [style=wire, in=90, out=-30] (89) to (93.center);
	\end{pgfonlayer}
\end{tikzpicture}
   \end{array}    
\end{align*}
\item[{\bf [r.3]}] Reverse Linear Rule:
  \begin{align*}
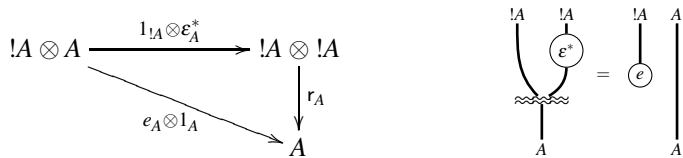
     \begin{array}[c]{c}
 \xymatrixcolsep{5pc}\xymatrix{\oc A \otimes A \ar[r]^-{1_{\oc A} \otimes \varepsilon^\ast_A} \ar[dr]_-{e_A \otimes 1_A} & \oc A \otimes \oc A \ar[d]^-{\mathsf{r}_A} \\
  & A    }
   \end{array} && 
   \begin{array}[c]{c}
\begin{tikzpicture}
	\begin{pgfonlayer}{nodelayer}
		\node [style=object] (94) at (-11.25, 11.5) {$\oc A$};
		\node [style=object] (95) at (-10.5, 11.5) {$A$};
		\node [style=object] (96) at (-10.5, 8.75) {$A$};
		\node [style=component] (97) at (-11.25, 10.25) {$e$};
		\node [style=port] (103) at (-12, 10.25) {$=$};
		\node [style=object] (104) at (-13.25, 8.75) {$A$};
		\node [style=component] (105) at (-12.75, 10.75) {$\varepsilon^\ast$};
		\node [style=differential] (106) at (-13.25, 9.75) {{\bf \aquarius\!\aquarius\!\aquarius}};
		\node [style=object] (107) at (-13.75, 11.5) {$\oc A$};
		\node [style=object] (110) at (-12.75, 11.5) {$\oc A$};
	\end{pgfonlayer}
	\begin{pgfonlayer}{edgelayer}
		\draw [style=wire] (94) to (97);
		\draw [style=wire] (95) to (96);
		\draw [style=wire, in=-90, out=30] (106) to (105);
		\draw [style=wire, in=-90, out=135] (106) to (107);
		\draw [style=wire] (106) to (104);
		\draw [style=wire] (110) to (105);
	\end{pgfonlayer}
\end{tikzpicture}
   \end{array}
\end{align*}
\item[{\bf [r.4]}] Reverse Chain Rule: 
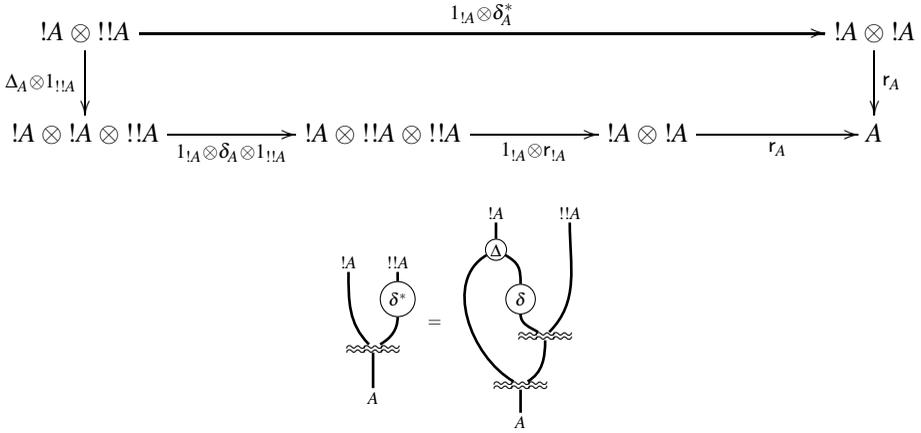
\begin{align*}
   \begin{array}[c]{c}
 \xymatrixcolsep{4pc}\xymatrix{\oc A \otimes \oc \oc A \ar[d]_-{\Delta_A \otimes 1_{\oc \oc A}} \ar[rrr]^-{1_{\oc A} \otimes \delta^\ast_A} &&& \oc A \otimes \oc A \ar[d]^-{\mathsf{r}_A} \\
    \oc A \otimes \oc A \otimes \oc \oc A \ar[r]_-{1_{\oc A} \otimes \delta_A \otimes 1_{\oc \oc A}} & \oc A \otimes \oc \oc A \otimes \oc \oc A \ar[r]_-{1_{\oc A} \otimes \mathsf{r}_{\oc A}} & \oc A \otimes \oc A \ar[r]_-{\mathsf{r}_A}  & A
  } 
   \end{array} 
\end{align*}
\begin{align*}
   \begin{array}[c]{c}
\begin{tikzpicture}
	\begin{pgfonlayer}{nodelayer}
		\node [style=port] (103) at (-12, 10.25) {$=$};
		\node [style=object] (104) at (-13.25, 8.75) {$A$};
		\node [style=component] (105) at (-12.75, 10.75) {$\delta^\ast$};
		\node [style=differential] (106) at (-13.25, 9.75) {{\bf \aquarius\!\aquarius\!\aquarius}};
		\node [style=object] (107) at (-13.75, 11.5) {$\oc A$};
		\node [style=object] (110) at (-12.75, 11.5) {$\oc \oc A$};
		\node [style=duplicate] (117) at (-10.75, 11.75) {$\Delta$};
		\node [style=object] (118) at (-10.75, 12.5) {$\oc A$};
		\node [style=differential] (121) at (-10.25, 9) {{\bf \aquarius\!\aquarius\!\aquarius}};
		\node [style=object] (122) at (-10.25, 8.25) {$A$};
		\node [style=component] (124) at (-10.25, 10.75) {$\delta$};
		\node [style=differential] (125) at (-9.75, 10) {{\bf \aquarius\!\aquarius\!\aquarius}};
		\node [style=object] (126) at (-9.25, 12.5) {$\oc \oc A$};
	\end{pgfonlayer}
	\begin{pgfonlayer}{edgelayer}
		\draw [style=wire, in=-90, out=30] (106) to (105);
		\draw [style=wire, in=-90, out=135] (106) to (107);
		\draw [style=wire] (106) to (104);
		\draw [style=wire] (110) to (105);
		\draw [style=wire] (118) to (117);
		\draw [style=wire] (121) to (122);
		\draw [style=wire, in=-90, out=150] (125) to (124);
		\draw [style=wire, in=-90, out=45] (125) to (126);
		\draw [style=wire, in=90, out=-30, looseness=1.25] (117) to (124);
		\draw [style=wire, in=150, out=-150] (117) to (121);
		\draw [style=wire, in=30, out=-90] (125) to (121);
	\end{pgfonlayer}
\end{tikzpicture}
   \end{array}    
\end{align*}
     \item[{\bf [r.5]}] Reverse Interchange Rule:
     \begin{align*}
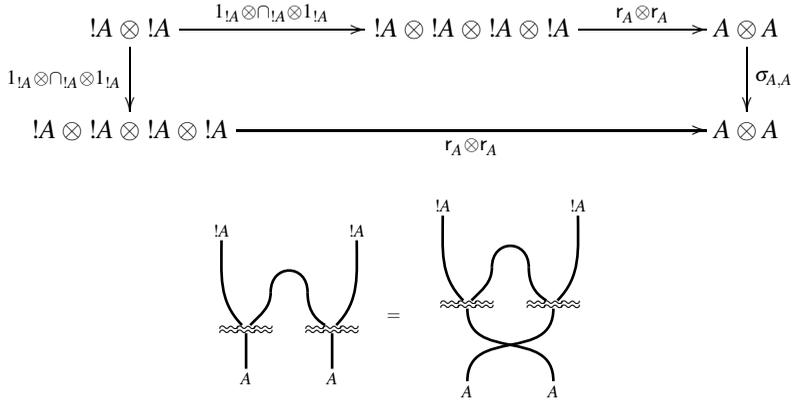
        \begin{array}[c]{c}
 \xymatrixcolsep{4pc}\xymatrix{\oc A \otimes \oc A  \ar[d]_-{1_{\oc A} \otimes \cap_{\oc A}  \otimes 1_{\oc A} }\ar[r]^-{1_{\oc A} \otimes \cap_{\oc A}  \otimes 1_{\oc A} } & \oc A \otimes \oc A \otimes \oc A \otimes \oc A \ar[r]^-{\mathsf{r}_A \otimes \mathsf{r}_A} & A \otimes A \ar[d]^-{\sigma_{A,A}} \\
         \oc A \otimes \oc A \otimes \oc A \otimes \oc A \ar[rr]_-{\mathsf{r}_A \otimes \mathsf{r}_A} && A \otimes A }
   \end{array} 
\end{align*}
\begin{align*}
 \begin{array}[c]{c}
\begin{tikzpicture}
	\begin{pgfonlayer}{nodelayer}
		\node [style=port] (103) at (3.25, 3.25) {$=$};
		\node [style=object] (104) at (0.25, 2) {$A$};
		\node [style=differential] (106) at (0.25, 3) {{\bf \aquarius\!\aquarius\!\aquarius}};
		\node [style=object] (107) at (-0.25, 5) {$\oc A$};
		\node [style=object] (127) at (2, 2) {$A$};
		\node [style=differential] (128) at (2, 3) {{\bf \aquarius\!\aquarius\!\aquarius}};
		\node [style=object] (129) at (2.5, 5) {$\oc A$};
		\node [style=none] (130) at (1.5, 3.75) {};
		\node [style=none] (131) at (0.75, 3.75) {};
		\node [style=object] (132) at (4.75, 1.75) {$A$};
		\node [style=differential] (133) at (4.75, 3.5) {{\bf \aquarius\!\aquarius\!\aquarius}};
		\node [style=object] (134) at (4.25, 5.5) {$\oc A$};
		\node [style=object] (135) at (6.5, 1.75) {$A$};
		\node [style=differential] (136) at (6.5, 3.5) {{\bf \aquarius\!\aquarius\!\aquarius}};
		\node [style=object] (137) at (7, 5.5) {$\oc A$};
		\node [style=none] (138) at (6, 4.25) {};
		\node [style=none] (139) at (5.25, 4.25) {};
	\end{pgfonlayer}
	\begin{pgfonlayer}{edgelayer}
		\draw [style=wire, in=-90, out=135] (106) to (107);
		\draw [style=wire] (106) to (104);
		\draw [style=wire, in=-90, out=45] (128) to (129);
		\draw [style=wire] (128) to (127);
		\draw [style=wire, bend right=90, looseness=2.00] (130.center) to (131.center);
		\draw [style=wire, in=45, out=-90] (131.center) to (106);
		\draw [style=wire, in=150, out=-90] (130.center) to (128);
		\draw [style=wire, in=-90, out=135] (133) to (134);
		\draw [style=wire, in=-90, out=45] (136) to (137);
		\draw [style=wire, bend right=90, looseness=2.00] (138.center) to (139.center);
		\draw [style=wire, in=45, out=-90] (139.center) to (133);
		\draw [style=wire, in=150, out=-90] (138.center) to (136);
		\draw [style=wire, in=90, out=-90, looseness=1.25] (136) to (132);
		\draw [style=wire, in=90, out=-90, looseness=1.25] (133) to (135);
	\end{pgfonlayer}
\end{tikzpicture}
   \end{array}    
\end{align*}
\end{description}
\end{definition}

Before we get to examples, it will be useful to describe the relationship between monoidal differential categories and monoidal reverse differential categories.  In particular, it shows that our definition satisfies requirement 1 described in the introduction to this section.

\begin{therm}\label{thm:rdc_to_dc} A reverse differential category is precisely a differential category which is also self-dual compact closed. Explicitly:
\begin{enumerate}[{\em (i)}]
    \item If $\mathbb{X}$ is a reverse differential category, then $\mathbb{X}$ is a differential category where the deriving transformation $\mathsf{d}_A: \oc A \otimes A \to \oc A$ is defined as: 
  \begin{align*}
   \begin{array}[c]{c}
\mathsf{d}_A := \xymatrixcolsep{3.5pc}\xymatrix{\oc A \otimes A \ar[r]^-{1_{\oc A} \otimes \cap_{\oc A} \otimes 1_A} & \oc A \otimes \oc A \otimes \oc A \otimes A  \ar[r]^-{\mathsf{r}_A \otimes \sigma_{\oc A, A}} & A \otimes A \otimes \oc A  \ar[r]^-{\cup_A \otimes 1_{\oc A}} & \oc A   
  }
   \end{array} 
\end{align*}
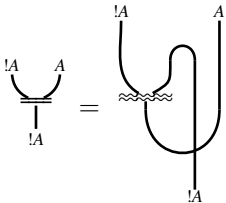
\begin{align*}
\begin{array}[c]{c} 
\begin{tikzpicture}
	\begin{pgfonlayer}{nodelayer}
		\node [style=object] (0) at (1.75, 2.75) {$A$};
		\node [style=object] (1) at (1.25, 1.25) {$\oc A$};
		\node [style=integral] (2) at (1.25, 2) {{\bf =\!=\!=\!=}};
		\node [style=object] (3) at (0.75, 2.75) {$\oc A$};
	\end{pgfonlayer}
	\begin{pgfonlayer}{edgelayer}
		\draw [style=wire, bend right] (2) to (0);
		\draw [style=wire] (1) to (2);
		\draw [style=wire, bend left] (2) to (3);
	\end{pgfonlayer}
\end{tikzpicture}
   \end{array} =    \begin{array}[c]{c}
\begin{tikzpicture}
	\begin{pgfonlayer}{nodelayer}
		\node [style=none] (0) at (1.5, 3.25) {};
		\node [style=none] (1) at (1, 2.25) {};
		\node [style=integral] (2) at (1, 2.5) {{\bf \aquarius\!\aquarius\!\aquarius}};
		\node [style=object] (3) at (0.5, 4.25) {$\oc A$};
		\node [style=none] (4) at (2, 3.25) {};
		\node [style=port] (5) at (2, 0.5) {$\oc A$};
		\node [style=object] (6) at (2.5, 4.25) {$A$};
		\node [style=none] (7) at (2.5, 2.25) {};
	\end{pgfonlayer}
	\begin{pgfonlayer}{edgelayer}
		\draw [style=wire, in=-90, out=30, looseness=1.50] (2) to (0.center);
		\draw [style=wire] (1.center) to (2);
		\draw [style=wire, in=-90, out=150] (2) to (3);
		\draw [style=wire, bend left=90, looseness=2.00] (0.center) to (4.center);
		\draw [style=wire] (4.center) to (5);
		\draw [style=wire, bend right=90, looseness=2.00] (1.center) to (7.center);
		\draw [style=wire] (6) to (7.center);
	\end{pgfonlayer}
\end{tikzpicture}
   \end{array}    
\end{align*}
   %Or in other words, $\mathsf{d}_A : = \mathsf{r}_A^{\dagger[A]}$. 
\item If $\mathbb{X}$ is a differential category which is also a self-dual compact closed category, then $\mathbb{X}$ is a reverse differential category where the reverse deriving transformation is defined as: 
    \begin{align*}
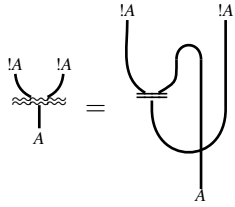
   \begin{array}[c]{c}
 \mathsf{r}_A := \xymatrixcolsep{3.5pc}\xymatrix{\oc A \otimes \oc A \ar[r]^-{1_{\oc A} \otimes \cap_{A} \otimes 1_{\oc A}} & \oc A \otimes  A \otimes A \otimes \oc A  \ar[r]^-{\mathsf{d}_A \otimes \sigma_{A, \oc A}} & \oc A \otimes \oc A \otimes A  \ar[r]^-{\cup_{\oc A} \otimes 1_A} & \oc A    }
   \end{array} 
\end{align*}
\begin{align*}
 \begin{array}[c]{c} 
\begin{tikzpicture}
	\begin{pgfonlayer}{nodelayer}
		\node [style=object] (0) at (1.75, 2.75) {$\oc A$};
		\node [style=object] (1) at (1.25, 1.25) {$A$};
		\node [style=integral] (2) at (1.25, 2) {{\bf \aquarius\!\aquarius\!\aquarius}};
		\node [style=object] (3) at (0.75, 2.75) {$\oc A$};
	\end{pgfonlayer}
	\begin{pgfonlayer}{edgelayer}
		\draw [style=wire, bend right] (2) to (0);
		\draw [style=wire] (1) to (2);
		\draw [style=wire, bend left] (2) to (3);
	\end{pgfonlayer}
\end{tikzpicture}
   \end{array} =    \begin{array}[c]{c}
\begin{tikzpicture}
	\begin{pgfonlayer}{nodelayer}
		\node [style=none] (0) at (1.5, 3.25) {};
		\node [style=none] (1) at (1, 2.25) {};
		\node [style=integral] (2) at (1, 2.5) {{\bf =\!=\!=\!=}};
		\node [style=object] (3) at (0.5, 4.25) {$\oc A$};
		\node [style=none] (4) at (2, 3.25) {};
		\node [style=port] (5) at (2, 0.5) {$A$};
		\node [style=object] (6) at (2.5, 4.25) {$\oc A$};
		\node [style=none] (7) at (2.5, 2.25) {};
	\end{pgfonlayer}
	\begin{pgfonlayer}{edgelayer}
		\draw [style=wire, in=-90, out=30, looseness=1.50] (2) to (0.center);
		\draw [style=wire] (1.center) to (2);
		\draw [style=wire, in=-90, out=150] (2) to (3);
		\draw [style=wire, bend left=90, looseness=2.00] (0.center) to (4.center);
		\draw [style=wire] (4.center) to (5);
		\draw [style=wire, bend right=90, looseness=2.00] (1.center) to (7.center);
		\draw [style=wire] (6) to (7.center);
	\end{pgfonlayer}
\end{tikzpicture}
   \end{array}    
\end{align*}
   %Or in other words, $\mathsf{r}_A : = \mathsf{d}_A^{\dagger[A]}$. 
\end{enumerate}
Furthermore, these constructions are inverses of each other. 
\end{therm}

\begin{remark}\label{twistremark} \normalfont As noted above, the story of these monoidal reverse differential categories could, in theory, be told without assuming the twist equation. If we drop that axiom, then the above constructions include an extra twist (depending on the convention of if $A^\ast$ if on the left/right for the cup/cap): 
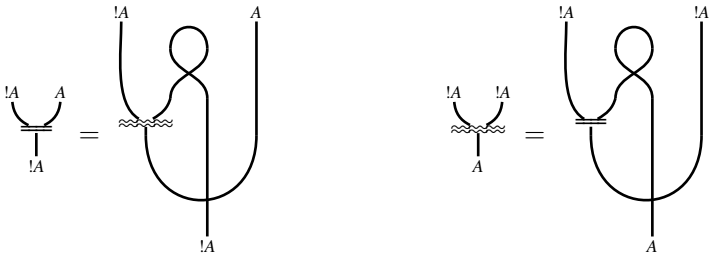
\begin{align*}
   \begin{array}[c]{c} 
\begin{tikzpicture}
	\begin{pgfonlayer}{nodelayer}
		\node [style=object] (0) at (1.75, 2.75) {$A$};
		\node [style=object] (1) at (1.25, 1.25) {$\oc A$};
		\node [style=integral] (2) at (1.25, 2) {{\bf =\!=\!=\!=}};
		\node [style=object] (3) at (0.75, 2.75) {$\oc A$};
	\end{pgfonlayer}
	\begin{pgfonlayer}{edgelayer}
		\draw [style=wire, bend right] (2) to (0);
		\draw [style=wire] (1) to (2);
		\draw [style=wire, bend left] (2) to (3);
	\end{pgfonlayer}
\end{tikzpicture}
   \end{array} =    \begin{array}[c]{c}
\begin{tikzpicture}
	\begin{pgfonlayer}{nodelayer}
		\node [style=none] (1) at (5.75, 0.75) {};
		\node [style=integral] (2) at (5.75, 1) {{\bf \aquarius\!\aquarius\!\aquarius}};
		\node [style=object] (3) at (5.25, 3.25) {$\oc A$};
		\node [style=object] (5) at (7, -1.5) {$\oc A$};
		\node [style=object] (6) at (8, 3.25) {$A$};
		\node [style=none] (7) at (8, 0.75) {};
		\node [style=none] (8) at (6.25, 2.5) {};
		\node [style=none] (9) at (7, 2.5) {};
		\node [style=none] (10) at (6.25, 1.5) {};
		\node [style=none] (11) at (7, 1.5) {};
	\end{pgfonlayer}
	\begin{pgfonlayer}{edgelayer}
		\draw [style=wire] (1.center) to (2);
		\draw [style=wire, in=-90, out=150, looseness=0.75] (2) to (3);
		\draw [style=wire, bend right=90, looseness=2.00] (1.center) to (7.center);
		\draw [style=wire] (6) to (7.center);
		\draw [style=wire, bend left=90, looseness=2.00] (8.center) to (9.center);
		\draw [style=wire, in=90, out=-90] (8.center) to (11.center);
		\draw [style=wire, in=-90, out=90] (10.center) to (9.center);
		\draw [style=wire, in=-105, out=30] (2) to (10.center);
		\draw [style=wire] (11.center) to (5);
	\end{pgfonlayer}
\end{tikzpicture}
   \end{array} &&  \begin{array}[c]{c} 
\begin{tikzpicture}
	\begin{pgfonlayer}{nodelayer}
		\node [style=object] (0) at (1.75, 2.75) {$\oc A$};
		\node [style=object] (1) at (1.25, 1.25) {$A$};
		\node [style=integral] (2) at (1.25, 2) {{\bf \aquarius\!\aquarius\!\aquarius}};
		\node [style=object] (3) at (0.75, 2.75) {$\oc A$};
	\end{pgfonlayer}
	\begin{pgfonlayer}{edgelayer}
		\draw [style=wire, bend right] (2) to (0);
		\draw [style=wire] (1) to (2);
		\draw [style=wire, bend left] (2) to (3);
	\end{pgfonlayer}
\end{tikzpicture}
   \end{array} =    \begin{array}[c]{c}
\begin{tikzpicture}
	\begin{pgfonlayer}{nodelayer}
		\node [style=none] (1) at (5.75, 0.75) {};
		\node [style=integral] (2) at (5.75, 1) {{\bf =\!=\!=\!=}};
		\node [style=object] (3) at (5.25, 3.25) {$\oc A$};
		\node [style=object] (5) at (7, -1.5) {$A$};
		\node [style=object] (6) at (8, 3.25) {$\oc A$};
		\node [style=none] (7) at (8, 0.75) {};
		\node [style=none] (8) at (6.25, 2.5) {};
		\node [style=none] (9) at (7, 2.5) {};
		\node [style=none] (10) at (6.25, 1.5) {};
		\node [style=none] (11) at (7, 1.5) {};
	\end{pgfonlayer}
	\begin{pgfonlayer}{edgelayer}
		\draw [style=wire] (1.center) to (2);
		\draw [style=wire, in=-90, out=150, looseness=0.75] (2) to (3);
		\draw [style=wire, bend right=90, looseness=2.00] (1.center) to (7.center);
		\draw [style=wire] (6) to (7.center);
		\draw [style=wire, bend left=90, looseness=2.00] (8.center) to (9.center);
		\draw [style=wire, in=90, out=-90] (8.center) to (11.center);
		\draw [style=wire, in=-90, out=90] (10.center) to (9.center);
		\draw [style=wire, in=-105, out=30] (2) to (10.center);
		\draw [style=wire] (11.center) to (5);
	\end{pgfonlayer}
\end{tikzpicture}
   \end{array}
\end{align*}
However, as explained above, we have elected to assume the twist equation to simplify our string diagrams. 
\end{remark}

\begin{proof} The axioms of a reverse deriving transformation $\mathsf{r}$ correspond precisely to the axioms of a deriving transformation $\mathsf{d}$. Naturality of $\mathsf{d}$ corresponds to the reverse naturality rule \textbf{[r.N]}, while \textbf{[d.n]} corresponds to \textbf{[r.n]} for $n = 1, 2, \hdots, 5$. The correspondence follows from using the snake equations (Definition \ref{SDCC}) and the sliding equations (Lemma \ref{sliding}). Since the computations are all similar, we will not work out the full proof in detail and will instead provide two examples, prove that the constructions are inverses of each other, and then leave the rest as an exercise for the reader. 

Starting with a reverse deriving transformation $\mathsf{r}$, we will show that the constructed $\mathsf{d}$ satisfies the Leibniz rule \textbf{[d.2]} by using the snake equations and sliding equations on the reverse Leibniz rule \textbf{[r.2]}: 
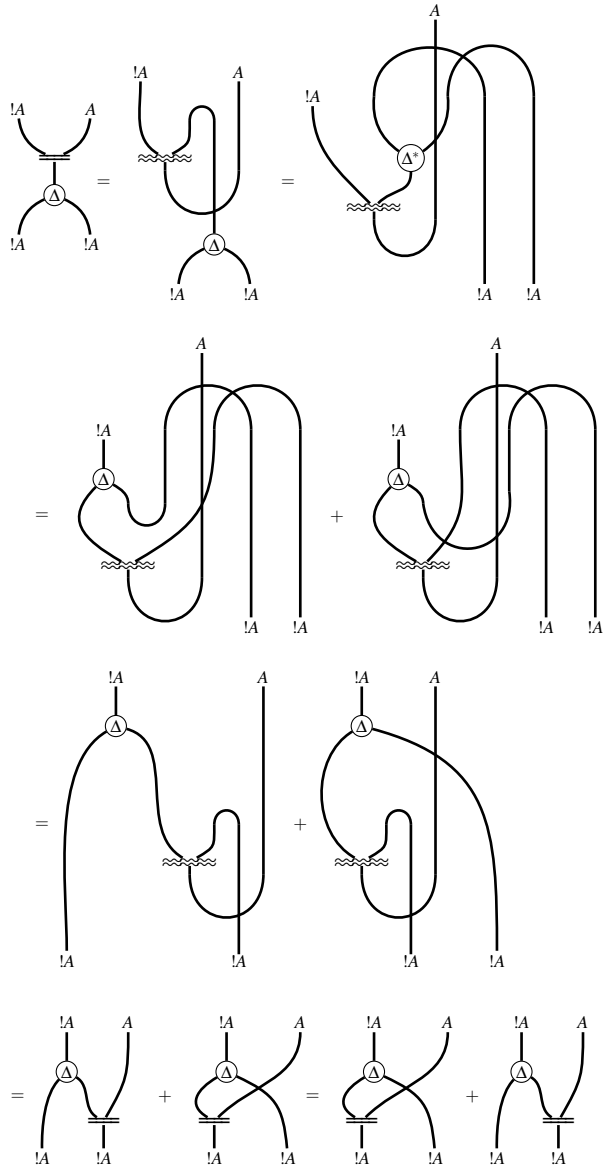
\begin{align*}
   \begin{array}[c]{c}
\begin{tikzpicture}
	\begin{pgfonlayer}{nodelayer}
		\node [style=differential] (0) at (13.75, 11) {{\bf =\!=\!=\!=}};
		\node [style=object] (1) at (14.5, 12) {$A$};
		\node [style=object] (2) at (13, 12) {$\oc A$};
		\node [style=object] (3) at (13, 9.25) {$\oc A$};
		\node [style=duplicate] (4) at (13.75, 10.25) {$\Delta$};
		\node [style=object] (5) at (14.5, 9.25) {$\oc A$};
		\node [style=object] (6) at (20.25, -6.5) {$\oc A$};
		\node [style=differential] (7) at (20, -8.5) {{\bf =\!=\!=\!=}};
		\node [style=object] (8) at (21.75, -6.5) {$A$};
		\node [style=duplicate] (9) at (20.25, -7.5) {$\Delta$};
		\node [style=object] (10) at (20, -9.25) {$\oc A$};
		\node [style=object] (11) at (21.5, -9.25) {$\oc A$};
		\node [style=object] (12) at (24.5, -6.5) {$A$};
		\node [style=differential] (13) at (24, -8.5) {{\bf =\!=\!=\!=}};
		\node [style=object] (14) at (24, -9.25) {$\oc A$};
		\node [style=object] (15) at (22.75, -9.25) {$\oc A$};
		\node [style=object] (16) at (23.25, -6.5) {$\oc A$};
		\node [style=duplicate] (17) at (23.25, -7.5) {$\Delta$};
		\node [style=port] (18) at (14.75, 10.5) {$=$};
		\node [style=port] (19) at (22.25, -8) {$+$};
		\node [style=none] (20) at (16.5, 11.75) {};
		\node [style=none] (21) at (16, 10.75) {};
		\node [style=integral] (22) at (16, 11) {{\bf \aquarius\!\aquarius\!\aquarius}};
		\node [style=object] (23) at (15.5, 12.75) {$\oc A$};
		\node [style=none] (24) at (17, 11.75) {};
		\node [style=object] (25) at (17.5, 12.75) {$A$};
		\node [style=none] (26) at (17.5, 10.75) {};
		\node [style=object] (27) at (16.25, 8.25) {$\oc A$};
		\node [style=duplicate] (28) at (17, 9.25) {$\Delta$};
		\node [style=object] (29) at (17.75, 8.25) {$\oc A$};
		\node [style=port] (30) at (19, -8) {$=$};
		\node [style=differential] (31) at (20.25, 10) {{\bf \aquarius\!\aquarius\!\aquarius}};
		\node [style=object] (32) at (19, 12.25) {$\oc A$};
		\node [style=none] (33) at (20.25, 12.25) {};
		\node [style=duplicate] (34) at (21, 11) {$\Delta^\ast$};
		\node [style=none] (35) at (21.75, 12.25) {};
		\node [style=object] (36) at (13.5, 3.75) {$=$};
		\node [style=object] (37) at (14.75, 5.5) {$\oc A$};
		\node [style=differential] (38) at (15.25, 2.75) {{\bf \aquarius\!\aquarius\!\aquarius}};
		\node [style=none] (39) at (17, 5.5) {};
		\node [style=duplicate] (40) at (14.75, 4.5) {$\Delta$};
		\node [style=none] (41) at (16, 5.5) {};
		\node [style=port] (42) at (19.5, 3.75) {$+$};
		\node [style=none] (43) at (16, 4) {};
		\node [style=none] (44) at (15.25, 4) {};
		\node [style=object] (45) at (20.75, 5.5) {$\oc A$};
		\node [style=differential] (46) at (21.25, 2.75) {{\bf \aquarius\!\aquarius\!\aquarius}};
		\node [style=none] (47) at (22, 5.5) {};
		\node [style=duplicate] (48) at (20.75, 4.5) {$\Delta$};
		\node [style=none] (49) at (23, 5.5) {};
		\node [style=none] (50) at (23, 4.25) {};
		\node [style=none] (51) at (21.25, 4) {};
		\node [style=object] (52) at (22.5, 8.25) {$\oc A$};
		\node [style=object] (53) at (23.5, 8.25) {$\oc A$};
		\node [style=none] (54) at (22.5, 12.25) {};
		\node [style=none] (55) at (23.5, 12.25) {};
		\node [style=object] (56) at (17.75, 1.5) {$\oc A$};
		\node [style=object] (57) at (18.75, 1.5) {$\oc A$};
		\node [style=none] (58) at (17.75, 5.5) {};
		\node [style=none] (59) at (18.75, 5.5) {};
		\node [style=none] (60) at (20.25, 9.75) {};
		\node [style=object] (61) at (21.5, 14) {$A$};
		\node [style=none] (62) at (21.5, 9.75) {};
		\node [style=none] (63) at (15.25, 2.5) {};
		\node [style=object] (64) at (16.75, 7.25) {$A$};
		\node [style=none] (65) at (16.75, 2.5) {};
		\node [style=none] (66) at (21.25, 2.5) {};
		\node [style=object] (67) at (22.75, 7.25) {$A$};
		\node [style=none] (68) at (22.75, 2.5) {};
		\node [style=none] (69) at (23, 5.5) {};
		\node [style=none] (70) at (22, 5.5) {};
		\node [style=object] (71) at (23.75, 1.5) {$\oc A$};
		\node [style=object] (72) at (24.75, 1.5) {$\oc A$};
		\node [style=none] (73) at (23.75, 5.5) {};
		\node [style=none] (74) at (24.75, 5.5) {};
		\node [style=port] (75) at (18.75, -2.5) {$+$};
		\node [style=object] (76) at (15, 0.5) {$\oc A$};
		\node [style=duplicate] (77) at (15, -0.5) {$\Delta$};
		\node [style=object] (78) at (14, -5.25) {$\oc A$};
		\node [style=none] (79) at (17, -2.5) {};
		\node [style=none] (80) at (16.5, -3.5) {};
		\node [style=integral] (81) at (16.5, -3.25) {{\bf \aquarius\!\aquarius\!\aquarius}};
		\node [style=none] (82) at (17.5, -2.5) {};
		\node [style=port] (83) at (17.5, -5.25) {$\oc A$};
		\node [style=object] (84) at (18, 0.5) {$A$};
		\node [style=none] (85) at (18, -3.5) {};
		\node [style=port] (86) at (13.5, -2.5) {$=$};
		\node [style=object] (87) at (20, 0.5) {$\oc A$};
		\node [style=duplicate] (88) at (20, -0.5) {$\Delta$};
		\node [style=object] (89) at (22.75, -5.25) {$\oc A$};
		\node [style=none] (90) at (20.5, -2.5) {};
		\node [style=none] (91) at (20, -3.5) {};
		\node [style=integral] (92) at (20, -3.25) {{\bf \aquarius\!\aquarius\!\aquarius}};
		\node [style=none] (93) at (21, -2.5) {};
		\node [style=port] (94) at (21, -5.25) {$\oc A$};
		\node [style=object] (95) at (21.5, 0.5) {$A$};
		\node [style=none] (96) at (21.5, -3.5) {};
		\node [style=object] (97) at (17.25, -6.5) {$\oc A$};
		\node [style=differential] (98) at (17, -8.5) {{\bf =\!=\!=\!=}};
		\node [style=object] (99) at (18.75, -6.5) {$A$};
		\node [style=duplicate] (100) at (17.25, -7.5) {$\Delta$};
		\node [style=object] (101) at (17, -9.25) {$\oc A$};
		\node [style=object] (102) at (18.5, -9.25) {$\oc A$};
		\node [style=object] (103) at (15.25, -6.5) {$A$};
		\node [style=differential] (104) at (14.75, -8.5) {{\bf =\!=\!=\!=}};
		\node [style=object] (105) at (14.75, -9.25) {$\oc A$};
		\node [style=object] (106) at (13.5, -9.25) {$\oc A$};
		\node [style=object] (107) at (14, -6.5) {$\oc A$};
		\node [style=duplicate] (108) at (14, -7.5) {$\Delta$};
		\node [style=port] (109) at (16, -8) {$+$};
		\node [style=port] (110) at (13, -8) {$=$};
		\node [style=object] (111) at (18.5, 10.5) {$=$};
	\end{pgfonlayer}
	\begin{pgfonlayer}{edgelayer}
		\draw [style=wire, bend right] (0) to (1);
		\draw [style=wire, bend left] (0) to (2);
		\draw [style=wire, bend right] (4) to (3);
		\draw [style=wire, bend left] (4) to (5);
		\draw [style=wire] (0) to (4);
		\draw [style=wire, in=-90, out=45] (7) to (8);
		\draw [style=wire, in=150, out=-150, looseness=1.50] (9) to (7);
		\draw [style=wire] (6) to (9);
		\draw [style=wire] (7) to (10);
		\draw [style=wire, bend left, looseness=1.25] (9) to (11);
		\draw [style=wire, in=-90, out=60, looseness=1.25] (13) to (12);
		\draw [style=wire, in=91, out=-135, looseness=0.75] (17) to (15);
		\draw [style=wire, in=150, out=-30] (17) to (13);
		\draw [style=wire] (16) to (17);
		\draw [style=wire] (13) to (14);
		\draw [style=wire, in=-90, out=30, looseness=1.50] (22) to (20.center);
		\draw [style=wire] (21.center) to (22);
		\draw [style=wire, in=-90, out=150] (22) to (23);
		\draw [style=wire, bend left=90, looseness=2.00] (20.center) to (24.center);
		\draw [style=wire, bend right=90, looseness=2.00] (21.center) to (26.center);
		\draw [style=wire] (25) to (26.center);
		\draw [style=wire, bend right] (28) to (27);
		\draw [style=wire, bend left] (28) to (29);
		\draw [style=wire] (24.center) to (28);
		\draw [style=wire, in=-90, out=135] (31) to (32);
		\draw [style=wire, bend left] (34) to (33.center);
		\draw [style=wire, bend right] (34) to (35.center);
		\draw [style=wire, in=45, out=-90, looseness=1.25] (34) to (31);
		\draw [style=wire, in=-90, out=30, looseness=1.25] (38) to (39.center);
		\draw [style=wire, in=150, out=-135, looseness=1.50] (40) to (38);
		\draw [style=wire] (37) to (40);
		\draw [style=wire, bend left=90, looseness=2.00] (43.center) to (44.center);
		\draw [style=wire] (41.center) to (43.center);
		\draw [style=wire, in=90, out=-30] (40) to (44.center);
		\draw [style=wire, in=-90, out=45, looseness=1.25] (46) to (47.center);
		\draw [style=wire, in=150, out=-135, looseness=1.50] (48) to (46);
		\draw [style=wire] (45) to (48);
		\draw [style=wire, bend left=90, looseness=2.00] (50.center) to (51.center);
		\draw [style=wire] (49.center) to (50.center);
		\draw [style=wire, in=90, out=-30] (48) to (51.center);
		\draw [style=wire, bend left=90, looseness=2.00] (35.center) to (55.center);
		\draw [style=wire, bend left=90, looseness=1.50] (33.center) to (54.center);
		\draw [style=wire] (54.center) to (52);
		\draw [style=wire] (55.center) to (53);
		\draw [style=wire] (58.center) to (56);
		\draw [style=wire] (59.center) to (57);
		\draw [style=wire, bend left=90, looseness=1.75] (41.center) to (58.center);
		\draw [style=wire, bend left=90, looseness=1.75] (39.center) to (59.center);
		\draw [style=wire, bend right=90, looseness=2.00] (60.center) to (62.center);
		\draw [style=wire] (61) to (62.center);
		\draw [style=wire] (31) to (60.center);
		\draw [style=wire, bend right=90, looseness=2.00] (63.center) to (65.center);
		\draw [style=wire] (64) to (65.center);
		\draw [style=wire] (38) to (63.center);
		\draw [style=wire, bend right=90, looseness=2.00] (66.center) to (68.center);
		\draw [style=wire] (67) to (68.center);
		\draw [style=wire] (46) to (66.center);
		\draw [style=wire] (73.center) to (71);
		\draw [style=wire] (74.center) to (72);
		\draw [style=wire, bend left=90, looseness=1.75] (70.center) to (73.center);
		\draw [style=wire, bend left=90, looseness=1.75] (69.center) to (74.center);
		\draw [style=wire] (76) to (77);
		\draw [style=wire, in=-90, out=30, looseness=1.50] (81) to (79.center);
		\draw [style=wire] (80.center) to (81);
		\draw [style=wire, bend left=90, looseness=2.00] (79.center) to (82.center);
		\draw [style=wire] (82.center) to (83);
		\draw [style=wire, bend right=90, looseness=2.00] (80.center) to (85.center);
		\draw [style=wire] (84) to (85.center);
		\draw [style=wire, in=90, out=-150, looseness=0.75] (77) to (78);
		\draw [style=wire, in=150, out=-15] (77) to (81);
		\draw [style=wire] (87) to (88);
		\draw [style=wire, in=90, out=-15, looseness=1.25] (88) to (89);
		\draw [style=wire, in=-90, out=30, looseness=1.50] (92) to (90.center);
		\draw [style=wire] (91.center) to (92);
		\draw [style=wire, bend left=90, looseness=2.00] (90.center) to (93.center);
		\draw [style=wire] (93.center) to (94);
		\draw [style=wire, bend right=90, looseness=2.00] (91.center) to (96.center);
		\draw [style=wire] (95) to (96.center);
		\draw [style=wire, bend right=60] (88) to (92);
		\draw [style=wire, in=-90, out=45] (98) to (99);
		\draw [style=wire, in=150, out=-150, looseness=1.50] (100) to (98);
		\draw [style=wire] (97) to (100);
		\draw [style=wire] (98) to (101);
		\draw [style=wire, bend left, looseness=1.25] (100) to (102);
		\draw [style=wire, in=-90, out=60, looseness=1.25] (104) to (103);
		\draw [style=wire, in=91, out=-135, looseness=0.75] (108) to (106);
		\draw [style=wire, in=150, out=-30] (108) to (104);
		\draw [style=wire] (107) to (108);
		\draw [style=wire] (104) to (105);
	\end{pgfonlayer}
\end{tikzpicture}
   \end{array}
\end{align*}
To show that $\mathsf{d}$ is natural and satisfies the rest of the deriving transformation axioms is similar. Conversely, starting with a deriving transformation $\mathsf{d}$, we will show that the constructed $\mathsf{r}$ satisfies the reverse chain rule \textbf{[r.4]} by using the snake equations and sliding equations on the chain rule \textbf{[d.4]}: 
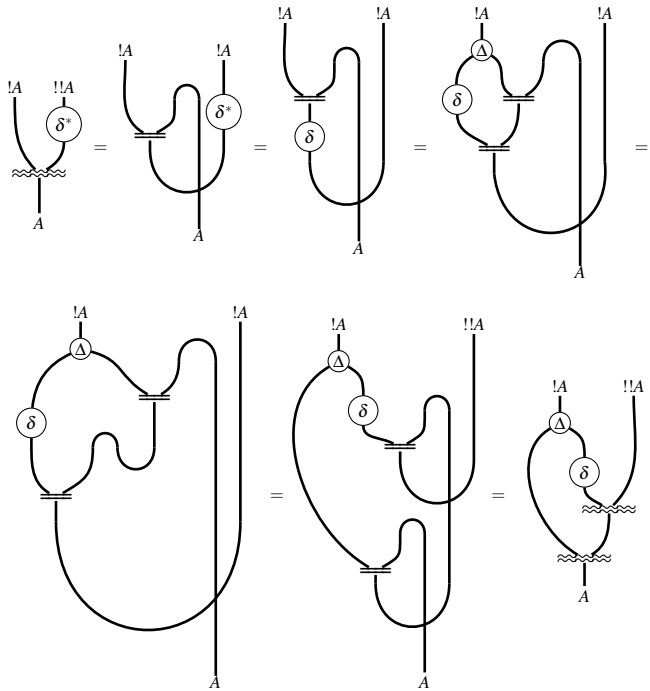
\begin{align*}
   \begin{array}[c]{c}
\begin{tikzpicture}
	\begin{pgfonlayer}{nodelayer}
		\node [style=port] (0) at (24, 12.5) {$=$};
		\node [style=object] (1) at (22.75, 11) {$A$};
		\node [style=component] (2) at (23.25, 13) {$\delta^\ast$};
		\node [style=differential] (3) at (22.75, 12) {{\bf \aquarius\!\aquarius\!\aquarius}};
		\node [style=object] (4) at (22.25, 13.75) {$\oc A$};
		\node [style=object] (5) at (23.25, 13.75) {$\oc \oc A$};
		\node [style=none] (6) at (25.5, 13.5) {};
		\node [style=none] (7) at (25, 12.5) {};
		\node [style=integral] (8) at (25, 12.75) {{\bf =\!=\!=\!=}};
		\node [style=object] (9) at (24.5, 14.5) {$\oc A$};
		\node [style=none] (10) at (26, 13.5) {};
		\node [style=port] (11) at (26, 10.75) {$A$};
		\node [style=object] (12) at (26.5, 14.5) {$\oc A$};
		\node [style=none] (13) at (26.5, 12.5) {};
		\node [style=component] (14) at (26.5, 13.25) {$\delta^\ast$};
		\node [style=port] (15) at (27.25, 12.5) {$=$};
		\node [style=none] (16) at (28.75, 14.25) {};
		\node [style=none] (17) at (28.25, 12.25) {};
		\node [style=integral] (18) at (28.25, 13.5) {{\bf =\!=\!=\!=}};
		\node [style=object] (19) at (27.75, 15.25) {$\oc A$};
		\node [style=none] (20) at (29.25, 14.25) {};
		\node [style=port] (21) at (29.25, 10.5) {$A$};
		\node [style=object] (22) at (29.75, 15.25) {$\oc A$};
		\node [style=none] (23) at (29.75, 12.25) {};
		\node [style=component] (24) at (28.25, 12.75) {$\delta$};
		\node [style=port] (25) at (30.5, 12.5) {$=$};
		\node [style=component] (26) at (31.25, 13.5) {$\delta$};
		\node [style=duplicate] (27) at (31.75, 14.5) {$\Delta$};
		\node [style=object] (28) at (31.75, 15.25) {$\oc A$};
		\node [style=differential] (29) at (32.5, 13.5) {{\bf =\!=\!=\!=}};
		\node [style=differential] (30) at (32, 12.5) {{\bf =\!=\!=\!=}};
		\node [style=none] (31) at (32, 12) {};
		\node [style=none] (32) at (34.25, 12.25) {};
		\node [style=none] (33) at (33, 14.25) {};
		\node [style=none] (34) at (33.75, 14.25) {};
		\node [style=port] (35) at (33.75, 10) {$A$};
		\node [style=object] (36) at (34.25, 15.25) {$\oc A$};
		\node [style=port] (37) at (35, 12.5) {$=$};
	\end{pgfonlayer}
	\begin{pgfonlayer}{edgelayer}
		\draw [style=wire, in=-90, out=30] (3) to (2);
		\draw [style=wire, in=-90, out=135] (3) to (4);
		\draw [style=wire] (3) to (1);
		\draw [style=wire] (5) to (2);
		\draw [style=wire, in=-90, out=30, looseness=1.50] (8) to (6.center);
		\draw [style=wire] (7.center) to (8);
		\draw [style=wire, in=-90, out=150] (8) to (9);
		\draw [style=wire, bend left=90, looseness=2.00] (6.center) to (10.center);
		\draw [style=wire] (10.center) to (11);
		\draw [style=wire, bend right=90, looseness=2.00] (7.center) to (13.center);
		\draw [style=wire] (12) to (14);
		\draw [style=wire] (14) to (13.center);
		\draw [style=wire, in=-90, out=30, looseness=1.50] (18) to (16.center);
		\draw [style=wire, in=-90, out=150] (18) to (19);
		\draw [style=wire, bend left=90, looseness=2.00] (16.center) to (20.center);
		\draw [style=wire] (20.center) to (21);
		\draw [style=wire, bend right=90, looseness=2.00] (17.center) to (23.center);
		\draw [style=wire] (22) to (23.center);
		\draw [style=wire] (18) to (24);
		\draw [style=wire] (24) to (17.center);
		\draw [style=wire, bend right] (27) to (26);
		\draw [style=wire] (28) to (27);
		\draw [style=wire, in=150, out=-30, looseness=1.25] (27) to (29);
		\draw [style=wire, in=30, out=-90] (29) to (30);
		\draw [style=wire, in=150, out=-90] (26) to (30);
		\draw [style=wire, bend right=90, looseness=2.00] (31.center) to (32.center);
		\draw [style=wire] (30) to (31.center);
		\draw [style=wire, bend left=90, looseness=2.00] (33.center) to (34.center);
		\draw [style=wire, in=-90, out=30] (29) to (33.center);
		\draw [style=wire] (34.center) to (35);
		\draw [style=wire] (36) to (32.center);
	\end{pgfonlayer}
\end{tikzpicture}
   \end{array} \\
    \begin{array}[c]{c}
    \begin{tikzpicture}
	\begin{pgfonlayer}{nodelayer}
		\node [style=port] (119) at (45.25, 12.5) {$=$};
		\node [style=duplicate] (120) at (46.5, 14) {$\Delta$};
		\node [style=object] (121) at (46.5, 14.75) {$\oc A$};
		\node [style=differential] (122) at (47, 11.25) {{\bf \aquarius\!\aquarius\!\aquarius}};
		\node [style=object] (123) at (47, 10.5) {$A$};
		\node [style=component] (124) at (47, 13) {$\delta$};
		\node [style=differential] (125) at (47.5, 12.25) {{\bf \aquarius\!\aquarius\!\aquarius}};
		\node [style=object] (126) at (48, 14.75) {$\oc \oc A$};
		\node [style=none] (127) at (43.25, 11.75) {};
		\node [style=none] (128) at (42.75, 10.75) {};
		\node [style=integral] (129) at (42.75, 11) {{\bf =\!=\!=\!=}};
		\node [style=none] (130) at (43.75, 11.75) {};
		\node [style=none] (131) at (44.25, 10.75) {};
		\node [style=duplicate] (132) at (42, 15.25) {$\Delta$};
		\node [style=object] (133) at (42, 16) {$\oc A$};
		\node [style=object] (134) at (43.75, 8.75) {$A$};
		\node [style=component] (135) at (42.5, 14.25) {$\delta$};
		\node [style=object] (136) at (44.75, 16) {$\oc \oc A$};
		\node [style=none] (137) at (43.75, 14.25) {};
		\node [style=none] (138) at (43.25, 13.25) {};
		\node [style=integral] (139) at (43.25, 13.5) {{\bf =\!=\!=\!=}};
		\node [style=none] (140) at (44.25, 14.25) {};
		\node [style=none] (141) at (44.75, 13.25) {};
		\node [style=component] (143) at (35.75, 14) {$\delta$};
		\node [style=duplicate] (144) at (36.75, 15.5) {$\Delta$};
		\node [style=object] (145) at (36.75, 16.25) {$\oc A$};
		\node [style=differential] (146) at (38.25, 14.5) {{\bf =\!=\!=\!=}};
		\node [style=differential] (147) at (36.25, 12.5) {{\bf =\!=\!=\!=}};
		\node [style=none] (148) at (36.25, 12) {};
		\node [style=none] (149) at (40, 12) {};
		\node [style=none] (150) at (38.75, 15.25) {};
		\node [style=none] (151) at (39.5, 15.25) {};
		\node [style=port] (152) at (39.5, 8.75) {$A$};
		\node [style=object] (153) at (40, 16.25) {$\oc A$};
		\node [style=port] (154) at (40.75, 12.5) {$=$};
		\node [style=none] (155) at (37.5, 13.5) {};
		\node [style=none] (156) at (38.25, 13.5) {};
		\node [style=none] (157) at (37, 13.5) {};
		\node [style=none] (158) at (37.5, 13.5) {};
	\end{pgfonlayer}
	\begin{pgfonlayer}{edgelayer}
		\draw [style=wire] (121) to (120);
		\draw [style=wire] (122) to (123);
		\draw [style=wire, in=-90, out=150] (125) to (124);
		\draw [style=wire, in=-90, out=45] (125) to (126);
		\draw [style=wire, in=90, out=-30, looseness=1.25] (120) to (124);
		\draw [style=wire, in=150, out=-150] (120) to (122);
		\draw [style=wire, in=30, out=-90] (125) to (122);
		\draw [style=wire, in=-90, out=30, looseness=1.50] (129) to (127.center);
		\draw [style=wire] (128.center) to (129);
		\draw [style=wire, bend left=90, looseness=2.00] (127.center) to (130.center);
		\draw [style=wire, bend right=90, looseness=2.00] (128.center) to (131.center);
		\draw [style=wire] (133) to (132);
		\draw [style=wire, in=90, out=-30, looseness=1.25] (132) to (135);
		\draw [style=wire, in=-90, out=30, looseness=1.50] (139) to (137.center);
		\draw [style=wire] (138.center) to (139);
		\draw [style=wire, bend left=90, looseness=2.00] (137.center) to (140.center);
		\draw [style=wire, bend right=90, looseness=2.00] (138.center) to (141.center);
		\draw [style=wire, in=150, out=-90] (135) to (139);
		\draw [style=wire] (136) to (141.center);
		\draw [style=wire] (140.center) to (131.center);
		\draw [style=wire, in=-150, out=150] (129) to (132);
		\draw [style=wire] (130.center) to (134);
		\draw [style=wire, bend right] (144) to (143);
		\draw [style=wire] (145) to (144);
		\draw [style=wire, in=150, out=-15, looseness=1.25] (144) to (146);
		\draw [style=wire, in=150, out=-90] (143) to (147);
		\draw [style=wire, bend right=90, looseness=2.00] (148.center) to (149.center);
		\draw [style=wire] (147) to (148.center);
		\draw [style=wire, bend left=90, looseness=2.00] (150.center) to (151.center);
		\draw [style=wire, in=-90, out=30] (146) to (150.center);
		\draw [style=wire] (151.center) to (152);
		\draw [style=wire] (153) to (149.center);
		\draw [style=wire, bend right=90, looseness=2.00] (155.center) to (157.center);
		\draw [style=wire, bend left=90, looseness=2.00] (156.center) to (158.center);
		\draw [style=wire] (146) to (156.center);
		\draw [style=wire, in=30, out=-90, looseness=1.25] (157.center) to (147);
	\end{pgfonlayer}
\end{tikzpicture}
     \end{array}
\end{align*}
To show that $\mathsf{r}$ satisfies the rest of the reverse deriving transformations axioms is similar. Lastly, we use the snake equations to prove that these constructions are inverses of each other: 
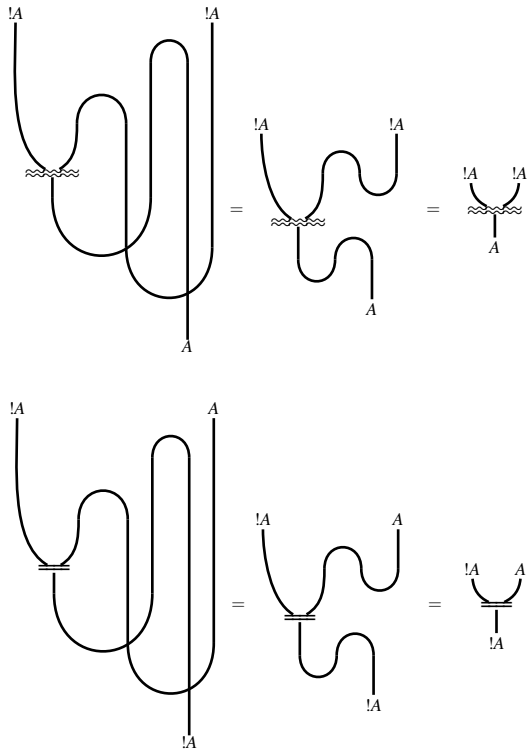
\begin{align*}
   \begin{array}[c]{c}
\begin{tikzpicture}
	\begin{pgfonlayer}{nodelayer}
		\node [style=none] (0) at (7, 9.75) {};
		\node [style=none] (1) at (6.5, 8.25) {};
		\node [style=integral] (2) at (6.5, 8.75) {{\bf \aquarius\!\aquarius\!\aquarius}};
		\node [style=object] (3) at (5.75, 12) {$\oc A$};
		\node [style=none] (4) at (8, 9.75) {};
		\node [style=none] (5) at (8.5, 8.25) {};
		\node [style=none] (6) at (8.5, 11) {};
		\node [style=none] (7) at (8, 7.25) {};
		\node [style=none] (8) at (9.25, 11) {};
		\node [style=port] (9) at (9.25, 5.25) {$A$};
		\node [style=object] (10) at (9.75, 12) {$\oc A$};
		\node [style=none] (11) at (9.75, 7.25) {};
		\node [style=object] (12) at (16, 8.75) {$\oc A$};
		\node [style=object] (13) at (15.5, 7.25) {$A$};
		\node [style=integral] (14) at (15.5, 8) {{\bf \aquarius\!\aquarius\!\aquarius}};
		\node [style=object] (15) at (15, 8.75) {$\oc A$};
		\node [style=port] (16) at (14.25, 8) {$=$};
		\node [style=none] (17) at (12.25, 7) {};
		\node [style=none] (18) at (11.5, 7) {};
		\node [style=none] (19) at (13, 7) {};
		\node [style=object] (20) at (13, 6) {$A$};
		\node [style=none] (21) at (12.25, 7) {};
		\node [style=object] (22) at (10.25, 8) {$=$};
		\node [style=none] (23) at (12.75, 8.75) {};
		\node [style=none] (24) at (13.5, 8.75) {};
		\node [style=object] (25) at (13.5, 9.75) {$\oc A$};
		\node [style=none] (26) at (12, 8.75) {};
		\node [style=none] (27) at (12.75, 8.75) {};
		\node [style=integral] (28) at (11.5, 7.75) {{\bf \aquarius\!\aquarius\!\aquarius}};
		\node [style=object] (29) at (10.75, 9.75) {$\oc A$};
	\end{pgfonlayer}
	\begin{pgfonlayer}{edgelayer}
		\draw [style=wire, in=-90, out=30] (2) to (0.center);
		\draw [style=wire] (1.center) to (2);
		\draw [style=wire, in=-90, out=150, looseness=0.75] (2) to (3);
		\draw [style=wire, bend left=90, looseness=2.00] (0.center) to (4.center);
		\draw [style=wire, bend right=90, looseness=2.00] (1.center) to (5.center);
		\draw [style=wire, bend left=90, looseness=2.00] (6.center) to (8.center);
		\draw [style=wire] (8.center) to (9);
		\draw [style=wire, bend right=90, looseness=2.00] (7.center) to (11.center);
		\draw [style=wire] (10) to (11.center);
		\draw [style=wire] (4.center) to (7.center);
		\draw [style=wire] (5.center) to (6.center);
		\draw [style=wire, bend right] (14) to (12);
		\draw [style=wire] (13) to (14);
		\draw [style=wire, bend left] (14) to (15);
		\draw [style=wire, bend left=90, looseness=2.00] (17.center) to (19.center);
		\draw [style=wire] (19.center) to (20);
		\draw [style=wire, bend right=90, looseness=2.00] (18.center) to (21.center);
		\draw [style=wire, bend right=90, looseness=2.00] (23.center) to (26.center);
		\draw [style=wire, bend left=90, looseness=2.00] (24.center) to (27.center);
		\draw [style=wire] (25) to (24.center);
		\draw [style=wire, in=-90, out=150, looseness=0.75] (28) to (29);
		\draw [style=wire] (28) to (18.center);
		\draw [style=wire, in=-90, out=30] (28) to (26.center);
	\end{pgfonlayer}
\end{tikzpicture}
   \end{array}
   \end{align*}
   \begin{align*}
  \begin{array}[c]{c}
\begin{tikzpicture}
	\begin{pgfonlayer}{nodelayer}
		\node [style=none] (0) at (7, 9.75) {};
		\node [style=none] (1) at (6.5, 8.25) {};
		\node [style=integral] (2) at (6.5, 8.75) {{\bf =\!=\!=\!=}};
		\node [style=object] (3) at (5.75, 12) {$\oc A$};
		\node [style=none] (4) at (8, 9.75) {};
		\node [style=none] (5) at (8.5, 8.25) {};
		\node [style=none] (6) at (8.5, 11) {};
		\node [style=none] (7) at (8, 7.25) {};
		\node [style=none] (8) at (9.25, 11) {};
		\node [style=port] (9) at (9.25, 5.25) {$\oc A$};
		\node [style=object] (10) at (9.75, 12) {$A$};
		\node [style=none] (11) at (9.75, 7.25) {};
		\node [style=object] (12) at (16, 8.75) {$A$};
		\node [style=object] (13) at (15.5, 7.25) {$\oc A$};
		\node [style=integral] (14) at (15.5, 8) {{\bf =\!=\!=\!=}};
		\node [style=object] (15) at (15, 8.75) {$\oc A$};
		\node [style=port] (16) at (14.25, 8) {$=$};
		\node [style=none] (17) at (12.25, 7) {};
		\node [style=none] (18) at (11.5, 7) {};
		\node [style=none] (19) at (13, 7) {};
		\node [style=object] (20) at (13, 6) {$\oc A$};
		\node [style=none] (21) at (12.25, 7) {};
		\node [style=object] (22) at (10.25, 8) {$=$};
		\node [style=none] (23) at (12.75, 8.75) {};
		\node [style=none] (24) at (13.5, 8.75) {};
		\node [style=object] (25) at (13.5, 9.75) {$A$};
		\node [style=none] (26) at (12, 8.75) {};
		\node [style=none] (27) at (12.75, 8.75) {};
		\node [style=integral] (28) at (11.5, 7.75) {{\bf =\!=\!=\!=}};
		\node [style=object] (29) at (10.75, 9.75) {$\oc A$};
	\end{pgfonlayer}
	\begin{pgfonlayer}{edgelayer}
		\draw [style=wire, in=-90, out=30] (2) to (0.center);
		\draw [style=wire] (1.center) to (2);
		\draw [style=wire, in=-90, out=150, looseness=0.75] (2) to (3);
		\draw [style=wire, bend left=90, looseness=2.00] (0.center) to (4.center);
		\draw [style=wire, bend right=90, looseness=2.00] (1.center) to (5.center);
		\draw [style=wire, bend left=90, looseness=2.00] (6.center) to (8.center);
		\draw [style=wire] (8.center) to (9);
		\draw [style=wire, bend right=90, looseness=2.00] (7.center) to (11.center);
		\draw [style=wire] (10) to (11.center);
		\draw [style=wire] (4.center) to (7.center);
		\draw [style=wire] (5.center) to (6.center);
		\draw [style=wire, bend right] (14) to (12);
		\draw [style=wire] (13) to (14);
		\draw [style=wire, bend left] (14) to (15);
		\draw [style=wire, bend left=90, looseness=2.00] (17.center) to (19.center);
		\draw [style=wire] (19.center) to (20);
		\draw [style=wire, bend right=90, looseness=2.00] (18.center) to (21.center);
		\draw [style=wire, bend right=90, looseness=2.00] (23.center) to (26.center);
		\draw [style=wire, bend left=90, looseness=2.00] (24.center) to (27.center);
		\draw [style=wire] (25) to (24.center);
		\draw [style=wire, in=-90, out=150, looseness=0.75] (28) to (29);
		\draw [style=wire] (28) to (18.center);
		\draw [style=wire, in=-90, out=30] (28) to (26.center);
	\end{pgfonlayer}
\end{tikzpicture}
   \end{array}
\end{align*}
So we conclude that a reverse differential category is precisely the same thing as differential category which is also self-dual compact closed. 
\end{proof} 

We note that a self-dual compact closed differential category is also a codifferential category (the dual of a differential category). Indeed, observe that if $(\oc, \delta, \varepsilon, \Delta, e)$ is a coalgebra modality on a self-dual compact closed category $\mathbb{X}$, then we can define an algebra modality (the dual of a coalgebra modality). The monad endofunctor $\wn: \mathbb{X} \to \mathbb{X}$ is defined on objects as $\wn A = \oc A$ and on maps $\wn f = \left( \oc(f^\ast) \right)^\ast$, while the remaining natural transformations are the duals of the coalgebra modality natural transformations. Explicitly, $(\wn, \delta^\ast, \varepsilon^\ast, \Delta^\ast, e^\ast)$ is an algebra modality on $\mathbb{X}$. It is crucial to observe that $\oc$ and $\wn$ are equal on objects but not maps, and that $\delta^\ast$, $\varepsilon^\ast$, $\Delta^\ast$ and $e^\ast$ are not necessarily natural with respect to $\oc$. For example, $\wn(f); \Delta^\ast = \Delta^\ast; (\wn(f) \otimes \wn(f))$ but $\Delta^\ast; \oc(f) $ may not be equal to $(\oc(f) \otimes \oc(f)); \Delta^\ast$. Furthermore, if $\mathbb{X}$ is also a differential category, where $\mathsf{d}$ is a deriving transformation for $(\oc, \delta, \varepsilon, \Delta, e)$, then $\mathbb{X}$ is also a codifferential category where $\mathsf{d}^\ast$ is a deriving transformation for $(\wn, \delta^\ast, \varepsilon^\ast, \Delta^\ast, e^\ast)$.

Before providing examples of reverse differential categories, let us first discuss the relation between the reverse deriving transformation and the coderiving transformation. Indeed, observe that if $(\oc, \delta, \varepsilon, \Delta, e)$ is a coalgebra modality on a self-dual compact closed category $\mathbb{X}$, then there is a canonical map of the desired type $\oc A \otimes \oc A \to A$ defined as the composite: 
\begin{align*}
   \begin{array}[c]{c}
\xymatrixcolsep{5pc}\xymatrix{\oc A \otimes \oc A \ar[r]^-{1_{\oc A} \otimes \mathsf{d}^\circ_A} & \oc A \otimes \oc A \otimes A  \ar[r]^-{\cup_{\oc A} \otimes 1_A} & A}    \end{array} &&    \begin{array}[c]{c}
\begin{tikzpicture}
	\begin{pgfonlayer}{nodelayer}
		\node [style=none] (93) at (3, -13) {};
		\node [style=none] (96) at (3.75, -13) {};
		\node [style=object] (101) at (3, -10.5) {$\oc A$};
		\node [style=object] (117) at (4.25, -10.5) {$\oc A$};
		\node [style=differential] (118) at (4.25, -11.25) {{\bf =\!=\!=\!=}};
		\node [style=object] (119) at (5.25, -14) {$A$};
		\node [style=none] (120) at (3.75, -12.25) {};
	\end{pgfonlayer}
	\begin{pgfonlayer}{edgelayer}
		\draw [style=wire, bend right=90, looseness=2.00] (93.center) to (96.center);
		\draw [style=wire] (101) to (93.center);
		\draw [style=wire] (117) to (118);
		\draw [style=wire, in=90, out=-150, looseness=1.25] (118) to (120.center);
		\draw [style=wire, in=90, out=-30, looseness=1.25] (118) to (119);
		\draw [style=wire] (120.center) to (96.center);
	\end{pgfonlayer}
\end{tikzpicture}
   \end{array}
\end{align*}
While this is a map of the right type, this is not automatically a reverse deriving transformation. This map is a reverse deriving transformation if and only if the coderiving transformation is a deriving transformation for the induced algebra modality.  

\begin{lemma}\label{lem:coder-rev} If $\mathbb{X}$ is a self-dual compact closed category, which is additive symmetric monoidal and equipped with a coalgebra modality $(\oc, \delta, \varepsilon, \Delta, e)$, the following are equivalent: 
\begin{enumerate}
    \item $\mathsf{r} := (1 \otimes \mathsf{d}^\circ);(\cup \otimes 1)$ is a reverse deriving transformation;  
   \item $\mathsf{d}^\circ$ is a deriving transformation for the algebra modality $(\wn, \delta^\ast, \varepsilon^\ast, \Delta^\ast, e^\ast)$; that is, the dual diagrams of Definition \ref{def:diffcat}  commute. 
    \end{enumerate} 
    Furthermore, in this case, $\mathsf{d}^\ast = \mathsf{d}^\circ$. 
\end{lemma}
\begin{proof} For $(i) \Rightarrow (ii)$: by Theorem \ref{thm:rdc_to_dc}, we obtain a deriving transformation $\mathsf{d}_A: \oc A \otimes A \to \oc A$. We will now show that $\mathsf{d}^\ast = \mathsf{d}^\circ$. So using the snake equations and Theorem \ref{thm:rdc_to_dc}, we compute: 
\begin{align*}
   \begin{array}[c]{c}
\begin{tikzpicture}
	\begin{pgfonlayer}{nodelayer}
		\node [style=none] (93) at (13, 12.25) {};
		\node [style=none] (96) at (13.75, 12.25) {};
		\node [style=object] (117) at (14.25, 14.75) {$\oc A$};
		\node [style=differential] (118) at (14.25, 14) {{\bf =\!=\!=\!=}};
		\node [style=object] (119) at (15.25, 11) {$A$};
		\node [style=none] (120) at (13.75, 13) {};
		\node [style=none] (121) at (5.25, 14) {};
		\node [style=none] (122) at (4.75, 12.5) {};
		\node [style=integral] (123) at (4.75, 13) {{\bf =\!=\!=\!=}};
		\node [style=none] (125) at (6.5, 14) {};
		\node [style=none] (126) at (7, 12.5) {};
		\node [style=object] (127) at (7, 15.25) {$\oc A$};
		\node [style=object] (128) at (6.5, 10.25) {$A$};
		\node [style=object] (143) at (7.75, 12.5) {$=$};
		\node [style=none] (151) at (4.25, 14) {};
		\node [style=none] (152) at (5.75, 14) {};
		\node [style=object] (153) at (5.75, 10.25) {$\oc A$};
		\node [style=object] (166) at (17.75, 11.5) {$A$};
		\node [style=object] (167) at (17.25, 13) {$\oc A$};
		\node [style=integral] (168) at (17.25, 12.25) {{\bf =\!=\!=\!=}};
		\node [style=object] (169) at (16.75, 11.5) {$\oc A$};
		\node [style=port] (170) at (16, 12.25) {$=$};
		\node [style=object] (180) at (10, 11) {$A$};
		\node [style=none] (182) at (8.75, 13.25) {};
		\node [style=object] (184) at (8.75, 11) {$\oc A$};
		\node [style=none] (185) at (9.5, 13.25) {};
		\node [style=none] (186) at (8.75, 13.25) {};
		\node [style=integral] (187) at (10, 12.25) {{\bf \aquarius\!\aquarius\!\aquarius}};
		\node [style=object] (188) at (10.75, 14.25) {$\oc A$};
		\node [style=none] (190) at (12.25, 13.25) {};
		\node [style=object] (191) at (12.25, 11) {$\oc A$};
		\node [style=none] (192) at (13, 13.25) {};
		\node [style=none] (193) at (12.25, 13.25) {};
		\node [style=object] (194) at (11.5, 12.5) {$=$};
	\end{pgfonlayer}
	\begin{pgfonlayer}{edgelayer}
		\draw [style=wire, bend right=90, looseness=2.00] (93.center) to (96.center);
		\draw [style=wire] (117) to (118);
		\draw [style=wire, in=90, out=-150, looseness=1.25] (118) to (120.center);
		\draw [style=wire, in=90, out=-30] (118) to (119);
		\draw [style=wire] (120.center) to (96.center);
		\draw [style=wire, in=-90, out=30] (123) to (121.center);
		\draw [style=wire] (122.center) to (123);
		\draw [style=wire, bend left=90, looseness=2.00] (121.center) to (125.center);
		\draw [style=wire, bend right=90, looseness=2.00] (122.center) to (126.center);
		\draw [style=wire] (125.center) to (128);
		\draw [style=wire] (126.center) to (127);
		\draw [style=wire, bend left=90, looseness=2.00] (151.center) to (152.center);
		\draw [style=wire, in=150, out=-90] (151.center) to (123);
		\draw [style=wire] (152.center) to (153);
		\draw [style=wire, bend left] (168) to (166);
		\draw [style=wire] (167) to (168);
		\draw [style=wire, bend right] (168) to (169);
		\draw [style=wire, bend left=90, looseness=2.00] (182.center) to (185.center);
		\draw [style=wire, in=-90, out=30, looseness=0.75] (187) to (188);
		\draw [style=wire, in=-90, out=150] (187) to (185.center);
		\draw [style=wire] (187) to (180);
		\draw [style=wire] (186.center) to (184);
		\draw [style=wire, bend left=90, looseness=2.00] (190.center) to (192.center);
		\draw [style=wire] (193.center) to (191);
		\draw [style=wire] (192.center) to (93.center);
	\end{pgfonlayer}
\end{tikzpicture}
   \end{array}
\end{align*}
So $\mathsf{d}^\ast = \mathsf{d}^\circ$. Therefore, by the above discussion, $\mathsf{d}^\ast = \mathsf{d}^\circ$ is a deriving transformation for the algebra modality $(\wn, \delta^\ast, \varepsilon^\ast, \Delta^\ast, e^\ast)$. Conversely, for $(ii) \Rightarrow (i)$: by the dual of the above discussion, we have that ${\mathsf{d}^\circ}^\ast$ is a deriving transformation for the coalgebra modality $(\oc, \delta, \varepsilon, \Delta, e)$. Then by Theorem \ref{thm:rdc_to_dc}, we obtain a reverse deriving transformation $\mathsf{r}_A: \oc A \otimes \oc A \to A$. Expanding out the construction, we compute: 
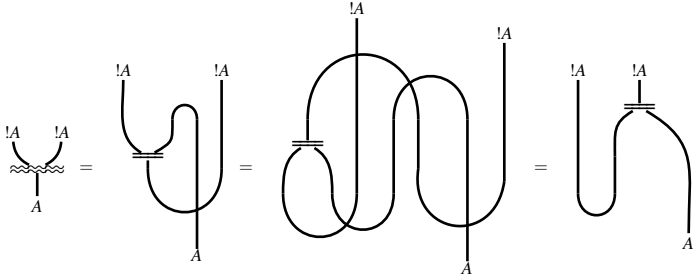
\begin{align*}
   \begin{array}[c]{c}
\begin{tikzpicture}
	\begin{pgfonlayer}{nodelayer}
		\node [style=none] (121) at (18.5, 2) {};
		\node [style=none] (122) at (18, 3.5) {};
		\node [style=integral] (123) at (18, 3) {{\bf =\!=\!=\!=}};
		\node [style=none] (125) at (19.75, 2) {};
		\node [style=none] (126) at (20.25, 3.5) {};
		\node [style=none] (151) at (17.5, 2) {};
		\node [style=none] (152) at (19, 2) {};
		\node [style=object] (153) at (19, 5.75) {$\oc A$};
		\node [style=port] (170) at (13.5, 2.5) {$=$};
		\node [style=object] (194) at (16.75, 2.5) {$=$};
		\node [style=none] (195) at (15.25, 3.5) {};
		\node [style=none] (196) at (14.75, 2.5) {};
		\node [style=integral] (197) at (14.75, 2.75) {{\bf =\!=\!=\!=}};
		\node [style=object] (198) at (14.25, 4.5) {$\oc A$};
		\node [style=none] (199) at (15.75, 3.5) {};
		\node [style=port] (200) at (15.75, 0.75) {$A$};
		\node [style=object] (201) at (16.25, 4.5) {$\oc A$};
		\node [style=none] (202) at (16.25, 2.5) {};
		\node [style=object] (203) at (13, 3.25) {$\oc A$};
		\node [style=object] (204) at (12.5, 1.75) {$A$};
		\node [style=integral] (205) at (12.5, 2.5) {{\bf \aquarius\!\aquarius\!\aquarius}};
		\node [style=object] (206) at (12, 3.25) {$\oc A$};
		\node [style=none] (207) at (20.25, 2.5) {};
		\node [style=none] (208) at (22, 2.25) {};
		\node [style=object] (209) at (22, 5.25) {$\oc A$};
		\node [style=none] (210) at (22, 2.25) {};
		\node [style=none] (211) at (19.75, 3.5) {};
		\node [style=none] (212) at (21.25, 3.5) {};
		\node [style=none] (213) at (19.75, 3.5) {};
		\node [style=none] (214) at (21.25, 3.5) {};
		\node [style=port] (215) at (21.25, 0.5) {$A$};
		\node [style=object] (216) at (22.75, 2.5) {$=$};
		\node [style=none] (224) at (23.5, 2) {};
		\node [style=none] (225) at (24.25, 2) {};
		\node [style=object] (226) at (23.5, 4.5) {$\oc A$};
		\node [style=object] (227) at (24.75, 4.5) {$\oc A$};
		\node [style=differential] (228) at (24.75, 3.75) {{\bf =\!=\!=\!=}};
		\node [style=object] (229) at (25.75, 1) {$A$};
		\node [style=none] (230) at (24.25, 2.75) {};
	\end{pgfonlayer}
	\begin{pgfonlayer}{edgelayer}
		\draw [style=wire, in=90, out=-30] (123) to (121.center);
		\draw [style=wire] (122.center) to (123);
		\draw [style=wire, bend right=90, looseness=2.00] (121.center) to (125.center);
		\draw [style=wire, bend left=90, looseness=2.00] (122.center) to (126.center);
		\draw [style=wire, bend right=90, looseness=2.00] (151.center) to (152.center);
		\draw [style=wire, in=-150, out=90] (151.center) to (123);
		\draw [style=wire] (152.center) to (153);
		\draw [style=wire, in=-90, out=30, looseness=1.50] (197) to (195.center);
		\draw [style=wire] (196.center) to (197);
		\draw [style=wire, in=-90, out=150] (197) to (198);
		\draw [style=wire, bend left=90, looseness=2.00] (195.center) to (199.center);
		\draw [style=wire] (199.center) to (200);
		\draw [style=wire, bend right=90, looseness=2.00] (196.center) to (202.center);
		\draw [style=wire] (201) to (202.center);
		\draw [style=wire, bend right] (205) to (203);
		\draw [style=wire] (204) to (205);
		\draw [style=wire, bend left] (205) to (206);
		\draw [style=wire, bend right=90, looseness=2.00] (207.center) to (208.center);
		\draw [style=wire] (126.center) to (207.center);
		\draw [style=wire] (209) to (210.center);
		\draw [style=wire, bend left=90, looseness=2.00] (211.center) to (212.center);
		\draw [style=wire] (213.center) to (125.center);
		\draw [style=wire] (214.center) to (215);
		\draw [style=wire, bend right=90, looseness=2.00] (224.center) to (225.center);
		\draw [style=wire] (226) to (224.center);
		\draw [style=wire] (227) to (228);
		\draw [style=wire, in=90, out=-150, looseness=1.25] (228) to (230.center);
		\draw [style=wire, in=90, out=-30, looseness=1.25] (228) to (229);
		\draw [style=wire] (230.center) to (225.center);
	\end{pgfonlayer}
\end{tikzpicture}
   \end{array}
\end{align*}
So we conclude that $\mathsf{r} := (1 \otimes \mathsf{d}^\circ);(\cup \otimes 1)$. 
\end{proof}

In the presence of Seely isomorphisms, if all the important structure maps are duals of one another, then the reverse deriving transformation is of the above form. 

\begin{definition} A reverse differential storage category is a reverse differential category with finite products whose coalgebra modality has Seely isomorphisms and such that $\eta = \varepsilon^\ast$, $\nabla= \Delta^\ast$, and $u = e^\ast$ (where $\eta$ is the induced codereliction, and $\nabla$ and $\mathsf{u}$ are the induced natural monoid structure). 
\end{definition}

\begin{corollary} A reverse differential storage category is precisely a differential storage category which is also self-dual compact closed and such that $\eta = \varepsilon^\ast$, $\nabla= \Delta^\ast$, and $u = e^\ast$. Furthermore, in a reverse differential storage category, the reverse deriving transformation is of the form $\mathsf{r} = (1 \otimes \mathsf{d}^\circ);(\cup \otimes 1)$. 
\end{corollary}
\begin{proof} The first part of the statement is simply an extension of Theorem \ref{thm:rdc_to_dc}. For the second part, recall that in a differential storage category the deriving transformation is of the form $\mathsf{d} = (1 \otimes \eta); \nabla$. By assumption, the dual of the deriving transformation is computed out to be: 
\[\mathsf{d}^\ast = \left(  (1 \otimes \eta); \nabla \right)^\ast = \nabla^\ast ; (1 \otimes \eta)^\ast = \nabla^\ast ; (1 \otimes \eta^\ast) = \Delta ; (1 \otimes \varepsilon) = \mathsf{d}^\circ \] 
So $\mathsf{d}^\ast = \mathsf{d}^\circ$. However, recall that since we are in the self-dual case, $\mathsf{d}^\ast = \mathsf{d}^\circ$ is a deriving transformation for the algebra modality $(\wn, \delta^\ast, \varepsilon^\ast, \Delta^\ast, e^\ast)$. Then by Lemma \ref{lem:coder-rev}, it follows that $\mathsf{r} = (1 \otimes \mathsf{d}^\circ);(\cup \otimes 1)$. 
\end{proof}

We conclude this section with examples of reverse differential categories. 

\begin{example} \normalfont Let $\mathsf{REL}$ be the category of sets and relations, where we recall that the objects are sets and the maps are relations between them; that is, a relation from a set $X$ to a set $Y$, denoted $R: X \to Y$, is a subset $R \subseteq X \times Y$. $\mathsf{REL}$ is a symmetric monoidal category where the monoidal product is given by the Cartesian product of sets, $X \otimes Y = X \times Y$, and where the monoidal unit is a chosen singleton $\lbrace \ast \rbrace$. With this monoidal structure, $\mathsf{REL}$ is also a self-dual compact closed category where for a set $X$, its cup $\cup_X: X \times X \to \lbrace \ast \rbrace$ and cap $\cap_X: \lbrace \ast \rbrace \to X \times X$ are the dual relations which relate the single element to all pairs of copies of elements of $X$: 
\begin{align*}
\cup_X = \left \lbrace \left( (x,x), \ast \right) \vert~ \forall. x \in X \right \rbrace \subset (X \times X) \times \lbrace \ast \rbrace && \cap_X = \left \lbrace \left( \ast, (x,x) \right) \vert~ \forall. x \in X \right \rbrace \subset \lbrace \ast \rbrace  (X \times X) 
\end{align*}
$\mathsf{REL}$ is also a (monoidal) differential category. The additive symmetric monoidal structure is induced by the biproduct, which is given by the disjoint union of sets, $X \sqcup Y$, and where the terminal object is the empty set $\emptyset$. As such, the sum of parallel relations ${R: X \to Y}$ and $S: X \to Y$ is defined as their union $R + S := R \cup S$, while the zero map $0: X \to Y$ is the empty relation $0 := \emptyset$. The coalgebra modality on $\mathsf{REL}$ is given by finite bags (also called finite multisets). So for a set $X$, let $\oc X$ be the set of all finite bags of $X$. This coalgebra modality has the Seely isomorphisms so $\oc(X \sqcup Y) \cong \oc X \times \oc Y$ and $ \oc \emptyset \cong \lbrace \ast \rbrace$. The deriving transformation ${\mathsf{d}_X: \oc X \times X \to \oc X}$ is defined as the relation which adds an element into the bag: 
\begin{align*}
\mathsf{d}_X := \left \lbrace \left( (B, x), B \sqcup \llbracket x \rrbracket \right) \vert~ \forall. B \in \oc X, x \in X \right \rbrace \subset (\oc X \times X ) \times \oc X 
\end{align*}
where $\llbracket x \rrbracket$ is the one element bag and $\sqcup$ is the (necessarily disjoint) union of finite bags. For more details on this differential category, see \cite[Section 2.5.1]{blute2006differential}. By applying Theorem \ref{thm:rdc_to_dc}, $\mathsf{REL}$ is also a reverse differential category where the reverse deriving transformation $\mathsf{r}_X: \oc X \times \oc X \to X$ is the relation that relates two bags that differ by one element to that said element: 
\begin{align*}
\mathsf{r}_X := \left \lbrace \left( (B, B \sqcup \llbracket x \rrbracket), x \right) \vert~ \forall. B \in \oc X, x \in X \right \rbrace \subset (\oc X \times X ) \times \oc X 
\end{align*}
\end{example}

\begin{example} \normalfont The above example generalizes to the weighted relational model \cite[]{journal:weighted-relational,ong2017quantitative}. The underlying category is the biproduct completion of a complete commutative semiring. Briefly recall that a complete commutative semiring is a commutative semiring where one can have sums indexed by arbitrary sets $I$, which we denote by $\sum \limits_{i \in I}$, such that these summation operations satisfy certain distributivity and partitions axioms, see \cite[Section III.B]{ong2017quantitative}. For a complete commutative semiring $R$, define the category $R^\Pi$ whose objects are sets $X$ and where a map from $X$ to $Y$ is a set function $f: X \times Y \to R$, and where composition and identities are defined as in \cite[Section III.B]{ong2017quantitative}. Note that when we take the two-element Boolean algebra $B = \lbrace 0, 1 \rbrace$, then $B^\Pi$ is isomorphic to $\mathsf{REL}$. For any complete commutative semiring $R$, $R^\Pi$ is a symmetric monoidal category where the monoidal product is given by the Cartesian product of sets, $X \otimes Y = X \times Y$, and where the monoidal unit is a chosen singleton $\lbrace \ast \rbrace$. $R^\Pi$ is also a self-dual compact closed category where for a set $X$, its cup $\cup_X: X \times X \to \lbrace \ast \rbrace$ and cap $\cap_X: \lbrace \ast \rbrace \to X \times X$ are the functions defined as follows: 
\begin{align*}
\cup_X: (X \times X) \times \lbrace \ast \rbrace &\to R & \cap_X: \lbrace \ast \rbrace \times (X \times X) &\to R \\
\left( (x,y), \ast) \right) &\mapsto \begin{cases} 0 & \text{ if } x \neq y \\
1 & \text{ if } x =y 
\end{cases} & \left(\ast, (x,y) \right) &\mapsto \begin{cases} 0 & \text{ if } x \neq y \\
1 & \text{ if } x =y 
\end{cases}
\end{align*}
$R^\Pi$ is also a differential category. The additive symmetric monoidal structure is induced by the biproduct, which is given by the disjoint union of sets, $X \sqcup Y$, and where the terminal object is the empty set $\emptyset$. The sum of ${f: X \to Y}$ and $g: X \to Y$ (which recall are functions $X \times Y \to R$) is defined pointwise, $(f+g)(x,y) = f(x,y) + g(x,y)$, while the zero map $0: X \to Y$ is the function which maps everything to zero, $0(x,y) = 0$. The coalgebra modality on $R^\Pi$ is again given by finite bags, that is, for a set $X$, let $\oc X$ be the set of all finite bags of $X$, and this coalgebra modality has the Seely isomorphisms as in the previous example. The deriving transformation ${\mathsf{d}_X: \oc X \times X \to \oc X}$ is defined as follows: 
\begin{align*}
\mathsf{d}_X: (\oc X \times X) \times \oc X &\to R \\
\left( (B,x), B^\prime \right) &\mapsto \begin{cases} 0 & \text{ if } B \sqcup \llbracket x \rrbracket \neq B^\prime \\
\vert B^\prime \vert = \vert B \vert + 1 & \text{ if } B \sqcup \llbracket x \rrbracket = B^\prime
\end{cases}
\end{align*} 
where $\vert B \vert$ is the cardinality of the finite bag. The image by $\vert B^\prime \vert = \vert B \vert + 1$ takes into account that if we were in the unordered case, there would be $n+1$ possible ways of putting an element into a bag of size $n$. Of course the factor disappears in the case that the semiring is additively idempotent (i.e. $1+1 = 1$), such as the two-element Boolean algebra $B$. Which is why the factor does not appear in the differential structure of $\mathsf{REL}$ as described in the previous example.  For more details on this differential category, see \cite[Section 6]{lemay2020convenient}. By applying Theorem \ref{thm:rdc_to_dc}, $R^\Pi$ is also a reverse differential category where the reverse deriving transformation $\mathsf{r}_X: \oc X \times \oc X \to X$ is defined as follows: \begin{align*}
\mathsf{r}_X: (\oc X \times \oc X) \times X &\to R \\
\left( (B,B^\prime), x \right) &\mapsto \begin{cases} 0 & \text{ if } B \sqcup \llbracket x \rrbracket \neq B^\prime \\
\vert B^\prime \vert = \vert B \vert + 1 & \text{ if } B \sqcup \llbracket x \rrbracket = B^\prime
\end{cases}
\end{align*} 
\end{example}

\begin{example} \normalfont Let $k$ be a field and $\mathsf{FVEC}_k$ the category of finite dimensional $k$-vector spaces and $k$-linear maps between them. While $\mathsf{FVEC}_k$ is a compact closed category, it is not canonically self-dual compact closed, since to give a self-dual structure corresponds to providing a basis. So let $\mathsf{FVEC}^{\mathcal{B}}_k$ be the category whose objects are pairs $(V, B_V)$ consisting of a finite dimensional $k$-vector space and a basis $B_V$ of $V$, and whose maps are arbitrary $k$-linear maps between the underlying vector spaces. $\mathsf{FVEC}^{\mathcal{B}}_k$ is a self-dual compact closed category, where the tensor product is defined as: 
\[ (V, B_V) \otimes (W, B_W) = (V \otimes W, B_V \otimes B_W = \lbrace v \otimes w \vert~ \forall. v \in B_V, w \in B_W \rbrace) \]
the monoidal unit is $(k, \lbrace 1 \rbrace)$, and where the cup $\cup_{(V,B_V)}: (V, B_V) \otimes (V, B_V) \to (k, \lbrace 1 \rbrace)$ is defined as on basis elements $v, w \in B_V$ as follows: 
\begin{align*}
\cup_{(V,B_V)}(v,w) = \begin{cases} 1 & \text{if } v = w \\
0 & \text{if } v \neq w
 \end{cases}
\end{align*}
and the cap $\cap_{(V,B_V)}: (k, \lbrace 1 \rbrace) \to (V, B_V) \otimes (V, B_V)$ is the $k$-linear map defined as: 
\begin{align*}
\cap_{(V,B_V)}(1) = \sum_{v \in B_V} v \otimes v
\end{align*}
which is well-defined since $B_V$ is a finite set. $\mathsf{FVEC}^{\mathcal{B}}_k$ is also an additive symmetric monoidal category where the additive structure is induced by the direct sum of vector spaces (which is the categorical biproduct): 
\[ (V, B_V) \oplus (W, B_W) = (V \oplus W, B_V \oplus B_W = \lbrace v \oplus w \vert~ \forall. v \in B_V, w \in B_W \rbrace )  \]
and where $(0, \emptyset)$ is the zero object. Unfortunately, however, as explained in \cite[]{lemayfhilb}, $\mathsf{FVEC}^{\mathcal{B}}_k$ does not usually have a (non-trivial) differential category structure. This problem is solved when we consider $k = \mathbb{Z}_2$, as was done in \cite[]{hyland2003glueing,lemayfhilb}. $\mathsf{FVEC}^{\mathcal{B}}_{\mathbb{Z}_2}$ is then a differential category where the coalgebra modality is induced by the exterior algebra, which is defined as follows:
\[ \oc(V, B_V) = \left( \mathsf{E}(V) =  \bigoplus^{\mathsf{dim}(V)}_{n=0} \bigwedge^n V , \mathsf{E}(B_V)  = \lbrace v_1 \wedge \hdots \wedge v_n \vert~ \forall n\in \mathbb{N}, v_i \in B_V \rbrace \right)   \]
Recall that the wedge product satisfies that $v \wedge v = 0$. Usually, one also has that the wedge product is anticommutative, that is, $v \wedge w = -w \wedge v$. But in the case of $\mathbb{Z}_2$, $1=-1$ and therefore $v \wedge w = w \wedge v$, which is key to obtaining a coalgebra modality. The deriving transformation $\mathsf{d}_{(V,B_V)}: \oc(V, B_V) \otimes (V, B_V) \to \oc(V, B_V)$ is defined on basis elements as follows: 
\begin{align*}
\mathsf{d}_{(V,B_V)} \left( (v_1 \wedge \hdots \wedge v_n) \otimes v \right) = v_1 \wedge \hdots \wedge v_n \wedge v
\end{align*}
See \cite[Example 2.6.(iii)]{lemayfhilb} for more details on this differential category. By applying Theorem \ref{thm:rdc_to_dc}, $\mathsf{FVEC}^{\mathcal{B}}_{\mathbb{Z}_2}$ is also a reverse differential category where the reverse deriving transformation $\mathsf{r}_{(V,B_V)}:  \oc(V, B_V) \otimes \oc(V, B_V) \to (V, B_V)$ is defined on basis elements as follows: 
\begin{align*}
\mathsf{r}_{(V,B_V)} \left( (v_1 \wedge \hdots \wedge v_n) \otimes (w_1 \wedge \hdots \wedge w_m) \right) = \begin{cases} v & \text{if }  v_1 \wedge \hdots \wedge v_n \wedge v = w_1 \wedge \hdots \wedge w_m \\
& \text{for a $v \in B_V$}\\
& \\ 
0 & \text{otherwise }
 \end{cases}
\end{align*}
\end{example}

\begin{example}\label{ex:quantum1} \normalfont Pagani, Selinger, and Valiron's categorical model of a quantum lambda calculus, $\overline{\mathbf{CPMs}}^\oplus$ \cite[Section 4.2]{journal:selinger-valiron-fully-abstract-quantum}, is a reverse differential storage category. The objects of $\overline{\mathbf{CPMs}}^\oplus$ are families of pairs of natural numbers and subgroups of permutations, while the maps of $\overline{\mathbf{CPMs}}^\oplus$ can be interpreted as completely positive (continuous) module homomorphisms \cite[Propostion 16]{journal:selinger-valiron-fully-abstract-quantum}. $\overline{\mathbf{CPMs}}^\oplus$ is self-dual compact closed \cite[Section 4.3.3]{journal:selinger-valiron-fully-abstract-quantum} and has (in)finite biproducts \cite[Section 4.3.1]{journal:selinger-valiron-fully-abstract-quantum}, and so is an additive symmetric monoidal category as well. Furthermore, $\overline{\mathbf{CPMs}}^\oplus$ is a Lafont category, that is, $\overline{\mathbf{CPMs}}^\oplus$ has a coalgebra modality $\oc$ which is given by cofree cocommutative comonoids, and these coalgebra modalities are called free exponential modalities \cite[Section 4.3.4]{journal:selinger-valiron-fully-abstract-quantum}. By \cite[Theorem 21]{lemay:LIPIcs.CALCO.2021.19}, in the presence of biproducts, any free exponential modality has a (canonical) deriving transformation, and so any Lafont category with biproducts is a differential (storage) category. Therefore, $\overline{\mathbf{CPMs}}^\oplus$ is a differential (storage) category. Since $\overline{\mathbf{CPMs}}^\oplus$ is also self-dual compact closed, by applying Theorem \ref{thm:rdc_to_dc}, we conclude that $\overline{\mathbf{CPMs}}^\oplus$ is also a reverse differential category. In future work, it would be interesting to study in more detail the consequence of reverse differential structure in this model of quantum lambda calculus. 
\end{example}

\begin{example}\label{ex:quantum2} \normalfont There is another interesting relationship between reverse differential categories and categorical quantum mechanics.  Every reverse differential storage category whose coalgebra modality is a free exponential modality is a model of Vicary's categorical quantum harmonic oscillator \cite[Defintion 3.1]{vicary2008categorical}. We note, however, that the converse is not necessarily true since the required base category for a categorical quantum harmonic oscillator need only be a $\dagger$-symmetric monoidal category instead of a ($\dagger$-)compact closed category. That said, as discussed in \cite[Section 6]{vicary2008categorical}, in future work it would be interesting to revisit Vicary's categorical quantum harmonic oscillators from the point of view of (reverse) differential categories. 
\end{example}

\subsection{From MRDCs to CRDCs}\label{sec:cokleisliRDC}

In this section we prove that our definition satisfies requirement 2 of an MRDC, that is, that the coKleisli category of an MRDC is a CRDC. 

First, note that we already have part of what we need.  By Theorem \ref{thm:characterization_of_crdc}, to give a CRDC is equivalent to giving a CDC and a contextual linear dagger.  Moreover, by Theorem \ref{thm:rdc_to_dc}, any MRDC is a differential category, and by Theorem \ref{coKleisliCDC}, if $\mathbb{X}$ is a differential category then its coKleisli category has the structure of a CDC.  Thus, putting this together, for any MRDC, its coKleisli category is a CDC. 

Thus, all that remains to show is that the coKleisli category has a contextual linear dagger, and for this, we need its linear fibration, $\mathcal{L}[\mathbb{X}_{\oc}]$ to be a dagger fibration.  However, Theorem \ref{thm:fibration_equivalence} showed that there is an isomorphism of fibrations $\mathcal{L}[\mathbb{X}_{\oc}] \cong \mathcal{L}_{\oc}[\mathbb{X}]$. Thus it will suffice to give a dagger fibration structure on $\mathcal{L}_{\oc}[\mathbb{X}]$:  

\begin{lemma}\label{lemma:compactclosed_to_dagger} If $\mathbb{X}$ is a self-dual compact closed category, then the fibration $\mathcal{L}_{\oc}[\mathbb{X}]$ has dagger fibration structure, where for a map $f: \oc X \otimes A \to B$, its dagger $f^{\dagger[X]}: \oc X \otimes B \to A$ is defined as the following composite: 
\begin{align*}
   \begin{array}[c]{c}
f^{\dagger[X]} := \xymatrixcolsep{5pc}\xymatrix{\oc X \otimes B \ar[r]^-{1_{\oc X} \otimes \cap_A \otimes 1_B} & \oc X \otimes A \otimes A \otimes B  \ar[r]^-{f \otimes \sigma_{A,B}} & B \otimes B \otimes A \ar[r]^-{\cup_B \otimes 1_A} & A   }
   \end{array}    \end{align*}
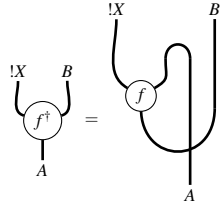
   \begin{align*}
   \begin{array}[c]{c}
\begin{tikzpicture}
	\begin{pgfonlayer}{nodelayer}
		\node [style=none] (0) at (1.5, 3.25) {};
		\node [style=none] (1) at (1, 2.25) {};
		\node [style=component] (2) at (1, 2.5) {$f$};
		\node [style=object] (3) at (0.5, 4.25) {$\oc X$};
		\node [style=none] (4) at (2, 3.25) {};
		\node [style=object] (5) at (2, 0.5) {$A$};
		\node [style=object] (6) at (2.5, 4.25) {$B$};
		\node [style=none] (7) at (2.5, 2.25) {};
		\node [style=object] (8) at (-1.5, 3) {$\oc X$};
		\node [style=object] (9) at (-0.5, 3) {$B$};
		\node [style=component] (10) at (-1, 2) {$f^\dagger$};
		\node [style=object] (11) at (-1, 1) {$A$};
		\node [style=port] (12) at (0, 2) {$=$};
	\end{pgfonlayer}
	\begin{pgfonlayer}{edgelayer}
		\draw [style=wire, in=-90, out=30, looseness=1.50] (2) to (0.center);
		\draw [style=wire] (1.center) to (2);
		\draw [style=wire, in=-90, out=150] (2) to (3);
		\draw [style=wire, bend left=90, looseness=2.00] (0.center) to (4.center);
		\draw [style=wire] (4.center) to (5);
		\draw [style=wire, bend right=90, looseness=2.00] (1.center) to (7.center);
		\draw [style=wire] (6) to (7.center);
		\draw [style=wire] (10) to (11);
		\draw [style=wire, in=165, out=-90] (8) to (10);
		\draw [style=wire, in=-90, out=15] (10) to (9);
	\end{pgfonlayer}
\end{tikzpicture}
   \end{array} 
\end{align*}
 Furthermore, for any map $f: A \to B$ in $\mathbb{X}$, the dagger of $e_X \otimes f: \oc X \otimes A \to B$ is $(e_X \otimes f)^{\dagger[X]} = e_X \otimes f^\ast: \oc X \otimes B \to A$, where $f^\ast: B \to A$ is defined as in Lemma \ref{sliding}. 
\end{lemma}
\begin{proof} Per Example \ref{ex:!daggerfibration}, it suffices to prove that the dagger operation satisfies contravariant functoriality, involution, and change of base. We begin by showing that the dagger is contravariant on composition: 
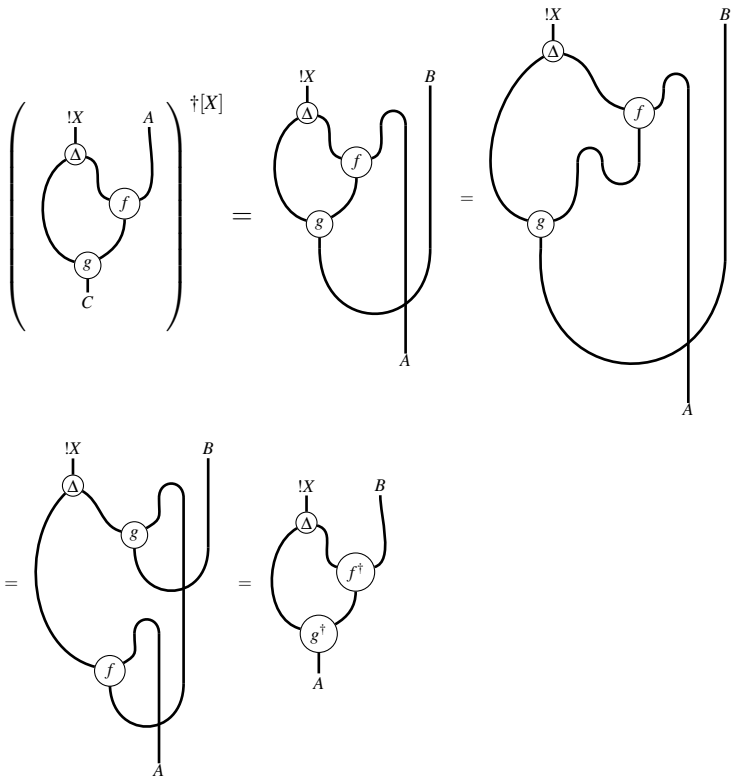
\begin{align*}
& \left(  \begin{array}[c]{c}
\begin{tikzpicture}
	\begin{pgfonlayer}{nodelayer}
		\node [style=duplicate] (1) at (9.5, 1) {$\Delta$};
		\node [style=object] (2) at (9.5, 1.75) {$\oc X$};
		\node [style=component] (3) at (10.5, 0) {$f$};
		\node [style=object] (4) at (11, 1.75) {$A$};
		\node [style=component] (5) at (9.75, -1.25) {$g$};
		\node [style=object] (6) at (9.75, -2) {$C$};
	\end{pgfonlayer}
	\begin{pgfonlayer}{edgelayer}
		\draw [style=wire] (2) to (1);
		\draw [style=wire, in=-90, out=15, looseness=0.75] (3) to (4);
		\draw [style=wire] (5) to (6);
		\draw [style=wire, in=30, out=-90] (3) to (5);
		\draw [style=wire, in=165, out=-150] (1) to (5);
		\draw [style=wire, in=165, out=-15, looseness=1.25] (1) to (3);
	\end{pgfonlayer}
\end{tikzpicture}
   \end{array} \right)^{\dagger[X]} =    \begin{array}[c]{c}
\begin{tikzpicture}
	\begin{pgfonlayer}{nodelayer}
		\node [style=duplicate] (0) at (0.5, -3.25) {$\Delta$};
		\node [style=object] (1) at (0.5, -2.5) {$\oc X$};
		\node [style=component] (2) at (1.5, -4.25) {$f$};
		\node [style=component] (3) at (0.75, -5.5) {$g$};
		\node [style=none] (4) at (2, -3.5) {};
		\node [style=none] (5) at (2.5, -3.5) {};
		\node [style=port] (6) at (2.5, -8.25) {$A$};
		\node [style=none] (7) at (0.75, -6) {};
		\node [style=object] (8) at (3, -2.5) {$B$};
		\node [style=none] (9) at (3, -6) {};
		\node [style=duplicate] (24) at (5.5, -2) {$\Delta$};
		\node [style=object] (25) at (5.5, -1.25) {$\oc X$};
		\node [style=component] (26) at (7.25, -3.25) {$f$};
		\node [style=component] (27) at (5.25, -5.5) {$g$};
		\node [style=none] (28) at (7.75, -2.75) {};
		\node [style=none] (29) at (8.25, -2.75) {};
		\node [style=port] (30) at (8.25, -9.25) {$A$};
		\node [style=none] (31) at (5.25, -6) {};
		\node [style=object] (32) at (9, -1.25) {$B$};
		\node [style=none] (33) at (9, -6.25) {};
		\node [style=none] (34) at (6.5, -4.25) {};
		\node [style=none] (35) at (7.25, -4.25) {};
		\node [style=none] (36) at (6, -4.25) {};
		\node [style=none] (37) at (6.5, -4.25) {};
		\node [style=object] (38) at (3.75, -5) {$=$};
	\end{pgfonlayer}
	\begin{pgfonlayer}{edgelayer}
		\draw [style=wire] (1) to (0);
		\draw [style=wire, in=30, out=-90] (2) to (3);
		\draw [style=wire, in=165, out=-150] (0) to (3);
		\draw [style=wire, in=165, out=-15, looseness=1.25] (0) to (2);
		\draw [style=wire, bend left=90, looseness=2.00] (4.center) to (5.center);
		\draw [style=wire] (5.center) to (6);
		\draw [style=wire, in=15, out=-90] (4.center) to (2);
		\draw [style=wire, bend right=90, looseness=2.00] (7.center) to (9.center);
		\draw [style=wire] (8) to (9.center);
		\draw [style=wire] (3) to (7.center);
		\draw [style=wire] (25) to (24);
		\draw [style=wire, in=165, out=-150] (24) to (27);
		\draw [style=wire, in=165, out=-15, looseness=1.25] (24) to (26);
		\draw [style=wire, bend left=90, looseness=2.00] (28.center) to (29.center);
		\draw [style=wire] (29.center) to (30);
		\draw [style=wire, in=15, out=-90] (28.center) to (26);
		\draw [style=wire, bend right=90, looseness=2.00] (31.center) to (33.center);
		\draw [style=wire] (32) to (33.center);
		\draw [style=wire] (27) to (31.center);
		\draw [style=wire, bend right=90, looseness=2.00] (34.center) to (36.center);
		\draw [style=wire, bend left=90, looseness=2.00] (35.center) to (37.center);
		\draw [style=wire] (26) to (35.center);
		\draw [style=wire, in=15, out=-90] (36.center) to (27);
	\end{pgfonlayer}
\end{tikzpicture}
   \end{array} \\
&    \begin{array}[c]{c}
\begin{tikzpicture}
	\begin{pgfonlayer}{nodelayer}
		\node [style=none] (10) at (12.25, -6) {};
		\node [style=none] (11) at (11.75, -7) {};
		\node [style=component] (12) at (11.75, -6.75) {$f$};
		\node [style=none] (13) at (12.75, -6) {};
		\node [style=port] (14) at (12.75, -8.75) {$A$};
		\node [style=none] (15) at (13.25, -7) {};
		\node [style=none] (16) at (12.75, -3.25) {};
		\node [style=none] (17) at (12.25, -4.25) {};
		\node [style=component] (18) at (12.25, -4) {$g$};
		\node [style=none] (19) at (13.25, -3.25) {};
		\node [style=object] (20) at (13.75, -2.25) {$B$};
		\node [style=none] (21) at (13.75, -4.25) {};
		\node [style=duplicate] (22) at (11, -3) {$\Delta$};
		\node [style=object] (23) at (11, -2.25) {$\oc X$};
		\node [style=object] (39) at (9.75, -5) {$=$};
		\node [style=duplicate] (40) at (15.75, -3.75) {$\Delta$};
		\node [style=object] (41) at (15.75, -3) {$\oc X$};
		\node [style=component] (42) at (16.75, -4.75) {$f^\dagger$};
		\node [style=object] (43) at (17.25, -3) {$B$};
		\node [style=component] (44) at (16, -6) {$g^\dagger$};
		\node [style=object] (45) at (16, -7) {$A$};
		\node [style=object] (46) at (14.5, -5) {$=$};
	\end{pgfonlayer}
	\begin{pgfonlayer}{edgelayer}
		\draw [style=wire, in=-90, out=30, looseness=1.50] (12) to (10.center);
		\draw [style=wire] (11.center) to (12);
		\draw [style=wire, bend left=90, looseness=2.00] (10.center) to (13.center);
		\draw [style=wire] (13.center) to (14);
		\draw [style=wire, bend right=90, looseness=2.00] (11.center) to (15.center);
		\draw [style=wire, in=-90, out=30, looseness=1.50] (18) to (16.center);
		\draw [style=wire] (17.center) to (18);
		\draw [style=wire, bend left=90, looseness=2.00] (16.center) to (19.center);
		\draw [style=wire, bend right=90, looseness=2.00] (17.center) to (21.center);
		\draw [style=wire] (20) to (21.center);
		\draw [style=wire] (19.center) to (15.center);
		\draw [style=wire] (23) to (22);
		\draw [style=wire, in=165, out=-30] (22) to (18);
		\draw [style=wire, in=165, out=-135] (22) to (12);
		\draw [style=wire] (41) to (40);
		\draw [style=wire, in=-90, out=15, looseness=0.75] (42) to (43);
		\draw [style=wire] (44) to (45);
		\draw [style=wire, in=30, out=-90] (42) to (44);
		\draw [style=wire, in=165, out=-150] (40) to (44);
		\draw [style=wire, in=165, out=-15, looseness=1.25] (40) to (42);
	\end{pgfonlayer}
\end{tikzpicture}
   \end{array} 
\end{align*}
Next we show that the dagger preserves identities: 
\begin{align*}
\left( \begin{array}[c]{c}
\begin{tikzpicture}
	\begin{pgfonlayer}{nodelayer}
		\node [style=object] (17) at (9.75, 1.75) {$\oc X$};
		\node [style=object] (19) at (10.5, 1.75) {$A$};
		\node [style=object] (21) at (10.5, -0.25) {$A$};
		\node [style=component] (22) at (9.75, 0.5) {$e$};
	\end{pgfonlayer}
	\begin{pgfonlayer}{edgelayer}
		\draw [style=wire] (19) to (21);
		\draw [style=wire] (17) to (22);
	\end{pgfonlayer}
\end{tikzpicture}
   \end{array}  \right)^{\dagger[X]} = \begin{array}[c]{c}
\begin{tikzpicture}
	\begin{pgfonlayer}{nodelayer}
		\node [style=object] (82) at (2.25, 9.75) {$\oc X$};
		\node [style=object] (83) at (4, 9.75) {$A$};
		\node [style=component] (85) at (2.25, 8.5) {$e$};
		\node [style=object] (86) at (9.5, 9.75) {$\oc X$};
		\node [style=object] (87) at (10.25, 9.75) {$A$};
		\node [style=object] (88) at (10.25, 7.75) {$A$};
		\node [style=component] (89) at (9.5, 8.5) {$e$};
		\node [style=none] (90) at (3, 9) {};
		\node [style=none] (91) at (3, 8.25) {};
		\node [style=none] (93) at (3.5, 9) {};
		\node [style=port] (94) at (3.5, 6.25) {$A$};
		\node [style=none] (95) at (4, 8.25) {};
		\node [style=object] (96) at (4.75, 8) {$=$};
		\node [style=object] (97) at (5.5, 9.75) {$\oc X$};
		\node [style=component] (98) at (5.5, 8.5) {$e$};
		\node [style=none] (99) at (7, 8.5) {};
		\node [style=none] (100) at (7.75, 8.5) {};
		\node [style=object] (101) at (7.75, 9.75) {$A$};
		\node [style=none] (102) at (6.25, 8.5) {};
		\node [style=object] (103) at (6.25, 7) {$A$};
		\node [style=none] (104) at (7, 8.5) {};
		\node [style=object] (105) at (8.5, 8) {$=$};
	\end{pgfonlayer}
	\begin{pgfonlayer}{edgelayer}
		\draw [style=wire] (82) to (85);
		\draw [style=wire] (87) to (88);
		\draw [style=wire] (86) to (89);
		\draw [style=wire, bend left=90, looseness=2.00] (90.center) to (93.center);
		\draw [style=wire] (93.center) to (94);
		\draw [style=wire, bend right=90, looseness=2.00] (91.center) to (95.center);
		\draw [style=wire] (90.center) to (91.center);
		\draw [style=wire] (83) to (95.center);
		\draw [style=wire] (97) to (98);
		\draw [style=wire, bend right=90, looseness=2.00] (99.center) to (102.center);
		\draw [style=wire] (102.center) to (103);
		\draw [style=wire, bend left=90, looseness=2.00] (100.center) to (104.center);
		\draw [style=wire] (101) to (100.center);
	\end{pgfonlayer}
\end{tikzpicture}
   \end{array}
\end{align*}
Next we show that the dagger is involutive using the snake equations: 
\begin{align*}
   \begin{array}[c]{c}
\begin{tikzpicture}
	\begin{pgfonlayer}{nodelayer}
		\node [style=none] (0) at (1.5, 3.25) {};
		\node [style=none] (1) at (1, 2.25) {};
		\node [style=component] (2) at (1, 2.5) {$f^\dagger$};
		\node [style=object] (3) at (0.5, 4.25) {$\oc X$};
		\node [style=none] (4) at (2, 3.25) {};
		\node [style=port] (5) at (2, 0.5) {$B$};
		\node [style=object] (6) at (2.5, 4.25) {$A$};
		\node [style=none] (7) at (2.5, 2.25) {};
		\node [style=object] (8) at (-1.5, 3.25) {$\oc X$};
		\node [style=object] (9) at (-0.5, 3.25) {$A$};
		\node [style=component] (10) at (-1, 2) {${f^\dagger}^\dagger$};
		\node [style=object] (11) at (-1, 0.75) {$B$};
		\node [style=port] (12) at (0, 2) {$=$};
		\node [style=none] (106) at (4.75, 4) {};
		\node [style=none] (107) at (4.25, 3) {};
		\node [style=component] (108) at (4.25, 3.25) {$f$};
		\node [style=object] (109) at (3.75, 5) {$\oc X$};
		\node [style=none] (110) at (5.25, 4) {};
		\node [style=none] (113) at (5.75, 3) {};
		\node [style=none] (114) at (5.75, 3.5) {};
		\node [style=none] (115) at (6.25, 3.5) {};
		\node [style=port] (116) at (6.25, -0.5) {$B$};
		\node [style=none] (117) at (5.25, 1.75) {};
		\node [style=object] (118) at (6.75, 5) {$A$};
		\node [style=none] (119) at (6.75, 1.75) {};
		\node [style=port] (120) at (3.25, 2) {$=$};
		\node [style=object] (121) at (8, 3.25) {$\oc X$};
		\node [style=object] (122) at (9, 3.25) {$A$};
		\node [style=component] (123) at (8.5, 2) {$f$};
		\node [style=object] (124) at (8.5, 0.75) {$B$};
	\end{pgfonlayer}
	\begin{pgfonlayer}{edgelayer}
		\draw [style=wire, in=-90, out=30, looseness=1.50] (2) to (0.center);
		\draw [style=wire] (1.center) to (2);
		\draw [style=wire, in=-90, out=150] (2) to (3);
		\draw [style=wire, bend left=90, looseness=2.00] (0.center) to (4.center);
		\draw [style=wire] (4.center) to (5);
		\draw [style=wire, bend right=90, looseness=2.00] (1.center) to (7.center);
		\draw [style=wire] (6) to (7.center);
		\draw [style=wire] (10) to (11);
		\draw [style=wire, in=165, out=-90] (8) to (10);
		\draw [style=wire, in=-90, out=15] (10) to (9);
		\draw [style=wire, in=-90, out=30, looseness=1.50] (108) to (106.center);
		\draw [style=wire] (107.center) to (108);
		\draw [style=wire, in=-90, out=150] (108) to (109);
		\draw [style=wire, bend left=90, looseness=2.00] (106.center) to (110.center);
		\draw [style=wire, bend right=90, looseness=2.00] (107.center) to (113.center);
		\draw [style=wire, bend left=90, looseness=2.00] (114.center) to (115.center);
		\draw [style=wire] (115.center) to (116);
		\draw [style=wire] (114.center) to (113.center);
		\draw [style=wire, bend right=90, looseness=2.00] (117.center) to (119.center);
		\draw [style=wire] (118) to (119.center);
		\draw [style=wire] (110.center) to (117.center);
		\draw [style=wire] (123) to (124);
		\draw [style=wire, in=165, out=-90] (121) to (123);
		\draw [style=wire, in=-90, out=15] (123) to (122);
	\end{pgfonlayer}
\end{tikzpicture}
   \end{array}
\end{align*}
Lastly we show that the substitution functors preserves the dagger, which is automatic by definition: 
\begin{align*}
\left( \begin{array}[c]{c}
\begin{tikzpicture}
	\begin{pgfonlayer}{nodelayer}
		\node [style=component] (0) at (9, 0.25) {$\delta$};
		\node [style=object] (1) at (10.5, 1.25) {$A$};
		\node [style=component] (2) at (9.75, -1.75) {$f$};
		\node [style=object] (3) at (9.75, -2.75) {$B$};
		\node [style=function2] (4) at (9, -0.75) {$h$};
		\node [style=object] (5) at (9, 1.25) {$\oc X$};
	\end{pgfonlayer}
	\begin{pgfonlayer}{edgelayer}
		\draw [style=wire] (2) to (3);
		\draw [style=wire, bend right, looseness=1.25] (4) to (2);
		\draw [style=wire] (0) to (4);
		\draw [style=wire, in=30, out=-90, looseness=0.75] (1) to (2);
		\draw [style=wire] (5) to (0);
	\end{pgfonlayer}
\end{tikzpicture}
   \end{array} \right)^{\dagger[X]} =    \begin{array}[c]{c}
\begin{tikzpicture}
	\begin{pgfonlayer}{nodelayer}
		\node [style=none] (131) at (-1.25, -4.25) {};
		\node [style=none] (132) at (-1.75, -5.25) {};
		\node [style=component] (133) at (-1.75, -5) {$f$};
		\node [style=none] (135) at (-0.75, -4.25) {};
		\node [style=port] (136) at (-0.75, -7) {$A$};
		\node [style=object] (137) at (-0.25, -2.25) {$B$};
		\node [style=none] (138) at (-0.25, -5.25) {};
		\node [style=component] (139) at (-2.5, -3.25) {$\delta$};
		\node [style=function2] (140) at (-2.5, -4.25) {$h$};
		\node [style=object] (141) at (-2.5, -2.25) {$\oc X$};
	\end{pgfonlayer}
	\begin{pgfonlayer}{edgelayer}
		\draw [style=wire, in=-90, out=30, looseness=1.50] (133) to (131.center);
		\draw [style=wire] (132.center) to (133);
		\draw [style=wire, bend left=90, looseness=2.00] (131.center) to (135.center);
		\draw [style=wire] (135.center) to (136);
		\draw [style=wire, bend right=90, looseness=2.00] (132.center) to (138.center);
		\draw [style=wire] (137) to (138.center);
		\draw [style=wire] (139) to (140);
		\draw [style=wire] (141) to (139);
		\draw [style=wire, in=165, out=-90] (140) to (133);
	\end{pgfonlayer}
\end{tikzpicture}
   \end{array}  =  \begin{array}[c]{c}
\begin{tikzpicture}
	\begin{pgfonlayer}{nodelayer}
		\node [style=component] (0) at (9, 0.25) {$\delta$};
		\node [style=object] (1) at (10.5, 1.25) {$B$};
		\node [style=component] (2) at (9.75, -1.75) {$f^\dagger$};
		\node [style=object] (3) at (9.75, -2.75) {$A$};
		\node [style=function2] (4) at (9, -0.75) {$h$};
		\node [style=object] (5) at (9, 1.25) {$\oc X$};
	\end{pgfonlayer}
	\begin{pgfonlayer}{edgelayer}
		\draw [style=wire] (2) to (3);
		\draw [style=wire, bend right, looseness=1.25] (4) to (2);
		\draw [style=wire] (0) to (4);
		\draw [style=wire, in=30, out=-90, looseness=0.75] (1) to (2);
		\draw [style=wire] (5) to (0);
	\end{pgfonlayer}
\end{tikzpicture}
   \end{array} 
\end{align*} 
So we conclude that $\mathcal{L}_{\oc}[\mathbb{X}]$ has dagger fibration. Next, for any map $f: A \to B$, using the snake equations and the sliding equations, we compute: 
\begin{align*}
\left(    \begin{array}[c]{c}
\begin{tikzpicture}
	\begin{pgfonlayer}{nodelayer}
		\node [style=object] (8) at (1, 3) {$A$};
		\node [style=component] (10) at (1, 2) {$f$};
		\node [style=object] (11) at (1, 1) {$B$};
		\node [style=object] (24) at (0, 3) {$\oc X$};
		\node [style=component] (25) at (0, 2) {$e$};
	\end{pgfonlayer}
	\begin{pgfonlayer}{edgelayer}
		\draw [style=wire] (10) to (11);
		\draw [style=wire] (8) to (10);
		\draw [style=wire] (24) to (25);
	\end{pgfonlayer}
\end{tikzpicture}
   \end{array} \right)^{\dagger[X]} =    \begin{array}[c]{c}
\begin{tikzpicture}
	\begin{pgfonlayer}{nodelayer}
		\node [style=none] (0) at (-7.5, 3.25) {};
		\node [style=none] (1) at (-7.5, 2.25) {};
		\node [style=none] (4) at (-7, 3.25) {};
		\node [style=none] (7) at (-6.75, 2.25) {};
		\node [style=port] (12) at (-6, 2.75) {$=$};
		\node [style=component] (21) at (-7.5, 2.75) {$f$};
		\node [style=object] (24) at (-8.25, 4.5) {$\oc X$};
		\node [style=component] (25) at (-8.25, 2.75) {$e$};
		\node [style=object] (26) at (2.75, 3.75) {$B$};
		\node [style=component] (27) at (2.75, 2.75) {$f^\ast$};
		\node [style=object] (28) at (2.75, 1.75) {$A$};
		\node [style=object] (29) at (1.75, 3.75) {$\oc X$};
		\node [style=component] (30) at (1.75, 2.75) {$e$};
		\node [style=object] (31) at (-6.75, 4.5) {$B$};
		\node [style=object] (32) at (-7, 0.75) {$A$};
		\node [style=none] (33) at (-4.25, 3.25) {};
		\node [style=none] (34) at (-4.25, 2.25) {};
		\node [style=none] (35) at (-3.75, 3.25) {};
		\node [style=none] (36) at (-3.25, 2.25) {};
		\node [style=component] (37) at (-3.25, 2.75) {$f^\ast$};
		\node [style=object] (38) at (-5, 4.5) {$\oc X$};
		\node [style=component] (39) at (-5, 2.75) {$e$};
		\node [style=object] (40) at (-3.25, 4.5) {$B$};
		\node [style=object] (41) at (-3.75, 0.75) {$A$};
		\node [style=port] (42) at (-2.5, 2.75) {$=$};
		\node [style=none] (43) at (-0.25, 3.25) {};
		\node [style=none] (44) at (-0.25, 2.25) {};
		\node [style=none] (45) at (-0.75, 3.25) {};
		\node [style=none] (46) at (0.25, 2.25) {};
		\node [style=component] (47) at (0.25, 2.75) {$f^\ast$};
		\node [style=object] (48) at (-1.5, 4.5) {$\oc X$};
		\node [style=component] (49) at (-1.5, 2.75) {$e$};
		\node [style=object] (50) at (0.25, 4.5) {$B$};
		\node [style=object] (51) at (-0.75, 0.75) {$A$};
		\node [style=port] (52) at (1, 2.75) {$=$};
	\end{pgfonlayer}
	\begin{pgfonlayer}{edgelayer}
		\draw [style=wire, bend left=90, looseness=2.00] (0.center) to (4.center);
		\draw [style=wire, bend right=90, looseness=2.00] (1.center) to (7.center);
		\draw [style=wire] (21) to (1.center);
		\draw [style=wire] (24) to (25);
		\draw [style=wire] (0.center) to (21);
		\draw [style=wire] (27) to (28);
		\draw [style=wire] (26) to (27);
		\draw [style=wire] (29) to (30);
		\draw [style=wire] (31) to (7.center);
		\draw [style=wire] (4.center) to (32);
		\draw [style=wire, bend left=90, looseness=2.00] (33.center) to (35.center);
		\draw [style=wire, bend right=90, looseness=2.00] (34.center) to (36.center);
		\draw [style=wire] (38) to (39);
		\draw [style=wire] (35.center) to (41);
		\draw [style=wire] (40) to (37);
		\draw [style=wire] (37) to (36.center);
		\draw [style=wire] (33.center) to (34.center);
		\draw [style=wire, bend right=90, looseness=2.00] (43.center) to (45.center);
		\draw [style=wire, bend right=90, looseness=2.00] (44.center) to (46.center);
		\draw [style=wire] (48) to (49);
		\draw [style=wire] (45.center) to (51);
		\draw [style=wire] (50) to (47);
		\draw [style=wire] (47) to (46.center);
		\draw [style=wire] (43.center) to (44.center);
	\end{pgfonlayer}
\end{tikzpicture}
   \end{array}
\end{align*}
So the desired equality holds. 
\end{proof}

The following then follows from the remarks above:

\begin{corollary}\label{coKlielsicondag} Let $\mathbb{X}$ be a differential category which is self-dual compact closed. Then the coKleisli category $\mathbb{X}_\oc$ is a Cartesian differential category with a contextual linear dagger, where for a map $\llbracket f \rrbracket: \oc(X \times A) \to B$ which is linear in context $X$, its dagger $\llbracket f^{\dagger[X]} \rrbracket: \oc (X \times B) \to A$ is defined as follows: 
\begin{align*}
   \begin{array}[c]{c}
\llbracket f^{\dagger[X]} \rrbracket = \mathsf{E}_X\left( \mathsf{E}^{-1}_X(\llbracket f \rrbracket)^{\dagger[X]} \right)    \end{array}&&    \begin{array}[c]{c}
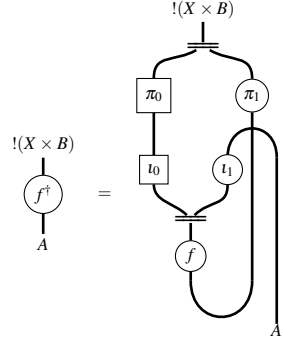
\begin{tikzpicture}
	\begin{pgfonlayer}{nodelayer}
		\node [style=none] (0) at (4.75, 2.75) {};
		\node [style=none] (1) at (4, 0.25) {};
		\node [style=none] (4) at (5.75, 2.75) {};
		\node [style=port] (5) at (5.75, -0.75) {$A$};
		\node [style=none] (7) at (5.25, 0.25) {};
		\node [style=object] (8) at (1, 3) {$\oc(X \times B)$};
		\node [style=component] (10) at (1, 2) {$f^\dagger$};
		\node [style=object] (11) at (1, 1) {$A$};
		\node [style=port] (12) at (2.25, 2) {$=$};
		\node [style=object] (15) at (4.25, 5.75) {$\oc (X \times B)$};
		\node [style=differential] (16) at (4.25, 5) {{\bf =\!=\!=\!=}};
		\node [style=component] (17) at (5.25, 4) {$\pi_1$};
		\node [style=function2] (18) at (3.25, 4) {$\pi_0$};
		\node [style=differential] (19) at (4, 1.5) {{\bf =\!=\!=\!=}};
		\node [style=component] (21) at (4, 0.75) {$f$};
		\node [style=component] (22) at (4.75, 2.5) {$\iota_1$};
		\node [style=function2] (23) at (3.25, 2.5) {$\iota_0$};
	\end{pgfonlayer}
	\begin{pgfonlayer}{edgelayer}
		\draw [style=wire, bend left=90, looseness=2.00] (0.center) to (4.center);
		\draw [style=wire] (4.center) to (5);
		\draw [style=wire, bend right=90, looseness=2.00] (1.center) to (7.center);
		\draw [style=wire] (10) to (11);
		\draw [style=wire] (8) to (10);
		\draw [style=wire] (15) to (16);
		\draw [style=wire, in=90, out=-150, looseness=1.25] (16) to (18);
		\draw [style=wire, in=90, out=-30, looseness=1.25] (16) to (17);
		\draw [style=wire] (19) to (21);
		\draw [style=wire, in=135, out=-90] (23) to (19);
		\draw [style=wire, in=-90, out=45, looseness=1.25] (19) to (22);
		\draw [style=wire] (21) to (1.center);
		\draw [style=wire] (0.center) to (22);
		\draw [style=wire] (18) to (23);
		\draw [style=wire] (17) to (7.center);
	\end{pgfonlayer}
\end{tikzpicture}
   \end{array}
\end{align*}
where $\mathsf{E}_X$ and $\mathsf{E}^{-1}_X$ are defined on in Corollary \ref{cor:fibres-equiv}, and the $\dagger[X]$ on the right-hand side is defined as in Lemma \ref{lemma:compactclosed_to_dagger}. 
\end{corollary}
\begin{proof} It follows from the equivalence of Theorem \ref{thm:fibration_equivalence}, that by giving a dagger fibration on $\mathcal{L}_\oc[\mathbb{X}]$ (Lemma \ref{lemma:compactclosed_to_dagger}), we obtain a dagger fibration on $\mathcal{L}[\mathbb{X}_\oc]$ defined as follows: 
  \[  \xymatrixcolsep{5pc}\xymatrix{ \mathcal{L}[\mathbb{X}_\oc]  \ar[r]^-{\mathsf{E}^{-1}} & \mathcal{L}_\oc[\mathbb{X}] \ar[r]^-{(-)^\dagger} & \mathcal{L}_\oc[\mathbb{X}]^\ast \ar[r]^-{\mathsf{E}^\ast} &  \mathcal{L}[\mathbb{X}_\oc]^\ast 
  } \]
Zooming in on the fibres, we have that the dagger on maps which are linear in context is defined as $\llbracket f^{\dagger[X]} \rrbracket = \mathsf{E}_X\left( \mathsf{E}^{-1}_X(\llbracket f \rrbracket)^{\dagger[X]} \right)$. It remains to show that each fibre also has dagger-biproducts. First note that in the fibres, the projections maps and injections maps are respectively: 
\begin{align*}
\llbracket 1_X \times \pi_i \rrbracket = \varepsilon_{X \times (A_0 \times A_1)}; (1_X \times \pi_X): \oc\left( X \times (A_0 \times A_1) \right) \to A_i \\
 \llbracket 1_X \times \iota_i \rrbracket = \varepsilon_{X \times A_i}; (1_X \times \iota_i): \oc (X \times A_i) \to A_0 \times A_1
\end{align*}
By applying $\mathsf{E}^{-1}_X$ to the projection we obtain the following:
\begin{align*}
\mathsf{E}^{-1}_X \left( \llbracket 1_X \times \pi_i \rrbracket \right) = e_X \otimes \pi_i: \oc X \otimes (A_0 \times A_1) \to A_i  \end{align*}
By Lemma \ref{lemma:compactclosed_to_dagger}, their dagger is given the dual: 
\begin{align*}
\mathsf{E}^{-1}_X \left( \llbracket 1_X \times \pi_i \rrbracket \right)^{\dagger[X]} = (e_X \otimes \pi_i)^{\dagger[X]} = e_X \otimes \pi_i^\ast: \oc X \otimes A_i \to A_0 \times A_1  \end{align*}
However by \cite[]{houston2008finite}, the dual of the projections are the injections (and vice-versa). So: 
\begin{align*}
\mathsf{E}^{-1}_X \left( \llbracket 1_X \times \pi_i \rrbracket \right)^{\dagger[X]} = e_X \otimes \pi_i^\ast = e_X \otimes \iota_i  \end{align*}
Lastly, applying $\mathsf{E}_X$ we finally obtain that: 
\begin{align*}
\mathsf{E}\left( \mathsf{E}^{-1}_X \left( \llbracket 1_X \times \pi_i \rrbracket \right)^{\dagger[X]}  \right) = \mathsf{E}\left( e_X \otimes \iota_i  \right) =  \varepsilon_{X \times A_i}; (1_X \times \iota_i) =  \llbracket 1_X \times \iota_i \rrbracket 
\end{align*}
So we conclude that $\llbracket 1_X \times \pi_i \rrbracket^{\dagger[X]} =   \llbracket 1_X \times \iota_i \rrbracket$, and that therefore each fibre $\mathcal{L}[X]$ has dagger biproducts. Thus, $\mathbb{X}_\oc$ is a Cartesian differential category with a contextual linear dagger. 
\end{proof}

We then obtain one of the main results of this paper: 

\begin{therm}\label{coKleisliCRDC} Let $\mathbb{X}$ be a reverse differential category with coalgebra modality $(\oc, \delta, \varepsilon, \Delta, e)$ and reverse deriving transformation $\mathsf{r}: \oc A \otimes \oc A \to A$, and finite (bi)products (which we denote here using the product notation). Then the coKleisli category $\mathbb{X}_\oc$ is a Cartesian reverse differential category with Cartesian left additive structure defined in Section \ref{cokleislisection} and reverse differential combinator $\mathsf{R}$ defined as follows on a coKleisli map $\llbracket f \rrbracket: \oc A \to B$: 
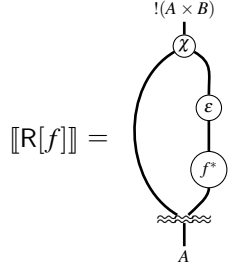
 \begin{align*}
 \llbracket \mathsf{R}[f] \rrbracket := \xymatrixcolsep{4pc}\xymatrix{\oc(A \times B) \ar[r]^-{\chi_{A \times B}} & \oc A \otimes \oc B \ar[r]^-{1_{\oc A} \otimes \varepsilon_B} & \oc A \otimes B \ar[r]^-{1_{\oc A} \otimes \llbracket f \rrbracket^\ast} & \oc A \otimes \oc A \ar[r]^-{\mathsf{r}_A} & A}
       \end{align*}
    \begin{align*}
 \begin{array}[c]{c}
\llbracket \mathsf{R}[f] \rrbracket 
   \end{array} =  \begin{array}[c]{c} 
\begin{tikzpicture}
	\begin{pgfonlayer}{nodelayer}
		\node [style=component] (0) at (7.5, 1.25) {$\varepsilon$};
		\node [style=differential] (1) at (7, -1) {{\bf \aquarius\!\aquarius\!\aquarius}};
		\node [style=object] (2) at (7, -1.75) {$A$};
		\node [style=duplicate] (3) at (7, 2.5) {$\chi$};
		\node [style=object] (4) at (7, 3.25) {$\oc (A \times B)$};
		\node [style=component] (5) at (7.5, 0) {$f^\ast$};
	\end{pgfonlayer}
	\begin{pgfonlayer}{edgelayer}
		\draw [style=wire] (4) to (3);
		\draw [style=wire, in=90, out=-30, looseness=1.25] (3) to (0);
		\draw [style=wire, in=150, out=-150] (3) to (1);
		\draw [style=wire] (1) to (2);
		\draw [style=wire] (0) to (5);
		\draw [style=wire, in=45, out=-90, looseness=0.75] (5) to (1);
	\end{pgfonlayer}
\end{tikzpicture}
   \end{array}
\end{align*}
where $\chi_{A \times B}: \oc (A \times B) \to \oc A \otimes \oc B$ is defined as in Definition \ref{Seelydef} and $(\_)^\ast$ is defined as in Lemma \ref{sliding}. Furthermore, the induced differential combinator is precisely that of Proposition \ref{coKleisliCDC}, and the induced contextual linear dagger is precisely that of Corollary \ref{coKlielsicondag}. 
\end{therm} 
\begin{proof} By Proposition \ref{coKleisliCDC}, $\mathbb{X}_\oc$ is a Cartesian differential category, and since $\mathbb{X}$ is compact closed, by Corollary \ref{coKlielsicondag}, $\mathbb{X}_\oc$ also has a contextual linear dagger. Therefore by Theorem \ref{thm:characterization_of_crdc}, $\mathbb{X}_\oc$ is a Cartesian reverse differential category where for a coKleisli map $\llbracket f \rrbracket: \oc X \to B$, its reverse derivative $\llbracket \mathsf{R}[f] \rrbracket: \oc (A \times B) \to A$ is defined as $\llbracket \mathsf{R}[f] \rrbracket = \llbracket \mathsf{D}[f]^{\dagger[A]} \rrbracket$. Expanding this out, we compute: 
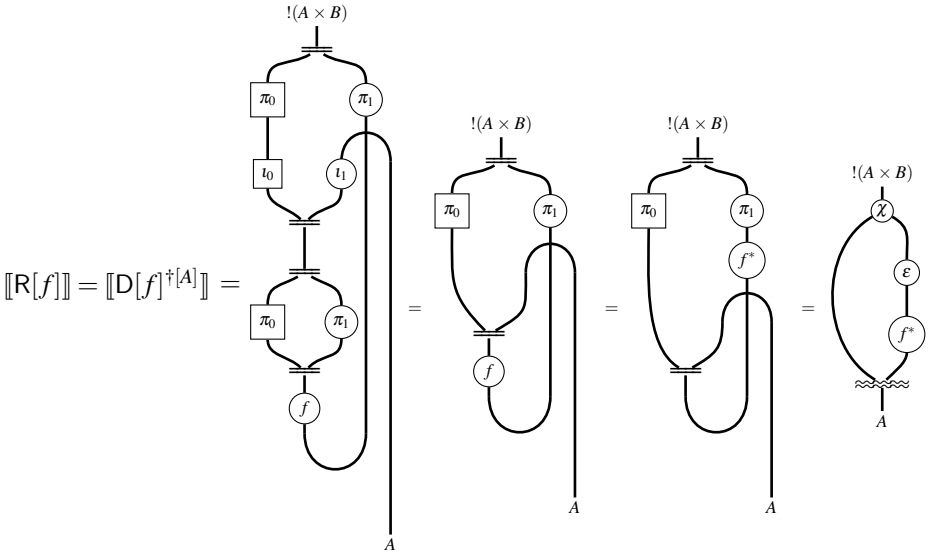
\begin{align*}
   \begin{array}[c]{c}
\llbracket \mathsf{R}[f] \rrbracket = \llbracket \mathsf{D}[f]^{\dagger[A]} \rrbracket
   \end{array} =    \begin{array}[c]{c}
\begin{tikzpicture}
	\begin{pgfonlayer}{nodelayer}
		\node [style=port] (42) at (-3, -7.5) {$=$};
		\node [style=none] (53) at (-4.5, -4.5) {};
		\node [style=none] (54) at (-5.25, -10) {};
		\node [style=none] (55) at (-3.5, -4.5) {};
		\node [style=object] (56) at (-3.5, -12.25) {$A$};
		\node [style=none] (57) at (-4, -10) {};
		\node [style=object] (62) at (-5, -1.5) {$\oc (A \times B)$};
		\node [style=differential] (63) at (-5, -2.25) {{\bf =\!=\!=\!=}};
		\node [style=component] (64) at (-4, -3.25) {$\pi_1$};
		\node [style=function2] (65) at (-6, -3.25) {$\pi_0$};
		\node [style=differential] (66) at (-5.25, -5.75) {{\bf =\!=\!=\!=}};
		\node [style=component] (68) at (-4.5, -4.75) {$\iota_1$};
		\node [style=function2] (69) at (-6, -4.75) {$\iota_0$};
		\node [style=differential] (70) at (-5.25, -8.75) {{\bf =\!=\!=\!=}};
		\node [style=component] (72) at (-5.25, -9.5) {$f$};
		\node [style=differential] (74) at (-5.25, -6.75) {{\bf =\!=\!=\!=}};
		\node [style=component] (75) at (-4.5, -7.75) {$\pi_1$};
		\node [style=function2] (76) at (-6, -7.75) {$\pi_0$};
		\node [style=none] (77) at (-0.75, -6.75) {};
		\node [style=none] (78) at (-1.5, -9.25) {};
		\node [style=none] (79) at (0.25, -6.75) {};
		\node [style=object] (80) at (0.25, -11.5) {$A$};
		\node [style=none] (81) at (-0.25, -9.25) {};
		\node [style=object] (82) at (-1.25, -3.75) {$\oc (A \times B)$};
		\node [style=differential] (83) at (-1.25, -4.5) {{\bf =\!=\!=\!=}};
		\node [style=component] (84) at (-0.25, -5.5) {$\pi_1$};
		\node [style=function2] (85) at (-2.25, -5.5) {$\pi_0$};
		\node [style=differential] (89) at (-1.5, -8) {{\bf =\!=\!=\!=}};
		\node [style=component] (90) at (-1.5, -8.75) {$f$};
		\node [style=port] (91) at (1, -7.5) {$=$};
		\node [style=none] (92) at (3.25, -7.75) {};
		\node [style=none] (93) at (2.5, -9.25) {};
		\node [style=none] (94) at (4.25, -7.75) {};
		\node [style=object] (95) at (4.25, -11.5) {$A$};
		\node [style=none] (96) at (3.75, -9.25) {};
		\node [style=object] (97) at (2.75, -3.75) {$\oc(A \times B)$};
		\node [style=differential] (98) at (2.75, -4.5) {{\bf =\!=\!=\!=}};
		\node [style=component] (99) at (3.75, -5.5) {$\pi_1$};
		\node [style=function2] (100) at (1.75, -5.5) {$\pi_0$};
		\node [style=differential] (101) at (2.5, -8.75) {{\bf =\!=\!=\!=}};
		\node [style=component] (102) at (3.75, -6.5) {$f^\ast$};
		\node [style=component] (103) at (7, -6.75) {$\varepsilon$};
		\node [style=differential] (104) at (6.5, -9) {{\bf \aquarius\!\aquarius\!\aquarius}};
		\node [style=object] (105) at (6.5, -9.75) {$A$};
		\node [style=duplicate] (106) at (6.5, -5.5) {$\chi$};
		\node [style=object] (107) at (6.5, -4.75) {$\oc (A \times B)$};
		\node [style=component] (108) at (7, -8) {$f^\ast$};
		\node [style=port] (109) at (5, -7.5) {$=$};
	\end{pgfonlayer}
	\begin{pgfonlayer}{edgelayer}
		\draw [style=wire, bend left=90, looseness=2.00] (53.center) to (55.center);
		\draw [style=wire] (55.center) to (56);
		\draw [style=wire, bend right=90, looseness=2.00] (54.center) to (57.center);
		\draw [style=wire] (62) to (63);
		\draw [style=wire, in=90, out=-150, looseness=1.25] (63) to (65);
		\draw [style=wire, in=90, out=-30, looseness=1.25] (63) to (64);
		\draw [style=wire, in=135, out=-90] (69) to (66);
		\draw [style=wire, in=-90, out=45, looseness=1.25] (66) to (68);
		\draw [style=wire] (53.center) to (68);
		\draw [style=wire] (65) to (69);
		\draw [style=wire] (64) to (57.center);
		\draw [style=wire] (70) to (72);
		\draw [style=wire, in=90, out=-150, looseness=1.25] (74) to (76);
		\draw [style=wire, in=135, out=-90] (76) to (70);
		\draw [style=wire, in=-90, out=45, looseness=1.25] (70) to (75);
		\draw [style=wire, in=90, out=-30, looseness=1.25] (74) to (75);
		\draw [style=wire] (66) to (74);
		\draw [style=wire] (72) to (54.center);
		\draw [style=wire, bend left=90, looseness=2.00] (77.center) to (79.center);
		\draw [style=wire] (79.center) to (80);
		\draw [style=wire, bend right=90, looseness=2.00] (78.center) to (81.center);
		\draw [style=wire] (82) to (83);
		\draw [style=wire, in=90, out=-150, looseness=1.25] (83) to (85);
		\draw [style=wire, in=90, out=-30, looseness=1.25] (83) to (84);
		\draw [style=wire] (84) to (81.center);
		\draw [style=wire] (89) to (90);
		\draw [style=wire] (90) to (78.center);
		\draw [style=wire, in=135, out=-90, looseness=1.25] (85) to (89);
		\draw [style=wire, in=30, out=-90, looseness=1.50] (77.center) to (89);
		\draw [style=wire, bend left=90, looseness=2.00] (92.center) to (94.center);
		\draw [style=wire] (94.center) to (95);
		\draw [style=wire, bend right=90, looseness=2.00] (93.center) to (96.center);
		\draw [style=wire] (97) to (98);
		\draw [style=wire, in=90, out=-150, looseness=1.25] (98) to (100);
		\draw [style=wire, in=90, out=-30, looseness=1.25] (98) to (99);
		\draw [style=wire, in=150, out=-90, looseness=0.75] (100) to (101);
		\draw [style=wire, in=30, out=-90, looseness=1.25] (92.center) to (101);
		\draw [style=wire] (101) to (93.center);
		\draw [style=wire] (96.center) to (102);
		\draw [style=wire] (99) to (102);
		\draw [style=wire] (107) to (106);
		\draw [style=wire, in=90, out=-30, looseness=1.25] (106) to (103);
		\draw [style=wire, in=150, out=-150] (106) to (104);
		\draw [style=wire] (104) to (105);
		\draw [style=wire] (103) to (108);
		\draw [style=wire, in=45, out=-90, looseness=0.75] (108) to (104);
	\end{pgfonlayer}
\end{tikzpicture}
   \end{array}
\end{align*}
So we conclude that the reverse differential combinator of $\mathbb{X}_\oc$ is induced by the reverse deriving transformation of $\mathbb{X}$. 
\end{proof}

Similarly to the differential combinator, the reverse differential combinator can also be expressed in terms of the coderiving transformation as follows on a coKleisli map $\llbracket f \rrbracket: \oc A \to B$: 
 \begin{align*} \llbracket \mathsf{R}[f] \rrbracket := \xymatrixcolsep{3pc}\xymatrix{\oc(A \times B) \ar[r]^-{\mathsf{d}^\circ_{A \times B}} & \oc (A \times B) \!\otimes\! (A \times B) \ar[r]^-{\oc(\pi_0) \otimes \pi_1} & \oc A \otimes B \ar[r]^-{1_{\oc A} \otimes \llbracket f \rrbracket^\ast} & \oc A \!\otimes\! \oc A \ar[r]^-{\mathsf{r}_A} & A
 }  \end{align*} 
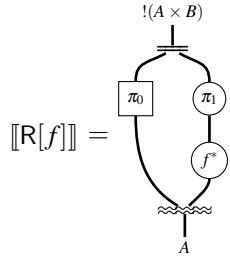
  \begin{align*} 
\begin{array}[c]{c}
\llbracket \mathsf{R}[f] \rrbracket 
   \end{array} =
   \begin{array}[c]{c}
\begin{tikzpicture}
	\begin{pgfonlayer}{nodelayer}
		\node [style=differential] (104) at (6.5, -9) {{\bf \aquarius\!\aquarius\!\aquarius}};
		\node [style=object] (105) at (6.5, -9.75) {$A$};
		\node [style=component] (108) at (7, -8) {$f^\ast$};
		\node [style=object] (117) at (6.25, -5) {$\oc (A \times B)$};
		\node [style=differential] (118) at (6.25, -5.75) {{\bf =\!=\!=\!=}};
		\node [style=component] (119) at (7, -6.75) {$\pi_1$};
		\node [style=function2] (120) at (5.5, -6.75) {$\pi_0$};
	\end{pgfonlayer}
	\begin{pgfonlayer}{edgelayer}
		\draw [style=wire] (104) to (105);
		\draw [style=wire, in=45, out=-90, looseness=0.75] (108) to (104);
		\draw [style=wire] (117) to (118);
		\draw [style=wire, in=90, out=-150, looseness=1.25] (118) to (120);
		\draw [style=wire, in=90, out=-30, looseness=1.25] (118) to (119);
		\draw [style=wire] (119) to (108);
		\draw [style=wire, in=150, out=-90] (120) to (104);
	\end{pgfonlayer}
\end{tikzpicture}
   \end{array}
\end{align*}

%We conclude this section by applying this construction to our main examples of reverse differential categories. 

\subsection{Other constructions of CRDCs}

The inconvenience of monoidal reverse differential categories is that the self-dual compact closed requirement is quite strong. Indeed, there are not many interesting or well-studied models of differential linear logic in the literature that are self-dual compact closed. In fact, from a linear logic perspective, such models are often considered somewhat ``degenerate'' \cite[Definition 3]{hyland2003glueing}. Therefore, examples of Cartesian reverse differential categories arising from monoidal reverse differential categories will often not appear naturally. There is, however, another way of construction Cartesian reverse differential categories from coKleisli categories. In particular, this slightly altered construction can be done with any monoidal differential category. Instead of requiring all objects in the base category to be self-dual, we can instead take the full subcategory of the coKleisli category of self-dual objects in the base category. The proof that this subcategory is a Cartesian reverse differential category is essentially the same as Theorem \ref{coKleisliCRDC}. Being a full subcategory of a Cartesian differential category that is closed under finite products implies that said subcategory is also a Cartesian differential category. Then using the self-duality, we can build a contextual linear dagger and we conclude that we have a Cartesian reverse differential category. 

\begin{definition} Let $(\oc, \delta, \varepsilon)$ be a comonad on a symmetric monoidal category $\mathbb{X}$. Define $\mathsf{R}[\mathbb{X}_\oc]$ as the full subcategory of the coKleisli category $\mathbb{X}_\oc$ whose objects are self-dual objects (Definition \ref{SDCC}) of $\mathbb{X}$, so triples $(A, \cup_A, \cap_A)$. Recall that by a full subcategory, the maps of $\mathsf{R}[\mathbb{X}_\oc]$ are all those of $\mathbb{X}_\oc$ between the underlying objects, that is, $\mathsf{R}[\mathbb{X}_\oc]\left( (A, \cup, \cap), (B, \cup, \cap) \right) = \mathbb{X}_\oc(A,B) = \mathbb{X}(\oc A, B)$, and both composition and identities are the same as in $\mathbb{X}_\oc$. 
\end{definition}

Suppose that the base symmetric monoidal category $\mathbb{X}$ has finite biproducts. The zero object is self-dual, where the cups and caps are simply the zero morphisms \cite[Lemma 3.19]{heunen2019categories}, and the biproduct of self-dual objects is again self-dual \cite[Lemma 3.23]{heunen2019categories}. Explicitly, if $(A, \cup_A, \cap_A)$ and $(B, \cup_B, \cap_B)$ are self-dual objects, then $(A \times B, \cup_{A \times B}, \cap_{A \times B})$ is also a self-dual object where the cup and cap are defined respectively as follows: 
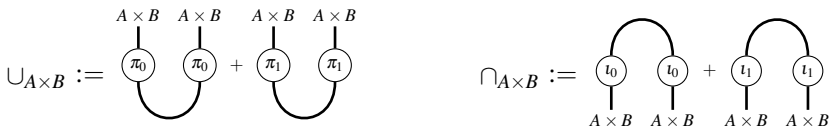
\begin{align*}
  \begin{array}[c]{c} \cup_{A \times B}    \end{array} : =    \begin{array}[c]{c}
\begin{tikzpicture}
	\begin{pgfonlayer}{nodelayer}
		\node [style=component] (0) at (7, -4.5) {$\pi_0$};
		\node [style=component] (1) at (8.25, -4.5) {$\pi_0$};
		\node [style=object] (4) at (9, -4.5) {$+$};
		\node [style=object] (5) at (7, -3.5) {$A \times B$};
		\node [style=object] (6) at (8.25, -3.5) {$A \times B$};
		\node [style=component] (7) at (9.75, -4.5) {$\pi_1$};
		\node [style=component] (8) at (11, -4.5) {$\pi_1$};
		\node [style=object] (9) at (9.75, -3.5) {$A \times B$};
		\node [style=object] (10) at (11, -3.5) {$A \times B$};
	\end{pgfonlayer}
	\begin{pgfonlayer}{edgelayer}
		\draw [style=wire, bend right=90, looseness=2.00] (0) to (1);
		\draw [style=wire] (5) to (0);
		\draw [style=wire] (6) to (1);
		\draw [style=wire, bend right=90, looseness=2.00] (7) to (8);
		\draw [style=wire] (9) to (7);
		\draw [style=wire] (10) to (8);
	\end{pgfonlayer}
\end{tikzpicture}
   \end{array} &&   \begin{array}[c]{c} \cap_{A \times B}    \end{array} : =    \begin{array}[c]{c}
\begin{tikzpicture}
	\begin{pgfonlayer}{nodelayer}
		\node [style=component] (0) at (7, -3.75) {$\iota_0$};
		\node [style=component] (1) at (8.25, -3.75) {$\iota_0$};
		\node [style=object] (4) at (9, -3.75) {$+$};
		\node [style=object] (5) at (7, -4.75) {$A \times B$};
		\node [style=object] (6) at (8.25, -4.75) {$A \times B$};
		\node [style=component] (7) at (9.75, -3.75) {$\iota_1$};
		\node [style=component] (8) at (11, -3.75) {$\iota_1$};
		\node [style=object] (9) at (9.75, -4.75) {$A \times B$};
		\node [style=object] (10) at (11, -4.75) {$A \times B$};
	\end{pgfonlayer}
	\begin{pgfonlayer}{edgelayer}
		\draw [style=wire, bend left=90, looseness=2.00] (0) to (1);
		\draw [style=wire] (5) to (0);
		\draw [style=wire] (6) to (1);
		\draw [style=wire, bend left=90, looseness=2.00] (7) to (8);
		\draw [style=wire] (9) to (7);
		\draw [style=wire] (10) to (8);
	\end{pgfonlayer}
\end{tikzpicture}
   \end{array}
\end{align*}
Therefore it follows that $\mathsf{R}[\mathbb{X}_\oc]$ has finite products. Note that any full subcategory of a Cartesian differential category whose objects are closed under finite products is again a Cartesian differential category. Therefore, if the starting base category $\mathbb{X}$ is a differential category, $\mathsf{R}[\mathbb{X}_\oc]$ will be a Cartesian differential category. 

\begin{lemma}  Let $\mathbb{X}$ be a differential category with finite biproducts. Then $\mathsf{R}[\mathbb{X}_\oc]$ is a Cartesian differential category where the differential combinator is defined as in Proposition \ref{coKleisliCDC}. 
\end{lemma}

Now that we have established that $\mathsf{R}[\mathbb{X}_\oc]$ is a Cartesian differential category, to show that it is also a Cartesian reverse differential category, it remains only to show that $\mathsf{R}[\mathbb{X}_\oc]$ has a contextual linear dagger. However, we may define the dagger in the same way that it was done in Corollary \ref{coKlielsicondag}. 

\begin{lemma} Let $\mathbb{X}$ be a differential category with finite biproducts. Then $\mathsf{R}[\mathbb{X}_\oc]$ is a Cartesian differential category with a contextual linear dagger, where the dagger is defined in the same way as in Corollary \ref{coKlielsicondag}.
\end{lemma}
\begin{proof} Using essentially the same proof as throughout Section \ref{sec:cokleisliRDC}, it follows from self-duality that we obtain a contextual linear dagger. 
\end{proof}

As a result, it follows that $\mathsf{R}[\mathbb{X}_\oc]$ is a Cartesian reverse differential category. Since the base category does not necessarily have a reverse deriving transformation, we will explicitly write the reverse differential combinator of $\mathsf{R}[\mathbb{X}_\oc]$ in terms of the deriving transformation and the cups and caps. 

\begin{proposition} Let $\mathbb{X}$ be a differential category with finite products. Then $\mathsf{R}[\mathbb{X}_\oc]$ is a Cartesian reverse differential category, where the reverse differential combinator is defined as follows for a coKleisli map ${\llbracket f \rrbracket: \oc A \to B}$: 
 \begin{align*}
    \begin{array}[c]{c}
 \llbracket \mathsf{R}[f] \rrbracket   \end{array}  :=      \begin{array}[c]{c}\xymatrixcolsep{4pc}\xymatrix{\oc(A \times B) \ar[r]^-{\chi_{A,B}} & \oc A \otimes \oc B \ar[r]^-{1_{\oc A} \otimes \varepsilon_B} & \oc A \otimes B \ar[r]^-{1_{\oc A} \otimes \cap_A \otimes 1_B} & \\
\oc A \otimes A \otimes A \otimes B \ar[r]^-{\mathsf{d}_A \otimes \sigma_{A,B}} & \oc A \otimes B \otimes A \ar[r]^-{\llbracket f \rrbracket \otimes 1_B \otimes 1_A} & B \otimes B \otimes A \ar[r]^-{\cup_B \otimes 1_A} & A 
 }
   \end{array} 
\end{align*}
\begin{align*}
     \begin{array}[c]{c}
\llbracket \mathsf{R}[f] \rrbracket 
   \end{array} =    \begin{array}[c]{c}
\begin{tikzpicture}
	\begin{pgfonlayer}{nodelayer}
		\node [style=none] (29) at (9.25, -6.75) {};
		\node [style=none] (30) at (8.5, -9.25) {};
		\node [style=none] (31) at (10.25, -6.75) {};
		\node [style=object] (32) at (10.25, -11.5) {$A$};
		\node [style=none] (33) at (9.75, -9.25) {};
		\node [style=object] (34) at (8.75, -3.75) {$\oc (A \times B)$};
		\node [style=component] (35) at (8.75, -4.5) {$\chi$};
		\node [style=component] (36) at (9.75, -5.5) {$\varepsilon$};
		\node [style=differential] (38) at (8.5, -8) {{\bf =\!=\!=\!=}};
		\node [style=component] (39) at (8.5, -8.75) {$f$};
	\end{pgfonlayer}
	\begin{pgfonlayer}{edgelayer}
		\draw [style=wire, bend left=90, looseness=2.00] (29.center) to (31.center);
		\draw [style=wire] (31.center) to (32);
		\draw [style=wire, bend right=90, looseness=2.00] (30.center) to (33.center);
		\draw [style=wire] (34) to (35);
		\draw [style=wire, in=90, out=-30, looseness=1.25] (35) to (36);
		\draw [style=wire] (36) to (33.center);
		\draw [style=wire] (38) to (39);
		\draw [style=wire] (39) to (30.center);
		\draw [style=wire, in=30, out=-90, looseness=1.50] (29.center) to (38);
		\draw [style=wire, in=150, out=-150] (35) to (38);
	\end{pgfonlayer}
\end{tikzpicture}
   \end{array}
\end{align*}
\end{proposition} 
\begin{proof} Since $\mathsf{R}[\mathbb{X}_\oc]$ is a Cartesian differential category with a contextual linear dagger, then by Theorem \ref{thm:characterization_of_crdc}, $\mathsf{R}[\mathbb{X}_\oc]$ is a Cartesian reverse differential category. By essentially the same calculations as in the proof of Theorem \ref{coKleisliCRDC}, we can show that the resulting reverse differential combinator is precisely the desired one.  
\end{proof}

We could have also expressed the reverse differential combinator in terms of the coderiving transformation as follows: 
 \begin{align*}
       \begin{array}[c]{c} \llbracket \mathsf{R}[f] \rrbracket    \end{array}  :=        \begin{array}[c]{c} \xymatrixcolsep{2.75pc}\xymatrix{\oc(A \times B) \ar[r]^-{\mathsf{d}^\circ_{A \times B}} & \oc (A \times B) \otimes (A \times B) \ar[r]^-{\oc(\pi_0) \times \pi_1} & \oc A \otimes B \ar[r]^-{1_{\oc A} \otimes \cap_A \otimes 1_B} & \\
\oc A \otimes A \otimes A \otimes B \ar[r]^-{\mathsf{d}_A \otimes \sigma_{A,B}} & \oc A \otimes B \otimes A \ar[r]^-{\llbracket f \rrbracket \otimes 1_B \otimes 1_A} & B \otimes B \otimes A \ar[r]^-{\cup_B \otimes 1_A} & A 
 }
   \end{array}
   \end{align*}
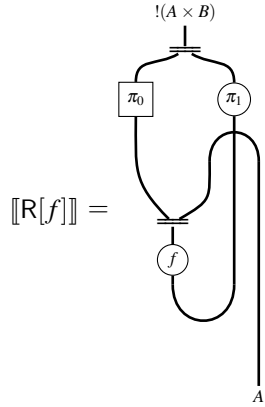
   \begin{align*}
       \begin{array}[c]{c}
\llbracket \mathsf{R}[f] \rrbracket 
   \end{array} =    \begin{array}[c]{c}
\begin{tikzpicture}
	\begin{pgfonlayer}{nodelayer}
		\node [style=none] (18) at (5.75, -6.75) {};
		\node [style=none] (19) at (5, -9.25) {};
		\node [style=none] (20) at (6.75, -6.75) {};
		\node [style=object] (21) at (6.75, -11.5) {$A$};
		\node [style=none] (22) at (6.25, -9.25) {};
		\node [style=object] (23) at (5.25, -3.75) {$\oc (A \times B)$};
		\node [style=differential] (24) at (5.25, -4.5) {{\bf =\!=\!=\!=}};
		\node [style=component] (25) at (6.25, -5.5) {$\pi_1$};
		\node [style=function2] (26) at (4.25, -5.5) {$\pi_0$};
		\node [style=differential] (27) at (5, -8) {{\bf =\!=\!=\!=}};
		\node [style=component] (28) at (5, -8.75) {$f$};
	\end{pgfonlayer}
	\begin{pgfonlayer}{edgelayer}
		\draw [style=wire, bend left=90, looseness=2.00] (18.center) to (20.center);
		\draw [style=wire] (20.center) to (21);
		\draw [style=wire, bend right=90, looseness=2.00] (19.center) to (22.center);
		\draw [style=wire] (23) to (24);
		\draw [style=wire, in=90, out=-150, looseness=1.25] (24) to (26);
		\draw [style=wire, in=90, out=-30, looseness=1.25] (24) to (25);
		\draw [style=wire] (25) to (22.center);
		\draw [style=wire] (27) to (28);
		\draw [style=wire] (28) to (19.center);
		\draw [style=wire, in=135, out=-90, looseness=1.25] (26) to (27);
		\draw [style=wire, in=30, out=-90, looseness=1.50] (18.center) to (27);
	\end{pgfonlayer}
\end{tikzpicture}
   \end{array} 
\end{align*}

As mentioned above, the advantage of this construction is that we construct a Cartesian reverse differential category from any differential category, with or without the Seely isomorphisms. In most cases, the self-dual objects of a differential category are of a ``finite-dimensional'' flavour. We conclude this section by applying this construction to well-known examples of differential categories to reconstruct some of the main examples of Cartesian reverse differential categories. 

\begin{example} \normalfont This example recaptures the reverse differentiation of polynomials from \cite[Example 14.1]{cockett_et_al:LIPIcs:2020:11661}. For simplicity, we will work with vector fields over a field but we note that this example can be generalized to the category of modules over any commutative semiring.  Let $k$ be a field and $\mathsf{VEC}_k$ be the category of $k$-vector spaces and $k$-linear maps between them. Then $\mathsf{VEC}^{\mathsf{op}}_k$ is a differential category where $\oc V$ is the free symmetric algebra over $V$: 
\[ \oc V = \bigoplus^{\infty}_{n=0} V^{\otimes_s^n}\]
where $\otimes_s^n$ is the $n$-fold symmetrized tensor product of $V$. If $X$ is a basis of $V$, then $\oc V \cong k[X]$, where $k[X]$ is the polynomial ring over $X$. From this point of view, the deriving transformation can be described as a map $\mathsf{d}_V: k[X] \to k[X] \otimes V$ which maps a polynomial to the sum of its partial derivatives: 
\[ \mathsf{d}(p(\vec x) = \sum \limits^n_{i=1} \frac{\partial p(\vec x)}{\partial x_i} \otimes x_i \]
Thus $\mathsf{VEC}^{\mathsf{op}}_{k}$ is a differential category, whose differential structure captures polynomial differentiation. For more details on this (co)differential category, see \cite[Section 2.5.3]{blute2006differential}. The self-dual objects in $\mathsf{VEC}^{\mathsf{op}}_k$ are precisely the finite-dimensional vector spaces, and since self-dual objects are self-dual, the same is true in $\mathsf{VEC}^{\mathsf{op}}_k$. Therefore, $\mathsf{R}[{\mathsf{VEC}^{\mathsf{op}}_k}_\oc]$ is equivalent, as Cartesian reverse differential categories, to $\mathsf{POLY}_k$ from \cite[Example 14.1]{cockett_et_al:LIPIcs:2020:11661}.
\end{example}

\begin{example} \normalfont This example recaptures the reverse differentiation of smooth functions from \cite[Example 14.2]{cockett_et_al:LIPIcs:2020:11661}. Let $\mathbb{R}$ be the field of real numbers. While the differential structure on $\mathsf{VEC}^{\mathsf{op}}_\mathbb{R}$ from the above example captures polynomial differentiation, $\mathsf{VEC}^{\mathsf{op}}_\mathbb{R}$ has another differential structure where this time the deriving transformation corresponds to differentiating (real) smooth functions. The key to this example is the notion of $C^\infty$-rings, which recall are defined as the algebras of the Lawvere theory whose morphisms are smooth maps between the Euclidean spaces $\mathbb{R}^n$. Equivalently, a $C^\infty$-ring is a set $A$ equipped with a family of functions ${\Phi_f: A^n \to A}$ indexed by the smooth functions $f: \mathbb{R}^n \to \mathbb{R}$ and which satisfies certain coherence equations. For example, $C^\infty(\mathbb{R}^n) = {\lbrace f: \mathbb{R}^n \to \mathbb{R} \vert~ f \text{ is smooth} \rbrace}$
 is a $C^\infty$-ring. For every $\mathbb{R}$-vector space $V$, there exists a free $C^\infty$-ring over $V$ \cite[Section 4]{cruttwell2019integral}, which we denote as $\mathsf{S}^\infty(V)$. If $V$ is finite dimensional of dimension $n$, then $\mathsf{S}^\infty(V) \cong C^\infty(\mathbb{R}^n)$ as $C^\infty$-rings, and  in particular, $\mathsf{S}^\infty(\mathbb{R}^n) = C^\infty(\mathbb{R}^n)$. Then $\mathsf{VEC}^{\mathsf{op}}_\mathbb{R}$ is a differential category with respect to the coalgebra modality $\mathsf{S}^\infty$ and whose deriving transformation is induced by differentiating smooth functions. In particular for $\mathbb{R}^n$, the deriving transformation $\mathsf{d}: C^\infty(\mathbb{R}^n) \to C^\infty(\mathbb{R}^n) \otimes \mathbb{R}^n$ maps a smooth function $f: \mathbb{R}^n \to \mathbb{R}$ to the sum of its partial derivatives: 
\[ \mathsf{d}(f) = \sum_{i=1}^{n} \frac{\partial f}{\partial x_i} \otimes x_i \] 
Hence $\mathsf{VEC}^{\mathsf{op}}_\mathbb{R}$ is a monoidal differential category, whose differential structure captures smooth function differentiation. For more details on this differential category, see \cite[]{cruttwell2019integral}. As explained in the above example, the self-dual objects of $\mathsf{VEC}^{\mathsf{op}}_\mathbb{R}$ are the finite-dimensional vector spaces. Therefore, $\mathsf{R}[{\mathsf{VEC}^{\mathsf{op}}_\mathbb{R}}_{\mathsf{S}^\infty}]$ is equivalent as a Cartesian reverse differential category to the example $\mathsf{SMOOTH}$ from \cite[Example 14.2]{cockett_et_al:LIPIcs:2020:11661}.
\end{example}

\section{Conclusions and Future Work}\label{sec:future_work}
In this paper we have filled in a gap in the literature on categorical differential structures by providing a definition of a \emph{monoidal reverse differential category}.  We have also provided key results to relate this structure to others, showing how monoidal reverse differential categories relate to monoidal differential categories, Cartesian differential categories, and Cartesian reverse differential categories.  This work provides many additional avenues for exploration; we briefly discuss some of them here.  
\begin{itemize}
    \item To understand what the structure of MRDCs should be, this paper started from an MDC and looked at what would happen if it's associated CDC was a CRDC.  However, there is another approach one could take.  In \cite[]{blute2015cartesian}, the authors look at what additional structure on a CDC would be necessary to form an MDC.  Thus, alternatively, one could start with a CRDC with such structure, and show that one gets an MRDC.  We leave this for future work.    
    \item In \cite[]{garner2021cartesian}, the authors describe how CDCs can be seen as a type of skew-enriched category, and use this result to demonstrate how every CDC embeds into a CDC associated to an MDC.  Similar results for CRDCS and MRDCs would be very useful.  
    \item In the world of ``reverse'' differential structures, the analog of tangent structures \cite[]{cockettCruttwellTangent} has yet to be described.  Such a structure would axiomatize the cotangent bundle in differential geometry.  Understanding such a structure's relationship to MRDC and CRDCs will then help further bridge the gap between differential geometry and differentiation in computer science.  
    \item All of the above items are theoretical; however, there is an important applied avenue which this work allows one to pursue.  As examples \ref{ex:quantum1} and \ref{ex:quantum2} demonstrated, several abstract models of quantum computation are MRDCs.  Then in particular, by Theorem \ref{coKleisliCRDC}, the coKleisli category associated to these models is a CRDC.  By the results of \cite[]{gradientBasedLearning}, this means that one could apply supervised learning techniques to these examples.  This possibility of combining quantum computation with supervised learning is an exciting direction we hope will be pursued in the future. 
\end{itemize}

\paragraph*{Acknowledgements:} For the work on this research project: Geoff Cruttwell was supported by an NSERC Discovery grant; Jonathan Gallagher was financially supported by an AARMS postdoctoral fellowship; Jean-Simon Pacaud Lemay was financially supported by an NSERC Postdoctoral Fellowship (PDF) - Award \#: 456414649; and Dorette Pronk was supported by an NSERC Discovery grant.

\bibliographystyle{msclike}      % mathematics and physical sciences
\bibliography{reversebib}   % name your BibTeX data base

\end{document}